\documentclass[english,reqno]{amsart}

\usepackage[margin=2cm]{geometry}
\usepackage{amsrefs}

\usepackage{color}
\usepackage{amsmath,amssymb, amsfonts, bbm}
\usepackage{mathtools}
\usepackage{verbatim}
\usepackage{enumerate}
\usepackage{enumerate}
\usepackage{comment}
\usepackage{esint}
\usepackage{multicol}
\usepackage{graphicx}

\usepackage{mathrsfs}
\usepackage[pdfstartpage=1,bookmarks=false,pdfstartview={FitH}]{hyperref}
\usepackage{cleveref}

\newcommand{\COMMENT}[1]{}
\DeclarePairedDelimiter\abs{\lvert}{\rvert}
\makeatletter
\let\oldabs\abs
\def\abs{\@ifstar{\oldabs}{\oldabs*}}
\makeatletter
\newcommand{\vast}{\bBigg@{4}}
\newcommand{\Vast}{\bBigg@{5}}
\makeatother
\newcommand{\ep}{\varepsilon}
\newcommand{\eps}{\varepsilon}

\newcommand{\N}{\mathbb{N}}
\newcommand{\R}{\mathbb{R}}
\newcommand{\Z}{\mathbb{Z}}
\newcommand{\bS}{\mathbb{S}}

\newcommand{\loc}{\textnormal{loc}}
\newcommand{\norm}[2][]{\left\|{#2}\right\|_{#1}}
\newcommand{\normsmall}[2][]{\|{#2}\|_{#1}}
\newcommand{\normbig}[2][]{\big\|{#2}\big\|_{#1}}

\newcommand{\set}[1]{\left\{#1\right\}}
\newcommand{\setsmall}[1]{\{#1\}}

\newcommand{\ceils}[1]{\lceil{#1}\rceil}
\newcommand{\floors}[1]{\lfloor{#1}\rfloor}

\newcommand{\dist}{{\rm dist}\, }

\newcommand{\cA}{\mathcal{A}}
\newcommand{\cB}{\mathcal{B}}
\newcommand{\ccB}{\mathfrak{B}}
\newcommand{\cC}{\mathcal{C}}

\newcommand{\cE}{\mathcal{E}}
\newcommand{\cF}{\mathcal{F}}
\newcommand{\cN}{\mathcal{N}}
\newcommand{\cG}{\mathcal{G}}
\newcommand{\cH}{\mathcal{H}}
\newcommand{\cI}{\mathcal{I}}
\newcommand{\ccI}{\mathfrak{I}}
\newcommand{\cJ}{\mathcal{J}}
\newcommand{\cK}{\mathcal{K}}

\newcommand{\cS}{\mathcal{S}}
\newcommand{\cT}{\mathcal{T}}

\newcommand{\cW}{\mathcal{W}}
\newcommand{\cX}{\mathcal{X}}
\newcommand{\cZ}{\mathcal{Z}}

\newcommand{\FB}{{\rm FB}}

\newcommand{\bc}{{\mathbf c}}
\newcommand{\rb}{r_\star}
\newcommand{\bN}{{\boldsymbol N}}
\newcommand{\bE}{{\boldsymbol E}}
\newcommand{\bA}{{\boldsymbol A}}

\newcommand{\dB}{d_\ccB}

\newcommand{\zz}{\mathbf{z}}

\newcommand{\epk}{{\widetilde \eps_k}}
\newcommand{\Rk}{{\widetilde R_k}}
\newcommand{\zk}{{\widetilde \zz_k}}
\newcommand{\rmin}{r_{{\rm min}} }

\newcommand{\cttg}{\gamma}
\newcommand{\ctta}{\alpha}
\newcommand{\cttb}{\beta_\circ}
\newcommand{\cttc}{\chi}
\newcommand{\valcttg}{\frac{11}{25}}
\newcommand{\valctta}{\frac{39}{50}}
\newcommand{\valcttb}{\frac{1}{20}}
\newcommand{\valcttc}{\frac{1}{500}}

\DeclareMathOperator{\Div}{div}

\DeclareMathOperator{\diag}{diag}

\DeclareMathOperator*{\argmax}{arg\,max}

\theoremstyle{plain}
\newtheorem{thm}{Theorem}[section]
\newtheorem{lem}[thm]{Lemma}
\newtheorem{cor}[thm]{Corollary}
\newtheorem{prop}[thm]{Proposition}

\newtheorem*{prop*}{Proposition}

\theoremstyle{definition}
\newtheorem{defn}[thm]{Definition}

\theoremstyle{remark}
\newtheorem{remark}[thm]{Remark}
\newcommand{\bremark}{\begin{remark} \em}
\newcommand{\eremark}{\end{remark} }

\numberwithin{equation}{section}

\definecolor{g1}{rgb}{0,0.5,0.1}
\definecolor{g2}{rgb}{0,0.6,0}
\definecolor{r2}{rgb}{0.8,0,0}

\title[
]{Global Stable Solutions to the Free Boundary Allen--Cahn and Bernoulli Problems in 3D are One-Dimensional}

\author{Hardy Chan}
\address{Universit\"at Basel, Spiegelgasse 1, 4051 Basel, 
Switzerland}
\email{hardy.chan@unibas.ch}

\author{Xavier Fern\'andez-Real}
\address{EPFL SB, Station 8, 1015 Lausanne, Switzerland}
\email{xavier.fernandez-real@epfl.ch}

\author{Alessio Figalli}
\address{ETH Zurich, R\"amistrasse 101, 8092 Zurich,
Switzerland}
\email{alessio.figalli@math.ethz.ch}

\author{Joaquim Serra}
\address{ETH Zurich, R\"amistrasse 101, 8092 Zurich,
Switzerland}
\email{joaquim.serra@math.ethz.ch}

\keywords{Stable solutions, Allen--Cahn equation, free boundary problems, Bernoulli problem, one-phase problem.}

\subjclass[2020]{35R35, 35J61, 35B35, 49Q20}

\begin{document}

\begin{abstract}
A long-standing conjecture of De Giorgi asserts that every monotone solution of the Allen--Cahn equation in \(\mathbb{R}^{n+1}\) is one-dimensional if \(n \leq 7\). A stronger version of the conjecture, also widely studied and often called ``the stable De Giorgi conjecture'', proposes that every stable solution in \(\mathbb{R}^n\) must be one-dimensional for \(n \leq 7\). To this date, both conjectures remain open for \(3 \leq n \leq 7\).  

An elegant variant of this problem, advocated by Caffarelli, C\'ordoba, and Jerison since the 1990s, considers a free boundary version of the Allen--Cahn equation. This variant features a step-like double-well potential, leading to multiple free boundaries. Locally, near each free boundary, the solution satisfies the Bernoulli free boundary problem. However, the interaction of the free boundaries causes the global behavior of the solution to resemble that of the Allen--Cahn equation.  

In this paper, we establish the validity of the stable De Giorgi conjecture in dimension \( 3\) for the free boundary Allen--Cahn equation and, as a corollary, we prove the corresponding De Giorgi conjecture for monotone solutions in dimension \(  4\).  
To obtain these results, a key aspect of our work is to address a classical open problem in free boundary theory of independent interest: the classification of global stable solutions to the one-phase Bernoulli problem in three dimensions. This result, which is the core of our paper, implies universal curvature estimates for local stable solutions to Bernoulli, and serves as a foundation for adapting some classical ideas from minimal surface theory---after significant refinements---to the free boundary Allen--Cahn equation.  
\end{abstract}

\maketitle

\setcounter{tocdepth}{1}
\tableofcontents

\section{Introduction}

The study of interfaces arising in nonlinear elliptic partial differential equations (PDEs) is a central theme in mathematical analysis, with significant implications for geometric analysis,  mathematical physics, and materials science. Interfaces---often also referred to as free boundaries, minimal surfaces, etc.---appear in models of phase transitions,  fluid dynamics, and other phenomena where different states or phases coexist and interact. 

A paradigmatic example of a PDE giving rise to interfaces is the classical Allen–Cahn equation. Originally proposed to describe phase separation in metal alloys \cite{AC72}, this equation has achieved mathematical prominence due to its profound connections with minimal surface theory (see, e.g., \cites{MM77, CC95, CM20}) and its close relation to several important phase field models—scalar, vectorial, or tensorial. Among the most closely related models are the Cahn–Hilliard equation, describing phase separation in binary fluids \cite{CH58}; the Peierls–Nabarro equation, modeling crystal dislocations \cites{Pei40, Nab47}; the Ginzburg–Landau theory, addressing superconductivity \cite{GL50}; and the Ericksen–Leslie model for liquid crystals \cites{Ericksen61, Leslie68}.

Interfaces also appear naturally in the study of free boundary problems such as the Bernoulli or one-phase problem. First studied from a mathematical viewpoint  by Alt and Caffarelli in 1981 \cite{AC81}, motivated by models in flame propagation and jet flows \cites{BL82, ACF82, ACF82b, ACF83}, it is related to shape optimization problems, capillary hypersurfaces, and minimal surfaces (among others; see, e.g., \cites{Buc12, KL18, KL19, CEL24, Tra14}). Also, it has been investigated as a one-phase problem, a two-phase problem, and in vectorial form \cites{ESV20, DSV21, FGKS24}.

During the past five decades, substantial progress has been made in understanding the structure of {\em absolute energy minimizers} for the Allen–Cahn equation, the Bernoulli problem, and related models. However, despite significant efforts, the structure of {\em stable solutions} remains largely elusive, even in the physically relevant three-dimensional space. Stable solutions are particularly important because they correspond to configurations observed in nature, representing physically stable states. Understanding stable solutions is thus a fundamental  challenging open question in the field.

In this paper, we introduce new analytical tools for studying three-dimensional stable solutions, focusing on two fundamental and deeply connected free boundary problems: the Allen--Cahn equation with a ``step potential'' and the Bernoulli problem.

\subsection{The variational model of phase transitions} The theory of minimal surfaces and phase transitions leads to considering energy functionals of the form
\[
    \cJ_{\varepsilon}(u;\Omega) = \int_\Omega \left\{ \varepsilon |\nabla u|^2 + \frac{1}{\varepsilon} W(u) \right\} \, dx,
\]
for $u\in H^1(\Omega)$, where $\varepsilon > 0$  is a small parameter, \( \Omega \subset \mathbb{R}^n \) is a bounded open set,  and $W :\mathbb R\to \mathbb R_+$ is a given (double well) potential. The function \( u \) is constrained (via the boundary datum) to satisfy \( -1 \leq u \leq 1 \), and the potential is such that $W(\pm 1) = 0$ and $W(t) > 0$ for $t\in (-1, 1)$.  

Prominent examples of such potentials \( W \) are given by the family of functions \((W_\alpha)_{0\le \alpha\le 2}\), 
\[ W_\alpha(u) :=
\begin{cases}
(1 - u^2)^\alpha &\qquad\text{for}\quad 0 < \alpha \le 2,\\
\mathbbm{1}_{(-1,1)}(u)& \qquad \text{for}\quad \alpha = 0,
\end{cases}\]
which give rise to the energy functionals
\begin{equation} \label{eq:energy-functional}
    \cJ^\alpha_{\varepsilon}(u;\Omega) := \int_\Omega \left\{ \varepsilon |\nabla u|^2 + \frac{1}{\varepsilon} W_\alpha(u) \right\} \, dx,\qquad\text{for}\quad \alpha\in [0, 2].
\end{equation}
This family of functionals was investigated by Caffarelli and C\'ordoba \cite{CC95} in what is considered one of the foundational papers in the Allen--Cahn literature.  The case \( \alpha = 2 \) corresponds to the classical Allen–Cahn energy \cite{AC72}---see, e.g.,  
\cite{Chodosh19} and references therein. The cases $\alpha\in [0,2)$ are considered, for example, in \cites{CC95, CC06, LWW17, Kam13, Val04, Val06, DGW22, Wan17} (and also mentioned in \cites{Sav09, Sav10, Savin17}).

\subsection{De Giorgi conjecture and its stable version} 

In 1978, De Giorgi stated the following celebrated conjecture \cite{DeG78}: 
\begin{center}
{\em Every solution
$ u \colon \mathbb{R}^{n+1} \to [-1,1] $ to the Allen--Cahn equation $-\Delta  u = u-u^3$\\
(equivalently, every critical point of the functional $\cJ_1^2$ defined in \eqref{eq:energy-functional})\\ that satisfies \( \partial_{n+1} u > 0 \) must be one-dimensional\footnote{That is, $u(x) = \phi(e\cdot x)$ for some $\phi:\R \to [-1,1]$ and $e\in \mathbb S^{n-1}$.}, at least for \( n \leq 7 \).}
\end{center}
This conjecture, often regarded as a PDE analogue of the classical Bernstein problem for minimal surfaces, has inspired substantial research over the past decades and has been resolved in certain cases: for \( n = 1 \) by Ghoussoub and Gui \cite{GG98}, and for \( n = 2 \) by Ambrosio and Cabré \cite{AC00} (see also \cite{AAC01}). In higher dimensions, Savin \cite{Sav09} proved the conjecture for \( n \leq 7 \) under the additional assumption that $u$ is an energy minimizer. For \( n \geq 8 \), del Pino, Kowalczyk, and Wei \cite{DKW11} constructed counterexamples showing that the conjecture fails in these dimensions.

It is well-known that solutions that are monotone in some direction are stable, namely, the second variation of $\cJ_1^2$ is non-negative (see \cite{AAC01}*{Corollary 4.3}).  Motivated by this fact, a stronger form of De Giorgi's conjecture---often called ``the stable De Giorgi conjecture''---asserts the following:
\begin{center}
{\em Every stable solution $ u \colon \mathbb{R}^{n} \to [-1,1] $  of the Allen–Cahn equation in \( \mathbb{R}^n \) must be one-dimensional for \( n \leq 7 \)}.
\end{center}
It is a known fact that the classical De Giorgi's conjecture for monotone solutions in $\R^{n+1}$ follows from its stable version in $\R^n$.\footnote{Indeed,
let $u:\R^{n+1}\to \R$ be a solution of the Allen--Cahn equation satisfying $\partial_{n+1}u>0$. First, since $u$ is stable (being monotone), one easily deduces that also the two functions \( u_{\pm}(y)= \lim_{x_{n+1} \to \pm\infty} u(y, x_{n+1}) \) are stable. Hence, if the conjecture for stable solutions is true in $\R^n$, then the two limits $u_{\pm}$ must be one-dimensional. One then checks that the one-dimensional solution is unique and increasing, so the functions $u_{\pm}$ are one-dimensional and increasing. This allows one to apply \cite[Theorem 1.3]{Jerison-Monneau} and deduce that $u$ is a minimizer, thus \cite{Sav09} applies and one concludes that also $u$ is one-dimensional.} 

The stable form of De Giorgi's conjecture has been proven only in dimension \( n = 2 \) in \cites{GG98, AC00}. For \( 3 \leq n \leq 7 \) it remains an open problem, while for \( n \geq 8 \) counterexamples exist \cite{PW13, LWW17b}. Again, Savin established the result for \( n \leq 7 \) under the additional assumption that \( u \) is an energy minimizer \cites{Sav09, Sav10, Savin17}.

It is worth emphasizing that both the De Giorgi's conjecture and its stable form, as well as the implication between them, are expected to hold for general double-well potentials. In fact, the majority of positive or partial results in the literature concerning either conjecture have been established directly in this more general setting.

Let us also mention that, in some applications (see, e.g., \cite{CM20}), it suffices to classify stable solutions to Allen--Cahn in \( \mathbb{R}^n \) satisfying the bounded energy growth condition  
\begin{equation}\label{Energygrowth}  
\sup_{R > 0} R^{1-n} \cJ_1^2(u, R) < +\infty.  
\end{equation}  
However, even under this additional assumption---which guarantees, using \cites{HutchinsonTonegawa00,TonegawaWickramasekera2012}, that blow-downs converge to hyperplanes with integer multiplicity in the appropriate sense\footnote{More precisely, 
Hutchinson and Tonegawa \cite{HutchinsonTonegawa00} showed that diffuse varifolds, constructed from the energy density and gradient direction of solutions to the Allen–Cahn equation, converge as $\eps\to 0$ to stationary integral varifolds, which generalize minimal surfaces and allow for singularities.}---the stable form of De Giorgi's conjecture has only been verified for \( n = 3 \) in \cite{AC00}.

\subsection{The connection to minimal surfaces}

Modica and Mortola \cite{MM77} rigorously established in 1977 the profound connection between phase transitions and minimal surfaces. They showed that, as \( \varepsilon \to 0 \),   minimizers of the energy \( \cJ_\varepsilon \) converge (in \( L^1_{\mathrm{loc}} \), up to subsequences) to the characteristic functions of sets with minimal perimeter.

Motivated by this result, De Giorgi proposed his conjecture in 1978 \cite{DeG78} as an analog of the classical Bernstein problem for area-minimizing graphs. Similarly, its stable version corresponds to the well-known problem of classifying complete, embedded, two-sided, stable minimal hypersurfaces in \( \mathbb{R}^n \) for \( n \leq 7 \), see \cites{DP79, FS80, Pog81,CL24, CL23, CMR24, CLMS24, Maz24}.

The influence of minimal surface theory is evident in many foundational developments in the study of the Allen–Cahn equation. Some examples are:
\begin{itemize}
    \item The Caffarelli–Córdoba density estimate for $ \{\cJ_\varepsilon^\alpha\}_{\alpha \in [0,2]}$ \cite{CC95} mirrors a similar property of minimal surfaces.
    \item The excess decay results of Savin \cite{Sav09} and Wang \cite{Wang17} for \( \cJ_\varepsilon^2 \) draw inspiration from the cornerstone theorems of De Giorgi and Allard in minimal surfaces.
    \item Modica’s monotonicity formula for the Allen–Cahn equation \cites{Modica85, Modica85b} is a clear analogue of Fleming’s monotonicity formula for minimal surfaces.
    \item The half-space theorems for the Allen–Cahn equation \cite{HLSWW21} are inspired by classical results for minimal surfaces.
\end{itemize}

It is worth emphasizing that these deep connections and analogies (of which the above points are just a few examples) between minimal surface theory and the class of  functionals \( \cJ_\varepsilon \) 
are valid for a very general class of double-well potentials $W$ that includes the family \(\{W_\alpha\}_{\alpha \in [0,2]} \). This fact has been well known to experts for some time (and has been confirmed in numerous works throughout the literature---see, e.g., \cites{CC95, CC06, LWW17, Kam13, Val04, Val06, Sav09, Sav10, Savin17, DGW22, Wan17}).
In particular, not only have techniques from minimal surface theory been adapted—often with significant modifications—to the study of phase transitions, but there are also striking instances where methods based on Allen–Cahn type equations have led to novel results in geometric analysis (see, for example, \cite{CM20,Mantoulidis21}).

\subsection{Recent progress and  challenges in stable phase transitions}

In recent years, significant progress has been made in the study of stable solutions to the Allen–Cahn equation: Wang and Wei \cite{WangWei18, WW19} and Chodosh and Mantoulidis \cite{CM20} have established key results on interface regularity and sheet separation estimates for stable solutions.

Even more recently, substantial advances have been achieved in understanding stable minimal hypersurfaces. The long-standing question of classifying complete minimal immersed hypersurfaces in \( \mathbb{R}^n \) for \( n \leq 7 \)—the analog of the stable De Giorgi conjecture—has been resolved in dimensions \( n \leq 6 \); see \cite{CL24, CL23, CMR24, CLMS24, Maz24}.

However, while the classification of complete stable minimal surfaces in \( \mathbb{R}^3 \) has been established since the 1980s \cites{DP79, FS80, Pog81}, the stable version of De Giorgi's conjecture in $\R^3$ remains unresolved.  
Put simply, despite significant progress and the development of sophisticated tools to study stable minimal surfaces and phase transitions, some deeper yet more rudimentary obstacles have prevented a proof of the conjecture in its stable form for decades. 
 
Even after Savin's breakthrough regularity results in 2009 \cite{Sav09}, which fully settled the case of \emph{minimizers} of $\cJ^\alpha_\varepsilon$ for \emph{all} $\alpha \in [0,2]$, the stable counterpart has remained elusive over the entire range of $\alpha$ (including the endpoint $\alpha = 2$). 
As we shall discuss now, addressing these challenges requires the creation of new techniques. 

A first deep heuristic obstacle is that while the large-scale behavior for \emph{absolute minimizers} of scale-dependent energies (which tend to the perimeter at large scales) mimics that of minimal surfaces, this correspondence does not need to hold for \emph{stable critical points}. As a concrete example, consider the functional
\begin{equation}
\label{eq:P eps}
\mathcal P_\varepsilon(E) 
\;:=\;
\int_{\partial E}
\bigl(1 + \varepsilon^2 \,|\mathrm{II}_{\partial E}|^2 \bigr)\,d\mathcal H^{2},\qquad E \subset \mathbb{R}^3,
\end{equation}
where $\mathrm{II}_{\partial E}$ denotes the second fundamental form of the boundary of $E$.
One can observe that $\mathcal P_\varepsilon$ behaves similarly to the perimeter on large scales (or, equivalently, as $\varepsilon \to 0$) and, in fact, it admits a Modica--Mortola-type $\Gamma$-convergence result. However, one can check that a catenoid of neck size $r>0$ is a stable critical point for this functional if and only if $r \le c \,\varepsilon$ (where $c$ is a universal constant). Hence, although the minimizers of $\mathcal{P}_\varepsilon$ do enjoy an $\varepsilon$-independent regularity theory (much like Savin’s theory for $\cJ^\alpha_\varepsilon$), stable critical points for $\mathcal{P}_\varepsilon$ do \emph{not} possess such uniform regularity; see \cite{Vassilev} for further discussion.

On a more technical level, whenever one attempts to adapt classification proofs from minimal surfaces to the Allen--Cahn setting, the following recurring difficulty arises: The elegant formulas and identities (e.g.\ Simons' identity, Gauss--Bonnet, etc.) that are fundamental in minimal surface theory:\\
- either do not admit  ``perturbative'' analogs for Allen--Cahn;\\
 - or, even in situations where they do, the usefulness of the ``generalized identities''  is far from clear.

These obstructions underscore why the classification of stable phase transitions remains both challenging and intriguing.
\subsection{The Allen--Cahn model with a step potential}

Motivated by the challenges described above, the primary objective of this work is to overcome, for the first time, the aforementioned barriers and introduce new methods and tools to prove a classification result for \emph{stable solutions} of a free boundary version of Allen--Cahn. Specifically, we consider stable critical points
$u \colon \mathbb{R}^3 \to [-1,1]$
of $\cJ^0_1$ (i.e., with the step potential $W_0$).
This corresponds to looking at solutions of the  free boundary problem
\begin{equation}\label{freebdryAC}
\left\{
\begin{aligned}
\Delta u \;&=\;  0 \quad &&\text{in } \{|u|<1\}\\
|\nabla u|\;&=\; 1 \quad &&\text{on } \partial\{|u|<1\}
\end{aligned}
\right.
\end{equation}
(corresponding to the first variation of the functional $\cJ_1^0$) that satisfy the stability inequality
\begin{equation}\label{freebdryAC-stability}
\int_{\{|u|<1\}}
\Bigl(
|D^2 u|^2 -|\nabla |\nabla u||^2\Bigr) \,\xi^2\,dx
\;\le\;
\int_{\{|u|<1\}}
|\nabla u|^2\, |\nabla \xi|^2\,dx
\quad
\text{for all }\;\xi \in C^\infty_c(\mathbb{R}^3)
\end{equation}
(see \Cref{defi:solutionsAC} and \Cref{lem:SZ_AC}).
For simplicity, to give a proper meaning to the equations above, we will assume that the free boundaries $\partial\{|u|<1\}$ are smooth surfaces and that $u$ is a classical solution of the PDE (namely, $u \in C^2(\{|u|<1\}) \cap C^1\big(\,\overline{\{|u|<1\}}\,\big)$).  However, these are mere qualitative assumptions (see also Remark~\ref{rmk:main thm}(iii) below).
 
\medskip

On a technical level, the case $\alpha = 0$ (and, more generally, when $0 \leq \alpha < 2$) has a substantial difference compared to the classical Allen--Cahn case $\alpha = 2$. Specifically, when $\alpha < 2$, the solutions satisfy a \emph{free boundary problem} rather than a global semilinear PDE.

That said, the primary difficulties in establishing the stable De Giorgi conjecture in $\mathbb{R}^3$ do not arise (at least not exclusively) from layer interactions, which are now relatively well understood \cite{Wang17, WangWei18,Chodosh19}. Indeed, in this context, the $\mathcal{P}_\varepsilon$ example in \eqref{eq:P eps} is particularly revealing: despite the absence of interactions between distinct sheets, stability remains compatible with the presence of small necks in solutions. This demonstrates that the main obstructions lie elsewhere.

One of the main purposes of this paper is to shed light on obstructions beyond sheet interactions and develop new techniques to tackle them. 
In this paper, we will focus on the case $\alpha=0$, but we believe that combining the ideas developed here with the techniques from \cite{Wang17, WangWei18, Chodosh19} will ultimately pave the way to addressing the case $\alpha = 2$.

\subsection{Microscopic necks: a new ``enemy.''}
Even if the free boundary formulation avoids certain layer-interaction issues, it gives rise to another profound difficulty:  In principle, two nearly flat free boundaries corresponding to $u=+1$ (or $u=-1$) could be joined by a \emph{microscopic neck} of size $r \ll 1$. Such a tiny neck contributes only a small amount (proportional to $r$) to the left-hand side of the stability inequality \eqref{freebdryAC-stability}; hence, having (possibly many) such microscopic necks might still be compatible with stability.

Concretely, suppose that inside $B_1\subset \mathbb{R}^3$ we have $\{u=1\}\cap B_1=\varnothing$.  Then the function $1 + u$ is a stable solution of the one-phase Bernoulli problem in $B_1$ (see below).  Yet it is known---see \cite{LWW21}---that certain global Bernoulli solutions (when rescaled) produce free boundaries with necks of arbitrarily small radius $r \ll 1$.  Existing examples of this type tend to be \emph{unstable} (albeit with finite Morse index, hence ``not too unstable''), but the question remains whether such ``microscopic-neck'' configurations \emph{could} ever be stable.

A significant portion (circa 80\%) of this paper is devoted to investigating these microscopic-neck configurations for the Bernoulli problem and proving that they must be necessarily unstable.  This is a delicate problem requiring refined PDE estimates and geometric arguments, which we will describe more thoroughly later.  In essence, the challenge is to show that any purportedly stable configuration with infinitely many small necks leads to contradictions with certain integral inequalities or regularity properties.

\subsection{The one-phase Bernoulli problem}

The one-phase Bernoulli free boundary problem arises from the study of the Alt--Caffarelli energy functional, namely, 
\[
 \cE(u; \Omega) = \int_\Omega\left\{ |\nabla u|^2 + \mathbbm{1}_{\{u>0\}}\right\}\, dx,
\]
where $\Omega\subset \R^n$ is a bounded open domain, and $u\in H^1(\Omega)$. Here, the function $u$ is constrained to satisfy $u \ge 0$.

First studied in 1981 by Alt and Caffarelli \cite{AC81}, the problem has received a lot of attention to date (see the monographs \cites{CS05, Vel23} for a nice introduction). Serving also as a model for semilinear PDEs, the study of the Bernoulli problem has gathered many tools and ideas from the theory of minimal surfaces, to the point where there is a formally established connection between the Bernoulli problem in dimension 2 and minimal surfaces in dimension 3 \cite{Tra14}.  This interplay highlights a geometric variational structure in the problem, bridging techniques from elliptic PDEs and geometric analysis.

In this direction, the regularity theory for free boundaries in minimizers mirrors that for minimal surfaces (with a shift in one dimension): a monotonicity formula, paired with a blow-up argument and an improvement of flatness, reduces the study of regular free boundaries to the classification of 1-homogeneous global solutions. The currently known results assert that minimizers have smooth free boundaries up to dimension 4 \cite{CJK04, Jerison-Savin}, while there are singular solutions in dimension 7 \cite{DJ09}. In dimensions 5 and 6 it remains as a challenging open problem.

Most of the theory for the one-phase problem has been developed for minimizers, such as the study of graphical solutions \cites{DJ09, EFY23}, the uniqueness of blow-ups at isolated singular points \cite{ESV20}, generic regularity \cites{FY24, FY25}, vectorial problems \cite{FGKS24}, etc. In recent years there has been a shift in trying to understand other solutions that do not necessarily arise as absolute minimizers: the rectifiability of free boundaries for stationary solutions \cite{KW24}, the nondegeneracy of stable solutions \cite{kamburov2022nondegeneracy}, the study of solutions in the plane \cite{Tra14, JK16, HHP11} and higher dimensions \cite{LWW21}, solutions with infinite topology \cites{BSS76, JK19}, etc. Even so, some fundamental questions remain open, with one of the most important being:
\[
\textit{Do global classical stable solutions of the Bernoulli problem in $\R^n$ have flat free boundaries, if $n\leq 6$? }
\]
It is well known that such a global rigidity result is equivalent to a local regularity property on curvature estimates for the free boundary of stable solutions (see, e.g., \cite{kamburov2022nondegeneracy} or the proof of \Cref{thm:main2} below). For $n = 2$, the answer to the previous question is affirmative thanks to a log cut-off argument that works for any semilinear PDE \cite{FV09, FR19, kamburov2022nondegeneracy}. 
For $n \ge 7$, the answer is negative by the recent construction in \cite{DJS22}.  One of the main results of the present paper is to positively answer the previous question for $n = 3$. The other dimensions remain a major open problem.

\subsection{Contributions of the paper}

Our first main result establishes the validity of the  stable De Giorgi conjecture for the functional $\cJ^0_1$ when $n=3$:
\begin{thm}
\label{thm:FBAC-main}
    Let $u: \R^3 \to [-1,1]$ be a  classical stable critical point of $\cJ^0_1$ (i.e., a global classical stable solution of \eqref{freebdryAC}, see \Cref{defi:solutionsAC}). Then $D^2u \equiv 0$ in $\{|u|<1\}$ and $u$ is one-dimensional.
\end{thm}

As a first corollary, we obtain the corresponding result for monotone solutions in dimension 4:
\begin{cor}
\label{cor:DeG_main}
 Let $u: \R^4 \to [-1,1]$ be a classical solution of \eqref{freebdryAC} (see \Cref{defi:solutionsAC}) satisfying $\partial_4 u>0$ in $\{|u|<1\}$. Then $D^2u \equiv 0$ in $\{|u|<1\}$ and $u$ is one-dimensional.  
\end{cor}

Combining \Cref{thm:FBAC-main} with the $\eps$-robust   $C^{1,1}$-to-$C^{2,\alpha}$  estimates for the level sets of solutions to \eqref{freebdryAC}---see \cite{An25}---we can also show the following:

\begin{cor}\label{cor:curvest}
Let $B_1\subset \R^3$ and  let $u_\varepsilon: B_1 \to [-1,1]$ be a classical stable critical point\footnote{Notice that $u_\eps$ is a classical stable critical point of $\cJ_\eps^0$ in $B_1$ if and only if $u_\eps(\eps\,\cdot\,)$ is a classical stable critical point of $\cJ_1^0$ in $B_{1/\eps}$.} of $\mathcal{J}_{\eps}^0$ in $B_1$, for $\eps> 0$ universally small. Assume that  $\partial\{|u_\eps|<1\}\cap B_{1/2}\neq \varnothing$.
Then, the principal curvatures of the level sets of $u_\eps$ inside $B_{1/2}$ are bounded by a universal constant.
\end{cor}

\begin{remark}
   As in the Allen--Cahn setting \cite{WW19, CM20}, once uniform curvature estimates for the level sets of the solutions $u_\varepsilon$ are established, the main result in \cite{An25} implies that their mean curvature goes to zero at an algebraic rate in $\varepsilon$. This reflects the natural expectation that the level sets approximate minimal surfaces in the limit as $\varepsilon\to 0$.
\end{remark}

As mentioned above, one of the most delicate estimates is to show that, thanks to stability,  the free boundaries satisfy universal curvature bounds. This can be phrased as a regularity result for the Bernoulli problem, which is of independent interest:

\begin{thm}
\label{thm:main1}
    Let $u:\R^3\to [0,\infty)$ be a classical  stable solution to the one-phase Bernoulli problem  (see \Cref{defi:solutions}). 
    Then $D^2u \equiv 0$ in $\{u>0\}$. In particular, the free boundary consists of either one or two hyperplanes.
\end{thm}

As a consequence, we obtain two corollaries. The first one is a classification of monotone solutions in $\R^4$:
\begin{cor}
\label{cor:DeG_Bernoulli}
 Let $u: \R^4 \to [0,\infty)$ be a classical solution to the one-phase Bernoulli problem (see \Cref{defi:solutions}) satisfying $\partial_4 u>0$ in $\{u>0\}$. Then $D^2u \equiv 0$ in $\{u>0\}$ and $u$ is one-dimensional.  
\end{cor}

The second are curvature estimates for local stable solutions to the Bernoulli problem:

\begin{cor}
\label{thm:main2}
    Let  $B_1\subset \R^3$ and let $u: B_1 \to [0,\infty)$ be a classical stable solution to the one-phase Bernoulli problem (see \Cref{defi:solutions})  with $0\in \partial\{u > 0\}$. 
    Then $|D^2u |\le C$ in $\overline {B_{1/2}\cap \{u>0\}}$, with $C$ universal. In particular, the principal curvatures of the free boundary are universally bounded. 
\end{cor}

Some comments are in order:
\begin{remark} 
\label{rmk:main thm}
{\em (i) Sharpness of the result:} For every $n \geq 2$, there exist classical solutions to the one-phase Bernoulli problem with catenoid-like free boundaries that have finite Morse index. In particular, these solutions are stable outside a compact set; see \cite{Tra14, LWW21}. Moreover, for $n = 2$, the Morse index has been shown to be exactly $1$ \cite{BK23}. In view of these examples, the stability assumption in Theorem~\ref{thm:main1} is necessary.

\noindent
{\em (ii) The role of $n=3$:} 
    Concerning Theorem~\ref{thm:main1}, the assumption $n=3$ is used to exploit a test function introduced by  Jerison and Savin in \cite{Jerison-Savin} to classify minimizing homogeneous solutions in $\R^4$. In view of this connection, it seems likely to us that if one could prove that minimizing homogeneous solutions in $\R^{k+1}$ are flat (which can be true only for $k\leq 5$), then our result could be extended to $\R^k$.
    Less crucially, the fact that $n = 3$ is also used in the classification of blow-downs in \Cref{prop:W}.\\
Regarding \Cref{thm:FBAC-main}, the dimensional assumption is used both to apply \Cref{thm:main2} and to exploit an argument from \cite{Pog81} based on Gauss--Bonnet. Still, if one could extend \Cref{thm:main1} (and thus \Cref{thm:main2}) to higher dimensions, in view of the recent breakthroughs in the classification of stable minimal surfaces \cite{CL24, CL23, CMR24, CLMS24, Maz24}, it seems plausible to us that one could attack \Cref{thm:FBAC-main} in higher dimensions as well.\\
\noindent
   {\em (iii) About the ``classical solution'' assumption:} 
   Since in $\R^3$ local minimizers (i.e., solutions that minimize the energy with respect to sufficiently small perturbations\footnote{That is, a function $u$ that minimizes the energy among all functions $v \in u  + H^1_0(B_1)$ with $\|v - u\|_{H^1(B_1)} < \delta$ for some $\delta > 0$.}) are classical stable solutions,
   Corollaries~\ref{cor:curvest} and \ref{thm:main2} apply to them. More generally, the curvature estimates from Corollaries~\ref{cor:curvest} and \ref{thm:main2}, as well as the classification results from Theorems~\ref{thm:FBAC-main} and \ref{thm:main1}, apply to all weak solutions that arise as local limits of classical stable solutions. This constitutes the largest class for which such results can be expected to hold\footnote{For instance, in the Bernoulli problem, the function $\mathbb{R}^2 \ni (x_1, x_2) \mapsto |x_1x_2|$ is a stationary critical point that is stable under domain variations. However, this is a spurious example. In particular, because of our results, 
it cannot be locally approximated by classical solutions, nor can it arise as a limit of solutions to a regularized problem  
$-\Delta u + \varepsilon F'(u/\varepsilon) = 0$,  
where $F$ is a mollified version of the indicator function of $(0, +\infty)$.}.
\end{remark}
 
\subsection{Structure of the paper}
The paper is organized as follows. 
To better guide the reader through the main arguments and techniques employed in the paper, in the next section we present a 
detailed overview of the key ideas and structure of the proofs of our main results.
Then, Sections \ref{sec:Bernoulli}--\ref{sec:closing} are dedicated to proving \Cref{thm:main1} and Corollaries~\ref{cor:DeG_Bernoulli} and \ref{thm:main2}, which form the backbone of our analysis. Building on these results, Section~\ref{sec:FB_AC} addresses the proofs of \Cref{thm:FBAC-main} and Corollaries~\ref{cor:DeG_main} and \ref{cor:curvest}, which crucially depend on \Cref{thm:main2}.

\subsection*{Acknowledgments}
H. C. was supported by the Swiss National Science Foundation (SNF grant
PZ00P2\_202012),  and by the Grant
CEX2019-000904-S funded by MCIN/AEI/10.13039/501100011033 and PID2020-113596GB-I00 (Spain).
X. F. was supported by the Swiss National Science Foundation (SNF grant
PZ00P2\_208930), by the Swiss State Secretariat for Education, Research and Innovation (SERI) under
contract number MB22.00034, and by the AEI project PID2021-125021NA-I00 (Spain).
A. F. is grateful to the Marvin V. and Beverly J. Mielke Fund for supporting his stay at IAS
Princeton, where part of this work has been done. 
J. S. was supported by the European Research Council under the Grant Agreement No 948029.

\section{Overview of the Proofs}

We now give a detailed overview of the structure of the proofs and the main ideas involved. The paper is structured following the logical dependencies of the results, and as such, the first part (Sections \ref{sec:Bernoulli}--\ref{sec:closing}) is devoted to showing \Cref{thm:main1} and Corollaries~\ref{cor:DeG_Bernoulli} and \ref{thm:main2}.

\subsection{The Bernoulli problem: Sections \ref{sec:Bernoulli}--\ref{sec:closing}} The main goal of these sections is to prove \Cref{thm:main1}. The key steps and main components of the proofs, along with their corresponding locations in the paper, are outlined below.

\subsubsection{The objects and their properties} \label{sssec:objects_properties} In Sections \ref{sec:Bernoulli}--\ref{sec:LpHessian}, we introduce the basic objects of interest and definitions, and all the necessary properties that will be used in a contradiction argument.

\begin{enumerate}[1.]

\item {\em A useful reduction}.  In Lemmas~\ref{lem:Lipbound} and~\ref{lem:reduction} we show that if there exists a classical stable solution $\bar u$ to the Bernoulli free boundary problem in $\R^3$ with $D^2 \bar u \not\equiv 0$, then there must exist a classical stable solution $u:\R^3\to [0,\infty)$ satisfying  the   bounds 
\[
|\nabla u| \le 1\quad \text{ in }  \R^3,\qquad |D^2 u|\leq 1
	\quad \text{ in }  \quad \{u>0\},
	\qquad 0 \in \FB(u),\qquad
|D^2 u(0)|=1,
\]
where $\FB(u)=\partial\{u>0\}$ denotes the free boundary of $u$
(cf. \cite{kamburov2022nondegeneracy}).  
Throughout the paper, we will assume that $u$ is as above and our goal will be to reach a contradiction.

\smallskip

\item {\em Preliminary results}. In \Cref{sec:Bernoulli} we start with some preliminary results on classical solutions to the Bernoulli problem: e.g., some variants of $\eps$-regularity (Lemmas \ref{lem:eps-reg-classical}, \ref{lem:L1Linfty_2}, and \ref{lem:eps-reg-Hess}) and a density estimate (\Cref{lem:cleanball}).

Then, in \Cref{sec:blowdown} we recall and establish some facts about classical \emph{stable} solutions. For example, it is well-known that the stability inequality (see  \eqref{eq:stab_ineq_H}--\eqref{eq:stab_ineq_H_2}) can be written in a {\em Sternberg--Zumbrun} form (as in \cite{SZ98}; cf. \Cref{lem:Stern_Zum} below). An immediate consequence is that, for all $y\in \R^3$ and $R>0$, we have
\[
R^2\fint_{B_R(y)\cap \{u>0\}} |D^2 u |^2\,dx
 \le  C,
\]
where $C>0$ is a universal constant. 

\smallskip

\item \noindent {\em Blow-down to vee}.  A first  consequence of stability is that the blow-down of a non-trivial global solution must be a \emph{vee} (\Cref{prop:W}). More precisely, we show that there exists a universal modulus of continuity $\omega$ for which the following holds:
If $y\in \FB(u)$ is such that $|D^2u(y)| \ge 1/\varrho >0$ ($\varrho \geq 1$ can be thought of as a radius of curvature), then for all $R\ge 1$ there exists $e =e(y,R)\in \mathbb S^{2}$
such that 
\[
\bigl\|
	u-V_{y,e}
\bigr\|_{L^\infty(B_R(y))}
\leq \omega(\varrho/ R)R,
\]
where $V_{y,e}$ is a {\em vee}, namely, a solution of the form
\[V_{y,e}(x) : = |e\cdot(x-y)|.\]

\smallskip

\item \label{it:ivabove} {\em Threshold radius, neck centers, neck radii, and ball tree}. Given $y\in \FB(u)$ we will define its associated {\em threshold radius} $\rb(y)$  as follows: 
\[
\rb(y) : = \sup\Big\{ r>0\ : \  \int_{B_{r}(x)\cap \{u>0\}}|D^2 u|^3\,dx <\eta_0^3\Big\},
\]
where $\eta_0>0$ will be a (fixed) small universal constant.

In Subsection~\ref{ssec:neck-set-def} we establish the existence of a discrete set $\cZ \subset \R^3$ (countable and locally finite), which we refer to as \emph{neck centers}, satisfying the following:  
\[
|D^2u(x)| \leq \frac{C}{\dist(x,\cZ)} \quad \forall\, x\in\{u>0\}, \qquad \text{and} \qquad |D^2u | \le \frac{C}{\rb(\zz)}
\quad \text{in } B_{\rb(\zz)}(\zz) \quad \forall\, \zz\in \cZ,
\]
see \Cref{lem:Hess-decay} and \Cref{cor:Hess-bd-neck}. The threshold radius at a neck center is called the \emph{neck radius}.   

The term ``neck'' is motivated by the following properties:
\begin{itemize}
    \item Away from necks, the free boundary consists of two (regular, nearly flat) disconnected sheets, and the positivity set $\{u>0\}$ consists of two disjoint connected components. (See \Cref{lemregball} for a precise statement.)
    \item Within a ball centered at a neck center with a radius comparable to the neck radius, the positivity set becomes connected through a \emph{neck-like} region. (See \Cref{lem:neckcenters_necks} for a precise statement.)
\end{itemize}
In addition, in \Cref{lem:X-blowdown} we show that, when centering at any given neck center and observing at a scale much larger than the neck radius, the solution $u$ becomes arbitrarily close to a vee. These structural properties imply that $\{u>0\}$ can be covered by a hierarchy of balls organized into a rooted tree structure, which we call the \emph{ball tree}. This covering consists of three types of balls:
\begin{itemize}
    \item \emph{Branching balls}: regions where the free boundary is concentrated within a thin slit, requiring further subdivision into smaller balls.
    \item \emph{Neck balls}: regions where the two disconnected positivity components merge, and the free boundary has a radius of curvature comparable to the ball radius.
    \item \emph{Regular balls}: regions where the free boundary has two regular nearly flat components.
\end{itemize}
See \Cref{fig:pictree} and \Cref{prop:geomtree} for further details.

\smallskip

\item{\em Symmetric $L^2$ excess}.  
For $\zz\in \cZ$ and $R>0$, we introduce the dimensionless quantity 
\begin{equation}
    \label{eq:Ez_def_intro}
\bE_\zz(u,R) : = 
\min_{e\in\bS^{2}}
 \sqrt{\frac{1}{R^2}\fint_{B_R(\zz)} |u- V_{\zz,e}|^2 \,dx}.
\end{equation}

\noindent {\em Small excess, small neck radii}: 
The  goal of Section~\ref{sec:LpHessian} (see \Cref{prop:E_to_rho}) is to show that neck radii are controlled by $\bE_\zz$. More precisely, we prove that for any  $\cttg\in (0, \frac49)$, $\zz\in \cZ$, and $R > 0$,  we have 
\begin{equation}\label{boundneckra2intro0}
\sup \bigg\{\frac{\rb(\zz')}{R} \ : \ \zz'\in \cZ\cap B_{3R/2}(\zz) \bigg\}  \le   C_\gamma \bE_\zz(u, 8R)^{3\cttg}.
\end{equation} 
Note that, by choosing $\cttg > \frac{1}{3}$, \eqref{boundneckra2intro0} implies that neck radii decay superlinearly with the symmetric excess---a crucial insight that lays the groundwork for the rest of the proof.

The proof of \eqref{boundneckra2intro0}  builds on the {\em Jerison--Savin}  test function in \cite{Jerison-Savin}: there exists a   1-homogeneous function $F$ of the Hessian such that $\bc= F(D^2u)^{1/3}$ is a subsolution of the linearized equation. In particular, 
\begin{equation}\label{eq:stab-intro}
\begin{split}
     \ccI(u,B_R(y)):&=\int_{B_R(y)\cap \{u>0\}} \bc\Delta \bc\,dx   
+ \int_{B_R(y)\cap \partial \{u>0\}} \bc(\bc_\nu + H\bc)\,d\cH^2 
 \le  \frac{C}{R^2} \int_{B_{2R}(y)\cap \{u>0\}} \bc^2 \,dx,
\end{split}
\end{equation}
where $H$ is the mean curvature of the free boundary, and all the integrands are non-negative. This motivates the definition of yet another dimensionless quantity: 
\begin{equation}
    \label{eq:varrho_def_intro}
\varrho_\zz(u,R) : =  
\frac{1}{R}\ccI(u, B_R(\zz))^3.
\end{equation}
From here, to establish \eqref{boundneckra2intro0}, we argue as follows:
\begin{enumerate}[(i)]
\item first, we bound the left-hand side of \eqref{boundneckra2intro0} by $\varrho_\zz(u,2R)$;
\item then, we bound $\varrho_\zz(u,2R)$ by the right-hand side in \eqref{boundneckra2intro0}.
\end{enumerate}
Step (i) is done in \Cref{prop:rb-l13}, by estimating  $\ccI\big(u,B_{\rb(\zz)}(\zz)\big)$ from below in a neck ball (\Cref{lem:I_neck_ball}). \\
Step (ii)
 is done in \Cref{prop:E_to_rho}, where the stability inequality \eqref{eq:stab-intro} is combined with a new ingredient: a local
$L^{\cttg'}$  estimate, with $\cttg'\in (0,1/2)$, for $D^2u$ in $B_{2R}(\zz)\cap \{u>0\}$ in terms of the excess $\bE_\zz(u, 4R)$ (see \Cref{lem:joaquim}).
This new delicate estimate strongly relies on the  ball tree structure described in point \ref{it:ivabove} above.
\end{enumerate}

\subsubsection{Sketch of the global contradiction argument}\label{subsection_sketch_of_proof_Bernoulli}  Very roughly, our strategy to prove \Cref{thm:main1} by contradiction in Sections~\ref{sec:selection}--\ref{sec:closing} can be summarized as follows: 
\begin{enumerate}[(i)]
    \item Assuming the neck set is non-empty, we pick a sequence of carefully chosen balls \(B_{R_k}(\zz_k)\) ---with $\zz_k$ neck centers and $R_k\to \infty$ as $k\to \infty$---  for which $\bE_{\zz_k}(u, 8R_k) =: \eps_k \downarrow 0$.
    \item By exploiting some special properties of the balls $B_{R_k}(\zz_k)$ we prove the existence of new centers $\zz'_k \in B_{R_k}(\zz_k)$ and scales $R'_k \ll R_k$ (as $k \to \infty$) such that
\begin{equation}
    \label{eq:chiintro}
\bE_{\zz'_k}(u, R'_k) \leq  \eps_k(R'_k/R_k)^\chi,\qquad \mbox{for some } \chi>0.
\end{equation}
\item Then, by using a new {\it Monneau--Weiss}-like approximate monotonicity formula with logarithmic errors, we show that the smallness of the excess in  $B_{R'_k}(\zz'_k)$ necessarily propagates to the larger ball, up to logarithmic errors: 
\begin{equation}\label{eq:iiiintro}
\eps_k = \bE_{\zz_k}(u,R_k) \leq C\bE_{\zz_k'}(u,4R_k) \le C\log(R_k/R_k')\bE_{\zz_k'}(u,R'_k) \leq C(R'_k/R_k)^\chi\log(R_k/R_k') \, \eps_k. 
\end{equation}
For $k$ sufficiently large this provides a contradiction, since $R'_k/R_k \to 0$ as $k \to \infty$.
\end{enumerate}

\vspace{5pt}
One of the cornerstones of the strategy outlined above is establishing a geometric excess decay between $B_{R_k}(\zz_k)$ and $B_{R_k'}(\zz_k')$. This decay is typically obtained through a linearization procedure, akin to those developed by De~Giorgi or Allard for minimal surfaces.
In this case, however---as explained further below---the situation is considerably more intricate: for a potentially large subset of center-scale pairs, the linearization approach may not be applicable.

Indeed, given a neck center $\zz$ and a scale $R > 0$, suppose we aim to improve the excess from $B_R(\zz)$ to $B_{R/4}(\zz)$. If the neck balls are very small (relative to $R$) and densely scattered throughout $B_R(\zz)$---a scenario that is entirely consistent with the estimate \eqref{boundneckra2intro0}---any attempt at linearization within this ball would be futile. A strategy completely different from linearization is required at these scales: we must harness instead the density of neck balls in a way that works to our advantage, enabling some form of improvement that can then be leveraged at smaller scales.

To address this challenge, we develop a new dichotomy-type argument that improves the excess for fundamentally different reasons at scales where linearization is possible and at scales where it is not. Interestingly, we are able to establish such ``dichotomic'' excess decay only around certain carefully selected neck centers; however, this suffices for our purposes.

\smallskip

We now describe in greater detail the main steps involved in the strategy described above.

\begin{enumerate}[1.]
\item[6.] {\em Careful selection of optimal center and scale}. Fix constants $\ctta\in (\frac 3 4,1)$ and $\cttg\in (0,\frac 4 9)$ such that $3\ctta \cttg>1$.
In Subsection~\ref{ssec:selection}, by suitably optimizing (with respect to $\zz$ and $R$) the quantity
\[
\frac{\bE_\zz(u, 8R)}{ \varrho_\zz(u,2R)^\ctta},
\]
we will show that there exist sequences $R_k\to \infty$ and $\zz_k\in \cZ$ such that
\begin{equation}
\label{eq:eps_kintro}
\eps_k : =  \bE_{\zz_k}(u,8R_k) \to 0\qquad \mbox{as } k\to \infty
\end{equation}
and for which, in addition, the  following crucial property holds:
\begin{equation}\label{key1intro}
\bE_\zz(u,8R) \le 2\eps_k \frac{\varrho_{\zz}(u,R)^\ctta }{\varrho_{\zz_k}(u,R_k)^\ctta } \qquad \mbox{for all $\zz\in \cZ$, $ R\le R_k$,}
\end{equation}
see \Cref{lem:zkRk}. It is worth emphasizing that the decay via dichotomy from the next steps crucially relies on the property \eqref{key1intro}, which holds only because of the careful selection of centers $\zz_k$ and scales   $R_k$. 

\smallskip

\item[7.] {\em Excess improvement when linearization is not possible}.
For \(\zeta \in (0, \tfrac14)\) and a given ball \(B_R(\zz) \subset \R^3\), we define the following quantity:
\begin{equation}
\label{eq:Ndefintro}
\bN\big(\zeta, B_R(\zz)\big) := (\zeta R)^{-3} \bigg|\bigcup_{\zz' \in \mathcal{A}^\zeta_{\zz, R}} B_{\zeta R}(\zz') \bigg|, \qquad \text{where} \quad \mathcal{A}^\zeta_{\zz, R} := \Big\{ \zz' \in B_R(\zz) \cap \cZ :  \rb(\zz') \leq \zeta R\Big\}.
\end{equation}
Essentially, $\bN(\zeta, B_R(\zz)\big)$ is the number of balls of radius $\zeta R$ needed to cover all neck centers with associated neck radius smaller than $\zeta R$. In some sense, the map $\zeta\mapsto \bN(\zeta, B_R(\zz)\big)$ quantifies the effective Minkowski dimension of the neck centers inside $B_R(\zz)$. This is essential because---as explained above---if neck centers are too densely scattered, any linearization attempt would be futile, and we need a different argument at that scale. 

Here a key idea is that---thanks to the careful selection of $B_{R_k}(\zz_k)$---whenever \( \bN(\zeta, B_{R_k}(\zz_k)) \) is ``too large'' for some resolution parameter \( \zeta \in (0,1) \), which would obstruct linearization, one can always find a smaller ball of radius \( \zeta R_k \) contained within $B_{R_k}(\zz_k)$ that still satisfies the key property \eqref{key1intro}, possibly in an even stronger form.

This observation allows us to show (see \Cref{lem:tildezktildeRk}) the existence of a new ball  $B_{\Rk}(\zk) \subset B_{R_k}(\zz_k) $, and $\epk \le\eps_k$, such that, for some $\cttb>0$ small:
\begin{equation}\label{wtiowhioh123}
    \begin{aligned}
        &\bN\big(\zeta, B_{\Rk}( \zk)\big) < C\zeta^{-\frac{1+\cttb}{3}} \qquad \mbox{for all }\zeta\in (0,\tfrac14),  \\
        &\bE_\zz(u,8R)  \le   2\Big(\tfrac{\Rk}{R}\Big)^\ctta    \widetilde\eps_k 
        \qquad \text{for all} \quad  \zz\in \cZ  \ \ \text{with}\ B_{R}(\zz)\subset B_{\Rk}(\zk)\text{ and }\ R\ge  \epk^{1/\alpha} \Rk.
    \end{aligned}
\end{equation}
Thanks to \eqref{wtiowhioh123}, we can show that linearization can be performed within $B_{\Rk}(\zk)$ (see point 8 below), therefore obtaining a geometric decay of the excess.

It is interesting to emphasize that this procedure to pass from $B_{R_k}(\zz_k)$ to $B_{\Rk}(\zk)$ involves a change in the center. This is non-standard in excess decay schemes, but it is necessary and effective for our purposes.

\smallskip

\item[8.] {\em Linearization regime: splitting of $\{u > 0\}$ and decay of an asymmetric excess.}
As already mentioned, thanks to \eqref{wtiowhioh123}, we can perform a linearization argument inside $B_{\Rk}(\zk)$ to improve the excess, and this can be iterated for a number of scales comparable to $|\log \epk|$.
However, for this linearization step, the symmetric excess is not the appropriate quantity to consider, and we will need to define an asymmetric $L^1$ excess as follows.

First, exploiting the tree structure described in point \ref{it:ivabove} above, in Subsection~\ref{ssec:Upm} we construct two disjoint open subsets $U_\pm$ such that
\[
U_\ast \subset U : =  \{ u>0\}\cap B_{\Rk}(\zk), \qquad \ast\in \{+, -\}.
\]
The sets $U_+$ and $U_-$ are essentially the two components (roughly half-spaces) in which the positivity set is split when removing all the necks (see \Cref{lem:obsUpm}).

Then, in \Cref{sec:linear} we define the asymmetric excess $\bA_{\zz}(u, R)$ for balls $B_R(\zz) \subset B_{\widetilde{R}_k}(\widetilde{\zz}_k)$ as
\begin{equation}
\label{eq:Az_def_intro}
\begin{split}
\bA_{\zz}(u, R) & := \max_{\ast\in \{+, -\}} \min_{a \in \mathbb{S}^2, \, b \in \mathbb{R}} \frac{1}{R |B_R(\zz)|} \int_{U_\ast \cap B_R(\zz)} \big| u(x) - a \cdot x - b \big| \, dx.
\end{split}
\end{equation}
Notice that, since we first minimize in $a$ and $b$, and only later compute the maximum in $\ast$, the `optimal' approximating planes that achieve the value of the asymmetric excess will have independent coefficients `on each side'.

Now, for fixed $\ast\in \{+, -\}$, it is natural to define
\[
v_\ast(x) := \frac{ u(x) -  a_\ast \cdot x -b_\ast }{\epk \Rk}, \qquad x\in U_\ast\cap B_{\Rk}(\zk).
\]
In \Cref{sec:linear} we prove that, thanks to \eqref{wtiowhioh123}, the function  $v_{\ast}$ is ``approximately'' a bounded weak solution  of 
\[
\mbox{
$\Delta w =0$ \quad in $B_{\Rk/2}(\zk) \cap\{a_\ast\cdot x+b_* >0\}$ \quad with \quad  $\partial_{a_\ast} w =0$ \quad on $B_{\Rk/2}(\zk) \cap\{a_\ast\cdot x+b_* =0\}$
}\] 
(see the proof of \Cref{prop:decaysqrt}). For this, the main challenge will be to prove estimates of the type 
\begin{equation}\label{wthuiwhiouwh}
\Rk^{p-3}\int_{U_\ast \cap B_{\Rk/2}(\zk)} |\nabla v_{\ast}|^p \,dx \lesssim 1  \quad \mbox{for some $p>1$}\qquad \mbox{and} \qquad \int_{\partial U_\ast \cap B_{\Rk/2}(\zk)} |(v_\ast)_\nu|^2 \,d\cH^{2}\ll1 
\end{equation}
(see \Cref{prop:neumann_bound} and \Cref{lem:hyp_comp}). 
Although from a very ``low-resolution'' perspective, this approach may appear similar to the classical linearization methods of Caffarelli and De Silva \cite{Caf89, DeSilva}, the key distinction is that, in our case, the difference in normal derivatives between solutions is small only in an \( L^p \) sense, rather than the usual \( L^\infty \) bound used in the viscosity approach. This \( L^p \) control is crucial because, at necks, the normal derivative is not small, leading to a large \( L^\infty \) norm. However, since the necks are relatively sparse and very small compared to the scale under consideration, the \( L^p \) approach remains effective. 

While the previous heuristic explanation justifies the use of $L^p$ topology, the actual proof of the linearization in this setup is much more subtle and requires utilizing all the properties of the ball $B_{\Rk}(\zz_k)$ described above---see Section \ref{sec:linear} for more details.
\smallskip

\item[9.] {\em Conclusion.} Thanks to the linearization step we establish \eqref{eq:chiintro} (see \Cref{prop:decaysqrt} and \Cref{lem:weiss_small}), and then we can conclude as in \eqref{eq:iiiintro} using the Monneau--Weiss-type monotonicity formula with a logarithmic error from \Cref{lem:Monneau}.

\end{enumerate}
\smallskip

Once  \Cref{thm:main1} is established, Corollaries~\ref{cor:DeG_Bernoulli} and \ref{thm:main2} follow (see Subsection~\ref{ssec:proofcorollaries}).
\subsection{The free boundary Allen--Cahn: Section~\ref{sec:FB_AC}}

Having now \Cref{thm:main2} at our disposal, we can proceed to describe the steps of the proof of \Cref{thm:FBAC-main}, which is done in Section~\ref{sec:FB_AC}. We argue by contradiction and assume that there exists a classical stable critical point of $\cJ^0_1$ in $\R^3$, denoted by $u$, that is not one-dimensional.
\begin{enumerate}[1.]
\item We start by recalling the Sternberg--Zumbrun inequality for stable solutions (\Cref{lem:SZ_AC}) and Modica's inequality  (\Cref{lem:Modica}), both in the context of the free boundary Allen--Cahn. Thanks to  \Cref{thm:main2}, we also observe that we have quantitative regularity in the set $\{|u| < 1\}$ (\Cref{lem:global-curv-bd}).  In particular, there are no `microscopic necks' in the free boundary.

\smallskip

\item We fix $\delta_\circ > 0$ small and  define the set $\cX(\delta_\circ)$ in \eqref{eq:Xdef}
 as those points $z\in \{|u| < 1\}$ for which the left-hand side in the Sternberg--Zumbrun inequality \eqref{eq:stab} is larger than $\delta_\circ$ in a ball $B_2(z)$. We want to show that, even if $\delta_\circ$ is chosen arbitrarily small,  $\cX(\delta_\circ)$ must be empty, which will directly yield our desired result.

To this aim, we  define $\cG(\delta_\circ)$ as the complement\footnote{In fact, we need to further divide 
$\{|u|<1\}\setminus \cX(\delta_\circ)$ into two sets, that we denote $\cG(\delta_\circ)$ and $\cW(\delta_\circ)$, according to whether the points are respectively close or far from the free boundary. Using stability (see \Cref{ssec:W}), an argument similar to the one described in point 3 here allows us to consider only the set $\cG(\delta_\circ)$ of points not in $\cX(\delta_\circ)$ that are close to the free boundary.} of $\cX(\delta_\circ)$ in $\{|u|< 1\}$, see \eqref{eq:Gdef}. In particular, it corresponds to points around which $u$ has flat level sets (in an $L^2$ sense). As a consequence, in \Cref{prop:good} we prove pointwise curvature  bounds of the solution $u$ around points in $\cG(\delta_\circ)$. 

\smallskip

\item The stability inequality in the form of Sternberg--Zumbrun also allows us to show that, for any given $\Lambda > 0$, there exist $z_\Lambda\in \cX(\delta_\circ)$ and $R_\Lambda > 1$ such that 
\begin{equation}
    \label{eq:cleanannulusintro}
\cX(\delta_\circ)\cap B_{R_\Lambda+\Lambda}(z_\Lambda)\subset B_{R_\Lambda}(z).
\end{equation}
That is, we can find arbitrarily thick annuli clean from $\cX(\delta_\circ)$ (even if their radius could be much larger than the thickness). This is done in \Cref{lem:freeannulus}.

\item For $\Lambda > 0$ fixed, given $z$ and $R > 0$ as in \eqref{eq:cleanannulusintro}, the curvature estimates in $\{|u| < 1\} \setminus \cX(\delta_\circ)$ ensure that the level sets are smooth submanifolds in this region. This makes it conceivable to test stability using a test function related to the intrinsic distance along the level sets $\{u = \lambda\}$ to $\cX(\delta_\circ) \cap B_R(z)$—an approach inspired by Pogorelov’s argument \cite{Pog81} for stable minimal surfaces in $\mathbb{R}^3$.

However, before proceeding in this direction, it is necessary to enlarge the set $\cX(\delta_\circ)$ by evolving it under the vector flow generated by $\nabla u$. Such a redefinition is possible due to the validity of curvature estimates at points on the boundary between $\cX(\delta_\circ)$ and $\cG(\delta_\circ)$. We denote the resulting enlarged set by $\mathfrak{B}$.

The intrinsic distance function along the level set $\{u = \lambda\}$ is then denoted by $\dB^\lambda$ (see \eqref{eq:dB-def}).

\smallskip

\item Using stability again,  by means of a cut-off function with gradient supported on suitably chosen dyadic scales, we show that for any $\lambda\in (-1, 1)$ there is some  $r\in (\Lambda^{1/4}, \Lambda/8)$ such that
\begin{equation}
\label{eq:asa}
\cH^2(\{u = \lambda\}\cap \{0 < \dB^\lambda < 2\}) \le \frac{C}{\delta_\circ |\log\Lambda|}\frac{\cH^2(\{u = \lambda\}\cap \{0 < \dB^\lambda < r\})}{r^2}
\end{equation}
(see Lemmas~\ref{lem:stab-on-bad} and \ref{lem:doubling-FBAC}). 
Moreover, we also obtain in \Cref{lem:doubling-FBAC} a precise doubling property.

\smallskip

\item In \Cref{prop:Xisempty} we conclude the proof of \Cref{thm:FBAC-main} as follows.
First, we use an integrated version of Gauss--Bonnet (see \Cref{lem:GB-rigor}) to obtain, roughly, that for any level set $\Sigma_\mu = \{u = \mu\}$,
  \[
  \cH^2(\Sigma_\mu \cap \{1<\dB^\mu<r\})\leq
\int_{1}^{r}\hspace{-1.5mm}\int_{1}^s\hspace{-1.5mm}  \hspace{0mm}\int_{\{1 <\dB^\mu<t\}}\hspace{-4mm}|K_{\Sigma_\mu}|\,d\cH^2\,dt\, ds
+Cr^2\cH^1(\Sigma_\mu \cap \{\dB^\mu=1\}),
\]
where $K_{\Sigma_\mu}$ is the Gauss curvature of $\Sigma_\mu$. From here one finds that, for all $r\in (1, \Lambda/8)$,
\[
\cH^2(\Sigma_\mu\cap \{0 <\dB^\mu<r\})
\leq
\frac{1}{4}\int_{\Sigma_\mu \setminus \cX(\delta_\circ)} |A_{\Sigma_\mu}|^2 (r-\dB^\mu)_+^2 \,d\cH^2
+Cr^2 \cH^2(\Sigma_\mu\cap \{0<\dB^\mu<2\}),
\]
where $|A_{\Sigma_\mu}|^2$ is the sum of the squares of the principal curvatures. 

Observe now that, due to the existence of a clean annulus  (see \eqref{eq:cleanannulusintro}), $(r-\dB^\mu)_+$ is an admissible test function for the stability inequality, which can be restricted to $\Sigma_\mu$ (up to a small multiplicative error) in view of the comparison across different level sets (see \Cref{lem:comp-dist}). 
Hence, thanks to stability, the co-area formula, and \eqref{eq:asa}, we find the existence of a level set $\{ u = \nu\}$ and $r\in (\Lambda^{1/4}, \Lambda/8)$ for which
\begin{equation} 
\begin{split}
    \frac{\cH^2(\Sigma_\nu\cap \set{0<\dB^\nu<r})}{r^2}
&\leq \frac{1+C\eta_\circ}{2}\cdot
    \frac{\cH^2(\Sigma_\nu \cap \set{0<\dB^\nu < r})}{r^2}+\frac{C}{\delta_\circ |\log\Lambda|}
    \frac{\cH^2(\Sigma_\nu \cap \set{0<\dB^\nu < r})}{r^2},
\end{split}
\end{equation}
where $\eta_\circ = o_{\delta_\circ}(1)$.  Choosing first $\eta_\circ$ small and then $\Lambda$ large, we  deduce that $\cH^2(\Sigma_\nu\cap \set{0<\dB^\nu<r}) = 0$, from which we easily get a contradiction.
\end{enumerate}

\smallskip

Again, once \Cref{thm:FBAC-main} is established, Corollaries~\ref{cor:DeG_main} and \ref{cor:curvest} follow (see Subsection~\ref{proofs:FBAC}).

\subsection{Notation}
\label{ssec:notation}
Throughout the paper, $C>1$ and $c\in(0,1)$ denote generic constants chosen conveniently large and small, respectively. Dependencies are denoted by subscripts or parentheses.

With $B_r(y)$ we denote the ball of radius $r>0$ centered at $y$. When $y=0$, we also write $B_r$ in place of $B_r(0)$. 
By $B_r(A)$ we denote the $r$-fattening of a set $A\subset \R^n$, namely $B_r(A):=\{x\in \R^n:\dist(x,A)<r\}$, which can also be seen as the Minkowski sum $B_r+A$.

Given three sets $A_1, A_2, A_3\subset \R^n$, we say 
\[
A_1 \subset A_2\qquad\text{in}\quad A_3\quad\stackrel{{\rm def}}{\Longleftrightarrow}\quad A_1 \cap A_3 \subset A_2\cap A_3. 
\]
We always assume that a modulus of continuity $\omega$ satisfies
\begin{equation}
\label{eq:w}
\omega:[0, +\infty)\to [0, +\infty)\ \text{is increasing and concave, with $\omega(t) \ge t$ for all $t\ge 0$}. 
\end{equation}
Given $y\in \R^n$ and $e\in \mathbb{S}^{n-1}$, we denote by $V_{y,e}$  a {\em vee}, namely, a function of the form 
\begin{equation}
    \label{eq:vee_def}
\R^n \ni x\mapsto V_{y,e}(x) : = |e\cdot(x-y)|.
\end{equation}
Given a ball $B_r(y)$, a unit vector $e\in \mathbb S^{2}$, and  $\eps\in (0,1)$, we define a \emph{slab} as
\begin{equation}
    \label{eq:Slab}
    {\rm Slab}\!\left(B_r(y), e, \eps\right)  : = \{x\in B_{r}(y) \ : \ |e\cdot(x-y)| \le \eps r\} = B_r(y) \cap \{V_{y, e}\le \eps r\}.
\end{equation}
We also denote by $\mathcal{H}^{k}$ the $k$-dimensional Hausdorff measure. 

Finally, let us collect, for the reader's convenience, some of the specific notations used throughout the work: 

\medskip

\begin{tabular}{ l p{13cm} }
\textbf{Symbol} & \textbf{Meaning/Location} \\
\hline
$u$ & Solution to the one-phase problem. \\
$\FB(u)$ & Free boundary of $u$, $\FB(u) = \partial\{ u > 0\}$. \\

$\nu$ & Inward unit normal to $\partial\{u > 0\}$. \\

$V_{y, e}$ & A \emph{vee} (blow-down model), see \eqref{eq:vee_def} above. \\

$\mathrm{Slab}(B_r(x),e,\eps)$ & Slab of thickness $\eps r$ around a plane orthogonal to $e$ (see \eqref{eq:Slab} above). \\

$\rb(x)$ & Threshold radius at $x\in \FB(u)$; see \eqref{eq:def_rB}. \\

$\cZ$ & Discrete set of \emph{neck centers}; see \eqref{eq:cZ_def}. \\

$\zz, \zz_k, \dots$ & Elements of $\cZ$. \\

 $\cN, \cN^{(\ell)}$ & All elements of the ball tree in $B_R(\zz)$ (as given in \Cref{defrootedtree} and \Cref{prop:geomtree}), and elements  of scale $\theta^\ell R$ as given by \Cref{prop:geomtree}, respectively.  \\

 $\cI, \cT, \cT^{\rm reg}, \cT^{\rm neck}$ & Subsets of the ball tree (as given in \Cref{defrootedtree} and \Cref{prop:geomtree}): respectively, \emph{internal nodes}, \emph{terminal nodes}, \emph{regular terminal}, and \emph{neck terminal} balls. See \Cref{defiregball}.  \\

 $B^{(\pm)}, B^{(\pm, \lambda)}, \Omega^{(\pm )}$ & Subdomains constructed in \Cref{def:Bpm_Omegapm}. \\

 $\ccI(u, U)$ & LHS in the stability inequality of Jerison--Savin, see \eqref{eq:def_I} (and also \eqref{eq:stab-intro} in the introduction). \\

 $\varrho_\zz(u, R)$ & Scale-invariant version of $\ccI(u, U)$ when $U = B_R(\zz)$; see \eqref{eq:varrho_def_body} (and \eqref{eq:varrho_def_intro}). \\

 $\bE_\zz(u, R)$ & Symmetric $L^2$-excess as defined in \eqref{eq:Ez_def_intro}.\\
  $\cZ_R$&  $\{\zz\in \cZ : \rb(\zz) \le R\}$, as defined in \eqref{eq:ZR}.  \\
 $\gamma, \alpha, \beta_\circ$ & Universal constants used from \Cref{sec:selection} onward, whose value is fixed in \eqref{eq:gamma_alpha_beta}. \\
$
\bN\big(\zeta, B_{  R}( \zz)\big)$  &   See \eqref{eq:Ndef}.
\\
 $\mathcal{A}^\zeta_{\zz, R}$ & $\cZ_{\zeta R} \cap B_{R}(\zz)$;  see \eqref{eq:Ndef}.\\
 $R_k, \zz_k, \eps_k$ & Initial sequences of radii, neck centers, and small values given by \Cref{lem:zkRk}.\\
 $\Rk, \zk, \epk$ & Refined sequences of radii, neck centers, and small values given by \Cref{lem:tildezktildeRk}.\\
$U_\pm, U_0$ & Sets of the two positive phases and removal, given by \Cref{def:UpUm}, . \\
$\bA_{\zz}(u, R)$ & Asymmetric $L^1$-excess in $B_R(\zz)$ as given by \eqref{eq:def_Az} (see also \eqref{eq:Az_def_intro}). \\
$\chi$ & Universal constant used in \Cref{sec:linear,sec:closing}, whose value is fixed in \eqref{eq:chi}. \\
$ R_k^\flat $ & Radius fixed in \Cref{prop:decaysqrt} as $\epk^\cttc \Rk$.\\
$ {\bf W}$, $ {\bf M}$, $ {\bf T}$ &  $ {\bf W}$ is the Weiss energy (introduced in \eqref{eq:W-def}). Quantities ${\bf M}$ and ${\bf T}$ introduced, respectively, in \eqref{eq:Monneau-def} and \eqref{eq:T-def}.\\ 
$\cX(\delta_\circ)$, $\cG(\delta_\circ)$, $\cW(\delta_\circ)$ & Sets defined in \eqref{eq:Xdef}, \eqref{eq:Gdef}, and \eqref{eq:Wdef_11},  respectively of curvature, flat transition and remainder. \\
$ \cS^*_\alpha(\delta_\circ) $, $\ccB_*(\delta_\circ)$ & Extensions of $\cX(\delta_\circ)$ defined in \eqref{eq:Sstardef}
  and \eqref{eq:Bstardef}, respectively. \\
  $ \cS_\alpha $, $\ccB$ & Local versions of $ \cS^*_\alpha(\delta_\circ) $, $\ccB_*(\delta_\circ)$ defined in \eqref{eq:cBdefAC}.\\
  $ \pi_\lambda(x)$ & Projection towards $\lambda$-level set as defined in \eqref{eq:pi lambda}. \\
  $ \dB^\lambda$ & Intrinsic distance to $\ccB$ along level sets; see \eqref{eq:dB-def}. \\
\hline
\end{tabular}

\section{The Bernoulli Problem: Preliminaries}
\label{sec:Bernoulli}
\subsection{The notions of solution}
\label{ssec:notion}
Given $u : B_R \to \R_+ : = [0,\infty)$ (where $B_R\subset \R^n$ denotes the  ball of radius $R>0$ centered at $0$), we define the Alt--Caffarelli energy functional by:
\[
\cE(u;B_R)
=\int_{B_R}
	\left\{|\nabla u|^2+\mathbbm{1}_{\set{u>0}}\right\}
\,dx.
\]
With this definition, critical points of $\cE$ solve the so-called one-phase Bernoulli problem.

In this paper we are interested in  \emph{classical solutions} of the Bernoulli problem: these are functions $u: B_R \to \R_+ $ such that 
\begin{equation}\label{eq:Bernoulli-main}
\text{$\{u > 0\}$ is locally a smooth domain in $B_R$}\qquad\text{and}\qquad
\begin{cases}
\Delta u=0
	& \text{ in } B_R \cap \set{u>0},\\
|\nabla u|=1
	& \text{ on } B_R \cap  \partial\set{u>0}.
\end{cases}
\end{equation}
The set $\partial\{u > 0\}$ is called the \emph{free boundary} and will also be denoted $\FB(u)$. In particular, a classical solution satisfies that $\{u > 0\}$ is locally the subgraph of a smooth function around each free boundary point (up to a rotation).

Classical solutions $u$ are \emph{stationary critical points} of $\cE$, that is, they satisfy 
\begin{equation}
\label{eq:stationary}
\frac{d}{dt}\bigg|_{t = 0} \cF(u\circ\Psi_t;B_R) = 0
	\quad \mbox{ for every $\Psi_t(x) := x+t \xi (x)$ with $\xi \in C^\infty_c(B_R; \R^n)$.}
\end{equation}
with $\cF = \cE$.  Stationary critical points $u$ are called \emph{stable} if they have non-negative second (inner) variations, i.e., they satisfy 
\begin{equation}\label{eq:stability-var}
\dfrac{d^2}{dt^2}\bigg|_{t=0} \cF(u\circ\Psi_t;B_R)\geq 0,
	\quad \mbox{ for every $\Psi_t(x) := x+t \xi (x)$ with $\xi \in C^\infty_c(B_R; \R^n)$,}
\end{equation}
for $\cF = \cE$.

In Sections \ref{sec:Bernoulli}--\ref{sec:closing}, a \emph{solution} will always refer to the one-phase Bernoulli (or Alt--Caffarelli)  problem. Moreover, we will distinguish among the following notions:

\begin{defn}
\label{defi:solutions}
Let $n \ge 2$ and $B_R\subset \R^n$ . In relation to the one-phase Bernoulli problem (i.e., taking $\cF = \cE$ in \eqref{eq:stationary}-\eqref{eq:stability-var}), we say that $u\in H^1(B_R)$ is:
\begin{itemize}
\item   a \emph{stationary solution} (or simply \emph{stationary}) in $B_R$ if satisfies \eqref{eq:stationary};
\item a \emph{classical solution} or a \emph{classical critical point} in $B_R$ if it satisfies \eqref{eq:Bernoulli-main} (in particular, it is stationary);
\item a \emph{stable solution} in $B_R$ if it is stationary and satisfies \eqref{eq:stability-var};
\item a \emph{classical stable solution} or \emph{classical stable critical point} in $B_R$ if it satisfies \eqref{eq:Bernoulli-main} and \eqref{eq:stability-var}.
\end{itemize}
If a function satisfies one of the previous definitions for all $R > 0$, we call it \emph{global}. 
\end{defn}

\subsection{Basic geometric properties of the free boundary}

We start by presenting some geometric properties of the free boundary for classical solutions to the Bernoulli problem. 

Before that, we recall the following well-known global boundedness of solutions (see, e.g., \cite[Proposition A.5]{kamburov2022nondegeneracy}):

\begin{lem}\label{lem:Lipbound}
Let $n \ge 2$ and let $u$ be a classical solution to the Bernoulli problem in $\R^n$. Then $|\nabla u|\le 1$ in~$\R^n$. 
\end{lem}

A first useful consequence of \Cref{lem:Lipbound} above is the following dimensional estimate for the area of the free boundary inside a ball: 

\begin{lem}\label{lem:perbound}
Let $n\ge 2$, and let $u$ be a global classical solution to the Bernoulli problem in $\R^n$. Then, we have that for any $R> 0$ and $y\in \{u=0\}\cap B_R$,
\[
\cH^{n-1} \left(  \FB(u)\cap B_\varrho(y)\right) \le C \varrho^{n-1}\qquad\text{for all}\quad \varrho \in (0, R/2),
\]
for some $C$ depending only on $n$.
\end{lem}
\begin{proof}
Since $\Delta u =  \cH^{n-1}|_{\FB(u)}$,
it suffices to consider a smooth non-negative cut-off function $\varphi_\rho \in C^\infty_c(B_{2\rho}(y))$ which satisfies $\varphi_\rho\equiv 1$ inside $B_\rho(y)$ and $|\nabla \varphi_\rho|\leq C\rho^{-1}$ to obtain (recall \Cref{lem:Lipbound})
$$
\cH^{n-1} \left(  \FB(u)\cap B_\varrho(y)\right) \leq \int \varphi_\rho\,\Delta u =-\int \nabla\varphi_\rho\cdot\nabla u\,dx
\leq C\rho^{-1}\|\nabla u\|_{L^\infty(\R^n)} |B_{2\rho}| \leq C \rho^{n-1},
$$
as desired.
\end{proof}

More generally, we have: 
\begin{lem}[Area of level sets]
\label{lem:area-level-sets}
Let $n \ge 2$, and let $u:\R^n\to [0, \infty)$ be $1$-Lipschitz, $\Delta u=0$ on $\set{u>0}$, and $|\nabla u|\geq c_\circ$ in some $\Omega\subset \R^n$. Then,
for any $B_R(x_\circ)\subset \R^n$ and $t>0$,
\[
c_\circ\cH^{n-1}\left(
	B_R(x_\circ)
	\cap \set{u=t} 
	\cap \Omega
\right)
\leq C R^{n-1},
\]
for some $C$ depending only on $n$.
\end{lem}

\begin{proof}
Integrating $\Delta u$ by parts inside $B_R(x_\circ) \cap \{u \ge t\}$ (notice that $u$ is harmonic there), we have 
\[
0 = \int_{B_R(x_\circ) \cap \{u \ge t\}} \Delta u\,dx
= \int_{\partial B_R(x_\circ) \cap \{u \ge t\}} \frac{x}{|x|}\cdot \nabla u \,d\cH^{n-1}
- \int_{\{u = t\}\cap B_R(x_\circ) }\partial_\nu u\,d\cH^{n-1},
\]
where $\nu$ is the normal unit vector to $\{u = t\}$ towards $\{u \ge t\}$ (so $\partial_\nu u=|\nabla u|$). Therefore,
\[
c_\circ\cH^{n-1}\left(
	B_R(x_\circ)
	\cap \set{u=t}
	\cap \Omega
\right)
\leq \int_{\{u = t\}\cap B_R(x_\circ) }
	\partial_\nu u
\,d\cH^{n-1}
=\int_{\partial B_R(x_\circ) \cap \{u \ge t\}} 
	\frac{x}{|x|}\cdot \nabla u 
\,d\cH^{n-1}
\leq C(n)R^{n-1},
\]
as we wanted.
\end{proof}

The following is a weak nondegeneracy property:

\begin{lem}[Clean ball property]\label{lem:cleanball}
Let $n\ge 2$. There exists $\eps_\circ=\eps_\circ(n)>0$ such that the following holds. 

Let $\varrho > 0$, $y\in \R^n$, and let $u$ be a   classical solution to the Bernoulli problem in $B_{2\varrho}(y)\subset \R^n$. 
Suppose that there is a connected component $U$ of $\{u>0\} \cap B_{2\varrho}(y)$ such that 
\[
|U \cap B_{2\varrho}(y)| \le \eps_\circ \varrho^n.
\]
Then,  $U \cap B_{\varrho}(y)= \varnothing$.
\end{lem}
\begin{proof}
Let $\bar u(x) : = \frac{1}{\varrho} (u {\mathbbm{1}}_{U})(y + \varrho x)$. Notice that $\bar u$ is a classical solution to the Bernoulli problem in $B_{2}$ with $\FB(\bar u)\cap B_1 \neq \varnothing$ for $\eps_\circ$ small,   with $|\nabla \bar u|\le C$ and $C$ depending only on $n$ (thanks to \cite[Lemma 11.19]{CS05}).
Therefore, for all $r\in (0,2)$, the divergence theorem applied to $\nabla\bar u$ inside the domain $B_r \cap \{\bar u>0\}$ gives
\begin{equation}\label{wnuibi1bu41biu1}
    \cH^{n-1} (\partial \{\bar u>0\}\cap B_r) \le  C\cH^{n-1} (\{\bar u>0\} \cap  \partial B_r).
\end{equation}
We can then use the argument in the classical proof of the density estimate for sets of minimal perimeter to conclude that $\{\bar u>0\}\cap B_{1}$ is empty, which is equivalent to $U  \cap B_{\varrho}(y)= \varnothing$.

More precisely, let $V(r)=|\{\bar u>0\}\cap B_r|$. Then, by coarea, $V(r) = \int_0^r \cH^{n-1} (  \{\bar u>0\}\cap \partial B_s )\,ds$. 
Combining the isoperimetric inequality in $\R^n$  (we denote by $c(n)$ the isoperimetric constant) with $\eqref{wnuibi1bu41biu1}$ this implies
\[
c(n) V(r)^{(n-1)/n} \le {\rm Per}(\{\bar u>0\}\cap B_r) \le (C+1) V'(r)
\]
for all $ r\in (0,2)$. Moreover,  by assumption, $V(2) \le \eps_\circ$. Then, a simple ODE analysis reveals that choosing $\eps_\circ$ small enough forces $V(1)=0$, that is,  $\{\bar u>0\}\cap B_{1} = \varnothing $.
\end{proof}
 \begin{remark}\label{rem:nondeg}
    The previous lemma is actually a nondegeneracy property of the positivity set for classical solutions. Namely, if $u$ is a classical solution to the Bernoulli problem and $x_\circ\in \{u > 0\}$, then by 
   \Cref{lem:cleanball}  applied to the connected component of $\{u > 0\}$ containing $x_\circ$ we have that, for any $r > 0$, 
    \[
    |\{u > 0\}\cap B_r(x_\circ)|\ge 2^{-n} \eps_\circ r^n.
    \]
\end{remark} 

\subsection{Regularity estimates for classical solutions to the Bernoulli problem}

In this section we present some basic regularity results for classical solutions. Several of these results actually hold for viscosity solutions, but we will not discuss this here.

The first result is a classical $\eps$-regularity estimate.

\begin{lem}[$\eps$-regularity]
\label{lem:eps-reg-classical}
Let $n\geq 2$. There exists $\eps_\circ=\eps_\circ(n) > 0$  such that the following holds.

Let $u$ be a classical solution to the Bernoulli problem in $B_1\subset \R^n$. If
\begin{equation}\label{eq:eps-reg-classical-flat}
\norm[L^\infty(B_1 \cap \set{u>0})]{u-x_n}\leq \eps\le \eps_{\circ},
\end{equation}
then, for any $k\in \N$, there exists $C_{n,k} > 0$, depending only on $n$ and $k$, such that
\[
\norm[C^k(B_{1/2} \cap \set{u>0})]{u-x_n}\leq C_{n,k}\eps\qquad \text{and}\qquad \text{$\FB(u)\cap B_{1/2}$ is a $C^k$ graph, with $C^k$-norm bounded by $C_{n,k}\eps$.}
\]
Moreover, $u$ is analytic in $B_{1/2} \cap \set{u>0}$.
\end{lem}
\begin{proof}
The regularity of the free boundary as well as the nonlinear bounds (i.e. without the dependence on $\eps$) on the $C^k$ norm of $u-x_n$ follow from the classical improvement of flatness and higher order regularity for the Bernoulli problem (see \cites{DeSilva,Kinderlehrer-Nirenberg}). The precise linear estimate (i.e., with the bound $C_{n,k}\eps$) stated here follows, e.g., from the recent results in \cite{Lian-Zhang} (see \Cref{prop:linear-estimate} in \Cref{app:linear-estimate}). Alternatively, see \cite[Proposition~5.1]{DJS22} combined with \Cref{lem:equivv}.\end{proof}

The next result states that a $C^2$ control follows from $L^1$-flatness. 
The main part of its proof is presented in  \Cref{app:linear-estimate}.

\begin{lem}[$L^1$ to $C^2$ estimate]
\label{lem:L1Linfty_2}
Let $n \ge 2$. There exists $\eps_\circ=\eps_\circ(n) > 0$ such that the following holds. 

Let $u$ be a classical solution to the Bernoulli problem in $B_1\subset \R^n$, with
\[
\int_{B_1\cap \{u > 0\}}|u -a\cdot x - b|\,dx \le \eps
\le \eps_{\circ}
\qquad\text{for some}\,\, a\in \mathbb{S}^{n-1},  \ b \in \R.
\]
Then  
\[
\|u - a\cdot x - b\|_{C^2(B_{1/2}\cap \{u > 0\})} \le C_n \eps,
\]
for some  $C_n$ depending only on $n$.
\end{lem}

\begin{proof}
Thanks to \Cref{prop:L1estimate}, $L^1$-flatness implies $L^\infty$-flatness, so the result follows from \Cref{thm:DeSilva-estimate} and \Cref{prop:linear-estimate}.
\end{proof}

In the next result, we show that a bound on the Hessian implies higher regularity as well, with estimates that are linear once the Hessian is bounded.

\begin{lem}[Higher regularity from the Hessian]
\label{lem:eps-reg-classical-2}
Let $n\geq 2$, and let $u$ be a classical solution to the Bernoulli problem in $B_1\subset \R^n$ satisfying
\begin{equation}\label{eq:Hess-C0}
\norm[L^\infty(B_1\cap\{u>0\})]{D^2u}\leq C_0 
\end{equation}
for some $C_0 > 0$. Then, for any $k \ge 2$,
\[
\normbig[L^\infty(B_{1/2} \cap \set{u>0})]{D^k u}
\leq C_{n,k} \max\{C_0^{k-2},1\}
	\norm[L^\infty(B_1\cap\{u>0\})]{D^2u},
\]
for some $C_{n,k}$ depending only on $n$ and $k$. 
\end{lem}

\begin{proof}
Let $\eps=\max\{C_0,1\}^{-1}\eps_{\circ}$, where $\eps_{\circ}$ comes from \Cref{lem:eps-reg-classical}, and let $x_\circ\in B_{1/2} \cap \overline{\set{u>0}}$. We separate into two cases:
\begin{itemize}   
    \setlength{\itemindent}{0pt}       
    \setlength{\leftskip}{-20pt}         
    \setlength{\labelwidth}{0pt}       
    \setlength{\labelsep}{5pt}   
  \item If $\dist(x_\circ,\FB(u))< \eps/4$, we let $y_\circ\in B_{3/4}\cap \FB(u)$ be the closest free boundary point to $x_\circ$, and we choose coordinates such that $e_n=\nabla u(y_\circ)$. Then $u_{y_\circ,\eps}=\frac{u(y_\circ+\eps\cdot)}{\eps}$ satisfies
\[
\norm[L^\infty(B_1 \cap \set{u_{y_\circ,\eps}>0})]{u_{y_\circ,\eps}-x_n}
\leq \norm[L^\infty(B_1 \cap \set{u_{y_\circ,\eps}>0})]{D^2 u_{y_\circ,\eps}}
= \eps\norm[L^\infty(B_\eps(y_\circ)\cap \set{u>0})]{D^2u}
\leq \eps_{\circ}.
\]
Thus \Cref{lem:eps-reg-classical} applies and gives
\[
\norm[C^k(B_{1/2}\cap \set{u_{y_\circ,\eps}>0})]{u_{y_\circ,\eps}-x_n}
\leq C_{n,k} \eps \norm[L^\infty(B_1\cap \set{u>0})]{D^2u}.
\]
In particular,
\[
|D^k u(x_\circ)|\hspace{-0.25mm}
\leq \hspace{-0.25mm}
\normbig[L^\infty(B_{\eps/2}(y_\circ)\cap \set{u>0})]{D^k u}
\hspace{-0.25mm}\leq\hspace{-0.25mm} \eps^{1-k} \norm[C^k(B_{1/2}\cap \set{u_{y_\circ,\eps}>0})]{u_{y_\circ,\eps}-x_n}
\hspace{-0.25mm}\leq \hspace{-0.25mm}C_{n,k} \eps^{2-k} \norm[L^\infty(B_1\cap \set{u>0})]{D^2u}.
\]

  \item If $\dist(x_\circ,\FB(u))\geq \eps/4$ or $\FB(u) = \varnothing$, the harmonicity of $D^2u$ in $B_{\eps/8}(x_\circ)$ gives the result. \qedhere
  \end{itemize}
\end{proof}

As a consequence, we also obtain linear bounds with respect to the $L^1$ norm of the Hessian, once it is bounded:

\begin{cor}[$\dot{W}^{2,1}$ controls $\dot{W}^{2,\infty}$]
\label{cor:Hess-W21}
Let $n\ge 2$, and let $u$ be a classical solution to the Bernoulli problem in $B_1\subset \R^n$ satisfying \eqref{eq:Hess-C0}. Then, 
\[
\normbig[L^\infty(B_{1/2} \cap \set{u>0})]{D^2 u}
\leq C_n
	\max\left\{
		C_0^n,1
	\right\}
	\normbig[L^1(B_1 \cap \set{u>0})]{D^2 u},
\]
for some $C_n$ depending only on $n$. 
\end{cor}

\begin{proof}
Combining the interpolation estimates from \Cref{lem:L1_Lip} with the regularity estimate in \Cref{lem:eps-reg-classical-2}, we get
\begin{align*}
\norm[L^\infty(B_{1/2}\cap \set{u>0})]{D^2u}^{n+1}
&\leq C\norm[L^1(B_{1/2} \cap \set{u>0})]{D^2u}
	\norm[L^\infty(B_{1/2} \cap \set{u>0})]{D^3u}^{n}+C\norm[L^1(B_{1/2} \cap \set{u>0})]{D^2u}^{n+1}\\
&\leq C\norm[L^1(B_{1/2} \cap \set{u>0})]{D^2u}
	\max\left\{
		C_0^n,1
	\right\}
	\norm[L^\infty(B_1 \cap \set{u>0})]{D^2u}^{n}+C\norm[L^1(B_{1/2} \cap \set{u>0})]{D^2u}^{n+1}.
\end{align*}
Applying this estimate to the rescalings $u_{z,r}=\frac{u(z+r\cdot)}{r}$ for $B_r(z)\subset B_1$, it follows that for any $\delta\in(0,1)$ there is $C_\delta>0$ such that
\[
r^n\norm[L^\infty(B_{r/2}(z)\cap \set{u>0})]{D^2u}
\leq C_\delta
	\max\left\{
		C_0^n,1
	\right\}
	\norm[L^1(B_1 \cap \set{u>0})]{D^2u}
	+\delta r^n\norm[L^\infty(B_r(z) \cap \set{u>0})]{D^2u}.
\]
By a standard covering argument (e.g. \cite[Lemma 2.27]{FR22}), the result follows.
\end{proof}

The following lemma provides an $\eps$-regularity result for solutions that are small in $\dot{W}^{2,n}$.

\begin{lem}[$\eps$-regularity for the Hessian] 
\label{lem:eps-reg-Hess}
Let $n\ge 2$. There exists $\eta_*=\eta_*(n)>0$ such that, for all $\eta\le \eta_*$, the following holds. 

Let $u$  be a classical solution to the Bernoulli problem in $B_1\subset \R^n$. Then
\begin{equation}\label{eq:eps_reg_Hess}
	\int_{B_1 \cap \set{u>0}}|D^2u|^n\,dx \leq \eta^n \quad \implies \quad \norm[L^\infty(B_{1/2} \cap \set{u>0})]{D^2u}
\leq C_n \eta,
\end{equation}
for some $C_n$ depending only on $n$. More generally, for $r>0$ and $k\ge 2$,  we have
\begin{equation}
\label{eq:eps_reg_Hess_k}
\int_{B_r\cap \set{u>0}}
	|D^2u|^n
\,dx
\leq \eta^n
	\quad \implies \quad
\normbig[L^\infty(B_{r/2}\cap \set{u>0})]{D^k u}
\leq \frac{C_{n,k}\eta}{r^{k-1}},
\end{equation}
for some $C_{n,k}$ depending only on $n$ and $k$. 
\end{lem}

\begin{proof}
We first show \eqref{eq:eps_reg_Hess}. Let
\[
x_0\in \argmax_{B_1\cap \overline{\{u > 0\}}}(1-|x|)|D^2u(x)|,\qquad r_0:=1-|x_0|, \qquad L_0:=|D^2u(x_0)|, 
\]
and we suppose by contradiction that, for some $C_*$ to be chosen later, 
\begin{equation}\label{eq:eps-reg-contra}
r_0L_0 > C_* \eta.
\end{equation}
Consider now 
$
v(y):=L_0 u(x_0+L_0^{-1}y).
$
Then $v$ is a classical solution to the Bernoulli problem in its domain. Also, since $|x_0+L_0^{-1}y| \leq |x_0|+\frac{r_0}{2}=1-\frac{r_0}{2}$ for $|y|\leq \frac{r_0L_0}{2}$, it follows from the definition of $x_0$ that
\[
|D^2v(y)|
=
	L_0^{-1}
	\frac{
		(1-|x_0+L_0^{-1}y|)|D^2u(x_0+L_0^{-1}y)|
	}{
		1-|x_0+L_0^{-1}y|
	}
\leq
	L_0^{-1}
	\frac{
		r_0L_0
	}{
		r_0/2
	}
= 2,\qquad\text{for}\quad y \in B_{r_0 L_0/2} \cap \overline{\{v > 0\}}. 
\]
This implies that the curvature of the free boundary of $v$ is universally bounded inside $B_{r_0 L_0/2}$. Since $0\in \overline{\{v >0 \}}$ and $|\nabla v| = 1$ on $\partial\{v > 0\}$, there exist a point $\bar y_0$ and a dimensional constant $c_n$ such that $0\in B_{c_n r_0 L_0}(\bar y_0)\subset B_{r_0 L_0/2} \cap \{v > 0\}$. Thus  $x_0 \in B_{c_n r_0}(\bar x_0)\subset B_{r_0/2}(x_0)\cap \{u > 0\}$ with $\bar x_0 = x_0+L_0^{-1}\bar y_0$. In particular, we can apply \Cref{lem:eps-reg-classical-2} with $k=3$ to the function $r_0^{-1} u(x_0+r_0x)$ to deduce that
\[
|D^3 u (x) |\leq C_{n,3} \frac{L_0}{r_0}\max\{1, L_0 r_0\} ,\qquad\text{for}\quad x\in  B_{c_n r_0}(\bar x_0).
\]
This implies that there exists a constant $c_*=c_*(n)>0$ such that 
\[
|D^2 u (x) |\ge L_0 - C_{n,3} \frac{L_0}{r_0}\max\{1, L_0 r_0\} |x-x_0| \ge \frac{L_0}{2},\qquad\text{for}\quad x\in  B_{c_n r_0}(\bar x_0)\cap B_{c_*\min\{L_0^{-1}, r_0\}}(x_0),
\]
therefore
\[
\eta^n \ge \int_{B_{c_n r_0}(\bar x_0)\cap B_{c_*\min\{L_0^{-1}, r_0\}}(x_0)} |D^2 u|^n \,dx \ge 2^{-n} |B_{c_n r_0}(\bar x_0)\cap B_{c_*\min\{L_0^{-1}, r_0\}}(x_0)|L_0^n. 
\]
Noticing now that $|B_{c_n r_0}(\bar x_0)\cap B_{c_*\min\{L_0^{-1}, r_0\}}(x_0)|\ge c \min\{L_0^{-n}, r_0^n\}\ge c L_0^{-n} \min\{1, C_*^n \eta^n\}$ (recall \eqref{eq:eps-reg-contra}), we obtain
\[
\eta^n \ge \hat c \min\{1, C_*^n \eta^n\},
\]
for some dimensional constant $\hat c=\hat c(n)>0$. However, choosing $C_*$ large enough so that $\hat cC_*^n\geq 2$, this inequality is impossible if $\eta$ is small enough. Thus \eqref{eq:eps-reg-contra} does not hold, and we obtain \eqref{eq:eps_reg_Hess}.

Rescaling by a factor of $r$, we get \eqref{eq:eps_reg_Hess_k} with $k=2$. Finally, \Cref{lem:eps-reg-classical-2} (together with a covering argument) yields \eqref{eq:eps_reg_Hess_k} for all $k\geq 3$.
\end{proof}

\subsection{Structural results for classical solutions} 
Here, we present the mean convexity of the free boundary and the regularity of solutions close to a vee.

\begin{lem}
\label{lem:vbounds}
Let $n\ge 2$, and let $u$ be a global classical solution to the Bernoulli problem in $\R^n$. Let $\nu$ denote the inward unit normal vector to $\partial\{u > 0\}$ at a given point, and let 
\begin{equation}
\label{eq:def_v}
v(x) := 1-|\nabla u(x)|^2\qquad\text{for}\quad x\in \R^n. 
\end{equation}
Then $v$ satisfies $0 \le v \le 1$ in $\R^n$, $v = 0$ on $\partial\{u > 0\}$, and 
\[
\left\{
\begin{array}{rcll}
\Delta v & \le & 0,&\qquad\text{in}\quad \{u > 0\},\\
\partial_\nu v & = & -2\partial_{\nu\nu}^2 u&\qquad\text{on}\quad \FB(u).
\end{array}
\right.
\]
In particular, whenever $u$ is not a half-space solution $(x\cdot e)_+$ or a vee $|x\cdot e|$ for some $e\in \mathbb{S}^{n-1}$, then $H = \tfrac12 \partial_\nu v > 0$ on $\FB(u)$, where $H$ denotes the mean curvature of $\FB(u)$ at a given point with respect to the outer unit normal $-\nu$.
\end{lem}
\begin{proof}
The bound $0 \leq v\leq 1$ comes from \Cref{lem:Lipbound}. Also, since $\Delta u(x) = 0$ in $\{u > 0 \}$, a simple computation yields
\[
\Delta v(x) = -2\, {\rm div} (D^2 u(x) \nabla u(x)) = -2|D^2 u(x)|^2\le 0 \qquad\text{for}\quad x\in \{u > 0\}
\]
and  
\[
\partial_{\nu} v(x) = \nabla u(x) \cdot \nabla v(x) = -2\nabla u(x) \cdot D^2 u(x) \nabla u(x) = - 2\, \partial_{\nu\nu}^2 u(x)
\qquad\text{for}\quad x\in \FB(u). 
\]
In particular, since $v$ is superharmonic,
either $v\equiv 0$ (in which case $u$ is either a half space $(x\cdot e)_+$ or a vee $|x\cdot e|$), or
$\partial_\nu v(x) > 0$ on $\partial \{u > 0\}$ by Hopf's lemma. Finally,  noticing that for $x\in \FB(u)$ we have $\partial^2_{\nu(x)\nu(x)} u = - \sum_{i = 1}^{n-1} \partial^2_{\tau_i(x)\tau_i(x)} u$ for some orthonormal basis $\{\tau_i(x)\}_{1\le i \le n-1}$ of the tangent plane to $\FB(u)$ at $x$, we deduce that 
\[
H(x) = \sum_{i = 1}^{n-1} \partial^2_{\tau_i(x)\tau_i(x)} u(x) = \frac12 \partial_{\nu(x)} v(x) > 0 \qquad\text{for}\quad x\in \FB(u), 
\]
as we wanted. 
\end{proof}

As a consequence of \Cref{lem:cleanball}, and thanks to the improvement of flatness, one obtains  additional properties needed to upgrade closeness to a vee into regularity: 
 
\begin{lem}[Closeness to vee and disconnectedness implies regularity]
\label{cor_closetoV_disc_reg}
Let $n\ge 2$. There exists $\eps_\circ=\eps_\circ(n)>0$ such that the following holds. 

Let $u$ be a global classical solution to the Bernoulli problem in $\R^n$ satisfying
\begin{equation}\label{sjiowoihw}
\big |u- V_{0,e_n}\big|\le \eps\varrho \le \eps_\circ\varrho\quad \mbox{in } B_{2\varrho},
\end{equation}
where $e_n$ is the $n$-th vector in the canonical basis. Suppose, in addition, that   the two points  $\varrho e_n$ and $-\varrho e_n$ lie in different connected components of the open set $\{u>0\}\cap B_{2\varrho}$. 

Then
\[\varrho^2\|D^2 u\|_{L^\infty(\{u > 0\}\cap B_{\varrho})}\leq C \eps \varrho\]
for some $C$ depending only on $n$. Moreover, 
\[
\{u>0\} = \{x_n > g^{(+)}(x_1, \dots , x_{n-1})\} \cup \{x_n < g^{(-)}(x_1, \dots , x_{n-1})\} \qquad\text{in}\quad B_\varrho,
\]
where $g^{(\pm)} : D_\varrho \to \mathbb{R}$ with $D_\varrho$ being the lower dimensional ball $\{ x_1^2 + \cdots  + x_{n-1}^2 < \varrho^2 \}$ in $\mathbb{R}^{n-1}$,  $g^{(-)}< g^{(+)}$, and 
\[
\|g^{(\pm)}\|_{L^\infty(D_\varrho)} + \varrho^2 \|D^2 g^{(\pm)}\|_{L^\infty(D_\varrho)} \leq C \eps \varrho.
\]
for some $C$ depending only on $n$. 
\end{lem}
\begin{proof}
Recalling that $V_{0,e_n}(x) = |x_n|$, it follows from \eqref{sjiowoihw} that
\[
\{u = 0\}\cap B_\varrho \subset \{x\in \R^n \ : \ |x_n|\le \eps \varrho\}.
\]
Let $U_+$ and $U_-$ be the connected components of $B_{2\varrho}\cap \{u>0\}$
respectively containing the two points $\pm \varrho e_n$. Then they necessarily contain the two sets $\{x\in B_\varrho: x_n > \eps \varrho\}$ and $\{x\in B_\varrho: x_n < -\eps \varrho\}$ respectively. Also, by assumption $U_+\cap  U_- =\varnothing$. 

Let $\bar u_\pm := u {\mathbbm{1}_{U_\pm}}$ and observe that $\bar u_+$ and $\bar u_-$ are classical solutions to the Bernoulli problem in  $B_{2\varrho}$ which satisfy
\[
\|\bar u_{\pm } \mp x_n\|_{L^\infty(B_{2\varrho}\cap \{\bar u_{\pm} > 0\})}\le \eps \varrho. 
\]
In particular, we can apply the classical epsilon-regularity theory in \Cref{lem:eps-reg-classical} to both $\bar u_+$ and $\bar u_-$ and deduce  the graphicality (hence ordering) of $\FB(\bar u_\pm)$ and the bound 
\[
\varrho|D^2\bar u_\pm|\le C\eps \qquad\text{in}\quad \{\bar u_\pm  >0\}\cap B_{\varrho}. 
\]
Moreover, thanks to \Cref{lem:cleanball}, for $\eps_\circ$ small enough we have 
\[\{u>0\} \cap B_\varrho =  \big(\{\bar u_+>0\} \cup \{\bar u_->0\}\big) \cap B_{\varrho}, 
\]
or, in other words, $u= \bar u_+ + \bar u_-$  in $B_\varrho$. The lemma now follows by \Cref{lem:eps-reg-classical} applied both to $\bar u_+$ and $\bar u_-$.
\end{proof}

The following is a useful auxiliary lemma  (recall the notion of Slab introduced in \eqref{eq:Slab}):

\begin{lem}\label{treeprelim1}
Let $n\ge 2$, and let $u$ be a  global classical solution to the Bernoulli problem in $\R^n$. Suppose that for some $y_1\in \R^n$, $r_1>0$, and $\bar e\in \mathbb S^{n-1}$, we have 
\begin{equation}\label{closenessep100}
\| u -V_{y_1, \bar e} \|_{L^\infty(B_{r_1}(y_1))} \le \eps r_1.
\end{equation}
Then
\begin{equation}\label{closenessepzero}
            \{u=0\}\cap B_{r_1}(y_1) \subset {\rm Slab}(B_{r_1}(y_1), \bar e, \eps) = \{x\in B_{r_1}(y_1)\ : \   |\bar e\cdot(x-y_1)|\le  \eps r_1\}.
\end{equation}
Moreover:

\begin{enumerate}[(a)]
\item For all $y_2\in \{u=0\}$ and $r_2>0$ such that $B_{r_2}(y_2)\subset B_{r_1}(y_1)$, we have 
\[
\| u -V_{y_2, \bar e} \|_{L^\infty(B_{r_2}(y_2))} \le 2\eps r_1. 
\]
\item We have 
\[
\dist(x, \{u = 0\}) \le C_n \eps r_1,\qquad\text{for all}\quad x\in {\rm Slab}\!\left(B_{r_1/2}(y_1), \bar e, 2\eps\right),
\]
for some $C_n$ depending only on $n$ (in particular, for $n=3$ one can choose $C_3= 16$). 
\end{enumerate}
\end{lem}
\begin{proof}
Equation \eqref{closenessepzero} is an immediate consequence of \eqref{closenessep100}. We now prove (a) and (b).
\begin{enumerate}[(a)]
\item It follows from
\[
\| u -V_{y_2, \bar e} \|_{L^\infty(B_{r_2}(y_2))} \le \| u -V_{y_1, \bar e} \|_{L^\infty(B_{r_1}(y_1))} + \| V_{y_2, \bar e} -V_{y_1, \bar e} \|_{L^\infty(B_{r_2}(y_2))},
\]
observing  that $\|V_{y_2, \bar e} -V_{y_1, \bar e} \|_{L^\infty(B_{r_2}(y_2))} 
\le |\bar e \cdot(y_2-y_1)|  = V_{y_1, \bar e}(y_2) \le \eps r_1$.

\item  Given $r<r_1/2$, we need to prove the following implication: 
\begin{equation}\label{jhgiohoiwhohith2}
   \{u=0\}\cap B_r(y) = \varnothing \mbox{ for some $y\in {\rm Slab}(B_\varrho(y_1), \bar e, 2\eps )$} \quad \implies \quad r< C_n\eps r_1.
\end{equation}
Indeed, $y\in {\rm Slab}\!\left(B_{r_1/2}(y_1), \bar e, 2\eps\right)$ is equivalent to $|y-y_1|<r_1/2$ and $ 
|(y-y_1)\cdot \bar e|
\le 2\eps r_1$. Also,  since $r< r_1/2$ we have $B_r(y)\subset B_{r_1}(y_1)$. 
Thus, from  \eqref{closenessep100} we obtain 
\[
u(y) \leq V_{y_1,\bar e}(y) + \eps r_1 = |\bar e\cdot(y-y_1)|  + \eps r_1 \leq 3\eps r_1. 
\]
On the other hand, still using \eqref{closenessep100} and the triangle inequality, we get
\[
\fint_{B_r(y)} u(x)\,dx \ge \fint_{B_r(y)} |\bar e\cdot(x-y_1)| \,dx -\eps r_1 \ge \fint_{B_r(y)} |\bar e\cdot(x-y)| \,dx -3\eps r_1 = r\fint_{B_1} |x_1| \,dx -3\eps r_1 = c_n r - 3\eps r_1.
\]
Since 
$ 
\fint_{B_r(y)} u = u(y)  
$ 
(recall that $u$ is harmonic in $B_r(y)$), this proves that 
$3\eps r_1 \ge c_n r - 3\eps r_1$, or equivalently $r \le \frac{6}{c_n} \eps r_1$, as wanted. (An explicit computation shows that $c_3 = \tfrac38$.) \qedhere
\end{enumerate}
\end{proof}

A variant of \Cref{cor_closetoV_disc_reg} that we will also use in the sequel is the following:  

\begin{lem}[Closeness to vee and bounded Hessian implies regularity]
\label{cor_closetoV_disc_reg2}
Let $n\ge 2$. Given $C_1\ge 1$ there exists $\eps_1>0$, depending only on $n$ and $C_1$, such that the following holds. 

Let $u$ be a global classical solution to the Bernoulli problem in $\R^n$. Suppose that $|D^2u| \le C_1\varrho^{-1}$ in $B_{2\varrho}\cap \{u>0\}$ and
\begin{equation}\label{sjiowoihw2}
\big |u- V_{0,e_n}\big|\le \eps\varrho \le \eps_1\varrho\quad \mbox{in } B_{2\varrho},
\end{equation}
where $e_n$ is the $n$-th vector in the canonical basis.
Then, the same conclusions as in \Cref{cor_closetoV_disc_reg} hold true.
\end{lem}
\begin{proof}
On the one hand,  the bound on the Hessian  implies that the principal curvatures of the free boundary inside $B_{2\varrho}$  are bounded by $C C_1\varrho^{-1}$ (recall that
$u=0$ and $\partial_\nu u=1$ on $\FB(u)$). 
On the other hand, \eqref{sjiowoihw2} implies that  $\FB(u)\cap B_{2\varrho}$ is contained in the slab $|x_n|\le \eps_1 \varrho$ (and it is non-empty, by \Cref{treeprelim1}(b)). 
The result follows.
\end{proof}

\section{Blow-down of global stable solutions}
\label{sec:blowdown}

The goal of this section is to prove that, in $\R^3$, non-flat global stable solutions  to the Bernoulli problem look like a vee at large scales. This is the content of the next:

\begin{prop}[Blow-down of non-flat solutions]
\label{prop:W}
 Given $\eps>0$, there exists $R_{\eps}>0$ depending only on $\eps$ such that for any $R\geq R_{\eps}$, the following holds.

Let $u$ be a global classical stable solution to the Bernoulli problem in $\R^3$, and $0\in \FB(u)$. If
\begin{equation}\label{eq:lem-W-Hess}
\norm[L^\infty(B_1 \cap \set{u>0})]{D^2 u}
\geq 1,
\end{equation}
then there exists $e_R\in \bS^{2}$ such that
\begin{equation}\label{eq:W-like-abs}
\big\| u-|e_R\cdot x| \big\|_{L^\infty(B_R)} \leq \eps R.
\end{equation}
In other words, there exists a universal modulus of continuity $\omega$ (of the form \eqref{eq:w}) such that 
\begin{equation}
	\label{eq:omega_mod}
	\big\| u-|e_R\cdot x| \big\|_{L^\infty(B_R)} \leq \omega(R^{-1}) R,\qquad\text{for all}\quad R>0.
\end{equation}
\end{prop}
To prove this result, we will need to develop a variety of tools that are of independent interest.

We first focus on results that are valid for classical stable solutions. We start by recalling the nondegeneracy of stable solutions recently obtained in \cite{kamburov2022nondegeneracy}. It is proved using a De Giorgi iteration with Michael--Simon--Sobolev inequality, where the mean curvature integral is estimated using the stability inequality with test function $|\nabla u|$.

\begin{lem}[Nondegeneracy of stable solutions \cite{kamburov2022nondegeneracy}]
\label{lem:nondegen}
Let $n\geq 2$, and let $u$ be a  global classical stable solution to the Bernoulli problem in $\R^n$. Then,
for all 
   $y\in \partial{\{u > 0\}}$ and $r>0$,
\begin{equation}\label{eq:density-B}
\fint_{\partial B_r(y)}u \,d\cH^{n-1}
\geq c r
	\quad \mbox{ and }\quad 
\mathcal{H}^{n-1}(\partial\{u > 0\}\cap B_r(y)) \ge c r^{n-1},
\end{equation}
for some $c> 0$ depending only on $n$.
\end{lem}

The following lemma is a direct consequence of a general result first obtained in the semilinear setting by Sternberg--Zumbrun \cite{SZ98}.

\begin{lem}[Sternberg--Zumbrun inequality]
\label{lem:Stern_Zum}
Let $n \ge 2$, $R > 0$, and let $u$ be a classical stable solution to the Bernoulli problem in $B_{2R}\subset \R^n$. Then, 
\[
\int_{B_{R}\cap \{u > 0\}} |D^2 u|^2 \,dx \le CR^{n-2},\quad\text{and therefore}\quad \fint_{B_R\cap \{u>0\}}|D^2 u|^p \,dx \le C R^{-p}\qquad\text{for any}\quad p\in [0, 2],
\]
for some $C$ depending only on $n$. 
\end{lem}
\begin{proof}
Recalling \Cref{lem:Lipbound}, to prove the first inequality we apply \Cref{lem:Stern_Zum_App} to $\frac1{2R} u(2R\,\cdot)$ with $\eta\in C^\infty_c(B_{1})$ non-negative and satisfying $\eta\equiv 1$ in $B_{1/2}$. Then, the second one follows from  H\"older's inequality. 
\end{proof}

We now introduce an important monotone quantity:
for $u\in H^1(B_r)$ and $0\in \FB(u)$, the Weiss boundary-adjusted energy (see \cite{Weiss1998})  is given by
\begin{equation}
\label{eq:W-def}
{\bf W}(u,r)
=\frac{1}{r^n}
	\int_{B_r}
		(|\nabla u|^2+\mathbbm{1}_{\set{u>0}})
	\,dx
-\frac{1}{r^{n+1}}
	\int_{\partial B_r}
		u^2
	\,d\cH^{n-1} = {\bf W}(u_r, 1),
\end{equation}
where $u_r$ denotes the natural dilation of $u$, namely  
\[
u_r(x):=\frac{u(rx)}{r},
	\quad \text{ for } r>0.
\]
Due to the Weiss monotonicity formula (see \cite[Theorem 3.1]{Weiss1998}), given  $u\in H^1(B_R)$ a  stationary solution to the Bernoulli problem,
then
$$
r\mapsto {\bf W}(u,r)\quad\text{ is non-decreasing on $(0,R)$}
$$
and
\begin{align}\label{eq:W-monotone}
\partial_r{\bf W}(u,r)
&= \frac{2}{r^{n+2}}
	\int_{\partial B_r}
		(u-x\cdot \nabla u)^2
	\,d\cH^{n-1} = \frac{2}{r}\int_{\partial B_1} (u_r - x\cdot \nabla u_r)^2 d\mathcal{H}^{n-1} \ge 0\quad \text{ for a.e.}\quad r\in(0,R)
\end{align}
(see also \cite[Section 9]{Vel23}).
In particular, any blow-down limit $u_\infty=\lim_{r_k\uparrow\infty}u_{r_k}$ satisfies  ${\bf W}(u_\infty,r)={\bf W}(u,\infty)$ (because $\lim_{r_k\uparrow\infty}{\bf W}(u,r_kr)=\lim_{r\uparrow\infty}{\bf W}(u,r)$). This implies that $r\partial_r {\bf W}(u_\infty,r)=0$, thus $u_\infty$ is 1-homogeneous.

Let us denote
\begin{equation}
\label{eq:alpha_n}
\alpha_n:={\bf W}((x_n)_+,1)
=2\int_{B_1 \cap \set{x_n>0}}\,dx
	-\int_{\partial B_1 \cap \set{x_n>0}}x_n^2\,d\cH^{n-1}\\
=\frac{\cH^{n-1}(\bS^{n-1})}{2n} = \frac12 |B_1|.
\end{equation}
It is clear that ${\bf W}(|x_n|,1)=2\alpha_n$. As a consequence of the next result, any classical stable solution $u$ to \eqref{eq:Bernoulli-main}--\eqref{eq:stability-var} in $\R^3$ with $0\in \FB(u)$ satisfies
\begin{equation}
\label{eq:W a 2a}
\alpha_3 \leq {\bf W}(u,1)\leq 2\alpha_3.
\end{equation}

\begin{lem}[Almost homogeneous solutions]
\label{lem:W-compact-H1}
For any $\eps\in(0,\frac{\alpha_3}{2})$, there exists $\delta\in(0,\frac{\alpha_3}{2})$ such that the following hold.

Let $u$ be a classical stable solution to the Bernoulli problem in $\R^3$ such that $0\in \FB(u)$ and
\begin{equation}\label{eq:H1-compact-W}
{\bf W}(u,2)-{\bf W}(u,1)<\delta.
\end{equation}
Then, either  
\begin{equation}\label{eq:H1-compact-half}
\normbig[L^\infty(B_1\cap\{u>0\})]{u-e\cdot x}<\eps
	\quad \text{ for some $e\in \bS^2$ and } \quad
{\bf W}(u,2)<\alpha_3+\eps,
\end{equation}
or
\begin{equation}\label{eq:H1-compact-abs}
\normbig[L^\infty(B_1)]{u-|e\cdot x|}<\eps
	\quad \text{ for some $e\in \bS^2$ and } \quad
{\bf W}(u,1)>2\alpha_3-\eps.
\end{equation}
\end{lem}

To prove \Cref{lem:W-compact-H1} we will need the following compactness result for sequences of stable solutions. 
\begin{lem}[Compactness] \label{lem:compact}
Let \( n \geq 2 \), and let \( v_k \in C^{0,1}_{\rm loc}(B_k) \) be a sequence of classical stable solutions to the Bernoulli problem in \( B_k \subset \R^n\), with \( 0 \in \FB(v_k) \) for all \( k \in \N \). Then the following hold:
\begin{enumerate}
    \item Up to a subsequence, \( v_k \) converges to some function \( v_\infty \) satisfying \( |\nabla v_\infty| \leq 1 \) in \( \R^n \), with strong convergence in \( (H^1_{\rm loc} \cap C^{0,\alpha}_{\rm loc})(\R^n) \) for all \( \alpha \in (0,1) \).
    
    \item \label{it:compact-2} The sets \( \overline{\{v_k>0\}} \), \( \{v_k=0\} \), and the free boundaries \( \FB(v_k) \), converge locally in the Hausdorff distance in \( \R^n \) to their corresponding sets for \( v_\infty \) (up to a subsequence). Specifically, 
    \[
    \overline{\{v_k>0\}} \to \overline{\{v_\infty>0\}}, \quad \{v_k=0\} \to \{v_\infty=0\}, \quad \text{and} \quad \FB(v_k) \to \FB(v_\infty),\quad\text{locally}.
    \]
    
    \item The limit function \( v_\infty \) is a stable solution in the sense of \Cref{defi:solutions}.
\end{enumerate}
\end{lem}

\begin{proof}
The proof is postponed to \Cref{sec:compact-stable-sol}.
\end{proof}

We can now prove \Cref{lem:W-compact-H1}.

\begin{proof}[Proof of \Cref{lem:W-compact-H1}]
We divide the proof into three steps. 

{\medskip \noindent \bf Step 1:} We argue by contradiction, and assume that there exists $\eps_0>0$ and a sequence of classical stable solutions $u_k$  
with $0\in\FB(u_k)$ and
\begin{equation}\label{eq:H1-compact-W-not}
{\bf W}(u_k,2)-{\bf W}(u_k,1)<\frac{1}{k},
\end{equation}
but
\begin{equation}\label{eq:H1-compact-not}
\begin{dcases}
\min_{e\in\bS^{2}}\normbig[L^\infty(B_1\cap \{u > 0\})]{u_k-e\cdot x}\geq \eps_0
	\quad \text{ or } \quad
{\bf W}(u_k,2)\geq \alpha_3+\eps_0,\\
\min_{e\in\bS^{2}}\normbig[L^\infty(B_1)]{u_k-|e\cdot x|}\geq \eps_0
	\quad \text{ or } \quad
{\bf W}(u_k,1)\leq 2\alpha_3-\eps_0.
\end{dcases}
\end{equation}
By \Cref{lem:compact}, along a subsequence we have
\[
u_k\to u_\infty
	\quad \text{ strongly in } \ (H^1_{\loc} \cap C^{0}_{\loc})(\R^n),
\]
for some global stationary and inner stable solution $u_\infty$ with $0\in \FB(u_\infty)$. 
Taking the limit in \eqref{eq:H1-compact-W-not} (using \Cref{lem:compact}) we obtain
\[
{\bf W}(u_\infty,2)-{\bf W}(u_\infty,1)=0
	\quad \implies \quad
r\partial_r {\bf W}(u_\infty,r)=0
	\quad \text{ for } r\in(1,2).
\]
In particular, $u_{\infty}$ is $1$-homogeneous in the open annulus $B_2\setminus \overline{B_1}$,  and therefore in $B_2$ by unique continuation.   Up to extending $u_\infty$ outside of $B_2$ in a 1-homogeneous way we can assume that it is defined in the whole $\R^3$. In Steps 2 and 3, we will show that there exists $e\in \bS^{2}$ such that
\begin{equation}
\label{eq:glob_class}
\text{ either } \quad
\begin{cases}
u_\infty=(e\cdot x)_+,\\
{\bf W}(u_\infty,\cdot)\equiv \alpha_3,
\end{cases}
	\quad \text{ or } \quad
\begin{cases}
u_\infty=|e\cdot x|,\\
{\bf W}(u_\infty,\cdot)\equiv 2\alpha_3,
\end{cases}
\end{equation}
which is in direct contradiction with \eqref{eq:H1-compact-not} in the limit $k\to\infty$ (using again \Cref{lem:compact}).

{\medskip \noindent \bf Step 2:} We first prove the validity of \eqref{eq:glob_class} ``up to a multiplicative constant''.

Since $u_\infty$ is 1-homogeneous, $\FB(u_\infty)$ is a cone. Let $y_\circ\in \mathbb{S}^2\cap \FB(u_\infty)$, and consider $\tilde u_\infty$ to be any blow-up of  $u_\infty$ at $y_\circ$ along a sequence $r_k \downarrow 0$, namely,
\[
\tilde u_\infty(x) = \lim_{k \to \infty} \frac{u_\infty(y_\circ + r_k x)}{r_k} = \lim_{k \to \infty}  u_\infty\left(\frac{y_\circ}{r_k} +  x\right).
\] 
Then $\tilde u_\infty$ is invariant in the $y_\circ$ direction. In particular,   $\tilde u_\infty$ is actually a 2-dimensional, 1-homogeneous, non-negative harmonic function. Hence, it must be of the form
\[
\tilde u_\infty(x) = a_+(x\cdot e)_+ + a_- (x\cdot e)_-\quad\text{for some}\quad a_+, a_- \ge 0,\quad e \in \mathbb{S}^2. 
\]
Also, up to changing $e$ with $-e$, we can assume that $a_-\leq a_+$.
We now distinguish two cases.\\
- If $a_- = 0$, since $0$ is a free boundary point for $\tilde u_\infty$ it must be $a_+ > 0$, and since it is a stationary solution then necessarily $a_+ = 1$.\\
- On the other hand, if $0<a_- \leq a_+$, then by stationarity we must have $a_+ = a_- = \tilde a$, and by the uniform 1-Lipschitz bound $\tilde a \le 1$. Observe also that, by the nondegeneracy of classical stable solutions \Cref{lem:nondegen}, we also have\footnote{We remark that any function of the form $\tilde u_\infty(x) = \tilde a |x\cdot e|$ for $\tilde a \ge 0$ is stationary and stable, according to \Cref{defi:solutions}.} that   $\tilde a \ge c > 0$ for some universal $c$. 

As a consequence of this discussion, we have two cases:
\begin{itemize}
\item If $\tilde u_\infty(x) = (x\cdot e_{y_\circ})_+$ for all $y_\circ \in \FB(u_\infty) \cap \mathbb{S}^2$, then the free boundary of $u$ is smooth everywhere outside of the origin,\footnote{This follows from the fact that if a stable solution is close to $(x_n)_+$ then it is close to $x_n$ inside its positivity set (see \Cref{lem:halfspacestable}), so the improvement of flatness in
\Cref{lem:eps-reg-classical} applies. Also, note that blow-ups of limits of classical solutions are themselves limits of classical solutions (by a diagonal argument).} so $u_\infty$ is a classical stable solution outside of the origin. Then, the classification of 1-homogeneous stable solutions in $\R^3$ from \cite{CJK04, Jerison-Savin} applies to our solution and implies that $u_\infty(x) = (x\cdot e')_+$ for some $e'\in \mathbb{S}^2$. Hence, we are in the first case of \eqref{eq:glob_class}.
\item Alternatively, if $\tilde u_\infty(x) = \tilde a_{y_\circ} |x\cdot e_{y_\circ}|$ for some $y_\circ \in \FB(u_\infty) \cap \mathbb{S}^2$ and $\tilde a_{y_\circ} \in (0, 1]$, then ${\bf W}(u_\infty(\cdot + y_\circ), 0^+) = 2\alpha_3$. On the other hand, the Weiss energy is also upper bounded by $2\alpha_3$: indeed, any blow-down of $u_\infty$ around any point is equal to $u_\infty$, which is 1-homogeneous, and for any $1$-homogeneous solution $v$ we have ${\bf W}(v, r) = \frac{1}{r^n} |\{v > 0\}\cap B_r| \leq |B_1| =2\alpha_3$. Therefore $$2\alpha_3=
{\bf W}(u_\infty(\cdot + y_\circ), 0^+) \le {\bf W}(u_\infty(\cdot + y_\circ), r)\leq {\bf W}(u(\cdot +y_\circ), \infty) \le 2\alpha_3,$$ which implies that the Weiss energy is constant, so $u_\infty$ is homogeneous around $y_\circ$. This implies that $u_\infty(x) = \tilde a |x\cdot e|$ for some $\tilde a\in  [c, 1]$ and some $e\in\mathbb{S}^2$ such that $e\cdot y_\circ = 0$. So, to conclude the proof, we only need to show that $\tilde a = 1$. This is the purpose of the next step.
\end{itemize}

{\medskip \noindent \bf Step 3:} It remains to prove that, in the second case, $\tilde a = 1$.

Up to subsequences and after a rotation, we know $u_k\to \tilde a |x_1|$ strongly in $(H^1_{\loc} \cap C^{0}_{\loc})(\R^3)$ for some $\tilde a\in [c, 1]$. Also, thanks to \Cref{lem:Stern_Zum},
\begin{equation}
\label{eq:contr_atilde}
\int_{B_1\cap \{u_k > 0\}} |D^2 u_k|^2\,dx 
\le C,
\end{equation}
for some $C > 0$ universal, independent of $k$.  

Now, assume by contradiction that $\tilde a < 1$. By harmonic estimates we have 
\begin{equation}\label{eq:limitatilde}
u_k\to \tilde a |x_1|	  
\qquad\text{
   in }\quad L^\infty(B_1)\cap C^1_{\rm loc}(B_1\setminus\{x_1 = 0\}), \quad\text{ for some }\quad 0<c< \tilde a< 1.
 \end{equation}
 The proof now follows along the lines of that of \Cref{lem:halfspacestable}. By Fubini's theorem, we know 
\[
\begin{split}
\int_{B_1\cap\{u_k  >0\}}|D^2 u_k|^2\,dx 
& \ge \int_{B'_{1/2}}\int_{[-1/2, 1/2]\cap \{u_k(t,\sigma) > 0\}} |D^2u_k|^2(t,\sigma)\, dt\, d\sigma \ge \int_{B'_{1/2}}\int_{t_{\sigma, k}}^{1/2} |D^2u_k|^2(t,\sigma)\, dt\, d\sigma,
\end{split}
\] 
where $B_r'\subset \R^2$ denotes the ball of radius $r$ in $\R^2$ and,  given $\sigma \in B_{1/2}'$ and $k\in \N$, $t_{\sigma, k}$ is the minimal value $t_*\in [-1/4, 1/4]$ (for $k$ large enough) such that $(t_{*}, 1/2)\subset \{u_k(\cdot, \sigma) > 0\}$. 

Let $\Pi_1:\R^3\to \R^2$ denote the orthogonal projection in the last two variables, that is $\Pi_1((x_1, x_2, x_3)) = (x_2, x_3)$, and define
 \[
 A_k := \Pi_1\big(\FB(u_k)\cap((-1/2, 1/2)\times B'_{1/2})\big).
 \]
Also, let $\delta>0$ be a small fixed constant. Note that $|\nabla u_k|^2=1$ on $\FB(u_k)$, while   $|\nabla u_k(\delta, \sigma) |^2\le \frac{1+\tilde a^2}{2}$ for $k \gg 1$ large enough (due to \eqref{eq:limitatilde} and harmonic estimates), therefore
$$
\int^{\delta }_{t_{\sigma, k}}
\bigl|
	\partial_1 |\nabla u|^2(t, \sigma)\big|  dt\geq 1-\frac{1+\tilde a^2}{2}=\frac{1-\tilde a^2}{2} \qquad \text{for all}\quad \sigma \in A_k
$$
(note that, if $k \gg 1$, then  $t_{\sigma, k}\in (-\delta, \delta)$ for $\sigma \in A_k$). Thus, thanks to the bound
 $\bigl|\nabla |\nabla u_k|^2\bigr|^2 \le 4|D^2 u_k|^2$, Cauchy--Schwarz, and \eqref{eq:contr_atilde}, this implies that 
\[
\frac{1-\tilde a^2}{2} |A_k| \le \int_{A_k} \int^{\delta }_{t_{\sigma, k}}
\bigl|
	\partial_1 |\nabla u|^2(t, \sigma)\big|  dt\, d\sigma\le  C \big(|A_k|\delta\big)^{1/2}\left( \int_{B_1\cap\{u_k  >0\}}|D^2 u_k|^2\right)^{1/2} \le C\big(|A_k|\delta\big)^{1/2},
\]
which proves 
\begin{equation}
    \label{eq:Ak}
    |A_k|\le \frac{C\delta}{1-\tilde a^2}.
\end{equation}
Consider now instead $\sigma \in B_{1/2}' \setminus A_k$. Then $t_{\sigma, k} =-\tfrac14 \le \delta$. Also, by \eqref{eq:limitatilde} we know $\partial_1 u_k(-\delta, \sigma) < -\frac{c}{2}$ and $\partial_1 u_k(\delta, \sigma) > \frac{c}{2}$, so that 
\[
\int_{-\delta}^{\delta} |\partial^2_{11} u_k(t, \sigma)|dt > c > 0,\quad\text{for $k$ large and $\sigma\in B_{1/2}'\setminus A_k$}.
\]
Hence, by $|\partial_{11}^2 u_k|^2 \le |D^2 u_k|^2$, Cauchy--Schwarz, and \eqref{eq:contr_atilde}, similarly to before we obtain 
\[
c |B'_{1/2}\setminus A_k| \le \int_{B_{1/2}'\setminus A_k} \int_{-\delta }^{\delta}
\bigl|
	\partial^2_{11} u_k\big|(t, \sigma)  dt\, d\sigma\le  C \big(|B_{1/2}'\setminus A_k|\delta\big)^{1/2}
\]
therefre $|B'_{1/2}\setminus A_k| \le C\delta$. Combining this bound with \eqref{eq:Ak}, we get a contradiction for $\delta$ sufficiently small.
\end{proof}

As a consequence of the previous result, if we can lower bound the Hessian of a solution at one point, then the solution cannot be energetically close to a half-space.

\begin{lem}[Lower bound of Weiss energy]
\label{lem:W-lower}
Let $\eps_\circ=\eps_\circ(3)$ and $C_{3,2}$ (i.e., $n=3$ and $k=2$) be the constants from \Cref{lem:eps-reg-classical}, and set $C_\circ := C_{3,2}\eps_\circ$. Let $\delta_\circ\in (0,\frac{\alpha_3}{2})$ be chosen from \Cref{lem:W-compact-H1} with $\eps=\eps_\circ$.

Let $u$ be a global classical stable solution to the Bernoulli problem in $\R^3$, with $0\in \FB(u)$. If
\begin{equation}\label{eq:lem-W-lower-Hess}
\norm[L^\infty(B_{1/2} \cap \set{u>0})]{D^2u}
\geq 2C_\circ,
\end{equation}
then
\begin{equation}\label{eq:lem-W-lower}
{\bf W}(u,2)
\geq \alpha_3+\delta_\circ.
\end{equation}
\end{lem}

\begin{proof}
Recalling that ${\bf W}(u,1)$ is always bounded from below by $\alpha_3$ (see \eqref{eq:W a 2a}),
if \eqref{eq:lem-W-lower} does not hold then 
\begin{equation}\label{eq:lem-W-lower-not}
{\bf W}(u,2)-{\bf W}(u,1)<(\alpha_3+\delta_\circ)-\alpha_3=\delta_\circ.
\end{equation}
Thus, by \Cref{lem:W-compact-H1}, either \eqref{eq:H1-compact-half} or \eqref{eq:H1-compact-abs} holds. The alternative \eqref{eq:H1-compact-half} can be ruled out, since \Cref{lem:eps-reg-classical} implies
\[
\norm[L^\infty(B_{1/2} \cap \set{u>0})]{D^2u}
\leq C_\circ,
\]
contradicting assumption \eqref{eq:lem-W-lower-Hess}. Thus we are left with the case \eqref{eq:H1-compact-abs}, in which case
\[
{\bf W}(u,2)
\geq {\bf W}(u,1)
>2\alpha_3-\eps_\circ
\geq \alpha_3+\delta_\circ,
\]
contradicting \eqref{eq:lem-W-lower-not}.
\end{proof}

We can finally upgrade the previous lemma to solutions close to vees and prove \Cref{prop:W}.

\begin{proof}[Proof of \Cref{prop:W}]
Let $C_\circ$ be as in \Cref{lem:W-lower}, and recall the notation $u_R(x)=\frac{1}R u(Rx)$. Then  $u_{2C_\circ}$ satisfies
\[
\norm[L^\infty(B_{1/2}\cap\setsmall{u_{2C_\circ}>0})]{D^2u_{2C_\circ}}
=2C_\circ\norm[L^\infty(B_{C_\circ}\cap\setsmall{u>0})]{D^2u}
\geq 2C_\circ\norm[L^\infty(B_{1}\cap\setsmall{u>0})]{D^2u}
\geq 2C_\circ
\]
therefore
\[
{\bf W}(u,4C_\circ)={\bf W}(u_{2C_\circ},2)\geq \alpha_3+\delta_\circ
\]
by \Cref{lem:W-lower}.

Now, given $\eps>0$, let $\delta_0:=\delta(\eps)>0$ be determined by \Cref{lem:W-compact-H1}.
Also, given $\eps_1:=\delta_0$, let $\delta_1:=\delta(\eps_1)$ be determined by applying \Cref{lem:W-compact-H1} one second time.

Let us now apply \Cref{lem:W-compact-H1} with $\eps_1,\delta_1$ to the functions $u_{2^{k+1}C_\circ}$ with $k=1,\dots, K$ (where $K=K(\eps_1)$ is to be chosen). We first check that the alternative \eqref{eq:H1-compact-half} does not hold for any $k$. Indeed, by \Cref{lem:eps-reg-classical}, \eqref{eq:H1-compact-half} implies
\[
\norm[L^\infty(B_{2^{k}C_\circ}\cap \setsmall{u>0})]{D^2 u}
=\frac{1}{2^{k+1}C_\circ}\norm[L^\infty(B_{1/2}\cap \setsmall{u_{2^{k+1}C_\circ}>0})]{D^2 u_{2^{k+1}C_\circ}}
\leq \frac{C_\circ}{2^{k+1}C_\circ}
<1,
\]
contradicting \eqref{eq:lem-W-Hess}. Hence, either
\begin{equation}\label{eq:lem-W-k}
{\bf W}(u,2^{k+2}C_\circ)-{\bf W}(u,2^{k+1}C_\circ)\geq \delta_1
	\qquad \text{for all}\quad k=1,\dots,K,
\end{equation}
or, by \eqref{eq:H1-compact-abs}, there exist $k\leq K$  
such that
\begin{equation}\label{eq:lem-W-stop}
{\bf W}(u,2^{k+1}C_\circ)>2\alpha_3-\eps_1.
\end{equation}
If \eqref{eq:lem-W-k} holds, then summing over $k$ from $1$ to $K:=\floors{\frac{\alpha_3}{\delta_1}}+1$ yields
\[
{\bf W}(u,2^{K+2}C_\circ)
\geq {\bf W}(u,4C_\circ)+K\delta_1
\geq \alpha_3+\delta_\circ+K\delta_1
\geq 2\alpha_3>2\alpha_3-\eps_1.
\]
In either case, recalling \eqref{eq:W a 2a} and choosing $R_\eps=2^{K+2}C_\circ$, for any $R\geq R_\eps$ it holds
\[
2\alpha_3\geq {\bf W}(u,R)\geq {\bf W}(u,R_\eps)>2\alpha_3-\eps_1.
\]
This implies that ${\bf W}(u_R,2)-{\bf W}(u_R,1)<\eps_1=\delta(\eps)$ and ${\bf W}(u_R,2)>2\alpha_3-\eps_1>\alpha_3+\eps$, so by applying  \Cref{lem:W-compact-H1} again we obtain
\[
\normbig[L^\infty(B_1)]{u_R-|e_R\cdot x|}<\eps
\]
for some $e_R\in \bS^{2}$, as desired.
\end{proof}

\section{Necks: definition and properties}
\label{sec:neck_set}
In this section we begin our study of global classical stable solutions.
We will need to properly define the ``neck'' regions (i.e., regions where the free boundary is not flat) and study their properties.

\subsection{Reduction}

Let us begin with the following reduction lemma.
From now on, whenever $u$ is a classical solution and $x_\circ \in \FB(u)$, we write $D^2u(x_\circ)$ to denote the limit of $D^2u$ from the positivity set: more precisely, $D^2u(x_\circ):=\lim_{\{u>0\}\ni x\to x_\circ} D^2 u(x).$

\begin{lem}[Reduction]
\label{lem:reduction}
Let $n \ge 2$, and suppose there exists a global classical stable  solution $v$ to the Bernoulli problem in $\R^n$  such that $|D^2 v|\not\equiv 0$ in $\{v > 0\}$. Then, there exists a global classical stable solution $u$ such that $0\in \FB(u)$, $|D^2 u|\le 1$ in  $\overline{\{u>0\}}$, and $|D^2 u|(0) = 1$. 
\end{lem}

\begin{proof} 
Notice that $v$ must have a free boundary; indeed, if not, it would be a positive harmonic function, so it would be constant, contradicting the assumption that  $|D^2 v| \not \equiv 0$.

Let us suppose first  $\sup_{\{v>0\}} |D^2 v|= \infty$. Consider $x_k\in B_k\cap \overline{\{ v > 0\}}$ such that
\[
h_k := |D^2 v(x_k)|\bigg(1 - \frac{|x_k|}k\bigg) = \max_{x\in B_k\cap \overline{\{v > 0\}}} |D^2 v(x)|\bigg(1 -\frac{|x|}k\bigg),
\]
which satisfies $h_k \geq \frac12 \max\limits_{B_{k/2}}|D^2 v|\to \infty$ as $k\to\infty$.
Let $d_k := |D^2 v(x_k)|$ and $\rho_k = 1-\frac{|x_k|}k$, and define the classical stable solutions
\[
u_k(y) := d_k \,v\bigg(x_k + \frac{y}{d_k}\bigg)\qquad\text{for}\quad y\in B_{d_k \rho_k }. 
\]
We have $0\in \overline{\{u_k > 0\}}$ and $|D^2 u_k(0)| = 1$. Also, by definition of $h_k$, 
for $x=x_k + \frac{y}{d_k}\in \{v>0\}$  with $|y|< d_k \rho_k$ we have 
\[
\left|D^2 v\left(x_k + \frac{y}{d_k}\right)\right| \le |D^2 v(x_k)| \frac{1-\frac{|x_k|}{k}  }{1-\frac{|x_k + y/d_k|}{k}} \le |D^2 v(x_k)| \frac{\rho_k}{\rho_k -\rho_k/k}.
\] 
Therefore,
\[
|D^2 u_k(y)| = \frac{1}{d_k} \abs{D^2 v\left(x_k + \frac{y}{d_k}\right)}
\le \frac{1}{1-1/k} \qquad\text{in}\quad B_{d_k \rho_k} \cap \overline {\{u_k > 0\}}. 
\]
Since $d_k\rho_k = h_k \to \infty$ as $k \to \infty$, \Cref{lem:compact} implies that (up to a subsequence) $u_k$ converges to some global stable solution $u$ with $0\in \overline {\{u>0\}}$. Moreover, thanks to the upper bound on the Hessian, the free boundaries are uniformly smooth and $|D^2 u|\le 1$ in $\overline{\{u>0\}}$.

We observe that \( u \) is a classical stable solution satisfying $|D^2u|(0)=1$. Indeed, given \( x_\circ \in \FB(u) \) and \( r > 0 \), the uniform bound \( |D^2 u_k| \leq 1 \)  implies—using the condition \( \partial_\nu u_k = 1 \) on \( \FB(u_k) \) and \Cref{lem:cleanball}—that the positivity sets \( \{ u_k > 0 \} \) are locally the union of at most two smooth hypographs (in opposite directions) inside \( B_r(x_\circ) \). Moreover, these hypographs have uniform curvature estimates for their boundaries.
Hence, applying the Arzelà-Ascoli theorem, these hypographs converge (up to a subsequence) to smooth hypographs as \( k \to \infty \), and the free boundaries \( \FB(u_k) \) converge smoothly to \( \FB(u) \). 
Also, since $|D^2u_k(0)|=1$ and the free boundaries converge smoothly, the Hessian of $u$ must be nonzero in an open set near $0$.

To verify that $u$ is classical we only need to rule out tangency situations: i.e., we must show that the boundaries of two locally connected components of \( \{ u > 0 \} \cap B_r(x_\circ) \) cannot touch. To show this, note that because of \Cref{lem:vbounds}, each component of \( \{ u > 0 \} \cap B_r(x_\circ) \) is mean concave. Thus, the presence of a tangency point \( x_\circ \) would force the mean curvature to be zero at \( x_\circ \), and therefore (again by \Cref{lem:vbounds}) \( D^2 u \equiv 0 \) in a neighborhood of \( x_\circ \). By unique continuation, this would imply \( D^2 u \equiv 0 \) in all of these two connected components of \( \{ u > 0 \} \), which would imply that $u$ is a vee. However, this contradicts the fact that the Hessian of $u$ is nonzero in an open set near $0$. Thus, no tangency point can exist, completing the argument.

Note now that, since $u$ is a classical solution, the bound $|D^2u_k(0)|=1$ implies in the limit that $|D^2u(0)|=1$.

It remains to consider the case $M: = \sup_{\{v>0\}} |D^2 v|\in (0, \infty)$.  In this case, it suffices to 
choose $x_k \in \{ v > 0\}$ such that $|D^2 v(x_k)|\to M $ and define $u$ as the limit of $u_k(x):=M\, v\left(x_k+\frac{x}{M}\right)$. 
Arguing similarly to above, $u_k$ converges locally uniformly to a  classical solution $u$ satisfying $|D^2 u|\le 1$ on $\overline {\{u>0\}}$  and $|D^2 u|(0)=1$. Again by the strong maximum principle ($|D^2 u|^2$ is subharmonic) we obtain $0\in \FB(u)$.
\end{proof}

\subsection{Fixing global assumptions}
\label{ssec:fixinghyp} 

  Let us now fix some global assumptions and variables. Throughout the rest of the paper, and until otherwise stated, we set $n=3$
and $u\in {\rm Lip}(\R^3)$ to be a fixed global classical stable solution to the Bernoulli problem in $\R^3$, with 
\begin{equation}
    \label{eq:assump_glob_u}
\text{$0\in \FB(u)$, \quad  $|\nabla u|\le 1$, \quad $|D^2 u(0)| = 1$, \quad  and \quad  $|D^2 u|\le 1$ in $\{u > 0\}$,}
\end{equation}
(as in \Cref{lem:reduction}). 

We also fix the global universal constant
\begin{equation}
    \label{eq:eta0}
    \eta_0 : =  \eta_*(3)
\end{equation}
where $\eta_*(3)$ is the  constant in \Cref{lem:eps-reg-Hess} for dimension $n=3$.

\subsection{Definition of neck centers}
\label{ssec:neck-set-def}
Given  $u$ and $\eta_0$ as above, we define the set of \emph{neck centers} $\cZ$  as follows. 
\begin{itemize}
\item First, for any $y\in \FB(u)$, we define its \emph{threshold radius} as 
\begin{equation}
\label{eq:def_rB}
\rb(y) :=  \inf\bigg\{r > 0 : \int_{B_r(y)\cap \{u > 0\}}|D^2 u|^3 \,dx \ge \eta_0^3 \bigg\}.
\end{equation}
Observe that, since we are assuming $|D^2 u|$ to be globally and universally bounded, we know that 
\begin{equation}\label{eq:rmin_def}
\rb(x) \geq \rmin = \rmin ( \eta_0):= c \, \eta_0 > 0\qquad\text{for all}\quad x\in \FB(u). 
\end{equation}

\item Then, for any $k\in \N_0$, we define
\begin{equation}\label{eq:Xktilde}
\tilde \cZ_k := \left\{ x \in \FB(u)  : \rb(x) \in [\rmin 2^k, \rmin 2^{k+1})\right\}.
\end{equation}
 
\item Given $\lambda > 0$ and $\mathcal Y \subset \FB(u)$, we denote 
\begin{equation}
\label{eq:cBlambdadef}
\cB_\lambda(\mathcal Y) := \bigcup_{y\in \mathcal Y} B_{\lambda \rb(y)}(y)
\end{equation}
(not to be confused with the notation $B_r(A) = A+B_r$ in \Cref{ssec:notation}). Thanks to Vitali's covering lemma, we can consider a countable subset of centers $\cZ_0\subset \tilde \cZ_0$ such that
\[
B_{\rb(z_1)}(z_1)\cap B_{\rb(z_2)}(z_2) = \varnothing\qquad\text{for all}\quad z_1, z_2\in \cZ_0,\quad z_1\neq z_2,
\]
and
\[
\tilde \cZ_0\subset \cB_1(\tilde \cZ_0)\subset \cB_4(\cZ_0). 
\]

\item Then, for $k\ge 1$, we recursively define
\begin{equation}
\label{eq:Xkprime}
\cZ_k' := \{x \in \tilde \cZ_k : B_{4\rb(x)}(x) \cap \cZ_{<k} = \varnothing\},
\end{equation}
where we have denoted $\cZ_{<k} := \bigcup_{i = 0}^{k-1} \cZ_i$. We take $\cZ_k$ to be the centers of a Vitali subcovering of $\cB_1(\cZ_k')$, namely, $\cZ_k\subset \cZ_k'$ is a countable subset such that
\[
B_{\rb(z_1)}(z_1)\cap B_{\rb(z_2)}(z_2) = \varnothing\qquad\text{for all}\quad z_1, z_2\in \cZ_k,\quad z_1\neq z_2,
\]
and 
\[
\cZ_k'\subset \cB_1(\cZ_k')\subset \cB_4(\cZ_k).
\]

\item Finally, we define 
\begin{equation}
\label{eq:cZ_def}
\cZ := \bigcup_{k\ge 0} \cZ_k.
\end{equation}
We call the points in $\cZ$ {\em neck centers} and denote the points in $\cZ$ by $\zz$, $\zz_k$, etc. The threshold radii of neck centers are simply called  \emph{neck radii}. 
\end{itemize}

The first observation is that $\cZ$ exists:

\begin{lem}
\label{lem:Znonempty}
There holds $\rb(0)\leq C\eta_0$. Consequently, the set of neck centers $\cZ$ is nonempty. 
\end{lem}
\begin{proof}
    Recalling \eqref{eq:assump_glob_u}--\eqref{eq:eta0}, \Cref{lem:eps-reg-Hess} gives $1=|D^2u(0)|\leq \normsmall[L^\infty(B_{\rb(0)/2}\cap \set{u>0})]{D^2u}\leq \frac{C\eta_0}{\rb(0)}$, as desired.
\end{proof}

From now on, the set $\cZ$ is fixed as above. 

\subsection{Basic properties of the neck centers and neck radii}
Given the previous definitions, we start to discuss some basic properties of the neck centers. 

The first one is the following, which says that while some points in \eqref{eq:Xktilde} are omitted in the definition of neck centers, they are covered by a fattening of the $\cZ_k$'s.

\begin{lem}[Covering ``omitted neck centers'']
	\label{lem:X<ktilde}
For any $k\geq 1$, we have
\[
\tilde{\cZ}_{<k}:=\bigcup_{j=0}^{k-1}\tilde{\cZ}_j
\subset  B_{2^{k+2}\rmin }(\cZ_{<k}) = \cZ_{<k}+B_{2^{k+2}\rmin } .
\]
\end{lem}

\begin{proof}
For any $j\geq 1$ and $x \in \tilde{\cZ}_j\setminus \cZ_j'$, it follows from \eqref{eq:Xkprime} that $x \in B_{4\rb(x)}(\cZ_{<j})$.
Recalling \eqref{eq:Xktilde}, this implies
\[
\tilde{\cZ}_j \setminus \cZ_j'
\subset B_{4\rmin \cdot 2^{j+1}}(\cZ_{<j}).
\]
Thus, using \eqref{eq:Xktilde} again for all $j\leq k-1$ (recall \eqref{eq:cBlambdadef}),
\begin{align*}
\tilde{\cZ}_{<k}
&\subset \bigcup_{j=0}^{k-1} \big( \cZ_j' \cup (\tilde{\cZ}_j \setminus \cZ_j') \big)
\subset \bigcup_{j=0}^{k-1} \Big( \cB_4(\cZ_j) \cup B_{4\rmin\cdot 2^k}(\cZ_{<j}) \Big)  
\subset \cB_4(\cZ_{<k}) \cup B_{4\rmin\cdot 2^k}(\cZ_{<k-1}).
\end{align*}
The result follows.
\end{proof}

 \begin{cor}\label{cor:Hess-bd-neck}
For any $\zz \in \cZ$, we have%
\[
\rb(\zz)\norm[L^\infty(B_{2 \rb(\zz)}(\zz)\cap \{u > 0\})]{D^2u}
\leq  C,
\]
for some $C$ universal.
\end{cor}

\begin{proof}
Let $\zz\in \cZ_k$. By  \Cref{lem:X<ktilde} together with \eqref{eq:Xktilde}--\eqref{eq:Xkprime},
\[
  \tilde{\cZ}_{< k-1}
\subset B_{2^{k+1}\rmin}(\cZ_{<k-1})
\subset B_{2 \rb(\zz)}(\cZ_{<k})
\subset \R^3\setminus B_{2\rb(\zz)}(\zz),
\]
therefore
\begin{equation}\label{eq:doubling-claim-10}
B_{2\rb(\zz)}(\zz) \cap \tilde{\cZ}_{<k-1} = \varnothing\quad\Longrightarrow\quad \rb(y) \ge \rb(\zz) /4\qquad\text{for all}\quad y\in \FB(u) \cap B_{2\rb(\zz)}(\zz). 
\end{equation}
(Here we understand that $\tilde{\cZ}_{<k-1}=\varnothing$ if $k-1\le 0$.) We now claim that
\begin{equation}\label{eq:doubling-claim-20}
\rb(y) \leq 4\rb(\zz)\qquad\text{for all}\quad y\in \FB(u) \cap B_{2\rb(\zz)}(\zz).
\end{equation}
Indeed, if not, then \eqref{eq:def_rB} yields\footnote{Here the strict inequality follows from the fact that, if $|D^2u|$ were to vanish in $\big(B_{\rb(y)}(y)\setminus B_{4\rb(\zz)}(y)\big)\cap \{u > 0\}$, then it would be zero inside $B_{4\rb(\zz)}(y)\cap \{u > 0\}$ by unique continuation.}
\[
\eta_0^3
=\int_{B_{\rb(\zz)}(\zz)\cap \{u > 0\}}
	|D^2u|^3
\,dx
\leq \int_{B_{4\rb(\zz)}(y)\cap \{u > 0\}}
	|D^2u|^3
\,dx
< \int_{B_{\rb(y)}(y)\cap \{u > 0\}}
	|D^2u|^3
\,dx
=\eta_0^3,
\]
a contradiction. Hence, \eqref{eq:doubling-claim-20} holds.

After a rescaling (considering $\frac{u(\zz+\rb(\zz) x)}{\rb(\zz)}$ instead of $u$), let us assume $\rb(\zz) = 1$. Then  $\rb(x)\in [\tfrac14, 4]$ for all $x\in \FB(u) \cap B_2$. Also, by \Cref{lem:eps-reg-Hess}, $|D^2 u|\le C \eta_0$ in $\{u > 0\}\cap B_{2}\cap \{\dist(\cdot, \FB(u)) < \tfrac18\}$.
Furthermore, since $u$ is $1$-Lipschitz, harmonic estimates imply that $|D^2 u|\le C$ in $\dist(\cdot, \FB(u)) \ge \tfrac18$. 
This shows that $|D^2 u|\le C$ in $\{u > 0\}\cap B_{2}$, which is the desired result (once one rescales the solution back).
\end{proof}

We now observe that the threshold radius controls the distance to the set $\cZ$ of neck centers:

\begin{lem}[$\rb$ controls distance to neck centers]
\label{lem:distX-r0}
Let $x \in \FB(u)$ and $\cZ$ be defined as above. Then
\[
\dist(x,\cZ) \leq 8 \rb(x).
\]
\end{lem}

\begin{proof}
By construction, there exists $k\in \N$ such that $x\in \tilde{\cZ}_k$. If $\dist(x,\cZ_{<k}) < 4\rb(x)$ then we are done, since $\cZ_{<k}\subset \cZ$. Otherwise,   $x\in \cZ_k'$ and there exists $\bar{x} \in \cZ_{k} \subset \cZ$ such that
\[
\dist(x,\cZ)
\leq |x-\bar{x}|
\le 4\rb(\bar{x}) \leq 4\rmin 2^{k+1}
=8\rmin 2^{k}
\leq 8\rb(x).
\qedhere
\] 
\end{proof}

Next, we show that the Hessian is controlled by its distance to $\cZ$:

\begin{lem}[Global Hessian decay]
	\label{lem:Hess-decay}
	We have:
	\[
	|D^2 u(x)|
	\leq C \min\biggl\{\frac{1}{\dist(x,\cZ)},1\biggr\},\qquad\text{for all}\quad x\in \{u > 0\},
	\]
 for some $C$ universal. 
\end{lem}

\begin{proof}
We divide the proof into two cases. Recall that $\eta_0 = \eta_*(3)$,  from \Cref{lem:eps-reg-Hess}. 

{\medskip \noindent \bf Case 1: $x_\circ \in \set{u>0}$ and $\dist(x_\circ,\FB(u)) \leq \tfrac{1}{25} \dist(x_\circ,\cZ)$.}
In this case, choose $y_\circ \in \FB(u)$ closest to $x_\circ$ so that, by triangle inequality, $$\dist(x_\circ,\cZ)\leq \dist(y_\circ,\cZ)+|x_\circ-y_\circ| \leq \dist(y_\circ,\cZ)+\tfrac{1}{25}\dist(x_\circ,\cZ).$$ By \Cref{lem:distX-r0}, this gives the chain of inequalities
\[
24|x_\circ-y_\circ|
\leq \frac{24}{25} \dist(x_\circ,\cZ)
\leq \dist(y_\circ,\cZ)
\leq 8\rb(y_\circ),
\]
and therefore
\[
x_\circ\in B_{\frac32 |x_\circ-y_\circ|}(y_\circ)
\subset B_{\rb(y_\circ)/2}(y_\circ).
\]
Now, using \eqref{eq:eps_reg_Hess_k} around $y_\circ$ and \Cref{lem:distX-r0},   
\[
|D^2u(x_\circ)|
\leq \norm[L^\infty(B_{ \rb(y_\circ)/2}(y_\circ)\cap \set{u>0})]{D^2u}
\leq \frac{C \eta_0}{\rb(y_\circ)}
\leq \frac{C \eta_0}{\dist(x_\circ,\cZ)}.
\]

{\medskip \noindent \bf Case 2: $x_\circ \in \set{u>0}$ and $\dist(x_\circ,\FB(u)) > \tfrac{1}{25} \dist(x_\circ,\cZ)$.}
In this case, we apply harmonic estimates to $u$ inside $B_{\frac{1}{2}\dist(x_\circ,\FB(u))}(x_\circ)$. Since $u(x) \leq \dist(x,\FB(u))$ (recall that $|\nabla u|\leq 1$), this yields
\[
	|D^2 u(x_\circ)|\le \frac{C }{\dist(x_\circ,\FB(u))}
	\leq \frac{C}{\dist(x_\circ,\cZ)}.
\]
This proves that $|D^2 u(x)|
	\leq \frac{C}{\dist(x,\cZ)}$. Recalling that $|D^2u|\leq 1$ (see \eqref{eq:assump_glob_u}), the result follows. 
\end{proof}

The next lemma says that, around a neck center and at scales much larger than the neck radius, the solution increasingly resembles a vee:

\begin{lem}[Blow-down around neck center]
\label{lem:X-blowdown}
For any $\eps>0$ there exists $M=M(\eps)\ge 1$ such that the following holds.

For every $\zz \in \cZ$ and every $R\geq M \rb(\zz)$, we have 
\begin{equation}\label{epscloseness11}
\min_{e\in \mathbb S^2}\big\|
u - V_{\zz, e}
\big\|_{L^\infty(B_{R}(\zz))}
\leq \eps R. 
\end{equation}
In particular (recall \eqref{eq:Ez_def_intro})
\begin{equation}\label{epscloseness11bis}
\bE_{\zz}(u,R)
\leq \eps . 
\end{equation}
More precisely, choosing $M_*:=|B_1|^{\frac{1}{3}}\eta_0^{-1}\ge 1$ and $\omega$ as in \eqref{eq:omega_mod}, the relation between $\eps$ and $M$ is implicitly given by
\[
\omega\big(M_*/M\big) = \eps
\]
\end{lem}

\begin{proof}
Let $u_{\zz,\rho}(x):=\frac{u(\zz+\rho x)}{\rho}$ with $\rho = M_* \rb(\zz)$ and $M_*:=|B_1|^{\frac{1}{3}}\eta_0^{-1}$. Then $0\in \FB(u_{\zz, \rho})$ and 
\[
1= \eta_0^{-3} \int_{B_{\rb(\zz)}(\zz)\cap \{u > 0\}}
|D^2u|^3
\,dx \le \|D^2 u\|^3_{L^\infty(B_{\rho}(\zz)\cap \{u > 0\})} \rho^3 = \norm[L^\infty(B_1\cap\set{u_{\zz,\rho}>0})]{D^2u_{\zz,\rho}}^3.
\]
By \Cref{prop:W}, for any $r>0$, 
\[
 \min_{e\in \mathbb S^2}	\big\| u_{\zz,\rho}-V_{0,e} \big\|_{L^\infty(B_r)} \leq \omega(r^{-1}) r,
	\quad \text{ or equivalently } \quad
	 \min_{e\in \mathbb S^2}	 \big\| u-V_{\zz,e} \big\|_{L^\infty(B_{\rho r}(\zz))} \leq \omega(r^{-1}) \rho r,
\]
In particular, given $\eps>0$, by choosing $M$ so large that $\omega(M_*/M)<\eps$ we obtain that \eqref{epscloseness11} holds for $R\geq M\rb(\zz)$.
\end{proof}

Up to now, it may not be completely clear why we introduced the notions of neck centers and neck radii. If $\zz\in \cZ$ is a given neck radius, then the previous lemma shows that, for $R\gg \rb(\zz)$, the zero set will be contained in some very thin strip ---see \Cref{treeprelim1}(b):
\[
\{u=0\} \cap B_{R}(\zz) \subset \{x\in \R^3 \ \ :\ \ |e\cdot (x-\zz)| = o(R)\},
\]
for some $e\in \mathbb S^2$ (depending on $\zz$ and $R$).

The next lemma actually shows that neck radii detect `necks' or `bridges' of the positivity set $\{u>0\}$ between the two sides of the set $\{|e\cdot (x-\zz)|> o(R)\}$. In other words, at scales  $R\gg \rb(\zz)$, $\{u>0\} \cap B_{R}(\zz)$ becomes connected. 

\begin{lem}[Neck centers detect `necks']
\label{lem:neckcenters_necks} There exists a large universal constant $\bar M\ge 1$ such that whenever $\zz\in \cZ$ and $\varrho\ge \bar M \rb(\zz)$, any two points of $\{u>0\}\cap B_{\varrho}(\zz)$ can be joined by a continuous path contained in $\{u>0\}\cap B_{2\varrho}(\zz)$.
\end{lem}
\begin{proof}
Let $\bar \eps  > 0$ be a small constant that will be fixed later and consider $\bar  M=M(\bar\eps)$ given by \Cref{lem:X-blowdown}. Then, for any $\varrho \ge \bar M \rb(\zz)$ we have
\[B_{2\varrho}(\zz)  \cap\{u=0\} \subset \{x \ :\    |e_{2\varrho, \zz} \cdot (x - \zz)| \le \bar \eps  \varrho\}.\]
Let $U_+$ and $U_-$ denote the connected components of $B_{2\varrho}(\zz) \cap \{u > 0\}$ that contain the sets 
\[\{x \in B_{2\varrho}(\zz) :  e_{2\varrho, \zz} \cdot (x - \zz) > \bar \eps \varrho\} \qquad \mbox{and}\qquad \{x \in B_{2\varrho}(\zz) :  e_{2\varrho, \zz} \cdot (x - \zz) < - \bar \eps  \varrho\} ,
\] 
respectively.  Suppose by contradiction that $U_+ \cap U_- = \varnothing$ and define $\bar u_\pm := u {\mathbbm{1}_{U_\pm}}$. 
By \Cref{lem:cleanball}, if $\bar \eps $ is chosen small enough, we have $u = \bar u_+ + \bar u_-$ inside $B_{\varrho}(\zz)$ (that is, there are no other connected components of $\{u>0\}$). Thus,   \Cref{lem:X-blowdown} and \Cref{cor_closetoV_disc_reg} imply
\[
\varrho|D^2 u|\le C \bar \eps  \qquad\text{in}\quad \{u > 0\}\cap B_{\varrho/2}(\zz). 
\]
In particular
\[
\eta_0^3 \le \int_{B_{\rb(\zz)}(\zz)\cap \{u > 0\}}|D^2 u|^3 \le \int_{B_{\varrho/2}(\zz)\cap \{u > 0\}} |D^2 u|^3 \,dx
\le C\bar \eps^3,
\]
which is a contradiction for $\bar \eps$ small enough. 
This proves that $U_+=U_-$ are the same connected component of $B_{2\varrho}(\zz)\cap \{u > 0\}$. Since by \Cref{lem:cleanball} we have already seen that any other connected component of $B_{2\varrho}(\zz)\cap \{u > 0\}$ lies outside of $B_\varrho(\zz)$, we obtain the desired result. 
\end{proof}

We finish this subsection with the following two related lemmas:

\begin{lem}\label{lem_aux_rb1}
For any $M\ge 1$ the following holds. 
    Given $\zz\in \cZ$ and $\varrho = M \rb(\zz)$, for all $\zz' \in \cZ \cap B_{\varrho} (\zz)$ we have
    \[
    \rb(\zz')\ge \frac{\rb(\zz)}{C_M}\qquad\text{and}\qquad \|D^2 u \|_{L^\infty(\{u>0\}\cap B_\varrho(\zz))} \le \frac{C_M}{\rb(\zz)},
    \]    
    for some  $C_M$ depending only on $M$.
\end{lem}
\begin{proof}
    For $M\le 2$, the comparability of the neck radii and the Hessian estimate follows from \eqref{eq:doubling-claim-10}--\eqref{eq:doubling-claim-20} and  \Cref{cor:Hess-bd-neck}, respectively.
    So let us assume $M>2$.
    
    Suppose for the sake of contradiction that there is  $\zz' \in \cZ \cap B_{\varrho} (\zz)$ with $
    \rb(\zz')< \frac{\rb(\zz)}{K}$, for $K$ sufficiently large to be chosen later (depending on $M$). 
    Then, since $3\varrho = 3M \rb(\zz) \ge 6\rb(\zz) \ge 6 K \rb(\zz')$, \Cref{lem:X-blowdown} implies that
\[
\big\|u - V_{\zz', e}
\big\|_{L^\infty(B_{3\varrho}(\zz'))}
\leq \eps(K) \varrho/2, 
\]
for some $e\in \mathbb{S}^2$, where $\eps(K)\downarrow 0$ as $K\uparrow \infty$. Therefore, thanks to  \Cref{treeprelim1}(a),
\[
\big\|u - V_{\zz, e}
\big\|_{L^\infty(B_{2\rb(\zz)}(\zz))}
\leq \eps(K) \varrho = \eps(K) M \rb(\zz).
\]
Thus, recalling \Cref{cor:Hess-bd-neck} and \Cref{cor_closetoV_disc_reg2}, if we choose $K$ large so that $\eps(K) M$ is sufficiently small we get
\[
\rb(\zz) |D^2 u|\le \frac{\eta_0}{100} \quad \mbox{in }B_{\rb(\zz)}(\zz)\cap \{u>0\}.
\]
Integrating in  $B_{\rb(\zz)}(\zz)\cap \{u>0\}$ we reach a contradiction with the definition of $\rb(\zz)$.

The second point is then a consequence of \Cref{cor:Hess-bd-neck} and \Cref{lem:Hess-decay}.
\end{proof}

\begin{lem}\label{lem_aux_rb2}
There exists $M_\circ>0$ universal such that if $\zz\in \cZ$ and $R\ge M_\circ\rb(\zz)$, then $\rb(\zz')\le \frac{R}{8}$ for all $\zz'\in \cZ\cap B_{3R/4}(\zz)$.
\end{lem}
\begin{proof}Let $\eps_\circ>0$ be a small constant to be fixed, and apply \Cref{lem:X-blowdown} to find $M_\circ>0$ such that, for $R\ge M_\circ\rb(\zz)$,
there exists $e\in \mathbb S^2$ such that 
\begin{equation}\label{Linfty000bis}
  \|u-V_{\zz,e}\|_{L^\infty(B_R(\zz))}\le  \tfrac{1}2\eps_\circ R\qquad \mbox{and}\qquad  \|u-V_{\zz',e}\|_{L^\infty(B_{R/4}(\zz'))}\le  \eps_\circ R,
\end{equation}
where the second bound follows from \Cref{treeprelim1}(a).

Now, assume by contradiction that $\rb(\zz') \ge \frac{R}{8}$. Then  \Cref{lem_aux_rb1}  implies that $|D^2 u|\le C/R$ in $B_{2R}(\zz')\cap \{u>0\}$,  with $C$ universal.  
By \Cref{cor_closetoV_disc_reg2} this gives $\rb(\zz) |D^2u|\le C\eps_\circ$ in $B_{\rb(\zz)}(\zz)\cap\{u>0\}$,
which integrated over $B_{\rb(\zz)}(\zz)$ contradicts the definition of $\rb(\zz)$ if $\eps_\circ$ is chosen small enough. 
\end{proof}

\subsection{Ball tree: `soft' geometric description of the zero set} 
\label{ssec:tree}

The goal of this section is to show  \Cref{prop:geomtree} below, which shall be very useful in the sequel.
To state it, we first recall the notion of {\em rooted tree}:

\begin{defn}[rooted tree]\label{defrootedtree}
    Let $\mathcal N$ be some given (finite, for simplicity) set. The elements $\nu \in \mathcal N$ will be called {\em nodes}. 
    Suppose that there exist a distinguished node $\nu_0\in \mathcal N$ (the {\em root}) and a map $p:\mathcal N\setminus \{\nu_0\}\to \mathcal N$ (the {\em predecessor map}) for which the following property holds: for all $\nu \in \mathcal N$ there is $\ell\in \N_{\geq 1}$ such that $p^\ell(\nu) =\nu_0$. 
    We then call the pair $(\mathcal N,p)$ a {\em rooted tree}. 

    Notice that $(\mathcal N,p)$ becomes naturally `graded' or `stratified' as follows: $\mathcal N= \bigcup_{\ell\ge 0} \mathcal N^{(\ell)}$ where $\mathcal N^{(0)} := \{\nu_0\}$ and $\mathcal N^{(\ell)} : = \{\nu\in \cN \ : \ p^\ell(\nu) =\nu_0\}$.  Notice also that, by definition, $p$ maps $\mathcal N^{(\ell)}$ to a subset of $\mathcal N^{(\ell-1)}$ (here $\ell\ge1$).

    Given $\nu\in \mathcal N$ we put ${\rm desc}(\nu) := p^{-1}(\{\nu\})$ and call it the {\em descendants} of $\nu$. Nodes $\nu$ with ${\rm desc}(\nu)= \varnothing$ are called {\em leaves} or {\em terminal nodes}.  Nodes $\nu$ with ${\rm desc}(\nu)\neq \varnothing$ are called {\em internal} or {\em branching nodes}.
\end{defn}

Intuitively, a node will be a given large ball. Then, the free boundary inside the node will be covered by the node's descendants in the next (smaller) scales, and such branching takes place  in balls that are large with respect to neck radii. In other words, one keeps zooming in until a neck or two regular phases are seen at the threshold scale, while keeping track of the intermediate balls, the closedness of $u$ to a vee as well as the tilting. See \Cref{fig:pictree}.

\begin{figure}[h!]
    \centering
    \includegraphics[width=0.8\textwidth]{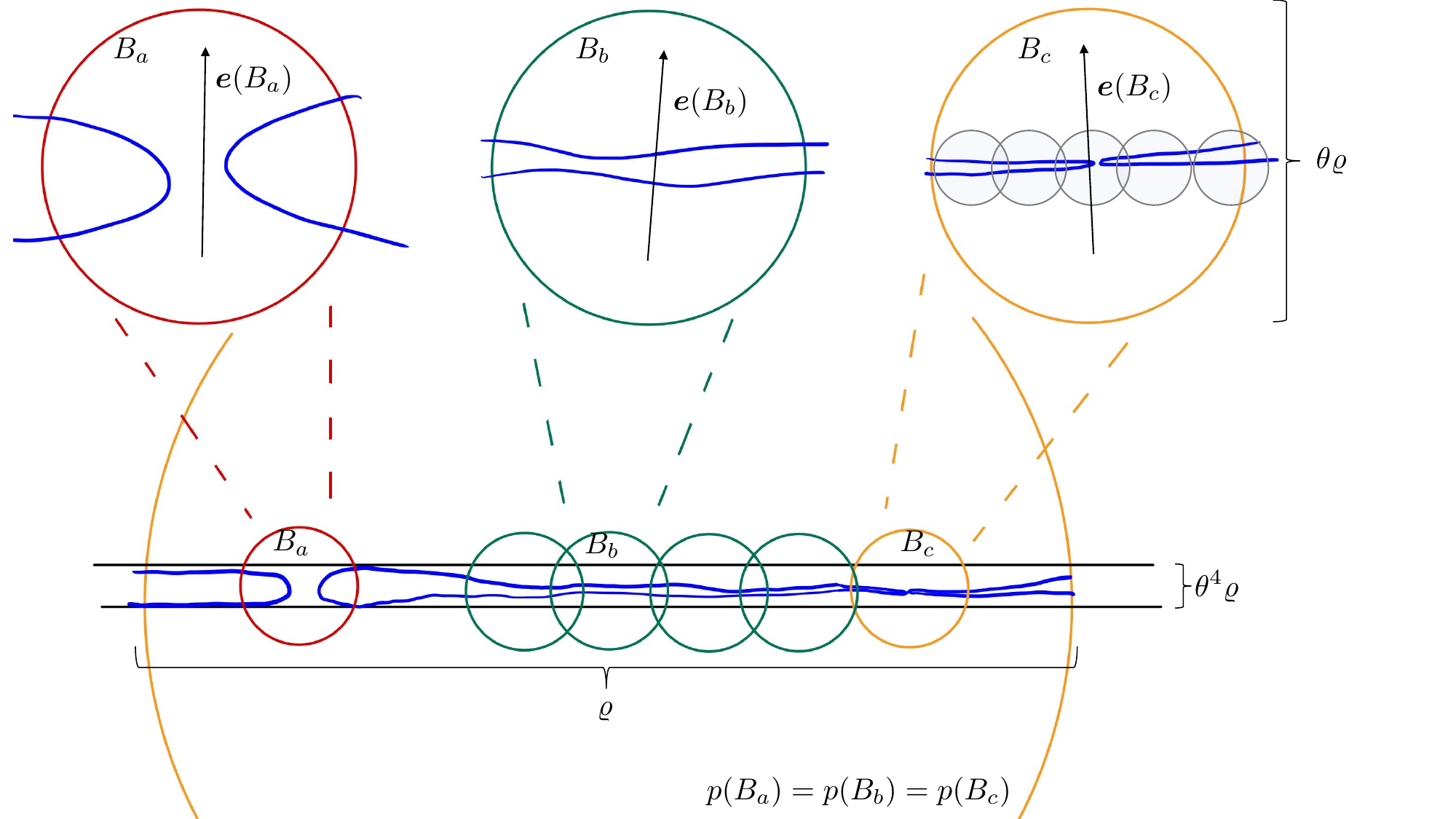}
    \caption{Illustration of branching structure in ball tree: \Cref{prop:geomtree} and  \Cref{defiregball}. From left to right: a neck-type terminal ball, a regular terminal ball, and a branching ball.}
    \label{fig:pictree}
\end{figure}
We can now give the following result concerning the geometric structure of $\{u>0\}$ (recall the definition of Slab in \eqref{eq:Slab}): 
\begin{prop}\label{prop:geomtree}
    There exists a small universal constant $\theta_\circ>0$ such that, for any given $\theta\in (0, \theta_\circ)$, there exists $M= M(\theta)\ge 1$ (large) such that the following holds true. 
    
    For any given $\zz\in \cZ$ and $R> M\rb(\zz)$, there exist:
    \begin{itemize}
        \item A finite collection $\mathcal N$ of balls of $\R^3$ with  $B_R(\zz)\in \mathcal N$. 
        \item A predecessor map $p:\mathcal N\setminus \{B_R(\zz)\}\to \mathcal N$ such that $(\mathcal N, p)$ is a rooted tree with root $B_R(\zz)$.
        \item A map $\boldsymbol e: \cN \to \mathbb S^2$ called {\em polarity map}.
    \end{itemize}
    The previous objects satisfy the following properties:
    \begin{enumerate}
        \item Every ball (or node) $B\in \mathcal N^{(\ell)}$, $\ell \ge 0$,  has radius $\varrho_\ell := \theta^\ell R$ and is centered at some point in $\{u=0\}\cap B_{R}(\zz)$.
        
        \item For every node $B=B_{\varrho} (y) \in \mathcal N^{(\ell)}$ (so that $\varrho= \varrho_\ell$) we have  
        \begin{equation}\label{closeness1}
            \big\| u - V_{y,e} \big\|_{L^\infty(B_{2 \varrho}(y))}\le  \theta^4 \varrho,
        \end{equation}
        where $e=\boldsymbol e(B)$ is the polarity of $B$.
        In particular, 
        \begin{equation}\label{closeness2}
            \{u=0\}\cap B_{2\varrho}(y) \subset \{x\in \R^3\ : \ |e\cdot(x-y)|\le  \theta^4 \varrho\}.
        \end{equation}

        \item A ball  $B= B_\varrho(y)$ in  $\mathcal N$ is an internal or  branching node whenever
        \begin{equation}\label{branchingcondition}
            \mbox{there exists $\zz\in B_{2\varrho}(y)\cap \cZ$ such that $  M\rb(\zz)\le \varrho$.}
        \end{equation}
               Otherwise, the ball is a terminal node.

        \item For every branching node $B \in \mathcal N$,  each of its descendants in ${\rm desc}(B)$ is centered at some point in $\{u=0\}\cap {\rm Slab}\!\left(B, e, \theta^4\right)$, where $e= \boldsymbol e(B)$ is the polarity of $B$. 
        Moreover, the union of the balls in ${\rm desc}(B)$ are a ``Vitali covering'' of ${\rm Slab}\!\left(B, e,   \theta^2\right)$ (namely, they cover ${\rm Slab}\!\left(B, e,   \theta^2\right)$ and the balls with the same centers and radii scaled by a factor $1/4$ are pairwise disjoint).
        In particular, the number of balls in ${\rm desc}(B)$ is bounded by $2^8\theta^{-2}$.

        \item For any $B'\in {\rm desc}(B)$ we have $|\boldsymbol e(B)- \boldsymbol e(B')|\le  \theta^3$. 
  
    \end{enumerate}
\end{prop}

\begin{proof} 
We will construct the tree $(\mathcal N, p)$ using an iterative procedure. The process begins at the root $B_R(\zz)$, which will always serve as a branching node. For any given node, we will define the criteria that determine whether it is branching or terminal, along with the procedure for constructing its descendants in the branching case.

This construction is divided into two steps:

\vspace{3pt}
\noindent{\bf Step 1.} We present a claim that acts as a fundamental step in the construction process. It governs the selection of the constant $M$ and outlines the procedure for determining descendants from a branching node.

\begin{quote}
    {\bf Claim.} For any given $\theta>0$ sufficiently small, there is $M=M(\theta)$ such that the following holds. 
    Suppose that $u(y)=0$ and $B= B_\varrho(y)\subset \R^3$ is some ball such that \eqref{branchingcondition} holds. Assume in addition that $e\in \mathbb S^2$ is a unit vector such that \eqref{closeness1} holds. 
    Then, there exists a collection of points $\{y_j\}_{1\le j \le N}$ in  $\{u=0\}\cap{\rm Slab}\!\left(B_\varrho(y), e, \theta^4\right)$ satisfying the following properties:
    \begin{enumerate}[(i)]
    \item  The balls $\{B_{\theta \varrho/4}(y_j)\}_{1\le j\le N}$ are disjoint. In particular, the number of points $N$ is bounded by $2^8\theta^{-2}$.
    \item The balls $\{B_{\theta \varrho}(y_j)\}_{1\le j\le N}$ cover ${\rm Slab}\!\left(B_\varrho(y), e, \theta^2\right)$
    \item There exists $\tilde e\in \mathbb S^{2}$ with $|e-\tilde e|\le \theta^3$ such that for all $1\le j\le N$ we have
        \begin{equation}\label{closeness1j}
            \big\| u(x) -|\tilde e\cdot(x-y_j)| \big\|_{L^\infty(B_{2\theta\varrho}(y_j))}\le  \theta^4 (\theta \varrho) = \theta^5\varrho,
        \end{equation}
    \end{enumerate}
\end{quote}

Let us prove this claim. By assumption \eqref{branchingcondition}, there exists $\zz\in B_{2\varrho}(y)\cap \cZ$ such that $   M\rb(\zz)\le  \varrho$. Within the setup of \Cref{lem:X-blowdown}, choose  $M=M(\theta)$ such that
\[
8\omega\big(M_*/M \big) \le \theta^6.
\]
Then,  \Cref{lem:X-blowdown} guarantees the existence of $\tilde e\in\mathbb S^2$ 
  such that 
\[
\big\|
u(x)-|\tilde e\cdot (x-\zz)|
\big\|_{L^\infty(B_{4\varrho}(\zz))}
\leq  \omega\big(M_*/M \big) 4 \varrho \le \tfrac 1 2 \theta^6 \varrho.
\] 
Since changing the sign to $\tilde e$ does not change the previous bound, we choose the sign giving $e\cdot \tilde e\ge0$.

Then using \Cref{treeprelim1}(a) (with $y_1=\zz$, $y_2=y$, $r_1=4\varrho$, $r_2 = 2\varrho$, $\eps=\theta^6/8$)  
\begin{equation}\label{closenesseps}
 \big\| u(x) -|\tilde e\cdot(x-y)| \big\|_{L^\infty(B_{2\varrho}(y))} \le \theta^6\varrho, 
\end{equation}
  therefore
\begin{equation}\label{closenesseps2}
    \{u=0\}\cap B_{2\varrho}(y) \subset \{x\in \R^3\ : \ |\tilde e\cdot(x-y)|\le  \theta^6 \varrho\}.
\end{equation}
Next, define the set \( \{y_j\}_{1 \le j \le N} \) as a subset of \( \{u = 0\} \cap \text{Slab}\left(B_\varrho(y), \tilde{e}, \theta^6\right) \) such that the balls \( B_{\theta \varrho / 4}(y_j) \) are pairwise disjoint. Furthermore, this subset is chosen to be maximal with respect to this disjointness property. Observe also that by \eqref{closeness2},  we have $$\{y_j\}_{1\le j \le N}\subset \{u=0\} \cap B_\varrho(y)\subset {\rm Slab}\!\left(B_\varrho(y), e, \theta^4\right).$$

Now, maximality implies (by a usual Vitali-type argument) that the triple balls are a cover:
\begin{equation}\label{jhgiohoiwhohith1}
   \{u=0\}\cap{\rm Slab}\!\left(B_\varrho(y), \tilde e, \theta^6\right) \subset 
\bigcup_{1\le j\le N} B_{3\theta \varrho/4}(y_j). 
\end{equation}
Also, by \eqref{closenesseps} and \Cref{treeprelim1}(b) we know  
\[
 {\dist(x, \{u = 0\}) \le 8 \theta^6 \varrho \qquad\text{for all}\quad x\in {\rm Slab}\!\left(B_\varrho(y), \tilde e, \theta^6\right)},
\]
i.e. ${\rm Slab}\!\left(B_\varrho(y), \tilde e, \theta^6\right) \subset \{u=0\} + \overline{B_{8\theta^6 \varrho}}$, upgrading \eqref{jhgiohoiwhohith1} to

\begin{equation*}%
   {\rm Slab}\!\left(B_\varrho(y), \tilde e, \theta^6\right) 
   \subset \bigcup_{1\le j\le N} B_{4\theta \varrho/5}(y_j),
\end{equation*}
provided $\theta$ is small so that $3\theta/4 + 8\theta^6 < 4\theta/5$.

Also for $\theta$ small enough (such that $2\theta^2+4\theta/5 \leq 5\theta/6$) %
\begin{equation}\label{jhgiohoiwhohith1upgrade}
{\rm Slab}\!\left(B_\varrho(y), \tilde e, 2\theta^2\right) \subset 
\bigcup_{1\le j\le N} B_{5\theta \varrho/6}(y_j). 
\end{equation}
Since the intersections of the balls $B_{\theta \varrho/4}(y_j)$ with the plane $\{x\in \R^3  \ :  \ \tilde e\cdot(x-y) =0\}$ are disjoint disks of radius $\ge \theta \varrho/8$, and they are all contained in ${\rm Slab}\!\left(B_{(1+\theta/3)\varrho}(y), \tilde e, 0\right)$ a simple comparison of areas gives 
\[
N (\theta\varrho/8)^2 \le (1+\theta/3)^2\varrho^2.
\]
Therefore since $\theta/3\le1$ we obtain $N \theta^2 \le 2^8$ as claimed. We have thus established (i).

To establish (iii) we observe first that (repeating similar triangle inequality arguments as above) from \eqref{closenesseps} and using $|\tilde e \cdot(y_j-y)| \le \theta^6\varrho$ we obtain that  \eqref{closeness1j} is automatically satisfied for all $j$  provided $2\theta^6\le \theta^5$. 

Similarly,  combining  \eqref{closeness1} and  \eqref{closenesseps} using the triangle inequality
we obtain 
\[
 \big\| V_{y,e} - V_{y,\tilde e} \big\|_{L^\infty(B_{2\varrho}(y))} = \big\|\, |e\cdot(\,\cdot\,-y)| -|\tilde e\cdot(\,\cdot\,-y)|\, \big\|_{L^\infty(B_{2\varrho}(y))}\le (\theta^4 + \theta^6)\varrho < 2\theta^4 \varrho.
\]
Recalling $e\cdot \tilde e\ge 0$ this implies $|e-\tilde e|\le 2\theta^4$, which is less than the claimed $\theta^3$ ($\theta$ is small). 

Finally, (ii) follows from \eqref{jhgiohoiwhohith1upgrade} 
together with $|e-\tilde e|\le \theta^3$. This finishes the proof of the claim.

\vspace{3pt}
\vspace{3pt}

\noindent {\bf Step 2.} We now use the claim to construct the tree $(\mathcal N, p)$.

We start by defining the root $\mathcal N^{(0)} := \{ B_R(\zz) \}$. Since $R > M \rb(\zz)$ by assumption, the conditions of the claim are satisfied for $B_R(\zz)$, thanks to \Cref{lem:X-blowdown}. This allows us to apply the branching procedure from the claim to $B_R(\zz)$, producing a finite collection of balls $\{ B_{\theta R}(y_j) \}_{1 \le j \le N}$, each centered at a point in $\{u = 0\} \cap {\rm Slab}\!\left(B_R(\zz), e, \theta^4\right)$ and satisfying the covering and disjointness properties of the claim in Step~1.

Next, for each branching node $B_\varrho(y) \in \mathcal N^{(k)}$ (at level $k$ of the tree), we apply the claim to generate its descendants, forming the next generation of nodes $\mathcal N^{(k+1)}$. If a ball satisfies the branching condition \eqref{branchingcondition}, it branches into a finite collection of descendants, where each ball in the descendant set satisfies the same geometric properties as the initial node. If a node fails the branching condition \eqref{branchingcondition}, it becomes a terminal node, and no further descendants are generated.

The predecessor map $p: \cN \setminus \{ B_R(\zz) \} \to \mathcal N$ is defined by setting $p(B') = B$ whenever $B'$ branches out from $B$. This establishes the rooted tree structure of $(\mathcal N, p)$, where $B_R(\zz)$ is the root.

The polarity map $\boldsymbol{e}: \cN \to \mathbb S^2$ is defined iteratively: for the root $B_R(\zz)$, we assign polarity $e$ given by \Cref{lem:X-blowdown}, and
for each descendant $B_{\theta \varrho}(y_j)$ of a branching node $B_\varrho(y)$, we assign the polarity $\tilde{e}$ from the claim, satisfying $|\boldsymbol{e}(B) - \boldsymbol{e}(B')| \le \theta^3$ for any descendant $B'$ of $B$.

Note that for the root $B_R(\zz)$, we may also arbitrarily assign the polarity $-e$. Once this sign is chosen, however, the signs of the polarities for all descendants are uniquely determined.

The iterative process continues until all nodes in the tree are either terminal or have their descendants constructed. The covering, disjointness, and approximation properties of the descendants are guaranteed by the claim, which ensures that every ball in $\mathcal N$ satisfies the conditions in \Cref{prop:geomtree}.

Finally, the number of descendants at each branching node is bounded by $2^8 \theta^{-2}$, and the radii of the balls decrease geometrically by a factor of $\theta$ at each generation. This ensures that the process terminates after a finite number of steps (since $\rmin > 0$), yielding a well-defined, finite tree structure.
\end{proof}

The following definition and lemmas extract the relevant analytic information from the rooted tree constructed in \Cref{prop:geomtree} to be used  in the following sections: 
\begin{defn}\label{defiregball}
Given $\theta > 0$ (sufficiently small), let $M = M(\theta)$ be the constant provided by \Cref{prop:geomtree}. Suppose $\zz \in \cZ$ and $R > M \rb(\zz)$. Let $(\mathcal N, p)$ denote the ball tree rooted at $B_R(\zz)$, and let $\boldsymbol{e}$ be the associated polarity map, both as described in \Cref{prop:geomtree}.
We partition $\cN$ into two sets: 
\[
\cN = \cI \cup \cT,
\]
where $\mathcal I$ consists of the {\em internal nodes} (branching balls), and $\mathcal T$ consists of the {\em terminal nodes} (balls that do not branch further).
A terminal ball $B = B_\varrho(y) \in \mathcal T$ is called {\em regular} if
\[
B_{2\varrho}(y) \cap \cZ = \varnothing.
\]
The set of regular terminal balls will be denoted as $\cT^{\rm reg}$. The non-regular terminal balls will be called {\em neck balls}.
We denote them by $\mathcal T^{\rm neck}$, so that $\cT= \cT^{\rm reg}\cup \cT^{\rm neck}$.
\end{defn}

The following lemma describes the free boundary inside regular terminal balls as two flat graphs contained in a thin slab perpendicular to the ball's polarity direction.

\begin{lem}\label{lemregball}
   In the setting of \Cref{defiregball}, let $\theta \in (0, \theta_\circ)$, where $\theta_\circ > 0$ is the universal constant provided by \Cref{prop:geomtree}.
   For every regular terminal ball $B = B_\varrho(y)\in \cT^{\rm reg}$, the set $\{ u > 0 \} \cap B_{3\varrho /2}(y)$ can be written as
\[
\{ u > 0 \} \cap B_{3\varrho/2}(y) = B^{(+,3/2)} \cup B^{(-,3/2)} ,
\]
where $B^{(+,3/2)} $ and $B^{(-,3/2)} $ are two disjoint connected components of $\{ u > 0 \}\cap B_{3\varrho/2}(y)$, characterized by containing the points $y \pm \varrho e$, where $e = \boldsymbol{e}(B)$ is the polarity of $B$.
In addition, we have:
\begin{equation}\label{gradvspolar}
    |\nabla u(x)-e|\le \theta^3 \quad \forall \,x\in B^{(+,3/2)} \qquad \mbox{and}\qquad  |\nabla u(x)+e|\le \theta^3 \quad \forall \,x\in B^{(-,3/2)}.
\end{equation}
Moreover, the two free boundaries $\partial \{u > 0\} \cap \partial B^{(\pm, 3/2)}$ are flat $C^{1,1}$ graphs. More precisely, if we choose an Euclidean coordinate system $(X_1, X_2, X_3)$ with origin at $y$ and $X_3$ pointing in the direction of $e$, we have
\[
B^{(+,3/2)}  = \{ X_3 > g^{(+)}(X_1, X_2) \} \cap B_{3\varrho/2}(y) \quad \text{and} \quad B^{(-,3/2)}  = \{ X_3 < g^{(-)}(X_1, X_2) \} \cap B_{3\varrho/2}(y),
\]
where the functions $g^{(\pm)} : D_{3\varrho/2} \to \mathbb{R}$, with $D_{3\varrho/2}$ being the disk $\{ X_1^2 + X_2^2 < (3\varrho/2)^2 \}$ in $\mathbb{R}^2$, are ordered ---that is   $g^{(-)}< g^{(+)}$--- and satisfy the estimates:
\[
\|g^{(\pm)}\|_{L^\infty(D_{3\varrho/2})} + \varrho^2 \|D^2 g^{(\pm)}\|_{L^\infty(D_{3\varrho/2})}
\leq \theta^3 \varrho.
\]
\end{lem}

\begin{proof}
It follows by combining \Cref{prop:geomtree} with \Cref{lem:Hess-decay} and \Cref{cor_closetoV_disc_reg}.
Indeed, if $B = B_\varrho(y)$ is a regular terminal ball then, by definition, $B_{2\varrho}(y)\cap \cZ =\varnothing$.
Hence, by \Cref{lem:Hess-decay} we obtain 
\[
|D^2 u|\le \frac{C_1}{\varrho} \qquad \mbox{in }\{u>0\}\cap B_{7\varrho/4}(y), 
\]
with $C_1$ universal. 
Recalling  \eqref{closeness1}---which holds thanks to \Cref{prop:geomtree}(2)--- we can use \Cref{cor_closetoV_disc_reg2} (with a covering argument) to conclude.
\end{proof}

\begin{defn}\label{def:Bpm_Omegapm}
For given $B = B_\varrho(y)\in \cN\setminus \cT^{\rm neck}$ we define $B^{(+)}$, $B^{(-)}$, as follows:  

\begin{itemize}
    \item If $B\in \mathcal I$ is an  internal ball and $e = \boldsymbol{e} (B)$, we define (the regular regions)
\[
    B^{(+)} := \{x \in B_\varrho(y) \ :\   e \cdot (x - y) > \theta^2\varrho\},  \quad 
    B^{(-)} := \{x \in B_{ \varrho}(y) \ :\   e \cdot (x - y) <- \theta^2\varrho\}.
\]
Similarly, for given $\lambda\in [1,3/2]$ we define
\[
 B^{(\pm,\lambda)} := \{x \in B_{\lambda\varrho}(y) \ :\  \pm e \cdot (x - y) > \theta^2\varrho\}.
\]

    \item If $B\in \cT^{\rm reg}$ is a regular terminal ball we define
    \[ 
    B^{(+)}:=  B^{(+,3/2)}  \cap B,  \quad B^{(-)}:=  B^{(-,3/2)} \cap B,
    \]   
    where $B^{(+,3/2)} $ and $B^{(-,3/2)} $ are as in \Cref{lemregball}. Also, for  a given $\lambda\in [1,3/2]$ we  define
    \[B^{(\pm,\lambda)}:=  B^{(\pm ,3/2)}  \cap B_{\lambda \varrho}(y).\]
\end{itemize}
Finally, given a ball tree $\cN$ with root $B_R(\zz)$, we define the two subsets $\Omega^{(+)}=\Omega^{(+)}(B_R(\zz))$ and $\Omega^{(-)}=\Omega^{(-)}(B_R(\zz))$ as
\[
\Omega^{(\pm)} :=  \bigcup\{B^{(\pm)} \ : \  B\in \mathcal I \cup \mathcal T^{\rm reg} \}.
\]
\end{defn}

\begin{remark}[Reversed polarity map]\label{reversed_polarity}
Notice that if $(\mathcal{N}, p)$ and $\boldsymbol{e}: \mathcal{N} \to \mathbb{S}^2$ are the tree and polarity map constructed in \Cref{prop:geomtree}, then replacing the map $\boldsymbol{e}$ with $-\boldsymbol{e}$ results in a new polarity map that satisfies exactly the same properties. 
In other words, we can always change the sign of the polarity at one node (e.g., the root), but this change must be propagated to all other nodes accordingly. 

It is also useful to observe that for each  ball $B$ the set $B^{(-)}$ defined using  the polarity $\boldsymbol{e}$ is the same as $B^{(+)}$ for the polarity $-\boldsymbol{e}$. Thus, the set $\Omega^{(-)}$ defined using the polarity $\boldsymbol{e}$ is the same as $\Omega^{(+)}$ for the polarity $-\boldsymbol{e}$.
\end{remark}

In the previous definition, we have a lower bound on the gradient inside $\Omega^{(\pm)}$:
 
\begin{lem}
    \label{lem:grad_lower_bound}
    In the setting of Definitions~\ref{defiregball} and \ref{def:Bpm_Omegapm}, let $\theta \in (0, \theta_\circ)$.    For all $B\in\cI \cup \cT^{\rm reg}$ we have:
\begin{equation}\label{gradvspolar2}
    |\nabla u(x)-e|\le C\theta^2 \quad \forall \,x\in B^{(+,3/2)} \qquad \mbox{and}\qquad  |\nabla u(x)+e|\le C\theta^2 \quad \forall \,x\in B^{(-,3/2)},
\end{equation}
where $e = \boldsymbol e(B)$ and $C$ is universal. In particular, $|\nabla u|\ge \tfrac{9}{10}$ in $\Omega^{(\pm)}$. 
\end{lem}
\begin{proof}
   For $B\in\cT^{\rm reg}$, this is implied by \eqref{gradvspolar}. For $B=B_\varrho(y)\in\cI$ this follows from interior estimates for the harmonic function $u$ using that 
$ 
\| u - V_{y,e} \|_{L^\infty(B_\varrho(y))} \le \theta^4\varrho 
$ 
and $B^{(\pm)} = \{x \ : \ \pm e\cdot(x-z) >  \theta^2\rho \}$.
\end{proof}

Notice also that, thanks to the tree structure, $\Omega^{(\pm)}$ cover the whole $\{u > 0\}\cap B_{R}(\zz)$ except $\cT^{\rm neck}$:
\begin{lem}
    \label{lem:Omegapmcovers}
    In the setting of Definitions~\ref{defiregball} and \ref{def:Bpm_Omegapm}, we have 
    \[
    \{u > 0\}\cap B_{R}(\zz) \subset \Omega^{(+)}\cup \Omega^{(-)} \cup \bigcup\cT^{\rm neck}.
    \]
\end{lem}
\begin{proof}
We reason by induction using the tree structure. Indeed, let 
$\mathcal N = \bigcup_{\ell\ge 0} \mathcal N^{(\ell)}$ be as in \Cref{defrootedtree} and define 
\begin{equation}
\label{eq:Omegaleqell}
\Omega^{(\ell, \pm )} : = \bigcup \{ B^{(\pm )} \  :\  B \in \mathcal N^{(\ell)} \cap (\mathcal I\cup \mathcal T^{\rm reg})\},\qquad\text{and}\qquad \Omega^{(\le \ell, \pm )}: = \bigcup_{k=0}^\ell \Omega^{(k, \pm )}.
\end{equation}
Then the result follows noticing that
\[
 B_{R}(\zz)\subset \Omega^{(\le \ell, \pm )}\cup \bigcup \cN^{(\ell)} \cup\bigcup \mathcal T.
\]
which is established for all $\ell\ge 0$  using \Cref{prop:geomtree}(4) ---more precisely, we use that the union of the balls in ${\rm desc}(B)$ covers all ${\rm Slab}\!\left(B, e, \theta^2\right)$---   and induction.
\end{proof}
 
\section{Estimating neck radii from symmetric excess} 
\label{sec:LpHessian}
The goal of this section is to estimate the size of neck radii using a test function introduced by Jerison and Savin in \cite{Jerison-Savin}.
More precisely,
given a global classical stable solution $u:\R^3\to \R$ to Bernoulli satisfying \eqref{eq:assump_glob_u},
following Jerison and Savin we define the functions 
$w$ and $\bc$ as
\begin{equation}
\label{eq:def_w}
w:=F(D^2u)=f(\lambda_1,\lambda_2, \lambda_3)
=\sqrt{\sum_{\lambda_i>0}\lambda_i^2+4\sum_{\lambda_i<0}\lambda_i^2},\qquad\text{and}\qquad \bc := w^{1/3}\qquad\text{in}\quad \overline{\{u > 0\}}, 
\end{equation}
where $\lambda_i$ are the eigenvalues of $D^2 u$ at a given point. We define (see \eqref{eq:stab-intro})
\begin{equation}
\label{eq:def_I}
\ccI(u, U) := \int_{\{u > 0\}\cap U} \bc \Delta \bc \, dx + \int_{\partial\{u > 0\}\cap U} \bc (\bc_\nu + H\bc) d\cH^2,
\qquad\text{for any open set }\quad U\subset \R^3. 
\end{equation}
We recall that $H $ denotes the mean curvature of the free boundary at a point, and in particular, $H(x) = -\partial^2_{\nu\nu} u(x)>0$ for $x\in \FB(u)$ (see \Cref{lem:vbounds}). Notice that, by the stability inequality \eqref{eq:stab-intro}, we have
\[
\ccI(u, B_1) 
\le C \norm[L^1(B_2\cap \set{u>0})]{|D^2u|^{\frac23}}.
\]
In addition, it follows from \cite{Jerison-Savin} that the two integrands in the definition of $\ccI$ are non-negative. In particular,  
\begin{equation}\label{eq:I monotone}
\ccI(u, U') \le \ccI(u, U)
\qquad\text{for any }\quad U'\subset U. 
\end{equation}
Finally, $\ccI$ enjoys the following scaling property:
\begin{equation}
\label{eq:rescaleI}
\ccI(u, U) = r^{\frac13} \ccI\left(u_{z,r}, \tfrac{1}{r} (U-z)\right),
\qquad\text{where}\quad u_{z,r}(x) := \frac{u(z+rx)}{r},\quad r > 0. 
\end{equation}
For notational convenience, for $\zz\in \cZ$ we define (recall \eqref{eq:varrho_def_intro}) 
\begin{equation}
    \label{eq:varrho_def_body}
\varrho_\zz(u, R) := \frac{1}{R} \ccI(u, B_R(\zz))^3. 
\end{equation}
The goal of this section is to establish the following two propositions. 
The first one provides a control on $\varrho_\zz$ by the symmetric excess (recall \eqref{eq:Ez_def_intro}):

\begin{prop}\label{prop:E_to_rho} For any  $\gamma\in (0, \frac49)$ there exists $C_\gamma>0$ such that 
\[ 
\varrho_\zz(u, R)\le C_{\cttg} \bE_\zz(u,  4R)^{3\cttg}\qquad\text{for all  $~\zz\in \cZ$ and $R>0$}.
\]
\end{prop}

The second proposition gives a control on the neck radii by $\varrho_\zz$: 

\begin{prop}
\label{prop:rb-l13}
There exists $C\ge 1$ universal such that, for any $\zz\in \cZ$ and $R \ge  \rb(\zz)$,
\begin{equation}\label{boundneckra1}
    \frac{\rb(\zz')}{R} \leq 
C \varrho_{\zz}(u, 2R)\qquad\text{for all}\ \zz'\in \cZ\cap B_{ 3R/2 }(\zz).
\end{equation}
\end{prop}

\noindent\textbf{Roadmap of the section (see also Subsection \ref{sssec:objects_properties}, point 5).}
We want to turn the (large-scale) flatness information encoded by the symmetric excess into a quantitative control of the neck radii introduced in \Cref{sec:neck_set}.
The bridge between these two notions is shown by estimating how much each neck ball contributes  to the left-hand side of the stability inequality \`a la Jerison--Savin, \eqref{eq:stab-intro}. The section is organized around the two complementary estimates above.

First, in Subsection~\ref{ssec:hessian_estimates} we prove \Cref{prop:E_to_rho}, showing that small symmetric excess forces the scale-invariant quantity $\varrho_\zz$ in \eqref{eq:varrho_def_body} to be small.
The key point, as carried out in \Cref{lem:joaquim}, is to exploit the multiscale ``ball tree'' description of the zero set from \Cref{sec:neck_set}: on each ball where the configuration is close to a vee, harmonic and boundary estimates yield quantitative integrability of the Hessian of $u$; summing these contributions along the tree and combining with stability inequalities transfers this control into an upper bound for $\varrho_\zz$.

Second, in Subsection~\ref{ssec:JerisonSavinNeck} we prove \Cref{prop:rb-l13}, which converts smallness of $\varrho_\zz$ into smallness of the neck radii.
Here one shows that each neck necessarily produces a definite contribution (a universal constant times the correct power of the neck radius) to the left-hand side of the stability inequality---in other words, we show a (localized, scaling-invariant) lower bound for $\varrho_\zz$ on neck balls;  see \Cref{thm:deltakappa} and \Cref{lem:I_neck_ball}.
This is obtained by revisiting the Jerison--Savin test function (more precisely, the computations that show that it is a subsolution of the linearized problem) and carefully keeping track of the nonnegative remainder terms in both the interior and boundary inequalities.

Combining \Cref{prop:E_to_rho} and \Cref{prop:rb-l13} yields the desired estimate of neck radii in terms of the symmetric excess, which will be a crucial input for the center/scale selection procedure in \Cref{sec:selection}.

\subsection{Hessian estimates in $L^{\cttg'}$: Proof of \Cref{prop:E_to_rho}}  
\label{ssec:hessian_estimates}

 We start by proving some estimates for positive harmonic functions in half-balls or flat-Lipschitz domains. The next two results follow from standard arguments, and we present their proofs in \Cref{appendixD} for the reader's convenience. 

\begin{lem}\label{lem_basic1}
Let $r > 0$, $n \ge 2$, and $w:B_{2r}\cap \{x_n > 0\}\to (0, \infty)$ a positive harmonic function. Then, denoting $x = (x', x_n)\in \R^{n-1}\times \R$, we have
\[
\int_{\{|x'|< 3r/2\}} w(x', t) \,dx' \le C\,  w( re_n) r^{n-1}\qquad\text{for all}\quad t\in (0, r),
\] 
where $C=C(n)$ is a dimensional constant.
\end{lem}

\begin{lem}\label{lem_basic2}
Let $n\ge 2$ and $w:B_{2r}\cap D \to (0, \infty)$ a positive harmonic function, where $D = \{x_n > \varphi(x')\}$ for some $\varphi:B_{2r}'\subset \R^{n-1}\to \R$ with 
\[ 
|\varphi|  + r|\nabla \varphi| \le c_\circ r.
\]
Let $\gamma' \in (0, \frac{1}{2})$. Then, for $c_\circ$ small enough depending only on $n$ and $\gamma'$, we have
\begin{equation}
\label{eq:twoeqs}
\quad r^{2\gamma'-n}\int_{B_{r}\cap D} |D^2 w|^{\gamma'} \,dx
\le C_{n,\gamma'} \, w(re_n)^{\gamma'},  
\end{equation}
for  some  $C_{n,\gamma'}$ depending only on $n$ and $\gamma'$. 
\end{lem}

With these two preliminary results at hand, we now focus on a series of estimates for our solution $u:\R^3\to \R$ to Bernoulli.
First, we prove that $L^1$-closeness to a vee implies $L^\infty$-closeness from below: 
\begin{lem}
\label{lem:boundbelow}
    Let  $y\in \R^3$, $e\in \mathbb{S}^2$, and $R> 0$. Then, the following implication holds for all $\eps\in (0,1)$:
    \[
     \frac{1}{R}\fint_{\{x\in B_R(y)\ :\ e\cdot(x-y)>R/8\}} (u(x)- e\cdot(x-y))_-\,dx\le \eps\qquad\implies\qquad   u(x) + C\eps R \ge e\cdot(x-y) \quad \forall \,x\in B_{3R/4}(y), 
    \]
    where $(t)_-= \max(-t,0)$ denotes the negative part of $t\in \R$ and $C$ is universal.

    In particular:
    \[
    \frac{1}{R}\fint_{B_R(y)} |u- V_{y, e}|\,dx \le \eps\qquad\implies\qquad  u+C\eps R \ge V_{y, e}\qquad\text{in}\quad B_{3R/4}(y).
    \]
\end{lem}
\begin{proof}
    Throughout the proof, we will assume without loss of generality that $\eps>0$ is sufficiently small, since otherwise the conclusion is trivially true (for some appropriately large $C$). 
    Defining 
    \[
        v(x) := u(x) - e\cdot(x-y),
    \]
   our goal is to show that
    \begin{equation}
        \label{eq:boundfrombelow}
        v\ge - C \eps R\qquad\text{in}\quad B_{7R/8}(y), 
    \end{equation}
    with $C$ universal. 
    
    Since $|\nabla u|\le 1$ we have $|\nabla v_-|\le 2$, so by our assumption and $L^1$-${\rm Lip}$ interpolation\footnote{Here we use a rescaled version of the classical interpolation inequality in  \Cref{lem:L1_Lip}.} we have, for $D_{R, y}:=B_R(y)\cap \{  x: e\cdot(x-y)>R/8\}$, 
$$
  \|v_-\|_{L^\infty(D_{R, y})}\le C R \eps^{1/4} \le R/8,  
$$
provided $\eps$ is small enough. In particular, $u > 0$ in $D_{R, y}$. 
Thus $v$ is harmonic and $v_-$ subharmonic in $D_{R, y}$. 
Hence, since by assumption
\[
\fint_{D_{R, y}} v_- \,dx \le C \eps R, 
\]
the $L^1$ to $L^\infty$ estimates for subharmonic functions give
$$v_-\le C \eps R\qquad \text{in}\quad B_{7R/8} (y) \cap \{x : (x-y)\cdot e  > R/6\}.$$
Since 
\[
\partial_e v(x) = \partial_e u(x) - 1 \le |\nabla u(x)|-1\le 0\qquad\text{in }\, \R^3,
\]
(recall that $|\nabla u|\le 1$), \eqref{eq:boundfrombelow} follows. 
\end{proof}

One of the cornerstones of this section is the following result, which strongly relies on the geometric information about $\{u > 0\}$ provided by the tree structure constructed in \Cref{prop:geomtree}. 
\begin{lem}
\label{lem:joaquim}
Let $n=3$. Given $\cttg'\in (0, \tfrac 12)$, there exist $M= M(\cttg')\ge 1$ and $C_{\cttg'}\ge 1$ (both large constants depending only on $\cttg'$) such that the following implication holds for every $\zz\in \cZ$,  $e\in \mathbb S^2$, and $R\ge M \rb(\zz)$:
\begin{equation}
\label{eq:joaq}
\frac{1}{R}
\fint_{B_{2R}(\zz)}
	\big|u - V_{\zz,e}\big| \,dx
\leq \eps\qquad\implies\qquad R^{\cttg'-3}\int_{B_{R}(\zz)\cap \{u>0\}} |D^2 u |^{\cttg'} \,dx 
\le C_{\cttg'} \eps^{\cttg'}.
\end{equation} 
\end{lem}

\begin{proof} 
Note that, by \Cref{lem:Stern_Zum} with $p=\cttg'$, 
\[
R^{\cttg'-3}\int_{B_{R}(\zz)\cap \{u>0\}} |D^2 u |^{\cttg'} \,dx 
\le C,
\]
with $C$ universal. Hence, to prove the lemma, we may assume without loss of generality that $\eps$ is sufficiently small.

As mentioned above, we will rely on the geometric information about $\{u > 0\}$ provided by the tree structure constructed in \Cref{prop:geomtree}. For a given $\cttg' \in (0, \tfrac{1}{2})$, we will choose an appropriately small $\theta > 0$ and consider the ball tree from \Cref{prop:geomtree}. The precise choice of $\theta$ in terms of $\cttg'$ will become clear later in the proof, but we can already note that $\theta \to 0^+$ as $\cttg' \to \left( \tfrac{1}{2} \right)^-$. 

 For $\theta>0$  small depending on $\cttg'$, we use \Cref{prop:geomtree} to obtain a ball tree $(\cN,p)$ rooted at $B_R(\zz)$ and polarity map $\boldsymbol e: \cN \to \mathbb S^2$. Also, we set $M(\cttg')$ equal to $M(\theta)$, the (large) constant from \Cref{prop:geomtree}.

We divide the proof into three steps:

\noindent {\bf Step 1:}  Define the function $v:B_{2R}(\zz) \to \R$ as 
\[
v(x) = u(x) - e\cdot(x-\zz).
\]
By \Cref{lem:boundbelow} we have
\begin{equation}\label{bybelow000}
    v \ge -C \eps R \quad \mbox{in } B_{3R/2}(\zz) \qquad\text{and}\qquad \{u = 0\}\cap B_{3R/2}(\zz) \subset {\rm Slab}\!\left(B_{3R/2}(\zz), e, C\eps\right)
\end{equation}
with $C$ universal. 
Since both $\varepsilon$ and $\theta$ are small enough, we can choose the sign of the polarity   $\boldsymbol{e}(B_R(\zz))$ so that
\begin{equation}\label{angle000}
    e \cdot \boldsymbol{e}(B_R(\zz)) \ge 1 - \frac{1}{100}.
\end{equation}
For a given ball $B = B_\varrho(y) \in \mathcal{N}$  we define its `positive pole' and `negative pole' as:
\[
\boldsymbol{P}^+(B) := y + \frac{\varrho}{2} \boldsymbol{e}(B)\qquad\text{and}\qquad \boldsymbol{P}^-(B) := y - \frac{\varrho}{2} \boldsymbol{e}(B).
\]
Define $U := B_{3R/2}(\zz) \cap \{u > 0\}$.  We claim that, if $w: U \to [0, \infty)$ is a non-negative harmonic function, then
\begin{equation}\label{thekeyequation111}
    \sum_{B \in \cN^{(\ell)}} w\big(\boldsymbol{P}^+(B)\big) \le (K\theta^{-2} )^\ell w\big(\boldsymbol{P}^+(B_R(\zz))\big),
\end{equation}
where $K$ is a universal constant; and analogously with $\boldsymbol{P}^-$.

Indeed, for any branching ball $B' \in \cN^{(\ell)}$, thanks to Harnack's inequality and \Cref{prop:geomtree}(4) we have
\[
\sum_{B \in \text{desc}(B')} w\big(\boldsymbol{P}^+(B)\big) \le  C(\theta\varrho_\ell)^{-3} \sum_{B \in \text{desc}(B')} \int_{B_{\theta\varrho_\ell/8}(\boldsymbol{P}^+(B))} w \,dx \le C (\theta\varrho_\ell)^{-3} \int_{B_{3\varrho_\ell/2}\cap S(B')} w\,dx,
\]
where $\varrho_\ell=\theta^\ell R$, $S(B') = \{\theta\varrho_\ell/16\le \boldsymbol e(B')\cdot (x-{\rm cent}(B')) \le 2\theta\varrho_\ell\}$, and ${\rm cent}(B')$ is the center of $B'$. 
Thus, applying   \Cref{lem_basic1} (integrated for $t\in \theta\varrho_\ell/16, 2\theta\varrho_\ell)$) using again the Harnack inequality we obtain
$$
    \sum_{B \in \text{desc}(B')} w\big(\boldsymbol{P}^+(B)\big) \le K \theta^{-2}   w\big(\boldsymbol{P}^+(B')\big).  
$$
This implies
$$
    \sum_{B \in \cN^{(\ell)}} w\big(\boldsymbol{P}^+(B)\big) \le K \theta^{-2}   \sum_{B \in \cN^{(\ell-1)}} w\big(\boldsymbol{P}^+(B)\big) \qquad \text{for all } \ell \ge 1,
$$
from which \eqref{thekeyequation111} follows.

\medskip 
    
\noindent {\bf Step 2:} We now show the existence of a constant \(C_{\cttg'}\), depending only on \(\cttg'\), such that:

\begin{enumerate}[(a)]
    \item For every internal or regular terminal ball \(B \in \cN^{(\ell)} \cap (\cI \cup \cT^{\rm reg})\), we have
    \begin{equation}\label{internal123}
        \int_{B^{(+)}} |D^2 w |^{\cttg'} \,dx \leq C_{\cttg'} (\theta^\ell R)^{3 - 2\cttg'} w\big(\boldsymbol{P}^+(B)\big)^{\cttg'}.
    \end{equation}
    
    \item For every neck-type terminal ball \(B \in \cN^{(\ell)} \cap \cT^{\rm neck}\), we have
    \begin{equation}\label{internal123_b}
        \int_{B \cap \{ u > 0 \}} |D^2 w |^{\cttg'} \,dx \leq C_{\cttg'} (\theta^\ell R)^{3 - 2\cttg'} w\big(\boldsymbol{P}^+(B)\big)^{\cttg'}.
    \end{equation}
\end{enumerate}
Indeed, provided that \(\theta > 0\) is chosen small enough, (a) follows from a direct application of \Cref{lem_basic2} to the non-negative harmonic function \(w\) in \(B^{(+,3/2)}\). 

The proof of (b) is much more involved and relies on Lemmas~\ref{lem:neckcenters_necks}, \ref{lem_aux_rb1}, and \ref{lem_basic2}, together with a suitable covering argument and a Harnack chain. We now provide the details of the proof.

Note first that, by the definition of a neck-type terminal ball, given \(B = B_\varrho(y) \in \cT^{\rm neck}\) there exists \(\zz \in B_{2\varrho}(y) \cap \cZ\) with \(M\rb(\zz) > \varrho\), $M = M(\theta)$ (recall that $M$ is the one given by \Cref{prop:geomtree}). 
Thus, setting \(\tilde \varrho := 4 \max\{M, \bar M\}\rb(\zz)\) (where $\bar M$ is the universal constant from \Cref{lem:neckcenters_necks}),  we have
\begin{equation}\label{hessbd123}
    |D^2 u| \leq \frac{C(M)}{\tilde \varrho} \quad \text{in } B_{3\tilde\varrho}(\zz) \cap \{ u > 0 \}\qquad \mbox{and}\qquad B\subset B_{\tilde\varrho}(\zz), 
\end{equation}
where the Hessian bound follows from \Cref{lem_aux_rb1}.
As a consequence of this (note that \(u = 0\) and \(\partial_\nu u = 1\) on $\partial \{ u > 0 \}$), we deduce the following:\\ For any given \(y' \in B_{2\varrho}(\zz) \cap \partial \{ u > 0 \}\), the connected component of \(B_{c\varrho}(y') \cap \{ u > 0 \}\) whose boundary contains \(y'\) satisfies the assumptions of \Cref{lem_basic2}  (in an appropriate Euclidean coordinate frame), where \(c > 0\) is a small constant depending only on \(M\) (i.e., depending only on \(\gamma'\)).

Thanks to this observation, we can argue as follows: first, we cover \(B_{2\tilde\varrho}(\zz) \cap \{ {\rm dist}(\cdot,\partial \{ u > 0 \}\leq c^2\tilde\varrho\} \) by a finite collection of balls \(\{B_{c\tilde\varrho/10}(y'_j)\}_{1\leq j \leq N=N(\cttg')}\) with \(y_j' \in B_{2\varrho}(\zz) \cap \partial \{ u > 0 \}\), and for each \(j\) we let \(y''_j\) to be a point such that \(B_{c\tilde\varrho/4}(y''_j) \subset B_{c\tilde\varrho}(y'_j) \cap \{ u > 0 \}\).

Then, we cover $B_{2\tilde\varrho}(\zz) \cap \{ {\rm dist}(\cdot, \{ u= 0 \}>c^2\tilde\varrho\}$ with finitely many balls $\{B_{c\tilde\varrho/8}(\hat y_j'')\}_{1\leq j\leq \hat{N}} $ such that $B_{c\tilde\varrho/4}(\hat y_j'')\subset \{ u > 0 \}$ and the balls $\{B_{c\tilde\varrho/4}(\hat y_j'')\}_{1\leq j\leq \hat{N}}$ have bounded overlapping. This guarantees that $\hat{N}$ is bounded by a constant depending only on $c$ (and thus only on $\gamma'$).
Now, by applying \Cref{lem_basic2}  inside each of the balls $\{B_{c\tilde\varrho/10}(y'_j)\}_{1\leq j \leq N}$, and interior harmonic estimates inside each of the balls $\{B_{c\tilde\varrho/4}(\hat y_j'')\}_{1\leq j \leq \hat{N}}$, since $\tilde \varrho$ is comparable to $\theta^\ell R$ we get
\begin{equation}
\label{eq:y''}
\int_{B \cap \{ u > 0 \}} |D^2 w |^{\cttg'} \,dx \leq C_{\cttg'} (\theta^\ell R)^{3 - 2\cttg'} \biggl(\sum_{1\leq j \leq N}w(y_j'')^{\cttg'}+\sum_{1\leq j\leq \hat{N}}w(\hat y_j'')^{\cttg'}\biggr).
\end{equation}
We now claim that, for any given point \(y'' \in B_{\tilde\varrho}(\zz)\) such that \(B_{c\tilde\varrho/4}(y'') \subset \{ u > 0 \}\), we have 
\begin{equation}\label{ahio174572}
    w(y'') \leq C_1 w(\boldsymbol{P}^+(B)),
\end{equation}
where \(C_1\) depends only on \(\gamma'\). Combining this claim with \eqref{eq:y''}, (b) follows immediately. So, we only need to prove  
\eqref{ahio174572}.

To this aim, we first apply \Cref{lem:neckcenters_necks} to find continuous path \(\Gamma : [0,1] \to B_{2\tilde \varrho(\zz)} \cap \{ u > 0 \}\) such that
\[
\Gamma(0) = y'', \quad \Gamma(1) = \boldsymbol{P}^+(B).
\]
Next, we show that there exists another continuous path \(\widetilde \Gamma : [0,1] \to B_{3\tilde \varrho(\zz)} \cap \{ u > 0 \}\) such that 
\[
\widetilde \Gamma(0) = y'', \quad \widetilde \Gamma(1) = \boldsymbol{P}^+(B), \quad \text{and} \quad \tilde \Gamma(s) + B_{c^4\tilde \varrho} \subset B_{3\tilde \varrho(\zz)} \cap \{ u > 0 \} \qquad \text{for all } s \in [0,1].
\]
Indeed, we can consider the vector field \(F = h(u) \nabla u\), where \(h : \mathbb{R} \to [0,\infty)\) is a smooth cut-off such that
\[
h(u) =
\begin{cases}
    1 & \text{if } u \leq c^3 {\tilde\varrho}, \\
    0 & \text{if } u \geq c^2{\tilde\varrho},
\end{cases}
\]
and define \(\widetilde \Gamma(s) := \Phi^F(\Gamma(s), c\tilde \varrho)\), where \(\Phi^F = \Phi^F(x,t)\) is the flow of $F$.\footnote{That is, $\Phi^F$ satisfies $\dot{\Phi}^F(x, t) =  F(\Phi^F(x, t))$ for $t > 0$, with $\Phi^F(x, 0) = x$. }

Finally, to establish \eqref{ahio174572}, we cover \(\widetilde \Gamma([0,1])\) by a number \(\tilde N\) (depending only on \(\gamma'\)) of balls of radius \(c^4\tilde\varrho/2\). This gives a `chain of balls' \(B_{c^4\tilde\varrho/2}(x_j)\) of length \(\tilde N\) such that 
\[
B_{c^4\tilde\varrho/2}(x_j) \cap B_{c^4\tilde\varrho/2}(x_{j+1}) \neq \varnothing, \quad 1\leq j<\tilde N.
\]
Then, applying Harnack inequality along the chain of balls \(\{B_{c^4\tilde\varrho}(x_j)\}_{1\leq j \leq \tilde N}\) (which have sufficient overlap between consecutive balls), we obtain \eqref{ahio174572}.

\medskip 
    
\noindent {\bf Step 3:} Consider the function  $w= v + C\eps R$.
As proved in Step 1, $w$ is non-negative in $B_{3R/2}(\zz)$. 
We then apply the estimates from Step 2 to $w$. 

More precisely, in $\Omega^{(+)}$ 
we sum over $\ell$ the estimate proved in (a),  recalling \Cref{def:Bpm_Omegapm} and that the number of descendants of each node in $\mathcal N$ is bounded by $2^8\theta^{-2}$ and thus $|\cN^{(\ell)}|\le (2^8\theta^{-2})^\ell$. In this way, by the concavity of $t\mapsto t^{\gamma'}$ and \eqref{thekeyequation111}, we obtain: 
\begin{equation}\label{boundhessiaOmega+}
\begin{split}
      \int_{\Omega^{(+)}} |D^2 v|^{\cttg'} \,dx &\le \sum_{\ell\ge0} \sum_{B\in \cN^{(\ell)}\cap (\cI\cup \cT^{\rm reg})}  \int_{B^{(+)}} |D^2 w |^{\cttg'} \,dx \\
      & \le C_{\cttg'}R^{3-2\cttg'}\sum_{\ell\ge0} 
      \theta^{(3-2\cttg')\ell} 
      \sum_{B\in \cN^{(\ell)}} w\big(\boldsymbol{P}^+(B)\big)^{\cttg'}
      \\
      &\le C_{\cttg'}R^{3-2\cttg'}\sum_{\ell\ge0} 
      \theta^{(3-2\cttg')\ell}  \,|\cN^{(\ell)}| \bigg(
      \frac{1}{|\cN^{(\ell)}|}
      \sum_{B\in \cN^{(\ell)}} w\big(\boldsymbol{P}^+(B)\big) \bigg)^{\cttg'}\\
      &\le C_{\cttg'} R^{3-2\cttg'}\sum_{\ell\ge0} 
      \theta^{(3-2\cttg')\ell}  \,((2^8\theta^{-2})^\ell) ^{1-\cttg'} 
      (K \theta^{-2})^{\cttg' \ell} 
      w\big(\boldsymbol{P}^+(B_R(\zz))\big) ^{\cttg'} 
      \\
      &\leq C_{\cttg'}R^{3-2\cttg'}\sum_{\ell\ge 0} \left( \theta^{1-2\cttg'} K^{\gamma'}2^{8(1-\gamma')}\right)^\ell \big(C\eps R\big)^{\cttg'}.
\end{split} 
\end{equation}
In the last line we used that, by the Harnack inequality and \eqref{bybelow000}, we have $w\big(\boldsymbol{P}^+(B_R(\zz))\big)\le C\eps R$, where $C$ is universal. Notice that, since $\cttg'<1/2$, we can choose $\theta= \theta(\cttg')$ sufficiently small so that $\theta^{1-2\cttg'} K^{\gamma'}2^{8(1-\gamma')}<1$ and the geometric series above converges to a constant depending only on $\cttg'$. 

The assumptions of the lemma do not change if we replace  $e$ by $-e$. But, by \eqref{angle000}, doing so reverses the polarity of the tree and therefore we obtain the same bounds over $\Omega^{(-)}$. 

Finally, we obtain a similar estimate for $\sum_{B\in \cT^{\rm neck}} \int_{B} |D^2 v|^{\cttg'} \,dx$ reasoning exactly as above but using (b) instead of (a).
Since $\{u>0\}\cap B_R(\zz) \subset \Omega^{(+)} \cup  \Omega^{(-)}\cup \big(\bigcup \cT^{\rm neck}\big)$ (by \Cref{lem:Omegapmcovers}), the proof is complete. 
\end{proof}

We can finally prove \Cref{prop:E_to_rho}.

\begin{proof}[Proof of \Cref{prop:E_to_rho}]
On the one hand, the Hessian estimate in \Cref{lem:joaquim} implies:
\begin{equation}\label{wthiowhiowhoi2}
R^{\cttg'} \fint_{B_{2R}(\zz) \cap \{u>0\}} |D^2 u|^{\cttg'} \, dx \le C_{\cttg'} \bE_\zz(u, 4R)^{\cttg'}, \quad \text{for any } \cttg' \in (0,1/2),
\end{equation}
where \( C_{\cttg'} \) is a constant depending on \( \cttg' \).
On the other hand, the Sternberg--Zumbrun stability inequality from \Cref{lem:Stern_Zum} gives:
\[
R^2 \fint_{B_{2R}(\zz) \cap \{u>0\}} |D^2 u|^2 \, dx \le C.
\]
Noticing that $\frac{2}{3} = \frac{4}{3(2-\cttg')}\cdot \cttg'  + \frac{2-3\cttg'}{3(2-\cttg')}\cdot 2$, we can combine these two inequalities using H\"older's inequality to  obtain
\[
R^{2/3} \fint_{B_{2R}(\zz) \cap \{u>0\}} |D^2 u|^{2/3} \, dx \le C_{\cttg} \bE_\zz(u, 4R)^{\cttg},
\]
where $\cttg := \frac{4\cttg'}{3(2-\cttg')} \to \left(\frac{4}{9}\right)^-$ as $\cttg' \to \left(\frac{1}{2}\right)^-$.

Finally, applying Jerison--Savin's stability inequality \eqref{eq:stab-intro},  and noting that \( \bc^2 = F(D^2u)^{2/3} \le C |D^2 u|^{2/3} \),   we get 
\[
R^{-1/3} \ccI(u, B_{R}(\zz)) \le C \bE_\zz(u, 4R)^{\cttg},
\]
as desired.
\end{proof}

 \subsection{The left-hand side of Jerison--Savin controls neck radii: Proof of \Cref{prop:rb-l13}} \label{ssec:JerisonSavinNeck}The goal of this subsection is to show the following result, which will imply \Cref{prop:rb-l13}: 
 
\begin{prop}
\label{thm:deltakappa}
There exists a large universal constant $\kappa > 0$ such that the following holds.

Let $u$ be a classical solution to the Bernoulli problem in $B_2\subset \R^3$ with $|\nabla u|\le 1$ in $B_2$ and $u(0) = 0$. If
$ 
\|D^2 u\|_{L^\infty(B_2\cap \set{u>0})} \le C_0
$ 
for some $C_0 > 0$, then 
\[
\|D^2 u\|^\kappa_{L^\infty(B_1 \cap \set{u>0})}\le C\,\ccI(u, B_2),
\]
for some $C$ depending only on $C_0$. 
\end{prop}

\begin{remark}[Scaling invariant form]
\label{rem:deltakappa}
By applying \Cref{thm:deltakappa} to $\tilde u(x) = u(rx)/r$, we obtain the following scaling invariant version: let $u$ be a solution in $B_{2r}\subset \R^3$ with $|\nabla u|\le 1$ in $B_{2r}$ and $u(0) = 0$. If
$ 
\| r D^2 u\|_{L^\infty(B_{2r}\cap \set{u>0})} \le C_0
$ 
for some $C_0 > 0$, then 
\[
\|r D^2 u\|^\kappa_{L^\infty(B_r \cap \set{u>0})}\le C\, r^{-1/3}\ccI(u, B_{2r}),
\]
for some $C$ depending only on $C_0$. 
\end{remark}

Before proving this result we note that, as its consequence, $\ccI$ is bounded from below at \emph{neck balls} (see \Cref{ssec:neck-set-def}):

\begin{lem}
\label{lem:I_neck_ball}
There exists $c  > 0$ universal such that
\[
\ccI(u, B_{2\rb(\zz)}(\zz)) \ge c\, \rb^{\frac13}(\zz) > 0,\qquad\text{for all}\quad \zz\in \cZ.
\]
\end{lem}


\begin{proof}
Let $r = \rb(\zz)$ and $\zz = 0$. By the definition of neck radius and \Cref{cor:Hess-bd-neck}, we have
\[
\int_{B_r \cap \{  u>0 \}}|D^2  u|^3 \,dx = \eta^3_0\qquad\text{and}\qquad r|D^2  u|\le C  \qquad\text{in}\quad B_{2r}\cap \{ u > 0\},
\]
for some $C$ universal. Thus,  we can apply  \Cref{thm:deltakappa} (rescaled as in \Cref{rem:deltakappa}), and obtain
\[
\eta_0^\kappa = \|D^2   u\|^{\kappa}_{L^3(B_r\cap\{u>0\})} 
\le C r^\kappa \|D^2   u\|^\kappa_{L^\infty(B_r\cap \{  u>0\})}
\le C r^{-1/3}\,\ccI( u, B_{2r}),
\]
as we wanted.
\end{proof}

Hence, we can prove \Cref{prop:rb-l13}.
\begin{proof}[Proof of \Cref{prop:rb-l13}]
Suppose first that $R \ge M_\circ \rb(\zz)$, with $M_\circ$ given by \Cref{lem_aux_rb2}. By the choice of $M_\circ$ we know that $B_{2\rb(\zz')}(\zz')\subset B_{2R}(\zz)$  for every $\zz'\in B_{3R/2}(\zz)\cap \cZ$. Thus, thanks to \Cref{lem:I_neck_ball} and the definition of $\ccI$ \eqref{eq:stab-intro}, we have
\[
c \rb^{1/3}(\zz')\le \ccI(u, B_{2\rb(\zz')}(\zz'))\le \ccI(u, B_{2R}(\zz))\qquad\text{for all}\quad \zz'\in B_{3R/2}(\zz)\cap \cZ,
\]
as desired.

Suppose now $\rb(\zz)\le R \le M_\circ \rb(\zz)$. On the one hand, \Cref{lem:I_neck_ball} gives
\[
\varrho_\zz(u, 2R)  = \frac{\ccI(u, B_{2R}(\zz))^3}{2R} \ge \frac{\ccI(u, B_{2\rb(\zz)}(\zz))^3}{2M_\circ \rb(\zz)} \ge c > 0,
\]
while on the other hand, thanks to \Cref{lem_aux_rb2} (now applied taking $R = 2M_\circ \rb(\zz)$ in the statement of \Cref{lem_aux_rb2}), 
$$
\rb(\zz') \le \tfrac14 M_\circ \rb(\zz) \le \tfrac14 M_\circ R\qquad\text{for all}\quad \zz'\in B_{3R/2}(\zz)\cap \cZ \subset B_{3 M_\circ \rb(\zz)/2}(\zz)\cap \cZ.$$
This gives the desired result. 
\end{proof}

The rest of this subsection will now be dedicated to the proof of \Cref{thm:deltakappa}.

\subsubsection{The test function revisited}
Until the end of the subsection, we assume that 
\begin{equation}
    \label{eq:u_ssec}
\text{$u$ is a classical solution to the Bernoulli problem in $B_2\subset \R^3$, with $\Omega = \{u > 0\}$}.
\end{equation}
We recall that, as a consequence of  \cite[Theorem 4.1]{Jerison-Savin}, the function $w = F(D^2 u)$ (as defined in \eqref{eq:def_w}) satisfies 
\begin{align*}
w\Delta w-\frac{2}{3}|\nabla w|^2 \geq 0
& \qquad \text{ in } \Omega=\set{u>0},\\
w_\nu+3Hw\geq 0
& \qquad \text{ on } \partial\Omega=\partial\set{u>0},
\end{align*}
where $\nu$ denotes the inward normal towards $\set{u>0}$, and $H$ is the mean curvature of $\partial \Omega$. In particular, $\Delta(w^{\alpha})\geq 0$ in $\Omega$ for $\alpha \geq \frac{1}{3}$, and $(w^{\alpha})_{\nu}+Hw^{\alpha}\geq 0$ on $\partial\Omega$ for $0\leq\alpha\leq \frac{1}{3}$. We aim to keep track of the remainder 
in the previous inequalities.

We start with the interior inequality. Notice that the function $f(\lambda_1,\lambda_2, \lambda_3)
=\sqrt{\sum_{\lambda_i>0}\lambda_i^2+4\sum_{\lambda_i<0}\lambda_i^2}$ is convex, therefore $(\lambda_i - \lambda_j)(f_{\lambda_i}-f_{\lambda_j}) \ge 0$ (here and in the sequel, $f_{\lambda_i} = \partial_{\lambda_i} f$). We recall that $(\lambda_i)_i$ denote the eigenvalues of $D^2 u$.

\begin{lem}[Remainder of interior inequality]
\label{lem:stab-rem-int}
Let $u$ be as in \eqref{eq:u_ssec}. At all points where  $\lambda_1,\lambda_2,\lambda_3$ are not all equal (to $0$),
\begin{align*}
w\Delta w -\frac{2}{3}|\nabla w|^2
&\geq 
\frac{2}{3}
\sum_{k=1}^{3}
	\frac{
		\sum_{1\leq i<j \leq 3,\,i,j\neq k}(\lambda_i-\lambda_j)(f_{\lambda_i}-f_{\lambda_j})
	}{
		\sum_{1\leq i\leq 3,\,i\neq k}(\lambda_i-\lambda_k)(f_{\lambda_i}-f_{\lambda_k})
	}
	w_k^2,\qquad\text{in}\quad \Omega.
\end{align*}
\end{lem}

\begin{proof}
Following the proof of \cite[Theorem 4.1]{Jerison-Savin}, we write 
\begin{align*}
w\Delta w
&\geq \frac{2}{n}\sum_{k=1}^{n}
	w_k^2
	\frac{nf}{
		\sum_{i\neq k}(\lambda_i-\lambda_k)(f_{\lambda_i}-f_{\lambda_k})
	}
=\frac{2}{n}\sum_{k=1}^{n}
	w_k^2
	\left (1+
		\frac{
			nf-\sum_{i\neq k}(\lambda_i-\lambda_k)(f_{\lambda_i}-f_{\lambda_k})
		}{
			\sum_{i\neq k}(\lambda_i-\lambda_k)(f_{\lambda_i}-f_{\lambda_k})
		}
	\right ).
\end{align*}
From \cite[eq. (4.7)]{Jerison-Savin}, we have that for each $k=1,\dots,n$,
\begin{align*}
nf-\sum_{i\neq k}(\lambda_i-\lambda_k)(f_{\lambda_i}-f_{\lambda_k})
&=\sum_{1\leq i<j\leq n}
	(\lambda_i-\lambda_j)(f_{\lambda_i}-f_{\lambda_j})
-\left (
	\sum_{1\leq i<k\leq n}
	+\sum_{1\leq k<i\leq n}
\right )
	(\lambda_i-\lambda_k)(f_{\lambda_i}-f_{\lambda_k})\\
&=
	\sum_{\substack{1\leq i<j\leq n\\i,j\neq k}}
	(\lambda_i-\lambda_j)(f_{\lambda_i}-f_{\lambda_j}),
\end{align*}
Rearranging gives the desired result, in particular for $n=3$.
\end{proof}

Next, we refine the lower bound on the boundary inequality.

\begin{lem}[Remainder of boundary inequality]
\label{lem:stab-rem-bdry}
 Let $u$ as in \eqref{eq:u_ssec}. Let $(\lambda_1,\lambda_2,\lambda_3)$ be the eigenvalues of $D^2 u$ evaluated at points on $\partial\{u>0\}$, and let us write $(\lambda_1,\lambda_2,\lambda_3)= ((\mu+1)H,-\mu H,-H)$ for some $\mu\geq -\frac12$. Then,
\[
w_\nu+3Hw
\geq \min\set{\frac{1}{4}(\mu-1)^2,1} Hw\qquad\text{on}\quad \partial\Omega.
\]
\end{lem}

\begin{proof}
Recall from \cite[Section 4.3]{Jerison-Savin} that
\[
\frac{w_\nu}{Hw}+3
=2-\frac{
	\sum_{\lambda_i>0}\lambda_i^3
	+4\sum_{\lambda_s<0}\lambda_s^3
	+4H\sum_{k=1}^{n}\lambda_k^2
}{Hw^2}.
\]
We want to find a positive lower bound of the right-hand side, which in both cases we denote by $g(\mu)$.

Suppose first that $\mu \ge 0$. Then, we have 
\begin{align*}
g(\mu)
&=2-\frac{(1+\mu)^3-4\mu^3+4((1+\mu)^2+\mu^2)}{(1+\mu)^2+4\mu^2+4}
 =\frac{(\mu-1)^2(3\mu+5)}{(1+\mu)^2+4\mu^2+4}
\geq \min\left\{\frac14 (\mu-1)^2, 1\right\}.
\end{align*}
Suppose now $\mu < 0$. In this case,  we get
\begin{align*}
g(\mu)
&=2-\frac{(1+\mu)^3-\mu^3+4((1+\mu)^2+\mu^2)}{(1+\mu)^2+\mu^2+4}
=\frac{5-7(1+\mu)\mu}{5+2(1+\mu)\mu}
\geq 1.
\end{align*}
In both cases, the proof is complete.
\end{proof}

\subsubsection{Mean curvature controls Hessian}
\label{ssec:mean_cuv_hessian}
 
Our next step consists in proving that the mean curvature of the free boundary controls the Hessian nearby. This is the purpose of the next:

\begin{prop}
\label{prop:H_to_Hessian}
 Let $u$ be as in \eqref{eq:u_ssec} with $|\nabla u|\leq 1$, $0\in \FB(u)$, and $
|D^2 u|\le C_0$ in $ B_2$. Then
\[
\normbig[L^\infty(B_{1/2} \cap \set{u>0})]{D^2 u}^2
\leq C\max\{C_0^{3},1\}
	\normsmall[L^\infty(B_2\cap \FB(u))]{H},
\]
for some $C>0$ universal. 
\end{prop}

To show this result, we first prove some preliminary lemmas. The first one shows that the mean curvature controls the $L^1$ norm of~$1-|\nabla u|^2 \in [0, 1]
$:

\begin{lem}[Controlling $1-|\nabla u|^2$]
	\label{lem:anal-exp-cover0}
Under the assumptions of \Cref{prop:H_to_Hessian}, let ${v}$ be defined as in \eqref{eq:def_v}.  Then
\[
\int_{B_{1/2}\cap \set{u>0}}
	 {v}
	\,dx
\leq C\max\left\{C_0^3,1\right\} \normsmall[L^\infty(B_2\cap \FB(u))]{H},
\]
for some universal $C$.
\end{lem}

\begin{proof}
For simplicity, let us denote $H_0 := \normsmall[L^\infty(B_2\cap \FB(u))]{H}$. We divide $B_{1/2}\cap \set{u>0}$ into slabs $\cS_k:=\setsmall{y\in B_{1/2}:\,2^{-k-1}\leq d_y \leq 2^{-k}}$, $k\ge 1$, where $d_y = \dist(y, \FB(u))$. Then:
\begin{itemize}
\item Since $|D^2u|\leq C_0$, we have $|\nabla u|\geq \frac12$ in $\{d_y \le \tfrac{1}{2C_0}\}$ and the area of $B_{1/2}\cap \setsmall{d_y=t}$ for $t\leq\frac{1}{2C_0}$ is comparable to the area of $B_{1/2}\cap \FB(u)$ (see e.g. \cite[eq. (3.4)]{WW19}), hence universally bounded (see \Cref{lem:perbound}). 
In particular, for $2^{-k}\leq \frac{1}{2C_0}$, $\cS_k$ can be covered by $C 2^{2k}$ balls of radius $2^{-k}$.

\item For $2^{-k}\geq \frac{1}{2C_0}$, $\cS_k$ can be (trivially) covered by $C C_0^3$ balls of radius $\frac{1}{4C_0}$.
\end{itemize}
We now observe the validity of the following
Hopf-type estimate: given $y \in B_{1/2}$, consider the superharmonic function $ {v}_{y,d_y}(z)= {v}(y+d_y z)$ for $z\in B_1$, and note that $B_1$ touches $\partial \{ v_{y,d_y}>0\}$ from the interior at some point $z_0$. Then, combining \Cref{lem:Hopf-quant} and \Cref{lem:vbounds}, we get
\[
\fint_{B_{d_{y}/2}(y)}
	{v}(x)\,dx = \fint_{B_{1/2}} {v}_{y,d_y}(z)\,dz
\leq C\,\partial_{\nu} {v}_{y,d_y}(z_0)
\leq C d_y H_0.
\]
Applying this bound inside each of the balls $B_{d_{y_{k,j}}}(y_{k,j})$ constructed above to cover the slabs $\cS_k$, since $d_{y_{k,j}}=2^{-k}$ we get  
\begin{align*}
\int_{B_{1/2}}
	 \hspace{-0.3mm} {v}(x)
\,dx
&
\leq \hspace{-0.3mm}
\bigg(
	\sum_{2^k=2C_0}^{\infty} \hspace{-1mm}
	\sum_{j=1}^{C 2^{2k}}
	\hspace{-1mm}+\hspace{-0.5mm}\sum_{2^k=1}^{4C_0}
	\sum_{j=1}^{CC_0^3}
\bigg)
	\int_{B_{d_{y_{k,j}}/{2}}(y_{k,j})}
		\hspace{-0.6mm} v(x)
	\,dx
\leq C 
\bigg(
	\sum_{2^k=2C_0}^{\infty} \hspace{-1.5mm}
	2^{2k}
	d_{y_{k,j}}^{4} 
	 \hspace{-0.6mm}+
	\sum_{2^k=1}^{4C_0}
	C_0^3
	d_{y_{k,j}}^{4}
\bigg)H_0\\
&\leq
	C \max\left\{C_0^3,1\right\}
	\sum_{k=1}^{\infty}
		2^{2k}2^{-4k}H_0
\leq C\max\left\{C_0^3,1\right\} H_0.
\qedhere
\end{align*}
\end{proof}

Thanks to the previous lemma, we now show that the mean curvature controls $L^2$ norm of the Hessian.

\begin{lem}[Controlling $D^2u$]
\label{lem:control_D2u}
Under the assumptions of \Cref{prop:H_to_Hessian}, we have 
\[
\int_{B_{1/4}\cap \set{ {u}>0}}
	|D^2 u|^2
\,dx
\leq C\max\left\{C_0^3,1\right\} \normsmall[L^\infty(B_2\cap \FB(u))]{H},
\]
for some $C$ universal.
\end{lem}

\begin{proof}
Let $\eta \in C_c^\infty(B_{1/2})$ be a non-negative
cut-off function satisfying $\eta\equiv 1$ inside $B_{1/4}$, and let $v$ be defined as in \eqref{eq:def_v}. We compute
\begin{align*}
\int_{B_{1/4}\cap \set{ {u}>0}}
	|D^2 {u}|^2
\,dx
&\leq \int_{B_{1/2}\cap \set{u>0}}
	|D^2{u}|^2 \eta
\,dx
=\frac{1}{2}
\int_{B_{1/2}\cap \set{u>0}}
	-\Delta {v}\cdot\eta
\,dx\\
&=\frac{1}{2}
	\int_{B_{1/2} \cap  \FB(u)}
		\partial_{\nu} {v}\cdot \eta
	\,d\cH^{n-1}
-\frac{1}{2}
	\int_{B_{1/2}\cap \set{ {u}>0}}
		 {v}\, (-\Delta \eta)
	\,dx\\
&\leq
	C \int_{B_{1/2} \cap  \FB(u)}
		 {H}
	\,d\cH^{n-1}
	+C \int_{B_{1/2}\cap \set{ {u}>0}}
		 {v}
	\,dx.
\end{align*}
Combining \Cref{lem:perbound} and \Cref{lem:anal-exp-cover0}, the result follows.
\end{proof}

We can now proceed with the proof of \Cref{prop:H_to_Hessian}.

\begin{proof}[Proof of \Cref{prop:H_to_Hessian}]
By a scaling and covering argument,
\Cref{lem:control_D2u} holds replacing $B_{1/4}$ with $B_1$ in the left-hand side. Hence, thanks to 
H\"older's inequality, we get
\[
\bigg( \int_{B_{1}\cap \set{ {u}>0}}
	|D^2 u|
\,dx
\bigg)^2 
\le C\int_{B_{1}\cap \set{ {u}>0}}
	|D^2 u|^2
\,dx
\leq C\max\left\{C_0^3,1\right\}
\normsmall[L^\infty(B_2\cap \FB(u))]{H},
\]
with $C$ universal. Recalling \Cref{cor:Hess-W21}, this concludes the proof.
\end{proof}

\subsubsection{Conclusion}
\label{ssec:analytic}
We can now finally prove
\Cref{thm:deltakappa}.
\begin{proof}[Proof of \Cref{thm:deltakappa}]  
Without loss of generality we can assume that 
\[
C_0 \leq \delta_0
\]
for some $\delta_0$ small universal constant to be chosen.

In fact, once the result is known when $C_0$ is sufficiently small, the general case follows by replacing $u$ with $u_r(x)=\frac{1}r u(rx)$ with $r=\delta_0/C_0$. 
Indeed $|D^2u_r|\leq \delta_0$ inside $B_{2/r}$, so applying the result to $u_r$ together with a covering   yields the desired estimate for $u$ near the free boundary. Using \Cref{prop:H_to_Hessian} to relate a Hessian bound on the boundary with a Hessian bound in the interior, yields the desired result (up to redefining $\kappa$).

So, from now on, we assume that $\|D^2 u\|_{L^\infty(B_2\cap \set{u>0})} \le \delta_0$ for some $\delta_0$ sufficiently small, to be fixed later.
We divide the proof into five steps.

\medskip 
\noindent {\bf Step 1:} We first perform an expansion of $u$ around the origin.

Since $\|D^2 u\|_{L^\infty(B_2\cap \set{u>0})} \le \delta_0$, we have
\begin{equation}\label{eq:anal-exp-H}
0<H\leq  C \delta_0
	\quad \text{ on } B_2 \cap \partial\set{u>0}.
\end{equation}
We select a subset on which $H$ satisfies a doubling property as follows. Let
\[
2\delta^3:=\max_{B_{3/2} \cap \partial\{u>0\}}
	\left(\tfrac32-|x|\right) H(x) \le C \delta_0
\]
be attained at $x_0$. Then, by \eqref{eq:anal-exp-H}, $r_0:=\frac32-|x_0|\geq \frac{1}{C\delta_0}\delta^3\geq \delta^3$. For $y\in B_{1/2}$, set
\[
\bar{u}(y)=\frac{1}{r_0}u(x_0+r_0y),
	\qquad
\bar{H}(y)=r_0H(x_0+r_0y)\quad \text{for }y \in \partial\{\bar u>0\}.
\]
Note that
\begin{equation}\label{eq:anal-exp-H-bar}
\bar{H}(0)=2\delta^3,
	\qquad
\bar{H}(y)\leq 4\delta^3
	\quad \text{ for } y\in B_{1/2}.
\end{equation}
In addition, since $|D^2 \bar u|\le C\delta_0$ in $B_{1}$, \Cref{lem:control_D2u} gives
\[
\int_{B_{1/8}\cap \{\bar u > 0 \}} |D^2 \bar u|^3 \,dx \le C \delta_0 \int_{B_{1/8}\cap \{\bar u > 0 \}}|D^2 \bar u|^2 \,dx \le C \delta^3, 
\]
and therefore (recall \Cref{lem:eps-reg-Hess}, assuming $\delta$ small),    
\begin{equation}
\label{eq:Dkbaru_bounds}
|D^k\bar u|\le C\delta \qquad\text{in}\quad B_{1/16}\cap \{\bar u > 0\},\quad\text{for}\quad k = 2, 3, 4. 
\end{equation}
Since $\delta^3 \le C \delta_0$, in order to have $\delta$ sufficiently small it is enough to assume $\delta_0$ small. 

Let us use expansions of $\bar u$ around $0$. Up to a rotation, we also assume $\nu(0)=\nabla \bar u(0)=e_3$. Thus, in principal coordinates, we have 
\begin{equation}
\label{eq:expbaru}
 \bar u(x)=x_3+ \sum_{i=1}^{3}a_ix_i^2
+\sum_{i=1}^{3}A_{iii}x_i^3
+\sum_{\substack{1\leq i \neq j\leq 3 }}A_{iij}x_i^2x_j
+A_{123}x_1x_2x_3
+O( |x|^4)\qquad\text{in}\quad \{\bar u > 0\}.
\end{equation}
where all the coefficients are bounded, and the big $O$ notation is with universal constants.

Thanks to this expansion, it follows that
\begin{equation}\label{eq:anal-exp-FBu}
\FB(\bar u)=\set{x_3=- \sum_{i=1}^{3}a_ix_i^2+O(|x|^3)}
=\set{x_3=- \sum_{i=1}^{2}a_ix_i^2+O(|x|^3)}.
\end{equation}
Also $a_3=-\delta^3$ (since $\bar H(0)=2\delta^3$), and the harmonicity of $\bar u$ inside $B_1\cap \set{\bar u>0}$ gives
\begin{equation}\label{eq:anal-exp-D2u}
0=\frac12\Delta \bar u
= (a_1+a_2-\delta^3)+3\sum_{i=1}^{3}A_{iii}x_i
+\hspace{-2mm}\sum_{\substack{1\leq i\neq j\leq 3 }}\hspace{-2mm}A_{iij}x_j+O(|x|^2)\ \Rightarrow \  \begin{cases}
a_1+a_2=\delta^3,\\
3A_{333}+A_{113}+A_{223}=0.
\end{cases}
\end{equation}
Let us now obtain a relation from the fact that $|\nabla \bar u|=1$ on 
$\FB(\bar u)$.
Since
\begin{align*}
\nabla \bar u
=\big( 
2  a_1x_1+O(|x|^2), \ \
2  a_2x_2+O(|x|^2) ,\ \
1+2  a_3x_3+A_{113}x_1^2+A_{223}x_2^2+A_{123}x_1x_2+O(x_3^2)+O(|x|^3)
\big)^\top,
\end{align*}
 then on $\FB(\bar u)$ we have (recall \eqref{eq:anal-exp-FBu} and that $a_3=-\delta^3$)
\begin{align*}
	(\nabla \bar u)_3
	=\begin{pmatrix}
1+2\delta^3(a_1x_1^2+a_2x_2^2)+A_{113}x_1^2+A_{223}x_2^2+A_{123}x_1x_2+O(|x|^3)
	\end{pmatrix},
\end{align*}
and therefore
\begin{align*}
1=|\nabla \bar u|^2
&=4 (a_1^2x_1^2+a_2^2x_2^2)
+1+4\delta^3(a_1x_1^2+a_2x^2)
+2A_{113}x_1^2+2A_{223}x_2^2+2A_{123}x_1x_2
+O(|x|^3)
\end{align*}
for $x \in \FB(\bar u)$.
Looking at the coefficients of $x_1^2$ and $x_2^2$,   
we deduce that
\begin{equation}
\label{eq:toapprox}
\begin{array}{rclrcl}
  \bar u_{113}(0)
=2A_{113}
=-4 a_1^2-4\delta^3 a_1,\qquad
\bar u_{223}(0)
=2A_{223}
=-4 a_2^2-4 \delta^3 a_2.
\end{array}
\end{equation}

\medskip 
\noindent {\bf Step 2:} By analyzing the remainder in the stability inequality, we identify two possible regimes.

More precisely, since the remainder of the boundary inequality in \Cref{lem:stab-rem-bdry} vanishes whenever $(\lambda_1,\lambda_2,\lambda_3)=(2\bar H,-\bar H,-\bar H)$, we analyze two cases depending on the closeness of $D^2\bar u(0)$ to $\diag(2\bar H(0),-\bar H(0),-\bar H(0))$.

To this aim, let us write the eigenvalues of $D^2\bar u$ on the $\FB(\bar u )$ as
\begin{equation}
\label{eq:mu_in_some_case}
\lambda_1=(\mu+1)\bar H,
\quad
\lambda_2=-\mu \bar H, \quad
\lambda_3=- \bar H,
\end{equation}
where $\mu:\FB(\bar u)\to [-\tfrac12,+\infty)$ is defined as 
\[
\mu(x)+1:=\max_{\tau \in \mathbb{S}^2 : \tau \cdot\nabla \bar u(x) = 0}\frac{\partial^2_{\tau \tau}\bar u (x)}{\bar H(x)}.
\]
Then, thanks to \Cref{lem:stab-rem-bdry}, $w := F(D^2 \bar u)$ satisfies
\begin{equation}\label{eq:stab-rem-bdry}
(w^{\frac13})_{\nu}+\bar Hw^{\frac13}
={\frac13}w^{-{\frac23}}\Big(w_\nu+3\bar Hw\Big)
\geq {\frac{1}{12}}\min\set{(\mu-1)^2,1}\bar Hw^{\frac13}.
\end{equation}
Consider a small threshold $\eps_{\rm E}\in(0,1)$ to be fixed later. At $0\in\FB(\bar u)$, we will distinguish between two cases:
\begin{enumerate}
\item \label{it:E-far} $|\mu(0)-1| \geq \eps_{\rm E}$;
\item \label{it:E-near} $|\mu(0)-1| < \eps_{\rm E}$.
\end{enumerate}

\medskip 
\noindent {\bf Step 3:} Case \eqref{it:E-far} holds.

We observe first that, in a neighborhood of $0$, $\mu(x)$ is a Lipschitz function. More precisely, since $\bar H (0) = 2\delta^3$ and $|D^3 \bar u|\le C \delta$, we have $\bar H(x) \ge  \delta^3 $ on $\FB(\bar u)\cap B_{c\delta^2}$ for some $c>0$ small. Therefore 
\[
\left|\nabla_{\tau'}\frac{\partial^2_{\tau \tau}\bar u(x)}{\bar H(x)} \right|\le C\bigg(\frac{\delta}{\bar H(x)}+\frac{\delta^2}{\bar H^2(x)}\bigg) \le C \delta^{-4} \qquad\text{for all}\quad x\in  \FB(\bar u)\cap B_{c\delta^2},\quad  \tau,\tau'\in \mathbb{S}^2\cap \nabla \bar u(x)^\perp.
\]
This implies that $\mu(x)$ is obtained as the maximum of Lipschitz functions with gradient bounded by $C \delta^{-4}$, thus
\[
|\nabla_{\tau'} \mu(x)|\le  C  \delta^{-4}\qquad\text{for all}\quad \FB(\bar u)\cap B_{c\delta^2},\quad  \tau'\in \mathbb{S}^2\cap \nabla \bar u(x)^\perp,
\]
for some universal $C$.
Thus, since we are in case \eqref{it:E-far},
\begin{equation}\label{eq:stab-rem-bdry-mu}
|\mu(x) - 1|\ge \frac{\eps_E}{2},\qquad\text{for all}\quad x\in \FB(\bar u)\cap B_{\delta^5},
\end{equation}
where $\delta$ is small enough depending on $\eps_E$, which will be fixed universal. Hence, since $w \ge \bar H$, thanks to \eqref{eq:stab-rem-bdry}--\eqref{eq:stab-rem-bdry-mu} we get
\[
\ccI(\bar u, B_1) \ge \ccI(\bar u, B_{\delta^5}) \geq \int_{\FB(\bar u)\cap B_{\delta^5}}  w^{\frac13}\left((w^{\frac13})_{\nu}+\bar Hw^{\frac13}\right)\,d\mathcal H^2\ge  \frac{\eps_E^2}{48}\int_{\FB(\bar u)\cap B_{\delta^5}} \bar H w^{\frac23}d\cH^2 \ge  c \eps^2_E \delta^{15}.
\]

\medskip 
\noindent {\bf Step 4:} Case \eqref{it:E-near} holds.

In this case we have that 
\begin{equation}\label{eq:stab-rem-2-1-1}
|D^2\bar u(0)-2\delta^3{\rm diag}(2,-1,-1)|\leq 2\eps_E \delta^3.
\end{equation}
Let $A(x) := \delta^{-3}D^2 \bar u(\delta^2 x)$.
Then, recalling \eqref{eq:Dkbaru_bounds}, 
\begin{equation}
    \label{eq:AboundsD2}
|A(0)  - 2{\rm diag}(2, -1, -1)|\le 2\eps_E,\qquad\text{and}\qquad |DA(x)|+|D^2 A(x)|\le C\quad\text{in}\quad B_1\cap \{\bar u(\delta^2\, \cdot\,) > 0\}. 
\end{equation}
In particular, if we denote by $\lambda_1^A(x)$ and $e_1^A(x)$ respectively the largest eigenvalue and the corresponding (unit) eigenvector of $A(x)$,
then $\lambda_1^A(x)$ is simple near the origin. Hence, because of \eqref{eq:AboundsD2},
\[
|\nabla \lambda_1^A(x)| + |D^2\lambda_1^A(x)| + |D e_1^A(x)|\le C\quad\text{for}\quad x\in B_{c}\cap \{\bar u(\delta^2\, \cdot\, ) > 0\},
\]
for some $c$ and $C$ universal. 
Thus, if we denote by $(\lambda_1(x), \lambda_2(x), \lambda_3(x))$ the eigenvalues of $D^2\bar u(x)$ at $x\in B_{\delta^3}\cap \overline{\{\bar u > 0\}}$, with $\lambda_1(x)$ being the largest one,  and by $e_1(x)$ the (unit) eigenvector corresponding  to $\lambda_1(x)$, then 
\begin{equation}
\label{eq:lambda1e1}
\delta^{-1}|\nabla \lambda_1(x)| + \delta |D^2\lambda_1(x)| + \delta^2|D e_1(x)|\le C\quad\text{for}\quad x\in B_{c\delta^2}\cap \{\bar u > 0\}.
\end{equation}
Furthermore, by \eqref{eq:stab-rem-2-1-1} and \eqref{eq:lambda1e1},
\begin{equation}
\label{eq:lambdabounds}
|\lambda_1(x) - 4\delta^3| \le 4\eps_E \delta^3,\qquad|\lambda_2(x) + 2\delta^3|\le 4\eps_E \delta^3,\qquad |\lambda_3(x) +2\delta^3|\le 4\eps_E \delta^3\qquad\text{for}\quad x\in B_{c\delta^2}\cap \{\bar u > 0\}.
\end{equation}
Recalling the interior inequality from  \Cref{lem:stab-rem-int}, since $ff_{\lambda_1}=\lambda_1$, $ff_{\lambda_2}=4\lambda_2$, $ff_{\lambda_3}=4\lambda_3$, and $(\lambda_i - \lambda_j) (f_{\lambda_i}- f_{\lambda_j})\ge 0$, we get (notice $w^{\frac13} \Delta w^{\frac13} = \tfrac13 w^{-\frac43}\left(w\Delta w -\tfrac23 |\nabla w|^2\right))$
\begin{equation}\label{eq:stab-remain}
\begin{split}
\ccI(\bar u, B_1) \ge \ccI(\bar u, B_{\delta^3}) &\ge \frac{2}{9}
\int_{\{\bar u > 0\}\cap B_{\delta^3} }
	\sum_{\substack{1\leq k\leq 3,\,i<j\\\set{i,j,k}=\set{1,2,3}}}
	\frac{
		(\lambda_i-\lambda_j)(f_{\lambda_i}-f_{\lambda_j})
	}{
		(\lambda_i-\lambda_k)(f_{\lambda_i}-f_{\lambda_k})
		+(\lambda_j-\lambda_k)(f_{\lambda_j}-f_{\lambda_k})
	}\,w_k^2 w^{-\frac43} 
\,dx\\
&\geq
\frac{2}{9}
\int_{\{\bar u > 0\}\cap B_{\delta^3}}
	 \sum_{\{j, k\} = \{2, 3\}} \frac{
		(\lambda_1-\lambda_j)
		(\lambda_1-4\lambda_j)}
	{
		(\lambda_1-\lambda_k)
		(\lambda_1-4\lambda_k)
		+4(\lambda_j-\lambda_k)^2
	} \,w_k^2 w^{-\frac43}\, dx
 \\
&\geq\frac{2}{9}
\int_{\{\bar u > 0\}\cap B_{\delta^3}}\hspace{-2mm}
	\frac{
		(6-8\eps_{\rm E})
		(12-20\eps_{\rm E})
	}{
		(6+8\eps_{\rm E})
		(12+20\eps_{\rm E}) + 256\eps_E^2
	}
	 \,(w_2^2 + w_3^2) w^{-\frac43}
\,dx \geq
	\frac{1}{9}
	\int_{\{\bar u > 0\}\cap B_{\delta^3}}\hspace{-2mm}
		 g\, w^{-\frac43-2} 
	\,dx,
\end{split}\end{equation}
for $\eps_E$ small universal, where we have denoted 
\[
g = w^2(w_2^2+w_3^2) = w^2(|\nabla w|^2 - w_1^2) = \frac14 \left(\big|\nabla (w^2)\big|^2-\big(e_1(x)\cdot\nabla (w^2)\big)^2\right).
\]
Notice that the function $g$ is well-defined, since the eigenvector $e_1(x)$ is simple around 0. Let us show it is Lipschitz. Indeed, since $w^2(x) = 4|D^2\bar u(x)|^2-3\lambda^2_1(x)$, it follows from \eqref{eq:Dkbaru_bounds} and \eqref{eq:lambda1e1} that
\[
|\nabla (w^2)|+|D^2 (w^2)|\le C\delta^2\quad\text{in}\quad B_{c\delta^2}\cap \{\bar u > 0\},
\]
and hence, using \eqref{eq:lambda1e1} again,
\begin{equation}
\label{eq:nablag0}
|\nabla g| \le C\delta^2\quad\text{in}\quad B_{c\delta^2}\cap \{\bar u > 0\}.
\end{equation}
We now want to evaluate $ww_3$ at the origin using that, at $0$, we can take $e_3$ as an eigenvector (with eigenvalue $-2\delta^3$).

Recalling \eqref{eq:toapprox} and observing that $2a_1 = \lambda_1(0)$ and $2a_2 = \lambda_2(0)$,
it follows that
$$
\bar u_{113}(0)=-\lambda_1(0)^2-2\delta^3 \lambda_1(0),\qquad \bar u_{223}(0)=-\lambda_2(0)^2-2\delta^3\lambda_2(0).
$$
Given that we are in Case \eqref{it:E-near}, this implies that
\[
|\bar u_{113}(0)+24\delta^6| + |\bar u_{223}(0)|   \le  C\eps_E \delta^6.
\]
Since $w^2 = f^2=\lambda_1^2+4\lambda_2^2+4\lambda_3^2$ and $ww_3
=\sum_{i = 1, 2, 3}f f_{\lambda_i} \bar u_{ii3}$ (see \cite[Section 4.1 and Eq. (4.4)]{Jerison-Savin}) we get
\begin{align*}
ww_3
&=\sum_{i = 1, 2, 3}f f_{\lambda_i} \bar u_{ii3} = \sum_{i=1, 2}(ff_{\lambda_i}-ff_{\lambda_3})\bar u_{ii3}
=(\lambda_1-4\lambda_3)\bar u_{113}
	+(4\lambda_2-4\lambda_3)\bar u_{223}\\
&=[12+O(\eps_E)]\delta^3
	\cdot [-24+O(\eps_E)]\delta^6
+O(\eps_E)\delta^3
	\cdot O(\eps_E) \delta^6=[-288+O(\eps_E)]\delta^9
\le -2\delta^9\qquad \text{at}\quad y = 0,
\end{align*}
for $\eps_E$ small  universal. In particular,
\[
g(0) \ge (ww_3)^2(0) \ge 4\delta^{18}. 
\]
Together with \eqref{eq:nablag0}, this implies
\[
g\ge \delta^{18}\quad\text{in}\quad B_{c\delta^{16}}\cap \{\bar u > 0\}. 
\]
Inserting this estimate in \eqref{eq:stab-remain} we obtain (notice also $w\le C\delta$)
\[
\ccI(\bar u, B_1) \ge \ccI(\bar u, B_{c\delta^{16}}) \ge 
	\frac{1}{9}
	\int_{\{\bar u > 0\}\cap B_{c\delta^{16}}}\hspace{-2mm}
		 g w^{-\frac43-2}\,dx \ge \delta^{15+16\cdot 3}.
\]

\medskip
\noindent{\bf Step 5:} We can now conclude the proof.

By Steps 3 and 4 we get that, in all cases,
\begin{equation}
\label{eq:cIdelta6}
\ccI(\bar u, B_1) \ge \delta^{63}. 
\end{equation}
We finally have all the ingredients to compelte the proof of \Cref{thm:deltakappa}.
Indeed, using the notation from the previous steps, we have 
\[
\ccI(u, B_1) \ge \ccI(u, B_{r_0}(x_0)) = r_0^{\frac13} \ccI(\bar u, B_1) \ge \delta^{64}
\] 
where we used \eqref{eq:rescaleI}, \eqref{eq:cIdelta6}, and the fact that $r_0 \ge \delta^3$. 
Thus, by the definition of $\delta$ and \Cref{prop:H_to_Hessian} we obtain (after a covering argument)
\[
C\, \ccI(u, B_1)^{\frac{1}{64}}
\ge \max_{x\in B_{4/3}\cap \partial\{u > 0\}} H(x)
\geq C^{-1}\normbig[L^\infty(B_{1/2} \cap \set{u>0})]{D^2 u}^2,
\]
as desired (in particular, we may take $\kappa = 128$). 
\end{proof}

\section{Selection of center and scale}

\label{sec:selection}

In the present section (and until the end of \Cref{sec:closing}), we fix the following universal constants:
\begin{equation}
    \label{eq:gamma_alpha_beta}
    \cttg := \valcttg = \frac{4}{9}-\frac{1}{225}, \qquad \ctta := \valctta = \frac{3}{4} + \frac{3}{100},\qquad \cttb : = \valcttb.
\end{equation}

We remark that 
\[
3\alpha \gamma = 1+\frac{37}{1250}  > 1. 
\]

\subsection{Selection of center and scale}\label{ssec:selection}
We now set up the contradiction argument that will yield the desired result.

For $\zz\in \cZ$ and $ R\ge  \rb(\zz)$, recall $\bE_\zz(u,R)$ and $\varrho_\zz(u,R)$ defined in \eqref{eq:Ez_def_intro} and \eqref{eq:varrho_def_intro}. By Propositions~\ref{prop:E_to_rho} and~\ref{prop:rb-l13}, we have 
\begin{equation}\label{eq:rbeps}
\frac{\rb(\zz')}{R} \le C \varrho_\zz(u, 2R)\le C \bE_\zz(u, 8R)^{3\cttg}\qquad\text{for any}\quad \zz'\in B_{3R/2}(\zz)\cap \cZ,\quad  R\ge \rb(\zz).
\end{equation}
In this section, it will be convenient to introduce the following definition: for given $R > 0$, we define
\begin{equation}
\label{eq:ZR}
\cZ_R := \{\zz\in \cZ : \rb(\zz) \le R\},
\end{equation}
which is nonempty for large $R$ due to \Cref{lem:Znonempty}.
 
The next lemma provides suitable centers and scales where we can start our argument:  
\begin{lem}
 \label{lem:zkRk}
  There exist sequences $R_k>0$ and $\zz_k\in \cZ_{R_k}$,
  with $R_k\to \infty$ as $k\to \infty$, such that 
 \begin{equation}
\label{eq:eps_k}
\frac{\rb(\zz_k)}{R_k} \le \varrho_{\zz_k}(u,2R_k)  \to 0 \quad \mbox{and} \quad \eps_k : =  \bE_{\zz_k}(u,8R_k)\to 0\qquad\text{as}\quad k \to \infty,
\end{equation}
and  
\begin{equation}\label{key1}
\bE_\zz(u,8R) \le 2 \frac{\varrho_{\zz}(u,2R)^\ctta }{\varrho_{\zz_k}(u,2R_k)^\ctta }\,\eps_k \qquad \text{for all}\quad  \zz\in \cZ_R,   \   R\le R_k.
\end{equation}
 \end{lem}
 \begin{proof}
 Let us define  the quotient
\[
F_u(R) : = \sup_{\zz\in \cZ_R}   \frac{ \bE_\zz(u,8R)}{ \varrho_\zz(u,2R)^\ctta } .  
\]
Notice that, because of \eqref{eq:rbeps} and \eqref{eq:rmin_def},
$F_u(R)\leq \sup_{\zz\in\cZ_R}(R/\rb(\zz))^\alpha \bE_\zz(u,8R) \leq CR^\alpha$ (here we use that $\bE_\zz(u,\cdot)$ is always bounded, since  $|\nabla u|\le 1$). So $F_u$ is well-defined.

Also, thanks to \eqref{eq:rbeps} and \Cref{lem:X-blowdown} (recall that $3\ctta\cttg > 1$),
\begin{equation}\label{limF}
\limsup_{R\to\infty}F_u(R)
\geq \limsup_{R\to\infty}\,  \sup_{\zz\in\cZ_R} \frac{\bE_\zz(u, 8R)}{C^\ctta \bE_\zz(u, 8R)^{3\ctta \cttg}}
\geq \frac{1}{C^\ctta } \limsup_{R\to\infty} \, \sup_{\zz\in\cZ_R}
	\omega\left(\frac{M_*\rb(\zz)}{8R}\right)^{1-3\ctta\cttg}
=+\infty.
\end{equation}
Consider now the `nondecreasing envelope' of $F_u$, namely
\[
\widetilde F_u(R) : = \sup_{ R'\le R} F_u(R'),
\]
and choose a monotone increasing sequence $R_k\to \infty$ such that, for each $k$, there exists $\zz_k\in \cZ_{R_k}$ satisfying
\begin{equation}
\label{eq:Ftildeineq}
\tfrac 1 2 \widetilde F_u(R_k) \le  \frac{\bE_{\zz_k}(u,8R_k)}{\varrho_{\zz_k}(u,2R_k)^\ctta}\le \widetilde F_u(R_k) 
\end{equation}
and let $\eps_k := \bE_{\zz_k}(u,8R_k)$.

Notice that the numerator in \eqref{eq:Ftildeineq} is always bounded using  $|\nabla u|\le 1$. Thus,   the only way $\widetilde F_u(R_k)$ may diverge is if the denominator in \eqref{eq:Ftildeineq} converges to zero. But then the numerator must converge to zero as well since, by \Cref{lem:X-blowdown} and \eqref{eq:rbeps},   $\bE_{\zz_k}(u, 8R_k)\le \omega(\rb(\zz_k)/8R_k) \le \omega(C\varrho_{\zz_k}(u, 2R_k))\to 0$. This shows \eqref{eq:eps_k}.

In addition, by the definition of  $\widetilde F_u$ we have 
\[
 \frac{\bE_\zz(u,8R)}{ \varrho_\zz(u,2R)^\ctta }  \le  \widetilde F_u(R_k) \le 2  \frac{\bE_{\zz_k}(u,8R_k)}{\varrho_{\zz_k}(u,2R_k)^\ctta }  =  \frac{2\eps_k}{\varrho_{\zz_k}(u,2R_k)^\ctta } \qquad \mbox{for all $\zz\in \cZ_R$, $R\le R_k$,}
\]
so \eqref{key1} follows.
\end{proof}

  Given $\zeta\in (0,1)$ and a ball $B_{  R}(  \zz)\subset \R^3$, recalling \eqref{eq:ZR} we define  
\begin{equation}
\label{eq:Ndef}
\bN\big(\zeta, B_{  R}( \zz)\big)  : =  (\zeta   R)^{-3} \bigg|\bigcup_{\zz'\in \mathcal{A}^\zeta_{\zz, R}} B_{\zeta   R} ( \zz') \bigg|, \qquad \text{where}\quad\mathcal{A}^\zeta_{\zz, R} := \cZ_{\zeta R} \cap B_{R}(\zz).
\end{equation}
This is roughly the number of balls of radius $\zeta R$ needed to cover  ${\cZ_{\zeta R} }\cap B_R(\zz)$. More precisely, we have the following: 
\begin{lem}
\label{lem:Nbound}
There exists $\tilde \cA_{\zz, R}^\zeta\subset   \cA_{\zz, R}^\zeta$, with $\# \tilde \cA_{\zz, R}^\zeta\le C\,\bN\big(\zeta, B_{  R}( \zz)\big)$ for some $C$ universal, such that
\[
\bigcup_{\zz'\in \cA_{\zz, R}^\zeta} B_{\zeta R}(\zz') \subset \bigcup_{\zz'\in \tilde \cA_{\zz, R}^\zeta} B_{2\zeta R}(\zz') .
\]
\end{lem}
\begin{proof}
Applying Besicovitch covering theorem to the family of balls $\{B_{\zeta   R} ( \zz')\}_{\zz' \in \cA_{\zz, R}^\zeta}$, we can find a subcovering  $\tilde \cA_{\zz, R}^\zeta$ of  $\cA_{\zz, R}^\zeta$ with bounded overlapping,
thus $\# \tilde \cA_{\zz, R}^\zeta\le C\,\bN\big(\zeta, B_{  R}( \zz)\big)$.
Also,
$$
\bigcup_{\zz'\in \tilde \cA_{\zz, R}^\zeta} B_{2\zeta R}(\zz')=
\bigcup_{\zz'\in \tilde \cA_{\zz, R}^\zeta} B_{\zeta R}(\zz')+B_{\zeta R}\supset \cA_{\zz, R}^\zeta+B_{\zeta R}=\bigcup_{\zz'\in \cA_{\zz, R}^\zeta} B_{\zeta R}(\zz').
$$
\end{proof}

Starting from \Cref{lem:zkRk}, we can define new sequences $\zk\in \cZ$ and $\Rk > 0$ satisfying the following:

\begin{lem} 
\label{lem:tildezktildeRk}
Let $R_k$ and $\zz_k$ be the sequences given by \Cref{lem:zkRk}. There exist  $\tilde \zeta_k \in (0, 1]$ and $\zk\in \cZ\cap B_{R_k}(\zz_k)$ such that, setting 
\[
\Rk := \tilde \zeta_k R_k\qquad \mbox{and} \qquad \epk := \tilde \zeta_k^{\ctta \cttb} \eps_k,
\]
we have $\Rk \to \infty$, $\epk \to 0$,  and  the following properties hold:
\begin{equation} \label{key3}
 B_{\Rk}(\zk)\subset B_{R_k} (\zz_k ),\qquad    \frac{\varrho_{\zk}(u, 2\Rk) }{\varrho_{\zz_k}(u,2R_k)} \le \tilde \zeta_k^{\cttb}, 
\end{equation}
and
\begin{equation}\label{strongalternativeA}
\bN\big(\zeta, B_{\Rk}(\tilde  \zz_k)\big)
   \le C \zeta^{-\frac{1+\cttb}{3}}\quad \mbox{for all } \zeta \in (0,1).
\end{equation}
Moreover, for all $k$ sufficiently large, we have:
\begin{equation}\label{eq:rbepk}
 \rb(\zz)  \le C\Rk\epk^{3\gamma} \le \frac{\Rk}{10}\epk^{1/\alpha}\qquad\text{for all}\quad \zz\in \cZ\cap B_{3\Rk/2}(\zk);
\end{equation}
and
\begin{equation}\label{keyeqn}
  \bE_\zz(u,8R)  \le   2\Big(\tfrac{\Rk}{R}\Big)^\ctta    \widetilde\eps_k \qquad \text{for all}\quad  \zz\in \cZ  \ \ \text{with}\ B_{R}(\zz)\subset B_{\Rk}(\zk)\text{ and }\ R\ge  \epk^{1/\alpha} \Rk.
\end{equation}
Here, the constant $C$ is universal and  $\bN$ is given by \eqref{eq:Ndef}.
\end{lem}

\begin{proof} We divide the proof into two steps.

\medskip
\noindent{\bf Step 1:}  
We first construct $\tilde \zeta_k$, show that $\Rk\to \infty$, and prove \eqref{key3} and \eqref{strongalternativeA}. 

Define
\[
\zeta_k  :=  \inf \left\{ \zeta>0  \  :  \  \mbox{there exists}\   \zz\in \cZ_{\zeta R_k}  \mbox{ s.t. } B_{\zeta R_k}(\zz)\subset B_{R_k} (\zz_k ) \mbox{ and }  \frac{\varrho_{\zz}(u,2\zeta R_k) }{\varrho_{\zz_k}(u,2R_k)} \le \zeta^{\cttb}    \right\}.
\] 
Notice that $\zeta=1$ and $\zz_k=\zz$ is always an admissible choice, therefore
$\zeta_k \le 1$. Also, since $\rb(\zz) \ge\rmin >0$  for all $\zz\in \cZ$ (recall \eqref{eq:rmin_def}), we must have $\zeta_k\geq  c/R_k >0$. 

Now, by the definition of $\zeta_k$, there exists $\tilde\zeta_k \in [ \zeta_k, \min\{2\zeta_k, 1\}]$ and  $\zk\in \cZ_{\Rk} \cap B_{R_k} (\zz_k )$ (where $\Rk : = \tilde \zeta_k R_k$) such that \eqref{key3} holds. 
Also, recalling \eqref{eq:rbeps} and that $\rmin>0$,
since $\varrho_{\zk}(u,2\Rk) \le \varrho_{  \zz_k}(u,   2R_k) \to 0$ as $k\to \infty$ we deduce that $\Rk \to \infty$. Furthermore $\epk \leq \eps_k \to 0$. 

We now prove \eqref{strongalternativeA}.
Notice that by the definition of $\zeta_k$ and the inclusion in \eqref{key3}, we must have
$$
 \frac{\varrho_{\zz}(u,2t\Rk) }{\varrho_{\zk}(u,2\Rk)} >t^{\cttb}\qquad \mbox{ for all }   t\in (0,1),\ \zz\in \cZ_{t\Rk}, \ B_{t \Rk}(\zz)\subset B_{\Rk} (\zk ). 
$$
or equivalently, recalling \eqref{eq:varrho_def_intro},
\begin{equation}
\label{eq:rhoI}
  \frac{\ccI(u,B_{2t\Rk}(\zz)) }{\ccI(u,B_{ 2\Rk}(\zk))} >t^{\frac{1+\cttb}{3}}\qquad \mbox{ for all }   t\in (0,1),\ \zz\in \cZ_{t\Rk}, \ B_{t \Rk}(\zz)\subset B_{\Rk} (\zk ). 
\end{equation}
Now, for $k$ fixed and $\zeta \in (0,1)$, define $\cA_\zeta := \cA_{\tilde\zz_k, \Rk}^\zeta$ (recall \eqref{eq:Ndef}), so that in the definition of $\bN(\zeta, B_{\Rk}(\zz_k))$,  we are considering the covering $\{B_{\zeta \Rk}(\zz') : \zz'\in \cA_\zeta\}$. Then, by Besicovitch covering theorem there exists a universal constant $C_3$ such that, for $C_3$ distinct subfamilies $\cA_{\zeta}^{(1)},\dots,\cA_{\zeta}^{(C_3)}\subset \cA_\zeta$, we have 
\begin{equation}
\label{eq:Besic_cupcup}
\bigcup_{\zz'\in \cA_{\zeta}} B_{\zeta \Rk }(\zz') \subset \bigcup_{j = 1}^{C_3} \bigcup_{\zz'\in \cA^{(j)}_{\zeta}} B_{2\zeta \Rk}(\zz'),
\end{equation}
and, for each $j\in \{1, \dots, C_3\}$, the family $\{B_{\zeta\Rk }(\zz') : \zz'\in \cA_\zeta^{(j)}\}$ consists of disjoint balls. Thus 
\[
\ccI(u, B_{2\Rk}(\tilde\zz_k)) \ge \ccI\biggl(u, \bigcup_{\zz'\in \cA_\zeta^{(j)}} B_{\zeta\Rk }(\zz')\biggr)\ge \sum_{\zz'\in \cA_\zeta^{(j)}} \ccI\left(u, B_{\zeta\Rk }(\zz')\right)\ge \#\cA_{\zeta}^{(j)}\cdot \min_{\zz'\in \cA_\zeta^{(j)}} \ccI\left(u, B_{\zeta\Rk }(\zz')\right),
\]
for every $j\in \{1, \dots, C_3\}$.
Also, because of \eqref{eq:rhoI}, 
\[
\min_{\zz'\in \cA_\zeta^{(j)}} \ccI\left(u, B_{\zeta\Rk }(\zz')\right) \ge (\zeta/2)^{\frac{1+\cttb}{3}} \ccI(u, B_{2\Rk}(\zk)),
\]
therefore
\[\#\cA_{\zeta}^{(j)} \le C\zeta^{-\frac{1+\cttb}{3}}\qquad\text{for every $j\in \{1, \dots, C_3\}$}, 
\]
for some universal constant $C$. Thus, choosing $\tilde \cA_\zeta:= \bigcup_{j = 1}^{C_3} \cA_\zeta^{(j)}$, it follows from \eqref{eq:Besic_cupcup} that
\begin{equation}
\label{eq:contained}
\bigcup_{\zz'\in \cA_{\zeta}} B_{\zeta \Rk }(\zz') \subset \bigcup_{\zz'\in \tilde \cA_\zeta}B_{2\zeta \Rk}(\zz')\qquad\text{with}\quad \#\tilde\cA_\zeta\le C \zeta^{-\frac{1+\cttb}{3}}, 
\end{equation}
and 
\[
\biggl|\bigcup_{\zz'\in \cA_{\zeta}} B_{\zeta \Rk }(\zz')\biggr|\le C  (2\zeta\Rk)^3 \#\tilde\cA_\zeta\le C (\zeta\Rk)^3  \zeta^{-\frac{1+\cttb}{3}}.
\]
Thus, \eqref{strongalternativeA} holds.

\medskip
\noindent{\bf Step 2:} We now prove \eqref{eq:rbepk} and \eqref{keyeqn}.

Note that, by the monotonicity of $\ccI$ (see \eqref{eq:I monotone}) and the definition of $\varrho$ (see \eqref{eq:varrho_def_intro}),  for all $\overline R>0$ and $\zz\in B_{\overline R} (\overline \zz)$ such that $B_{R} (\zz)\subset B_{\overline R} (\overline \zz)$ we have
\[
R \,\varrho_{\zz}(u, 2R)   \le   \overline R \,   \varrho_{\overline \zz}(u, 2\overline R).
\]
Combined with \eqref{key1}, this gives
\begin{equation}\label{keyeqnbis}
  \bE_\zz(u,8R)  \le   2\Big(\tfrac{\Rk}{R}\Big)^\ctta    \widetilde\eps_k \qquad \text{for all}\quad  \zz\in \cZ_R  \ \text{and }\ R\le \Rk,\ \text{with}\ B_R(\zz)\subset B_{\Rk}(\zk).
\end{equation}
Then, using \eqref{eq:rbeps} and \eqref{keyeqnbis} with $\zz = \zk\in \cZ_{\Rk}$ and $R = \Rk$, \eqref{eq:rbepk} follows (recall that $3\cttg\ctta > 1$).  
Consequently, noticing that
\begin{equation}\label{ZR_equals_Z}
    \cZ_{R}\cap B_{3\Rk/2}(\zk)= \cZ\cap B_{3\Rk/2}(\zk) \qquad \mbox{for all } R\ge \frac12\epk^{1/\alpha} \Rk,
\end{equation}
\eqref{keyeqnbis} implies \eqref{keyeqn}.
\end{proof}

\subsection{Definition of $U_+$ and $U_-$} \label{ssec:Upm} Recall the sets $\Omega^{(\pm)}$ introduced in \Cref{Omegapusminus}. Our first goal is showing the following lemma---a structural property that says that the sets $\Omega^{(\pm)} = \Omega^{(\pm)}\subset \{u > 0\}$  are connected, disjoint open sets of $\{u > 0\}$ (roughly, two half-spaces in $B_R(\zz)$, which only miss $\cT^{\rm neck}$):  

\begin{lem}\label{Omegapusminus}
Let $(\mathcal N,p)$ be the tree provided by \Cref{prop:geomtree} with root $B_R(\zz)$, and let $\Omega^{(\pm)} = \Omega^{(\pm)}(B_R(\zz))$ be as in \Cref{def:Bpm_Omegapm}.

Then there exists $\theta_\circ'>0$ universal such that, for all $\theta\in(0,\theta_\circ')$, $\Omega^{(+)}$ and $\Omega^{(-)}$ are disjoint, open, connected, and satisfy 
\[
 \{\pm e\cdot(x-\zz) > \theta^4 R\} \subset \Omega^{(\pm )}\subset \{\pm e\cdot(x-\zz) > - \theta^4 R\}\qquad\text{in}\quad B_R(\zz), 
\]
where $e$ is the polarity of the root.
Moreover, their union covers $\{u > 0\} \cap B_R(\zz)$ minus the union of `neck-type' terminal balls:
\[
\big(\{u > 0\} \cap B_R(\zz)\big) \setminus \bigcup \cT^{\rm neck}  \subset \Omega^{(+)} \cup \Omega^{(-)}.
\]
 \end{lem}

The proof of \Cref{Omegapusminus} relies crucially on the following result:

\begin{lem}\label{gradmap}
Under the same assumptions as in \Cref{Omegapusminus}, define $F := \nabla u|_{\{u>0\}}$ and denote by $\Phi^F = \Phi^F(x,t)$ the associated flow (with maximal domain): 
\begin{equation}\label{eq:flow}
\left\{
\begin{array}{rcll}
 \dot{\Phi}^F(x, t) & =&  F(\Phi^F(x, t)))&\qquad\text{for}\quad t > 0,\\
\Phi^F(x, 0) & = & x.&
\end{array}
\right.
\end{equation}
For any given $B\in \cN \setminus \{ B_R(\zz) \}$  and $x\in B^{(+,5/4)}$  there exists $\tau>0$ such that $\Phi^F(x,t) \in B^{(+,3/2)}$ for $t\in [0,\tau]$ and 
\[
\Phi^F(x,\tau') \in p(B)^{(+,5/4)}, \quad \text{for some $\tau' \in  (0, \tau)$ }.
\]
(Recall that $p(B)$ is the predecessor of $B$.)
The same statement holds with $+$ replaced by $-$.
\end{lem}
\begin{proof}
We can always assume that $\theta \in (0,\theta_\circ)$, where $\theta_\circ$ is as in \Cref{prop:geomtree}. 
Given some ball $B = B_{\varrho}(y)\in \cN$  with polarity $e$ and its predecessor $p(B) = B_{\varrho'}(y')$  ---thus $\varrho =\theta \varrho'$--- with polarity $e'$, by \Cref{prop:geomtree} we know that
\begin{equation}\label{gradvspolar2_3}
y\in B_{\varrho'}(y')\cap \{x\ : \  |e\cdot (x-y')|\le \theta^4\varrho'\}, \qquad |e-e'|\le \theta^3.
\end{equation}
Combining this with \Cref{lem:grad_lower_bound} one can easily see that (for $\theta$ universally small) the integral curves of $\nabla u$ starting at 
$B^{(+,5/4)}$ meet $p(B)^{(+,5/4)}$ for some universal time before leaving $B^{(+,3/2)}$.
\end{proof}

We can now show \Cref{Omegapusminus}.

\begin{proof}[Proof of \Cref{Omegapusminus}]
We will exploit the tree structure to reason by induction. Note that the covering property follows from \Cref{lem:Omegapmcovers}. 

Set 
$\mathcal N = \bigcup_{\ell\ge 0} \mathcal N^{(\ell)}$ as in \Cref{defrootedtree}, and let $\Omega^{(\le \ell, \pm)}$ as in \Cref{lem:Omegapmcovers}  (see \eqref{eq:Omegaleqell}). 
We need to show that the sets $\Omega^{(+)}$ and $\Omega^{(-)}$ are disjoint, connected, and open.

The openness directly follows because each set is a finite union of open sets, intersected with an open ball.

For the connectedness, we will show by induction over $\ell=0,1,2,\dots$ that the two sets $\Omega^{(\le \ell, +)}$ and $\Omega^{(\le \ell, -)}$ are connected. Since the tree is finite, these two sets will eventually coincide with $\Omega^{(+)}$ and $\Omega^{(-)}$.
Since the root is always internal (see \Cref{prop:geomtree}),  $\Omega^{(\le 0,+)}$ and $\Omega^{(\le 0,-)}$ coincide respectively  with $B^{(+)}$ and $B^{(-)}$  as in  \Cref{def:Bpm_Omegapm} (with $B$ being the root and $e$ its polarity). Each of these two sets is connected (and they are disjoint). Now assuming that   $\Omega^{(\le \ell-1, \pm)}$ are connected    open sets for some $\ell\ge 1$, the result follows by induction from the following observation, which is a consequence of \Cref{prop:geomtree} (for $\theta$ small): for any given  $B\in\mathcal N^{(\ell)} \cap (\mathcal I\cup \mathcal T^{\rm reg})$ we have 
\[
   B^{(\pm)}\cap\Omega^{(\le \ell-1, \pm)}\neq \varnothing.
\]
 Since for any $B\in\mathcal N^{(\ell)}$ we have $p(B)\in\mathcal N^{(\ell-1)} \cap \mathcal I$, we obtain $p(B)^{(\pm)}\subset \Omega^{(\le \ell-1, \pm)}$, so the connectedness follows. 

To show that $\Omega^{(+)}$ and $\Omega^{(-)}$ are disjoint we use \Cref{gradmap} iteratively. Indeed, if $\bar x\in \Omega^{(+)}\cap \Omega^{(-)}$, repeated iterations of \Cref{gradmap} for both $+$ and $-$ (notice that the flow is always well defined, since the value of $u$ increases along it) imply that 
$$
\Phi(\bar x, T_+) \in B_{5R/4}(\zz)\cap \{e\cdot(x-\zz) > \theta^2 R\}\qquad \text{ and }\qquad \Phi(\bar x, T_-) \in B_{5R/4}(\zz)\cap \{e\cdot(x-\zz) <- \theta^2 R\}
$$
for some $T_\pm > 0$, where $e = \boldsymbol{e}(B_R(\zz))$. In addition, $\Phi(\bar x, t) \in B_{3R/2}(\zz)$ for all $t < \max\{T_+, T_-\}$. 
Assume now, without loss of generality, that $T_+ < T_-$.  Then, since $\Phi(\bar x, T_+) \in B_{5R/4}(\zz)\cap \{e\cdot(x-\zz) > \theta^2 R\}$ and $\nabla u$ is very close to $e$ (see \eqref{gradvspolar2}), for $t \geq T_+$ the flow goes outside of $B_{3R/2}(\zz)$ without crossing $\{e\cdot(x-\zz) = 0\}$, a contradiction to the fact that $\Phi(\bar x, T_-) \in B_{5R/4}(\zz)\cap \{e\cdot(x-\zz) <- \theta^2 R\}$. Hence, $\Omega^{(+)}$ and $\Omega^{(-)}$ are disjoint. 
\end{proof}

 We can now define the sets $U_\pm $ and $U_0$.

 \begin{defn}
 \label{def:UpUm}
Given $\theta>0$ universal such that \Cref{Omegapusminus} holds, let $\Rk$ and  $\zk$ be given by \Cref{lem:tildezktildeRk}, and let $\Omega^{(\pm )} = \Omega^{(\pm )}(B_{\Rk}(\zk))$ be the two  open connected subdomains of $\{u>0\}\cap B_{\Rk}(\zk)$ constructed in \Cref{def:Bpm_Omegapm}. 
Fix a subset $\tilde \cA\subset \cA^{\tilde \zeta}_{\zk, \Rk}= \cZ_{\tilde \zeta \Rk} \cap B_{\Rk}(\zk)$   (recall \eqref{eq:Ndef}) with $\tilde \zeta =\frac12  \epk^{1/\alpha}$ as in \Cref{lem:Nbound} such that 
\begin{equation}
\label{eq:tildeA}
(\mathcal Z\cap B_{\Rk}(\zk))+B_{\frac12\Rk \epk^{1/\alpha}} \subset \bigcup_{\zz\in\tilde \cA} B_{\Rk \epk^{1/\alpha}} (\zz),\qquad \#\tilde \cA\le C \epk^{-\frac{1+\cttb}{3\alpha}}
\end{equation}
  (note that $\cA_{\zk, \Rk}^{\tilde \zeta}  = \cZ\cap B_\Rk (\zk)$ by \eqref{ZR_equals_Z}), 
  and define 
\begin{equation}\label{depU1U2}
U_0  : =  \bigcup_{\zz\in\tilde \cA} B_{\Rk \epk^{1/\ctta}} (\zz),
\qquad
U_+  : = \left( \Omega^{(+)}\cap B_{\Rk/2}(\zk)\right)\setminus \overline{U_0},
\qquad \mbox{and}\qquad 
U_-  : = \left(\Omega^{(-)}\cap B_{\Rk/2}(\zk)\right)\setminus \overline{U_0}.
\end{equation}
\end{defn}

We note that since $\Omega^{(+)}\cap \Omega^{(-)}=\varnothing$ by \Cref{Omegapusminus}, the sets
$U_+$ and $U_-$ are disjoint open subsets of $B_{\Rk}(\zk)$. 
We also remark that the sets $\Omega^{(\pm)}$, $U_{\pm}$, and $U_0$  depend on $k$, but we drop this dependence in our notation for the sake of readability.

The following observations on $U_+, U_-$ will be used several times in the sequel:
\begin{lem}
    \label{lem:obsUpm}
Let $\Omega^{(\pm)} = \Omega^{(\pm)}(B_\Rk(\zk))$ and $U_\pm$ be as in \Cref{def:UpUm} above. Then: 

\begin{enumerate}[(i)]
\item The (disjoint open) sets $U_+$ and $U_-$ satisfy
\begin{equation}\label{bdryUpm1}
     U_+ \cup  U_- =  \left(\{u>0\}\cap B_{\Rk/2}(\zk)\right)\setminus \overline{U_0}\quad 
\end{equation}
and
\begin{equation}\label{bdryUpm}
    \left((\partial U_+ \cup \partial U_-)\cap B_{\Rk/2}(\zk) \right)\setminus \overline{U_0} \subset \partial\{u>0\}.
\end{equation}

\item For all $\zz\in \cZ$ and  $R\ge \epk^{1/\alpha} \Rk$ such that $B_{R}(\zz)\subset B_{\Rk/2}(\zk)$ there is $e\in \bS^2$ 
such that
\begin{equation} \label{flatnessatallscalesofinterest}
 \left\{\pm e\cdot(x-\zz) > C\epk (\Rk/R)^\alpha R\right\} \subset U_\pm \subset \left\{\pm e\cdot(x-\zz) > - C \epk(\Rk/R)^\alpha R\right\}\qquad\text{in}\quad B_R(\zz), 
\end{equation}
for some $C$ universal.
\end{enumerate}
\end{lem}
\begin{proof}
{\em (i)}  
We first show that the union of all neck balls of $\cN = \cN (B_{\Rk}(\zk))$ that intersect $B_{\Rk/2}(\zk)$ is contained in $U_0$, namely,
\begin{equation}\label{eq:necksCU0}
\bigcup \left\{ B\in \cT^{\rm neck} \ : \ B\cap B_{\Rk/2}(\zk) \neq \varnothing\right\} \subset U_0.     
\end{equation}
To prove this, let $B=B_\varrho(y)\in \cT^{\rm neck}$ with $B\cap B_{\Rk/2}(\zk) \neq \varnothing$.
On the one hand, by definition\footnote{See \Cref{prop:geomtree} and \Cref{defiregball}} of neck ball, we have $y\in B_{\Rk}(\zk)\cap\{u=0\}$, $\varrho\le \theta \Rk$, $B_{2\varrho}(y)\cap \cZ$ is nonempty, and for every $\zz\in B_{2\varrho}(y)\cap \cZ$ it holds $\varrho <M(\theta)\rb(\zz)$.
On the other hand, picking an arbitrary $\zz \in B_{2\varrho}(y)\cap \cZ$, \eqref{eq:rbepk} implies that $\rb(\zz) \le C\Rk \epk^{3\gamma} \ll \Rk \epk^{1/\ctta}$.
Thus, we conclude that $\varrho\ll  \Rk\epk^{1/\alpha}$ as $k\to \infty$,
and therefore $B\subset U_0$ (recall \eqref{depU1U2}). This proves \eqref{eq:necksCU0}.

Recall now that, by \Cref{lem:Omegapmcovers}, $\Omega^{(+)}$, $\Omega^{(-)}$, and the neck balls cover all of $\{u>0\}\cap B_{\Rk}(\zk)$. 
Thus \eqref{eq:necksCU0} implies \eqref{bdryUpm1}, from which  \eqref{bdryUpm} is a direct consequence.

{\em (ii)} Let $C_*$ be a large universal constant to be chosen later. Notice that  \eqref{flatnessatallscalesofinterest} becomes trivially true at scales $R<C_* \Rk \epk^{1/\alpha}$ with $C = C_*^\alpha$, so we can assume $R\ge C_*  \Rk\epk^{1/\alpha}$.

From \eqref{keyeqn} and \Cref{lem:boundbelow} we know that, for some $e\in \mathbb{S}^2$, 
\[
\{u = 0\}\cap B_{R}(\zz) \subset {\rm Slab}\!\left(B_{R}(\zz), e, C_1   (\Rk/R)^\alpha \epk\right). 
\]
with $C_1>1$ universal. Hence, recalling \eqref{depU1U2} and that $R\ge \Rk \epk^{1/\alpha}$, we obtain
\[
(\{u = 0\}\cup U_0)\cap B_{R}(\zz) \subset {\rm Slab}\!\left(B_{R}(\zz), e, 2C_1(\Rk/R)^\alpha\epk\right).
\]
This will imply \eqref{flatnessatallscalesofinterest},  provided we can show that both $U_+$ and $U_-$ must intersect the set $$\left(\{u > 0\}\cap B_{R}(\zz) \right)\setminus {\rm Slab}\!\left(B_{R}(\zz), e, 2C_1(\Rk/R)^\alpha\epk\right)$$ for $R\ge C_*\Rk \epk^{1/\alpha}$. 
To show this we recall that, by the proof of (i), the radii of the neck balls are much smaller than $\epk^{1/\alpha}\Rk$. Thus, if $C_*$ is sufficiently large (depending on the parameter $\theta$ in the construction of the ball tree), there exist balls $B\in \mathcal N$ {(in particular, centered at points of $\{u = 0\}$)} that:\\
- are either regular terminal or interior;\\
-  are fully contained in $B_R(\zz)$;\\
- and their radius is larger than $cR$, for some $c=c(\theta)>0$.\\
Reasoning with these balls (and recalling  \Cref{def:Bpm_Omegapm} and \Cref{prop:geomtree}), we deduce that both $B^{(+)}$ and $B^{(-)}$ (and therefore both $U_+$ and $U_-$) must intersect $\{u > 0\}\cap B_{R}(\zz) \setminus  {\rm Slab}\!\big(B_{R}(\zz), e, 2C_1(\Rk/R)^\alpha\epk\big)$.
This concludes the proof. 
\end{proof}

\section{Linearization}
\label{sec:linear}
In this section, we fix the sets $U_\pm\subset B_{\Rk/2}(\zk)$ and $U_0$ from \Cref{def:UpUm}, where $\Rk$ and  $\zk$ are given by \Cref{lem:tildezktildeRk}. We define 
 the asymmetric excess $\bA_{\zz}(u, R)$  for balls $B_R(\zz) \subset B_{\widetilde{R}_k/2}(\widetilde{\zz}_k)$   as follows (see~\eqref{eq:Az_def_intro}):
\begin{equation}
\label{eq:def_Az}
\begin{split}
\bA_{\zz}(u, R) & := \max_{\ast\in \{+, -\}} \min_{a \in \mathbb{S}^2, \, b \in \mathbb{R}} \frac{1}{R |B_R(\zz)|} \int_{U_\ast \cap B_R(\zz)} \big| u(x) - a \cdot x - b \big| \, dx\\
& \,= \max_{\ast\in \{+, -\}} \min_{a \in \mathbb{S}^2, \, \bar b \in \mathbb{R}} \frac{1}{R |B_R(\zz)|} \int_{U_\ast \cap B_R(\zz)} \big| u(x) - a \cdot (x - \zz) - \bar b \big| \, dx.
\end{split}
\end{equation}
In this section (and for the following ones) we also fix (in addition to the constants in \eqref{eq:gamma_alpha_beta}) the constant
\begin{equation}
    \label{eq:chi}
    \cttc := \valcttc. 
\end{equation}
The goal of this section is to show the following result on the decay of the excess, with a two-scale behaviour depending on the size of the radius. More precisely, we first prove that up to some mesoscopic scale (depending on $\epk$) there is a $C^{1,1/3}$ improvement in flatness, while
until a second smaller mesoscopic scale there is a sort of ``preservation on average'' of the $L^\infty$ norm.

\begin{prop}
\label{prop:decaysqrt}
Let $\Rk$ and  $\zk$ be given by \Cref{lem:tildezktildeRk}, and let $R_k^\flat  := \epk^\cttc \Rk$. Then, for  every $\zz\in B_{\Rk/16}(\zk)\cap \cZ$,
\begin{equation}
\label{eq:prop62}
\bA_\zz(u, R) \le \left(\frac{R}{\Rk}\right)^{1/3}  \epk \qquad\text{for all}\quad R\in  \left[R_k^\flat, \Rk/C\right]
\end{equation}
for some $C>0$ universal. 
Moreover, there exist $a^\flat_+, a^\flat_-\in \mathbb S^2$ and $b^\flat_+, b^\flat_-\in \mathbb R$  
such that 
\begin{equation}\label{equation:weakestdecay}
\fint_{U_\ast\cap B_{r}(\zz)} |u-a^\flat_\ast \cdot x- b^\flat_\ast| \,dx \le C \epk^{1+\cttc/3} R_k^\flat \qquad\text{for}\quad \ast\in \{+, -\},\qquad \text{for all}\quad r\in [\epk^{1+2\cttc} \Rk, R_k^\flat ],
\end{equation}
with $C>0$ universal.
\end{prop}

\noindent\textbf{Roadmap of the section (see also  subsection \ref{subsection_sketch_of_proof_Bernoulli}, point 8).}
Our goal now is to enter the \emph{linearization regime} provided by the density control from \Cref{sec:selection}, and to upgrade it into a quantitative decay of the (new) \emph{asymmetric excess} in \eqref{eq:def_Az}.  Concretely, after fixing the two open sets $U_\pm$ from~\Cref{def:UpUm} (which represent the two ``sides'' of $\{u>0\}$ once the necks are removed), we allow \emph{independent} affine approximations on $U_+$ and $U_-$ and measure the resulting error by $\bA_{\zz}(u, R)$ in $B_R(\zz)$; this is the right quantity for iteration because it decouples the two half-space behaviors across the neck region.  The main result is \Cref{prop:decaysqrt}, which gives a two-scale improvement scheme for $\bA_{\zz}$: up to a first mesoscopic threshold one gets a genuine improvement of flatness, while down to a second (smaller) mesoscopic threshold one retains a controlled ``preservation of average oscillation’’.   

The proof has three components.
In Subsection~\ref{ssec:linearization}, we establish auxiliary estimates showing that, once the optimal affine functions on each side are subtracted, the resulting error function is \emph{approximately} harmonic with small/vanishing Neumann condition, where the key difficulty is that the mismatch in normal derivatives is only small in an \emph{integrated} sense (rather than pointwise), reflecting the presence of small but non-negligible necks. 

In Subsection~\ref{ssec:compactness} we prove abstract compactness/approximation lemmas tailored to this setting: sequences of bounded functions solving the linearized equation up to small errors in a flat-Lipschitz domain converge (after normalization) to even harmonic limits (or low-degree harmonic polynomials), providing the rigidity needed for excess decay. 

Finally, in Subsection~\ref{ssec:decaysqrt}, we apply these compactness statements in an iterative Campanato-type argument: Step~1 yields an algebraic decay of the asymmetric excess from scale to scale at a fixed center, and Step~2 keeps track of how the affine coefficients evolve across scales, producing the two-scale conclusions of \Cref{prop:decaysqrt}.

 \subsection{Linearization: auxiliary results} \label{ssec:linearization} We start by proving the following:

\begin{lem}
\label{lem:AzEz} For  $B_R(\zz)\subset B_{\Rk/2}(\zk)$ with $R\ge\epk^{1/\alpha} \Rk$, we have 
\[
  \bA_\zz(u,R)  \le   C\Big(\tfrac{\Rk}{R}\Big)^\ctta    \widetilde\eps_k ,
\]
where $C$ is universal.
\end{lem}

\begin{proof}
Fix $\zz\in \cZ$ and $R\ge\epk^{1/\alpha} \Rk$ such that $B_R(\zz)\subset B_{\Rk/2}(\zk)$. Also, define $\eta= \Big(\tfrac{\Rk}{R}\Big)^\ctta\epk$.

On the one hand, by \eqref{keyeqn} and H\"older's inequality, we obtain
\begin{equation}\label{aux111}
\int_{B_R(\zz)} |u - V_{\zz, e}|  \,dx  \le   C\eta R^4.
\end{equation}
On the other hand, by \eqref{flatnessatallscalesofinterest},
\begin{equation}\label{aux222}
D_{\pm} : = \big(U_\pm\setminus \{\pm e\cdot(x-\zz) > C\eta R\} \big)\cap B_{R}(\zz) \subset {\rm Slab}\!\left(B_{R}(\zz), e, C\eta\right).
\end{equation}
(We notice that in the proof of \eqref{flatnessatallscalesofinterest}, the vector $e$ is given by \Cref{lem:boundbelow}, from which we deduce that the unit vectors $e$ in \eqref{aux111} and \eqref{aux222} are indeed the same.)

Also, since $u(\zz)=0$ (because $\zz\in \cZ$), from the gradient bound $|\nabla u|\le 1$  we obtain
\begin{equation}\label{aux333}
 \sup_{B_R(\zz)} |u - V_{\zz, e}|\le CR.
\end{equation}
Combining \eqref{aux111},\eqref{aux222},  and \eqref{aux333},
we get
\[
\int_{U_\pm } |u(x)\mp(e\cdot x)|\, dx \le C\eta R^4 + C|D_{\pm}|R \le C\eta R^4,
\]
thus $\bA_{\zz}(u,R) \le C\eta$, as wanted.
\end{proof}

We next state an abstract lemma that will be applied in the context of \Cref{lem:tildezktildeRk}. It provides pointwise gradient and flux bounds in a linearization regime. 

\begin{prop}[Linearization] \label{proplinearization}
There exist $a_+\in \bS^{2}$ and $b_+\in \R$ such that, denoting
$ 
v(x) := u(x) - a_+\cdot x-b_+,
$
we have
\[
\frac{1}{\Rk}
\fint_{B_{\Rk} (\zk)\cap U_+} |v| \,dx \le C\,  \epk,
\]
for some $C$ universal. Moreover, for all $\bar{x}\in   \overline{U_+} \cap B_{\Rk/4}(\zk)$,
\[
  |\nabla v (\bar x)| = | \nabla u(\bar{x})-a_+|
\leq  C\biggl(\frac{\Rk}{\dist(\bar x, \cZ)}\biggr)^{\ctta}  \epk.
\]
In particular, for all $\bar{x}\in  \partial U_+\cap B_{\Rk/4}(\zk)$ with $\dist(\bar{x},\cZ) \geq 
\Rk \epk^{1/\ctta} $, 
we have $\bar x \in \FB(u)$  and
\[
|\partial_\nu v(\bar x)| = |1-a_+\cdot \nu(\bar{x})|
\leq  C\biggl(\frac{\Rk}{\dist(\bar x, \cZ)}\biggr)^{2\ctta}  \epk^{2},
\]
where $\nu(\bar x)$ is the inward unit normal vector to $\FB(u)$, and the constant  $C$ is universal

The same statement holds with $a_+, b_+, U_+$ replaced by $a_-, b_-, U_-$.
\end{prop}

\begin{proof}
By \Cref{lem:AzEz} we have 
\begin{equation}\label{eq:Azbound}
\bA_{\zz}(u, R) \le C \bigg(\frac{\Rk}{R}\bigg)^\alpha \epk, 
\end{equation}
as long as $B_R(\zz)\subset B_{\Rk/2}(\zk)$ and $R\ge\epk^{1/\alpha} \Rk$. 
In particular, choosing $\zz=\zk$ and $R=\Rk/2$ and using  the definition of the asymmetric excess in \eqref{eq:def_Az}, it follows that there exist $a_+ \in \bS^{2}$ and $b_+ \in \R$ such that
 \begin{equation}
 \label{eq:u R}
      \frac{1}{R^4} \int_{U_\ast \cap B_R(\zz)} \big| u(x) - a_+ \cdot x - b_+ \big| \, dx \le C \epk,
 \end{equation}
 so the first inequality in the statement of the proposition holds.

 Now, fix $\bar{x} \in \overline{U_+} \cap B_{\Rk/4}(\zk)$ and define $\rho:=|\bar x - \zz(\bar x)|\le \frac{\Rk}{4}$, where $\zz(\bar x)\in \cZ\cap B_{\Rk/2}(\zk)$ is such that $\dist(\bar{x},\cZ) = \rho$. Since $| \nabla u(\bar{x})-a_+|\le 2$, our desired estimate is trivially true if $\rho \le C_* \Rk \epk^{1/\alpha}$, so we can assume $\rho \ge C_*\Rk \epk^{1/\alpha}$ for a universal $C_*\ge 2$. Moreover, since $\rb(\zz(\bar x))\le \tfrac{1}{10}\Rk \epk^{1/\alpha}$ (by \eqref{eq:rbepk}), we have $\overline{B_{\rho/2}(\bar x)}\cap \overline{U_0} = \varnothing$. 
  
Now, on the one hand, set $\rho_j:=2^{j}\rho$ with $0\leq j\leq j_{\max}:=\floors{\log_2(\frac{\Rk}{2\rho})}$.
Then \eqref{eq:Azbound} gives the existence of  $a_j\in \bS^{2}$ and $b_j\in\R$ such that
\begin{equation}\label{eq:lin-asym-j}
\frac{2}{\rho_j}
\fint_{B_{\rho_j/2}(\bar{x})\cap U_+}
	|u(x)-a_{j}\cdot x-b_{j}|
\,dx
\leq \frac{C}{2\rho_j}
\fint_{B_{2\rho_j}(\zz(\bar x))\cap U_+}
	|u(x)-a_{j}\cdot x-b_{j}|
\,dx\leq  C \bigg(\frac{\Rk}{\rho_j}\bigg)^\ctta \epk
\end{equation}
for $0 \leq j \leq j_{\max}$. Hence, if we set 
 $a_{j_{\max}+1}=a_+$ and $b_{j_{\max}+1}=b_+$, 
it follows from the bounds above and \eqref{eq:u R} that 
\[
|a_{j}-a_{j-1}|
\leq 
\frac{C }{\rho_j}
\fint_{B_{\rho_j(\zz(\bar x))}\cap U_+}
	|(a_{j}-a_{j-1})\cdot x|
\,dx
\leq C\bigg(\frac{\Rk}{\rho_j}\bigg)^\ctta   \epk,
	\qquad 1\leq j\leq j_{\rm max }+1
\]
 (note that, because of \eqref{flatnessatallscalesofinterest}, at the scales of interest $U_+$ is roughly a half-space).

On the other hand, since $B_{\rho/2}(\bar x)\cap \overline{U_0} = \varnothing$, 
it follows from \eqref{eq:necksCU0} 
that $u\big|_{U_+}$  is a (classical) solution to the Bernoulli problem in $B_{\rho/2}(\bar x)$. Hence, by \eqref{eq:lin-asym-j} for $j=0$, we can apply a rescaled version of \Cref{lem:L1Linfty_2} to obtain
\[
|\nabla u(\bar{x})-a_0| \leq C\bigg(\frac{\Rk}{\rho}\bigg)^\alpha \epk,  \quad \mbox{in } U_+\cap B_{\rho/4}(\bar{x}).
\]
Summing up,
\begin{align*}
|\nabla u(\bar x)-a_+|
&\leq C  \sum_{j=1}^{j_{\max}+1}
  \left(\frac{\Rk}{\rho_j}\right)^\ctta \epk
\leq C \bigg(\frac{\Rk}{\rho}\bigg)^\ctta\epk.
\end{align*}
This proves the desired bound on $\nabla v$.

Finally, assume $\bar x \in \partial U_+\cap B_{\Rk/4}(\zk)$ with $\dist(\bar x, \cZ) \ge \Rk \epk^{1/\ctta}$. Thanks to \eqref{bdryUpm}, since  $\bar x\notin U_0$ then $\bar x\in \FB(u)$. Thus, since $a_+$ and $\nabla u(\bar x)$ are unit vectors, we get
\begin{align*}
|1-a_+\cdot \nabla u(\bar{x})|
&=\frac{1}{2}|a_+-\nabla u(\bar{x})|^2
\leq C\bigg(\frac{\Rk}{\rho}\bigg)^{2\ctta}\epk^2, 
\end{align*}
as desired. 
\end{proof}

\begin{prop}
\label{prop:neumann_bound}
Let $v$ be as in \Cref{proplinearization}. For any given $\beta\in [0,1]$ satisfying $12 \ctta\beta< 5 -\cttb$, we have
\[
\int_{(\partial U_+\cup \partial U_-)\cap B_{\Rk/8}(\zk)} |\partial_\nu v|^2  \,d\cH^{2}
\le C\,
\Rk^2\epk^{4\beta} ,\]
where $\nu$ is the unit inward normal vector and $C$ depends only on  $\beta$. 
\end{prop}
\begin{remark}\label{rem:pequal2}
Recalling \eqref{eq:gamma_alpha_beta}, one can choose $\beta := \frac{2+\delta_\circ}{4}$ with $\delta_\circ := \tfrac{1}{10}$. 
\end{remark}

\begin{proof}
We prove it for $U_+$, the same proof works for $U_-$. Observe that, for $\bar x\in \FB(u)$,  \Cref{proplinearization} implies that 
\[
|\partial_\nu v(\bar x)| \le C \biggl(\frac{\Rk}{ \dist(\bar x, \cZ )}\biggr)^{2\ctta} \epk ^2 
\qquad\text{for}\quad 
\bar x\in \partial U_+\cap B_{\Rk/4}(\zk)\cap\left\{ {\rm dist}(\cdot, \cZ) \ge \Rk \epk^{1/\ctta}\right\}. 
\]
In particular, since 
$
|\partial_\nu v(\bar x)|  \le 2
$, 
for any $\beta\in [0,1]$ we have 
\begin{equation}
\label{eq:gtlc}
|\partial_\nu v(\bar x)| \le 2|\partial_\nu v(\bar x)|^\beta \le C  \bigg(\frac{\Rk}{\dist(\bar x, \cZ)}\bigg)^{2\ctta\beta} \epk ^{2\beta} 
\quad \text{ for} \quad 
\bar x\in \partial U_+\cap B_{\Rk/4}(\zk)\cap\left\{ {\rm dist}(\cdot, \cZ) \ge \Rk \epk^{1/\ctta}\right\}.
\end{equation}
Now, for  each $t \in (0, \Rk)$, consider the sets 
\[
D_t : = \bigcup_{\zz'\in \cZ \cap B_\Rk(\zk)}  B_{ t}(\zz') \cap \partial U_+\cap B_{\Rk/8}(\zk).
\] 
Notice that, by the definition of $U_0$ (cf. \Cref{def:UpUm}), for $ t\ge \Rk \epk^{1/\ctta}$ the set $D_t$ covers  $\overline{U}_0 \cap \partial U_+$ . In particular, if $\tilde t := \Rk \epk^{1/\ctta}$ then $D_t\setminus D_{\tilde t}\subset \FB(u)$ (see \eqref{bdryUpm}). Hence, by \eqref{strongalternativeA}, \Cref{lem:Nbound}, and  the perimeter bound from \Cref{lem:perbound}, we obtain
\[
\cH^2(D_t\setminus D_{\tilde t}) \le  C t^2 (t/\Rk)^{-\frac{1+\cttb}{3}}\qquad\text{for}\quad  t\ge \tilde t.
\]
In addition, again by \eqref{bdryUpm}, $\cH^2(D_{\tilde t})\le \cH^2(\FB(u)\cap D_{\tilde t})+\cH^2(\partial U_0)$. Thus, arguing as above,
\[
\cH^2(D_{\tilde t})\le C (2\tilde t)^2 (2\tilde t/\Rk)^{-\frac{1+\cttb}{3}} +C {\tilde t}^2 (\tilde t/\Rk)^{-\frac{1+\cttb}{3}}\le C {\tilde t}^2 (\tilde t/\Rk)^{-\frac{1+\cttb}{3}}.
\]
Hence, combining the last two estimates, we conclude that
\begin{equation}
\label{eq:boundDt}
\cH^2(D_t ) \le  C t^2 (t/\Rk)^{-\frac{1+\cttb}{3}}\qquad\text{for}\quad  t\ge   \Rk \epk^{1/\alpha}.
\end{equation}
This allows us to obtain the desired estimate, using the following standard `layer cake' formula:\\
If $(E_t)_{t\in [a,b]}$ is an increasing  collection of  (measurable) sets with $t\mapsto \cH^k(E_t)$ continuous, $f:E_b \to [0,\infty)$ is integrable and satisfies  $0\le f \le g(t)$ in   $E_b\setminus E_t$ and $0 \leq f \leq g(a)$ in $E_a$, where $g\in C^1([a,b])$ is nonincreasing, then  
\[
\int_{E_b} f \, d\cH^{k} \le \int_a^b   \cH^k(E_t)\, |g'(t)| \,dt  +\cH^k(E_b)\,g(b) + \cH^k(E_a)\,g(a) .
\]
Using this formula with $E_t=D_t$, $a=\Rk\epk^{1/\ctta}$, $b=\Rk$, $f = |\partial_\nu v|^2$, and $g(t)=\min\left\{\left(\tfrac{\Rk}{t}\right)^{4\ctta \beta }\epk^{4\beta},1\right\}  $  (see \eqref{eq:gtlc}), thanks to \eqref{eq:boundDt} we obtain
\[
\begin{split}
\int_{\partial U_+\cap B_{\Rk/8}(\zk)} |\partial_\nu v|^{2} \,d\cH^{2} &\le C  \int_{\Rk\epk^{1/\ctta}}^{\Rk}  \cH^2(D_t) \left|\tfrac{d}{dt}\Big(  \Big(\tfrac{\Rk}{t}\Big)^{4\ctta \beta } \epk^{4\beta}  \Big) \right|\,dt  +C\Rk^2 \epk^{4\beta }+C \Rk^2\epk^{\frac{5-\cttb}{3\ctta}}
\\& \le   C \,\Rk^{4\ctta\beta + \frac{1+\cttb}{3} }\epk^{4\beta}  \int_{0}^{\Rk}    t^{2 -\frac{1+\cttb}{3}-4\ctta\beta-1}  dt   +C\Rk^2 \epk^{4\beta  }+C \Rk^2\epk^{\frac{5-\cttb}{3\ctta}}\\
& \le  C_{   \beta} \, \Rk^2\epk^{4\beta}  
\end{split}
\]
provided  that $2 -\frac{1+\cttb}{3}-4\ctta\beta-1>-1$, that is, $4\ctta\beta< 5/3-\cttb/3$.  
\end{proof}

\subsection{The compactness argument}\label{ssec:compactness}

We now present two abstract compactness results that will be used to show \Cref{prop:decaysqrt}:

\begin{lem}
\label{lem:compactness}
Let $n\ge 2$  and $p > 1$. For any $\eta > 0$ there exists $\delta = \delta(\eta, n, p)>0$ such that the following holds. 

Let $\Omega_\delta\subset \R^n$ be a Lipschitz and locally piecewise smooth domain, and let $v\in H^1(B_1)$ satisfy
\[
\Delta v  =  0  \quad\text{in}\quad \Omega_\delta\cap B_1 \qquad\text{and}\qquad 
\| v\|_{W^{1,p}(\Omega_\delta \cap B_1)}\le 1. 
\] 
Suppose that  
\[
B_1\cap \{ x_n \ge \delta\} \subset \Omega_\delta, \qquad \Omega_\delta\cap B_1 \subset \{ x_n \ge- \delta\},\qquad \text{and}\qquad \int_{\partial \Omega_\delta\cap B_1} |v_\nu| \,d\cH^{n-1}\le \delta,
\]
where $\nu$ is the inwards unit normal to $\partial \Omega_\delta$.
Then there exists $w:B_1\to \R$, harmonic in $B_1$ and even in $x_n$, such that 
\[
\int_{  \Omega_\delta\cap B_1} |v-w|\,dx \le \eta.
\]
\end{lem}
\begin{proof}
We argue by contradiction. Suppose that the statement does not hold. Then there exists some $\eta_\circ>0$ such that, for each $k\in \N$, there is some $v_k$ and $\Omega_k$ with 
\[
\Delta v_k  =  0  \quad\text{in}\quad \Omega_k\cap B_1 \qquad\text{and}\qquad 
\| v_k\|_{W^{1,p}(\Omega_k \cap B_1)}\le 1,
\]
and
\[
B_1\cap \{ x_n \ge 1/k\} \subset \Omega_k, \qquad \Omega_k\cap B_1 \subset \{ x_n \ge- 1/k\},\qquad \text{and}\qquad \int_{\partial \Omega_k\cap B_1} |(v_k)_\nu| \,d\cH^{n-1} \le \frac{1}{k},
\]
 such that 
\[
\int_{  \Omega_k\cap B_1} |v_k-w| \,dx >\eta_\circ > 0\qquad\text{for all $w$ harmonic in $B_1$ and even in $x_n$. }
\]
Notice first that, by harmonic estimates, up to subsequences we have that $v_k$ converges locally uniformly in $  \{x_n > 0\}\cap B_1$ to some function $v_\infty$ which satisfies 
\[
\Delta v_\infty  =  0  \quad\text{in}\quad \{x_n > 0 \}\cap B_1 \qquad\text{and}\qquad 
\| v_\infty\|_{W^{1,p}(\{x_n > 0\} \cap B_1)}\le 1. 
\]
We now want to show that $\partial_nv_\infty=0$ on $\{x_n=0\}\cap B_1$. To this aim,
let $\varphi\in C_c^\infty(B_1)$ and $\mu > 0$ (small) be fixed.  For $k\in \N\cup\{\infty\}$ with $k > 1/\mu$ and   $\Omega_\infty :=\{x_n > 0\}$, 
\[
\left|\int_{\Omega_k \cap B_1}\nabla v_k\cdot\nabla \varphi \,dx \right|\le \left|\int_{\{x_n > \mu\}\cap B_{1-\mu}}\nabla v_k\cdot\nabla \varphi\,dx \right|+\left|\int_{A^k_\mu}\nabla v_k\cdot\nabla \varphi\,dx \right|,
\]
where $A^k_\mu := (\Omega_k \cap B_1)\setminus (\{x_n > \mu\} \cap B_{1-\mu})$. In particular, since $|A^k_\mu|\le C \mu$, by H\"older's inequality with $\frac{1}{p}+\frac{1}{p'} = 1$ we get
\begin{equation}
\label{eq:vkAk}
\left|\int_{A^k_\mu}\nabla v_k\cdot\nabla \varphi\,dx\right|\le \|\nabla v_k\|_{L^p(\Omega_k\cap B_1)}\|\nabla\varphi\|_{L^\infty(B_1)}|A_\mu^k|^{\frac{1}{p'}}\le C \|\nabla\varphi\|_{L^\infty(B_1)} \mu^{\frac{1}{p'}}.
\end{equation}
We compute now
\[
\begin{split}
\left|\int_{\{x_n > 0\}\cap B_1}\nabla v_\infty\cdot\nabla \varphi\,dx\right|& \le \left|\int_{\{x_n > \mu\}\cap B_{1-\mu}}\nabla v_\infty\cdot\nabla \varphi\,dx\right|+C \|\nabla\varphi\|_{L^\infty(B_1)} \mu^{\frac{1}{p'}}\\
& \hspace{-1cm}\le \left|\int_{\{x_n > \mu\}\cap B_{1-\mu}}\nabla v_k\cdot\nabla \varphi\,dx\right|+C \|\nabla\varphi\|_{L^\infty(B_1)}\left( \|\nabla v_k - \nabla v_\infty\|_{L^\infty(\{x_n > \mu\}\cap B_{1-\mu})} +  \mu^{\frac{1}{p'}}\right) \\
& \hspace{-1cm}\le \left|\int_{\Omega_k \cap B_1}\nabla v_k\cdot\nabla \varphi\,dx\right|+C \|\nabla\varphi\|_{L^\infty(B_1)}\left( \|\nabla v_k - \nabla v_\infty\|_{L^\infty(\{x_n > \mu\}\cap B_{1-\mu})} +  2\mu^{\frac{1}{p'}}\right) ,
\end{split}
\]
where, in the last inequality, we have used \eqref{eq:vkAk}.  By the assumption on $v_k$ for $k\in \N$, we know
\[
\left|\int_{\Omega_k \cap B_1}\nabla v_k \cdot \nabla \varphi\,dx\right| = \left|\int_{\partial \Omega_k \cap B_1}\nu \cdot \nabla v_k \, \varphi \,d\cH^{n-1}\right|\le \frac{\|\varphi\|_{L^\infty(B_1)}}{k}. 
\]
Thus, we have 
\[
\left|\int_{\{x_n > 0\}\cap B_1}\nabla v_\infty\cdot\nabla \varphi\,dx\right| \le \frac{\|\varphi\|_{L^\infty(B_1)}}{k} +C \|\nabla\varphi\|_{L^\infty(B_1)}\left( \|\nabla v_k - \nabla v_\infty\|_{L^\infty(\{x_n > \mu\}\cap B_{1-\mu})} +  2\mu^{\frac{1}{p'}}\right). 
\]
Letting $k \to\infty$, since $v_k \to v_\infty$ smoothly in the interior of $\{x_n > 0\}\cap B_1$ (by harmonic estimates) we deduce that 
\[
\left|\int_{\{x_n > 0\}\cap B_1}\nabla v_\infty\cdot\nabla \varphi\,dx\right| \le  C \|\nabla\varphi\|_{L^\infty(B_1)}\mu^{\frac{1}{p'}},
\]
and so, by the arbitrariness of  $\mu>0$,
\[
\int_{\{x_n > 0\}\cap B_1}\nabla v_\infty\cdot \nabla\varphi\,dx = 0\qquad\forall\,  \varphi\in C^\infty_c(B_1). 
\]
This is the weak formulation of
\[
\left\{
\begin{array}{rcll}
\Delta v_\infty & = & 0 & \qquad\text{in}\quad   \{x_n > 0\}\cap B_1, \\
\partial_n v_\infty & = & 0 & \qquad\text{on}\quad \{x_n = 0 \}\cap B_1.
\end{array}
\right. 
\] 
In particular, $v_\infty$ extends evenly to a harmonic function $\bar v_\infty$ defined in the whole $B_1$.
Also, for any $\mu > 0$ we have
\[
\int_{\{x_n > \mu\}\cap B_{1-\mu}} | v_k - \bar v_\infty|\,dx\to 0\qquad\text{as}\quad k\to \infty. 
\]
Hence, since $\|v_k\|_{L^p(\Omega_\delta\cap B_1)}\le 1$ for all $k \in \N\cup\{\infty\}$, again by H\"older inequality we get
\[
\int_{A_\mu^k} | v_k - \bar v_\infty|\,dx \le C\|v_k - \bar v_\infty\|_{L^p(\Omega_k\cap B_1)} |A_\mu^k|^{p'}\le  C \mu^{p'}.
\]
Hence, choosing $\mu$ small enough so that $ C \mu^{p'} \le \frac{\eta_\circ}{2}$, we reach a contradiction for $k$ large enough. 
\end{proof}

From the previous compactness result, we obtain the following global version:

\begin{lem}
\label{lem:compactness2}
Let $n\ge 2$, $p > 1$, and $d \ge 0$. For any $\eta > 0$ there exists $\delta = \delta(\eta, n, p, d)$ small such that the following holds. 

Let $\Omega_\delta\subset \R^n$ be a Lipschitz and locally piecewise smooth domain, and let $v\in H^1(B_{1/\delta})$ satisfy
\[
\begin{split}
\Delta v  =  0  \quad \text{in}\quad  \Omega_\delta\cap B_{1/\delta},\quad\text{and}\quad \left(\fint_{\Omega_\delta\cap B_\rho} |v|^p\,dx\right)^{1/p} & + \rho\left( \fint_{\Omega_\delta\cap B_\rho} |\nabla v|^p\,dx\right)^{1/p} \le \rho^{d+1/2}\quad \text{for}\quad 1\le \rho \le \tfrac{1}{\delta}.
\end{split} 
\] 
Suppose that 
\[
\begin{split}
 & \{ \tilde e_\rho \cdot x \ge \rho \delta\} \subset \Omega_\delta \subset  \{\tilde e_\rho \cdot x\ge- \rho\delta\}\quad\text{in}\quad B_{\rho},
 \qquad\text{for}\quad 1\le \rho \le \tfrac{1}{\delta},\\
 &\qquad \frac{1}{\rho^{n-2}} \int_{\partial \Omega_\delta\cap B_\rho} |v_\nu| \,d\cH^{n-1}\le \delta \rho^{d+1/2}\qquad \text{for}\quad 1 \le \rho \le \tfrac{1}{\delta},
 \end{split}
\]
where $\nu$ is the outwards unit normal to $\partial \Omega_\delta$, for some $\tilde e_\rho \in \bS^{n-1}$ depending on $\rho$.
Then there exists a harmonic polynomial $p_d$ of degree at most $d$ such that
\[
\int_{  \Omega_\delta\cap B_1} |v-p_d|\,dx \le \eta,
\]
where $p_d$ is even in the $\tilde e_1$ direction and  satisfies  $\|p_\delta\|_{L^1(B_1)}\le |B_1|$.
\end{lem}
\begin{proof}
We argue by compactness/contradiction. Suppose that the statement is not true, so that there is a sequence $v_k$ satisfying the previous hypotheses for $\delta_k \downarrow 0$ but the conclusion fails for a certain $\eta = \eta_\circ>0$.

Applying \Cref{lem:compactness} inside each ball $B_\rho$ with $1\le \rho\le \frac{1}{\delta_k}$ and a standard diagonal argument, we obtain that $v_k$ converges in $L^{1}_{\rm loc}(\R^n)$ to some harmonic function $v_{\infty}$, even with respect to $\{\tilde e_1 \cdot x= 0\}$. Also, $v_\infty$ satisfies the growth bound 
\[
\fint_{B_\rho} |v| \,dx \le \rho^{d+1/2}\qquad \text{for all }\rho \geq 1.
\]
By the Liouville theorem, $v_\infty$ must be a harmonic polynomial of degree $\le d$, thus reaching a contradiction for $k$ large enough. 
Finally, the bound $\|p_\delta\|_{L^1(B_1)}\le |B_1|$ comes from the growth bound with $\rho=1$.
\end{proof}

In analogy with Lemmas~\ref{lem_basic1}--\ref{lem_basic2}, we also have the following general estimates for monotone harmonic functions:

\begin{lem}\label{lem:abst_osc2}
Suppose that $n\geq 2$ and $w: B_{2} \cap \{x_n>0\}\to (0,\infty)$  is a harmonic function satisfying $\partial_n w\le 0$ in $B_2\cap \{x_n>0\}$. Then, denoting $x = (x', x_n)\in \R^{n-1}\times \R$,  for every $q\in (1,\infty)$ we have
\[
\int_{\{|x'| < 3/2\}}
|\nabla w|^q (x', t)\,dx'
\leq C t^{(n-1)(1-q)}
\int_{B_{2} \cap \set{x_n \geq 1/4}}
    |\nabla w|^q
\,dx\qquad\text{for all}\quad t\in (0, 1),
\]
where $C$ depends only on $n$ and  $q$. 
\end{lem}
\begin{proof}
    Denote $B_r^+ =B_r\cap \{x_n>0\}\subset \R^n$ and $B_r'=\{x' : |x'| < r\}\subset \R^{n-1}$.   By harmonic estimates and Poincar\'e inequality,
    \[
    \|w-c\|_{L^\infty(B_{7/4}\cap \{x_n\ge 1/3\})}\le C\|w-c\|_{L^q(B_2\cap \{x_n\ge 1/4\})} \le C\|\nabla w\|_{L^q(B_2\cap \{x_n\ge 1/4\})}, 
    \]
    where $C=C(n,q)\geq $ is a constant. 
    Now, up to replacing $w$ by $$\frac{w-c}{2C_q\|\nabla w\|_{L^q(B_2\cap \{x_n\ge 1/4\})}}+\frac12,$$
    we can assume that 
    $0 \leq w \leq 1$ inside $B_{7/4}\cap \{x_n\ge 1/3\}$.
    Thus, since  $\partial_n w \le 0$, it follows that $w(e_n)\le 1$ and  $w\ge 0$ in $B_{3/2}^+$,
    and under these assumptions we need to prove that
    \[
    \int_{\{|x'| < 3/2\}}
|\nabla w|^q (x', t)\,dx'
\leq C_q t^{(n-1)(1-q)}.
\]
To this aim, set $w_n := \partial_n w$.  By standard  Poisson kernel bounds, since $w$ is harmonic with $w(e_n)\le 1$ and $w\ge 0$ in $B_{3/2}^+$, it follows that $\|w(x', 0)\|_{L^1(B'_{5/4})}\le C$.
Similarly, since $w_n$ is harmonic with $|w_n(e_n)|\le C$ and $w_n\le 0$ in $B_{3/2}^+$, we also get $\|w_n(x', 0)\|_{L^1(B'_{5/4})}\le C$.

Now, let us denote respectively by $\bar w$ and $\bar w_n$, the harmonic extensions inside $\{x_n > 0\}$ of $w(\cdot , 0)\mathbbm{1}_{B'_{4/3}}$ and $w_n(\cdot , 0)\mathbbm{1}_{B'_{4/3}}$. 
Then, by boundary Harnack, we have $\norm[C^1(B^+_{5/4})]{\frac{w-\bar{w}}{x_n}}\leq C$. In particular,
\begin{equation}
\label{eq:vvbar}
\|w-\bar w\|_{C^1(B_{5/4}^+)} +\|w_n-\bar w_n\|_{L^\infty(B_{5/4}^+)} \le C.
\end{equation}
Now, the Poisson representation for the half-space $\{x_n>0\}\subset \R^n$ reads
\[
\bar w(\,\cdot\,, t) = P(\cdot , t)*_{x'} \bar w(\cdot, 0),\quad \bar w_n(\,\cdot\,, t) = P(\cdot , t)*_{x'} \bar w_n(\cdot, 0),
    \qquad
P(x',t) : = c_n \frac{t}{(|x'|^2 + t^2)^{n/2}}.
\]
Thus, recalling that $\|\bar w(\cdot, 0)\|^q_{L^1(\R^{n-1})}+\|\bar w_n(\cdot, 0)\|^q_{L^1(\R^{n-1})} \leq C$, by Young's inequality and a direct computation we get
\[
\|\bar w(\cdot, t)\|^q_{L^q(\R^{n-1})}+\|\bar w_n(\cdot, t)\|^q_{L^q(\R^{n-1})} 
\le C \| P(\cdot, t)\|^q_{L^q(\R^{n-1})}
\le 
C_q t^{(n-1)(1-q)}\qquad \forall\,q \geq 1.
\]
In particular, thanks to \eqref{eq:vvbar}, 
\[
\|w(\cdot, t)\|^q_{L^q(B'_{5/4})}+\|w_n(\cdot, t)\|^q_{L^q(B'_{5/4})} \le C_q t^{(n-1)(1-q)}. 
\]
Finally, since $\bar w_n(\cdot, t) =  -(-\Delta)_{x'}^{1/2}  \bar  w(\cdot, t) $, by $W^{1, q}$ estimates\footnote{Combining classical Calder\'on--Zygmund estimates for the Riesz transform  with interior estimates for $\tfrac12$-harmonic functions, one obtains the following: If $(-\Delta)^{1/2} u = f$ in $B_1\subset \R^d$ with $f\in L^p(B_1)$ and $p \in (1,\infty)$, then 
\[
\|u\|_{W^{1,p}(B_{1/2})}\le C_{d,p} \left(\|f\|_{L^p(B_1)}+ \int_{\R^d} \frac{|u(y)|}{1+|y|^{d+1}}\, dy\right). 
\]} for   $(-\Delta_{x'})^{1/2}$ imply that 
\[
\|\nabla'\bar w(\cdot, t)\|^q_{L^{q}(B'_1)} \le  C_q \|w(\cdot, t)\|^q_{L^q(B'_{5/4})}+\|w_n(\cdot, t)\|^q_{L^q(B'_{5/4})} \leq C_q t^{(n-1)(1-q)} \qquad \forall\,q \in (1,\infty).
\]
Using again \eqref{eq:vvbar}, we get the desired bound on $\nabla w$. 
\end{proof}

\normalcolor

As shown in \Cref{appendixD},  \Cref{lem:abst_osc2} implies the following result in flat-Lipschitz domains: 

\begin{lem}\label{lem:abst_osc1}
Let $n\ge 2$ and fix  $q\in(1,\frac{n}{n-1})$. Assume that $w:B_{2r}\cap D \to (0, \infty)$ is a positive harmonic function inside $D = \{x_n > \varphi(x')\}$, where $\varphi:B_{2r}'\subset \R^{n-1}\to \R$ satisfies
\[ 
|\varphi|  + r|\nabla \varphi| \le c_\circ r.
\]
Assume, in addition, that $\partial_n w\le 0$ in $B_{2r}\cap D$.
Then, for $c_\circ$ small enough depending only on $n$ and $q$, we have
\[ 
\int_{
    B_r\cap D
}
    |\nabla w|^q
\,dx
\leq C_q 
\int_{B_{2r} \cap \set{x_3 \geq r/4}}
    |\nabla w|^q
\,dx,
\]
where $C_q$ depends only on $n$ and $q$.
\end{lem}

To show \Cref{prop:decaysqrt}, we will need to apply  \Cref{lem:compactness2} to a suitable sequence. The following result will ensure that the sequence satisfies the hypotheses of \Cref{lem:compactness2}:

\begin{lem}
\label{lem:hyp_comp}
Let $p\in [1, \tfrac32]$ and $\bar \gamma = \tfrac{1}{10}$. There exists $\eps_\circ$ depending only on  $p$ such that if $ 0< \epk< \eps_\circ$ the following holds for any $\zz\in B_{\Rk/8}(\tilde\zz_k)\cap \cZ$. 

 Let $a_+\in \mathbb{S}^2$ and $b_+\in \R$ be given by \Cref{proplinearization}, and assume that for some $\delta>0$ and $R \in (\Rk \epk^{1+{\bar \gamma}}, \Rk/8)$ we have 
\[
\frac{1}{R}\fint_{U_+\cap B_{R}(\zz)} |u-a\cdot x - b|\,dx
\le   \epk\delta ,\qquad \text{for some $a\in \mathbb{S}^2$ with $  |a- a_+|\le \epk^{1/2} $ and $b\in\R$}.
\]
Then
\[
\frac{1}{R^p}\fint_{U_+\cap B_{R/2}(\zz)} 
    |u-a\cdot x-b|^p \,dx
+ \fint_{U_+\cap B_{R/2}(\zz)}
    |\nabla u-a|^p \,dx
 \le C \bigg(  (\epk\delta)^p + \frac{\Rk}{R} \epk^{1+{\bar \gamma}}\bigg),
\]
where $C$ depends only on  $p$.
\end{lem}

\begin{proof}
Up to a rotation, we can assume that $a_+ = e_3$. We write $v=u-x_3-b$ and divide the proof into two steps.

\medskip
\noindent{\bf Step 1:} we prove  the  ${\dot{W}}^{1,p}$ bound.

We begin by noticing that if $u(x) \ge \Rk \epk^{1+{\bar \gamma}}$ then\footnote{Recall that $|\nabla u|\le 1$ in $\R^3$ and $\cZ\subset \{u=0\}$.}  $\dist(x, \cZ)\ge \Rk \epk^{\frac{2}{1+\ctta}}\gg \Rk \epk^{\frac{1}{\ctta}}$ as long as $1+{\bar \gamma} < \frac{2}{1+\ctta}$ (the chosen value $\bar \gamma = \tfrac{1}{10}$ works since $\ctta=\valctta$).
Thus, thanks to \Cref{proplinearization}, 
\begin{equation}
\label{eq:gradbounde3}
|\nabla u  - e_3|\le C\epk^{\frac{1-\ctta}{1+\ctta}} \ll 1 \qquad\text{in}\quad   \Omega^{{\bar \gamma}}_+\cap B_{\Rk/4}(\zk),\quad\text{where}\quad   \Omega^{\bar \gamma}_+ := U_+ \cap \left\{u \ge \Rk\epk^{1+{\bar \gamma}}\right\}.
\end{equation}
Note that \eqref{eq:gradbounde3} implies that $\partial_{\tilde e} u >0 $ in $\Omega_+^{\bar \gamma}$ as long as $\tilde e\cdot e_3\ge  \epk^{\frac{1-\ctta}{2}} \gg C\epk^{\frac{1-\ctta}{1+\ctta}} $. Hence, noticing that $\{u = \Rk \epk^{1+\bar\gamma}\}\cap \partial U_+\cap B_{\Rk /5}(\zz) = \varnothing$, we deduce that $\partial\Omega_+^{{\bar \gamma}}\cap B_{\Rk/5}(\zz)$ is an $\epk^{\frac{1-\ctta}{2}}$-Lipschitz graph both in the $e_3$ direction and in the $a$ direction (recall that, by assumption, $|a-e_3| = |a-a_+|\le \epk^{1/2}$). 
In particular,  for any $y_\circ\in \partial \Omega^{{\bar \gamma}}_+\cap B_{\Rk/8}(\zz)$ we have (using $\bar \gamma<\frac{1-\ctta}{2}$)
\begin{equation}
\label{eq:epskgammaflat}
  \{x_3 \ge \rho \epk^{\bar \gamma}\}\subset  \Omega_+^{\bar \gamma} -y_\circ \subset \{x_3 \ge - \rho \epk^{{\bar \gamma}}\} \qquad\text{in}\quad B_{\rho},\qquad\text{for any}\quad \rho\in (0, \Rk/8). 
\end{equation}
We  divide the set $ U_+ \cap B_{R/2}(\zz)$ into two regions:
\begin{equation}
    \label{eq:A1A2}
U_+ \cap B_{R/2}(\zz) = A_1\cup A_2,\qquad\text{where}\quad A_1 :=  \Omega_+^{{\bar \gamma}}\cap B_{R/2}(\zz),\quad  A_2 :=  (U_+ \setminus  \Omega_+^{{\bar \gamma}})\cap B_{R/2}(\zz).  
\end{equation}
Now, we first use  \Cref{lem:abst_osc1} in $A_1$ (recall $A_1$ is a flat-Lipschitz domain in the $e_3$ direction, and notice that $\partial_{3} v \le |\nabla u|-1\le 0$) with a covering argument to obtain
\[
\int_{A_1} |\nabla v|^p \,dx
\le  C\int_{B_{2R/3}(\zz)\cap \{(x-\zz)\cdot e_3 \ge R/12\}} |\nabla v|^p\, dx \leq C\int_{B_{2R/3}(\zz)\cap \{(x-\zz)\cdot a \ge R/15\}} |\nabla v|^p\, dx 
\le CR^3 (\epk \delta)^p,
\] 
where the last inequality follows by interior harmonic estimates and the $L^1$-smallness assumption of $v$ inside $U_+\cap B_R(\zz)\supset B_{2R/3}(\zz)\cap \{(x-\zz)\cdot a \ge R/20\}$.

Concerning $A_2$, \Cref{lem:area-level-sets} together with the coarea formula and the fact that the gradient is lower bounded in $\Omega^{(\pm)}\supset U_\pm$ (see \Cref{lem:grad_lower_bound}) imply that $|A_2|\le C R^2 \Rk \epk^{1+\bar\gamma}$. Hence, since $|\nabla v|\le 2$,
\[
\int_{A_2} |\nabla v|^p \,dx
\le C R^2\Rk \epk^{1+{\bar \gamma}}.
\] 
This proves the desired bound on $|\nabla v|^p$.

\medskip
\noindent{\bf Step 2:} We now prove the $L^p$ bound.

As before, consider the sets $A_1$ and $A_2$ as in \eqref{eq:A1A2}.
Notice first that, since $|\nabla v| \leq 2$, by the $L^1$ bound of $v$ in $B_R(\zz)$ we deduce that $|v|\le CR$ in $A_2$. Hence,
\[
\int_{A_2}|v|^p\,dx \le C R^p |A_2|\leq C R^{p+2} \Rk\epk^{1+\gamma}.
\]
On the other hand, since $A_1$ is a Lipschitz domain, by Sobolev embedding (see, for example, \cite[Theorem 3]{AF77}) and H\"older inequality  we get
\begin{multline*}
\left(\fint_{A_1} |v|^p\,dx\right)^{1/p} \leq \left(\fint_{A_1} |v|^{3/2}\,dx\right)^{2/3}
\le C\fint_{A_1} |v|\,dx + CR\fint_{A_1} |\nabla v|\,dx \\
\leq C\fint_{A_1} |v|\,dx + CR\left(\fint_{A_1} |\nabla v|^p\,dx \right)^{1/p}\qquad \text{whenever} \quad 1 \le p \le \frac{3}{2}.
\end{multline*}
Thus, using Step 1 and the assumption on $v$, we get the desired estimate.
\end{proof}

\subsection{Proof of \Cref{prop:decaysqrt}} \label{ssec:decaysqrt} We conclude by proving the main result of this section.

\begin{proof}[Proof of \Cref{prop:decaysqrt}]
We split the proof into two steps. 

\medskip
\noindent
{\bf Step 1:} 
We first show an algebraic decay of $\bA$ from scale $\Rk$ to scale $R_k^\flat$ (with given center $\zz$). More precisely, we start by showing 
\begin{equation}
\label{eq:prop62p}
\bA_\zz(u, R) \le C_\circ \left(\frac{R}{\Rk}\right)^{1/2}\epk , \qquad\text{for all}\quad R\in  \left[R_k^\flat, \Rk/16\right],
\end{equation}
for some ${C_\circ}$ universal. Note that \eqref{eq:prop62} follows directly from this bound, choosing  $C=C_\circ^6$.

Let us denote $R_\ell := 2^{-\ell}\Rk$,
and let $\ell_0 \in \mathbb N$ be a large constant to be fixed. Thanks to  \Cref{lem:AzEz}, up to choosing $C_\circ$ sufficiently large (depending only on $\ell_0$), we can assume that \eqref{eq:prop62p} holds for $R = R_4,R_5, R_6, R_7, \dots, R_{\ell_0}$.

We now argue by induction and prove the following: if 
\eqref{eq:prop62p} holds for $R =R_4, R_5, R_6, R_7, \dots, R_{\ell}$ for some $\ell \geq \ell_0$ such that $2^{-\ell}\geq \epk^\cttc$,
then \eqref{eq:prop62p} holds for $R_{\ell+1}$. This will imply the desired bound.

To prove the inductive step, we  will apply \Cref{lem:compactness2} to a suitable function. 
Recall that $\bA_\zz$ is given by the maximum of two integrals, one inside $U_+$ and one inside $U_-$ (see \eqref{eq:def_Az}).
Here we just prove the estimate for $U_+$, since the case of $U_-$ is completely analogous.

 Fix $\zz\in B_{\Rk/16}(\zk)\cap \cZ$. By the inductive hypothesis, for all $m=4,5,6,\dots,\ell$ there exist $a_m\in \mathbb{S}^2$ and $b_m\in \R$ such that
\begin{equation}\label{eq:vm small}
 \frac{1}{|B_{R_m}|}\int_{ U_+\cap B_{R_m}(\zz)} |v_m| \,dx \le {C_\circ} 2^{-m/2} R_m \epk, \qquad \mbox{where} \quad  v_m(x) := u(x) - a_m\cdot x - b_m.
\end{equation}
Then, by the triangle inequality  (similarly to  \Cref{proplinearization}, using that $ U_+\cap B_\rho(\zz)$ is roughly a half-space at scales $\rho \gg \Rk \epk^{1/\alpha}\ge \rb(\zz)$)  we get 
\[
  R_m |a_m-a_{m+1}| +  |b_m-b_{m+1}|\le C {C_\circ}  2^{-m/2}R_m\epk.
\]
In particular, this implies that for any $4 \leq \ell_1\le \ell_2\le \ell$ we have 
\begin{equation}
    \label{eq:boundsab}
  |a_{\ell_1}-a_{\ell_2}| \le C{C_\circ}  2^{-\ell_1/2}\epk , \qquad  |b_{\ell_1}-b_{\ell_2}|\le C{C_\circ}  2^{-\ell_1/2}   R_{\ell_1} \epk .
\end{equation}
Furthermore, if we consider the function $v(x) = u(x) -a_+\cdot x -b_+$ provided by \Cref{proplinearization}, by the very same reason  we also have
\begin{equation}
    \label{eq:boundsab2}
  |a_{\ell_1}-a_+| \le C{C_\circ}  \epk  \qquad\text{for any}\quad 4 \le \ell_1\le \ell. 
\end{equation}
Fix now $p>1$ satisfying 
$\tfrac{1+{\bar \gamma}-\cttc}{p}\ge 1+\frac{\cttc}2$ 
(for instance, one can choose $p=1+\frac1{20}$), and recall that by assumption  $2^{-\ell} \ge \epk^\cttc$. Then, thanks to \eqref{eq:vm small}, we can apply \Cref{lem:hyp_comp} with $\delta = {C_\circ} 2^{-m/2}$ to deduce that, for any $4\le m \le \ell$,  
\[
\frac{1}{R_m^p}\fint_{U_+\cap B_{R_m/2}(\zz)} |v_m|^p \,dx +\fint_{U_+\cap B_{R_m/2}(\zz)} |\nabla v_m|^p \,dx \le (C {C_\circ}2^{-m/2}\epk)^p. 
\]
Using \eqref{eq:boundsab}, this implies that 
\[
 \frac{1}{R_m^p}\fint_{U_+\cap B_{R_m/2}(\zz)} |v_\ell|^p\,dx +\fint_{U_+\cap B_{R_m/2}(\zz)} |\nabla v_\ell|^p \,dx \le   (\tilde C {C_\circ}2^{-m/2}\epk)^p \qquad \forall\, m \in \{\ell-\ell_0,\ldots,\ell\}, 
\]
for some $\tilde C$ universal. Hence, if we define 
$
\tilde v_\ell(x):=(\tilde C {C_\circ} 2^{-\ell/2}R_\ell\epk)^{-1} v_\ell (\zz+R_{\ell} x)$,
we get  
\begin{equation}
    \label{eq:twocond}
\bigg( \fint_{R_{\ell}^{-1}(U_+-\zz)\cap B_{\rho/2}} |\nabla \tilde v_\ell|^p \,dx \bigg)^{1/p} \le    \rho^{1/2},  \quad  \bigg( \fint_{R_{\ell}^{-1}(U_+-\zz)\cap B_{\rho/2}} |\tilde v_\ell|^p \,dx \bigg)^{1/p}\le    \rho^{3/2},\quad\text{for $\rho \in  \{2^0,2^1,2^2,\dots, 2^{\ell_0} \}$.} 
\end{equation}
On the other hand,  \Cref{prop:neumann_bound}  and \Cref{rem:pequal2} yield
\[
\int_{\partial U_+\cap B_{R_4}(\zz)} |\partial_\nu v|^2  \,d\cH^{2}
\le C\,
\Rk^2\epk^{2+\delta_\circ} .\]
Also, by the triangle inequality and \eqref{eq:boundsab2},
\[
|1-\nu\cdot a_\ell|^2 = \frac14 |\nabla u - a_\ell|^4 \le |\nabla u - a_+|^4 + (C{C_\circ}\epk)^4 = 4|1-\nu\cdot a_+|^2 + (C{C_\circ}\epk)^4\qquad\text{on}\quad \FB(u),
\]
therefore
$$
|\partial_\nu v_\ell|^2 \leq 4|\partial_\nu v|^2+(C{C_\circ}\epk)^4\qquad\text{on}\quad \FB(u).
$$
Recalling that $\cH^2(\partial U_0)\le C\Rk^2 \epk^{\frac{5-\cttb}{3\ctta}} \le C\Rk^2 \epk^{2+\delta_\circ}$ (see \eqref{eq:boundDt}) 
and that $\cH^2(\partial U_+\cap B_{R_4}(\zz)) \le C\Rk^2 +\cH^2(\partial U_0\cap B_{\Rk/2}(\zk)) \le C\Rk^2$ (see \Cref{lem:perbound} and \eqref{depU1U2}), we then obtain
\begin{multline*}
\int_{\partial U_+\cap B_{R_4}(\zz)} \hspace{-3mm}|\partial_\nu v_\ell|^2\, d\mathcal{H}^2 \le C \Rk^2 \epk^{2+\delta_\circ} + \int_{\partial U_+\cap B_{R_4}(\zz)\setminus\partial{U_0}} \hspace{-3mm}|\partial_\nu v_\ell|^2\, d\mathcal{H}^2  \\ \le C \,\Rk^2\left( \epk^{2+\delta_\circ} +  C_0^4  \epk^{4}\right) + 4\int_{\partial U_+\cap B_{R_4}(\zz)\setminus\partial{U_0}} \hspace{-3mm}|\partial_\nu v|^2\, d\mathcal{H}^2   \leq C \,\Rk^2\left( \epk^{2+\delta_\circ} +  C_0^4  \epk^{4}\right).
 \end{multline*}
By H\"older's inequality, using again $\cH^2(\partial U_+\cap B_{R_4}(\zz)) \le C\Rk^2$, this implies
\[ 
 \int_{\partial U_+\cap B_{R_4}(\zz)} \hspace{-3mm}|\partial_\nu v_\ell|\, d\mathcal{H}^2\le C\,\Rk^{2} (   \epk^{1+ \delta_\circ/2} +C_\circ^{2}  \epk^{2} )\quad \implies\quad 
 \int_{R_\ell^{-1}(\partial U_+-\zz)\cap B_{2^{\ell-4}}} |\partial_\nu \tilde v_\ell (x )|\,d \mathcal{H}^2\le C2^{5\ell/2} (  \epk^{\delta_\circ/2}+C_\circ^2 \epk) .
\]
Hence, taking $\epk$ small enough depending on ${C_\circ}$ so that $C_\circ^2 \epk\le \epk^{1/2}$,  since $2^\ell \le (\epk)^{-\chi}$ and $\delta_\circ/2-5\cttc/2 \ge \delta_\circ/4$ we get
\begin{equation}
    \label{eq:twocond2}
 \int_{R_\ell^{-1}(\partial U_+-\zz)\cap B_{\rho}} |\partial_\nu \tilde v_\ell (x )|\, d\mathcal{H}^2\le C \,2^{5\ell/2}\epk^{\delta_\circ/2}  \le C\,2^{5\ell/2} \rho^{5/2} \epk^{\delta_\circ/2}  \le C\rho^{5/2} \epk^{\delta_\circ/4},\qquad\text{for $\rho\in [1, 2^{\ell_0-4}]$.}
\end{equation}
To go further, we note that
\eqref{keyeqn} and \Cref{lem:obsUpm}(ii) imply the existence of a vector $e_R = e_{R, \zz}\in \mathbb{S}^2$ such that
\begin{equation}
\label{eq:eRtilde}
\fint_{B_R(\zz)}|u-V_{\zz,e_R}|\,dx
\le C \bigg(\frac{\Rk}{R}\bigg)^\alpha  R\, \epk
\end{equation} 
(where H\"older's inequality is used) and
$$
 \bigg\{e_R\cdot x > C\bigg(\frac{\Rk}{R}\bigg)^\alpha R \epk \bigg\}\subset U_+-\zz \subset \bigg\{e_R\cdot x \ge -C\bigg(\frac{\Rk}{R}\bigg)^\alpha R \epk\bigg\}\qquad\text{in}\quad B_R. 
$$
(As already noted in the proof of \Cref{lem:AzEz}, $e_R$ is the same for both estimates.)
Thus, if we denote $\rho = 2^\ell R \Rk^{-1}$ and $\tilde {e}_\rho = e_{2^{-\ell}\rho\Rk}$, as long as $2^{-\ell}\ge \epk^\cttc \gg \epk$ we have that
the domain $R_\ell^{-1}( U_+-\zz)$ satisfies
\begin{equation}
    \label{eq:twocond3}
 \left\{\tilde e_\rho \cdot x \ge   C \epk^{1-\alpha\chi} \rho \right\}\subset R_\ell^{-1}(U_+-\zz) \subset \left\{\tilde e_\rho\cdot x \ge -  C \epk^{1-\alpha\chi} \rho \right\}\qquad\text{in}\quad B_{\rho},\qquad\text{for}\quad 1\le  \rho \le 2^{\ell_0}.  
\end{equation}
Also, by \eqref{eq:eRtilde} with $R=R_\ell$, we have
\[ 
\fint_{B_R(\zz)}\big|u-|\tilde e_1 \cdot (x-\zz)|\big| \,dx
\le C \epk^{1-\alpha\chi} R_\ell
\]
and in particular, because of \eqref{eq:vm small}, it follows that $|a_\ell - \tilde e_1| \le C {C_\circ} \epk^{1-\alpha\chi}\le C {C_\circ} \epk^{1/2}$.

Thanks to \eqref{eq:twocond}, \eqref{eq:twocond2}, and \eqref{eq:twocond3}, we can now apply \Cref{lem:compactness2} with $d=1$ 
 to $\tilde v_\ell$. As a consequence, given   $\eta>0$  fixed (to be chosen universally), for $\epk$ small enough and $\ell_0$ large enough we have 
\[
\int_{R_\ell^{-1}( U_+-\zz)\cap B_{1/2}} |\tilde v_\ell - a\cdot x - b|\, dx\le \eta, 
\]
for some $a\in \R^{2}$ with $a\cdot \tilde e_1 = 0$ and $b\in \R$, with $|a|+|b|\le C$. In terms of $u$, this gives
\[
R_\ell^{-3} \int_{U_+\cap B_{R_\ell/2}(\zz)} \left| u(x) -a_\ell\cdot x - b_\ell - \tilde C{C_\circ}  \epk2^{-\ell/2} a\cdot (x-\zz) - \tilde C {C_\circ}  R_\ell \epk2^{-\ell/2} b\right|\, dx\le  \tilde C {C_\circ}  R_\ell \epk2^{-\ell/2} \eta.
\]
Let us denote $b_{\ell+1} := b_\ell + \tilde C{C_\circ}  R_\ell \epk 2^{-\ell/2} b+ \tilde C{C_\circ} \epk 2^{-\ell/2} a\cdot \zz$, as well as $\tilde a_{\ell+1} = a_\ell + \tilde C{C_\circ}  \epk 2^{-\ell/2} a$ and $a_{\ell+1} = \frac{\tilde a_{\ell+1}}{|\tilde a_{\ell+1}|}$. Then, thanks to the fact that $|a_\ell - \tilde e_1| \le C {C_\circ} \epk^{1/2}$ and $a\cdot \tilde e_1 = 0$ we deduce that $|a\cdot a_\ell|\le  C {C_\circ} \epk^{1/2}$ and thus
$
|\tilde a_{\ell+1} - a_{\ell+1}|\le C C_\circ^2 \epk^{3/2} 2^{-\ell/2}.
$
Combining all together, we obtain 
\[
\fint_{U_+\cap B_{R_\ell/2}(\zz)} | u(x) -a_{\ell+1}\cdot x - b_{\ell+1}|\, dx\le  C {C_\circ} R_\ell \epk2^{-\ell/2} \left(\eta + {C_\circ}\epk^{1/2}\right).
\]
We now choose $\eta$ small so that $C\eta \le \tfrac14$, which in turn fixes $\ell_0$, ${C_\circ}$, and an upper bound for $\epk$. Then, choosing $\epk$ small enough so that $C{C_\circ} \epk^{1/2}\le \tfrac14$, we get  \eqref{eq:prop62p}, as desired.
 This concludes the proof of  \eqref{eq:prop62}.

\medskip
\noindent
{\bf Step 2:}  We now show \eqref{equation:weakestdecay} for $\ast = +$; the same proof holds for $\ast = -$. Recall that $R_k^\flat  = \epk^\cttc \Rk$. Then, for $\zz\in B_{\Rk/8}(\zk)\cap \cZ$ fixed, we define
$w :=  u-a^\flat_+ \cdot x- b^\flat_+$,
where $a^\flat_+\in \mathbb S^2$ and $b^\flat_+\in \mathbb R$  are such that
\[
\fint_{U_+\cap B_{R_k^\flat}(\zz)} |u-a^\flat_+ \cdot x- b^\flat_+| \,dx  \le C \bA_\zz(u, R_k^\flat)  \le \epk^{1+\cttc/3} R_k^\flat.
\]
To prove \eqref{equation:weakestdecay}, we will show that, for some small universal constant ${\bar \beta}>0$, it holds that
\begin{equation}\label{weakdecay11}
   \text{for all}\quad r \in [\epk^{1+2\cttc} \Rk, R_k^\flat]\quad\text{there exists $c = c(r)$ such that }   \fint_{U_+\cap B_r(\zz)} |w-c| \,dx \le C \epk^{1+\cttc/3} R_k^\flat  (r/R_k^\flat)^{\bar \beta}.
\end{equation}
Notice that this directly yields the desired result by adding a geometric series, since   by the triangle inequality, we have  $|c(r_1) - c(r_2)|\le C \epk^{1+\cttc/3} R_k^\flat  (r_1/R_k^\flat)^{\bar \beta}$ for all $r_2 \in (r_1/2, r_1)$ and $r_1 \in [\epk^{1+2\cttc} \Rk, R_k^\flat]$ (and we may take $c(R_k^\flat) = 0$).

The first key observation is that, by \Cref{prop:neumann_bound}, arguing as in Step 1 we get 
\begin{equation}\label{flux1234}
\int_{\partial U_+\cap B_{\Rk/16}(\zz)} |\partial_\nu w|^2\, d\mathcal{H}^2\le C  \epk^{2+\delta_\circ} \Rk^{2}  
\quad\Longrightarrow\quad 
 \int_{\partial U_+\cap B_{r}(\zz)} |\partial_\nu  w |^2\, d\mathcal{H}^2\le C\epk^{2+2\cttc} (R_k^\flat)^2  (r/R_k^\flat)^{2{\bar \beta}}
\end{equation}
for all $r \in [\epk^{1+2\cttc} \Rk, R_k^\flat]$, provided that $\delta_\circ  >4\cttc+2\bar\beta(1+2\cttc)$ (this is true, for example, choosing $\bar \beta \leq  \tfrac{1}{50}$).

Set $r_\ell := 2^{-\ell}R_k^\flat$. As in Step 1, we can assume that \eqref{weakdecay11} holds for $r = r_0,r_1, r_2, r_3, \dots, r_\ell$ for some $\ell \geq \ell_0$, and we will show its validity for $r_{\ell+1}$ as well, as long as $2^{-\ell}\ge \epk^{1+2\chi}$.

Again as in Step 1, by assumption there exist $c_m\in \R$ such that,  if we define
$w_m(x) = w(x) -c_m,$
then 
\begin{equation}\label{industep2}
\fint_{ U_+\cap B_{r_m}(\zz)} |w_m|\,dx  \le C 2^{-{\bar \beta} m} \epk^{1+\cttc/3} R_k^\flat,\qquad   |c_m-c_{\ell}|\le C_{\bar \beta} 2^{-{\bar \beta} m}\epk^{1+\cttc/3}  R_k^\flat,\qquad \text{for all }0\le m \le \ell.
\end{equation}
Thus, applying \Cref{lem:hyp_comp} with $R = r_m$ and $\delta =  C 2^{-{\bar \beta} m} \epk^{\cttc/3} R_k^\flat/r_m$ and using the induction hypothesis, for any $2\le m \le \ell$ we get
\[
 \fint_{U_+\cap B_{r_m/2}(\zz)} |w_m|^p 
 \,dx+r_m^p\fint_{U_+\cap B_{r_m/2}(\zz)} |\nabla w_m|^p \,dx   \le C  \left(\epk^{1+\cttc/3} R_k^\flat 2^{-{\bar \beta} m}\right)^p, 
\]
  for some $p > 1$ 
 and ${\bar \beta} > 0$ sufficiently small (more precisely, we need $\bar\beta p \leq p-1$ and $\tfrac{1+\bar\gamma-\cttc}{p} > 1+\tfrac{\cttc}{3}$).  By the triangle inequality and \eqref{industep2}, the same holds for $w_\ell$. Thus, if we define 
\[
\tilde w_\ell(x):= \frac{w_\ell (\zz+r_{\ell} x)}{2C \epk^{1+\cttc/3} R_k^\flat 2^{-{\bar \beta} \ell}},
\]
since ${\bar \beta} < 1/2$ we get
\begin{equation}
    \label{eq:twocondbis}
\left( \fint_{r_{\ell}^{-1}(U_+-\zz)\cap B_{\rho/2}} |\tilde w_\ell|^p \,dx\right)^{1/p}+\rho \left(  \fint_{r_{\ell}^{-1}(U_+-\zz)\cap B_{\rho/2}} |\nabla \tilde w_\ell|^p \,dx\right)^{1/p} \le    \rho^{1/2}\qquad\text{for $\rho \in \{2^0,2^1,2^2,\dots,2^{\ell_0}\} $.} 
\end{equation}
Also,  by \eqref{flux1234} and H\"older inequality, we have
\begin{equation}
    \label{eq:twocond2bis}
 \frac{1}{\rho}\int_{r_\ell^{-1}(\partial U_+-\zz)\cap B_{\rho}} |\partial_\nu \tilde w_\ell (x )|\, d\mathcal{H}^2\le C\, \epk^{\cttc/3} \qquad\text{for $\rho\in [1, 2^{\ell_0}]$.}
\end{equation}
Finally,  since $r_\ell \ge \epk^{1+2\chi}\Rk  \gg \rb(\zz)$ (see \eqref{eq:rbepk}), as in Step 1 we obtain
\begin{equation}
    \label{eq:twocond3bis}
 \left\{\tilde e_\rho \cdot x \ge   o_{\epk}(1) \rho \right\}\subset r_\ell^{-1}(U_+-\zz) \subset \left\{\tilde e_\rho\cdot x \ge -  o_{\epk}(1) \rho  \right \}\qquad\text{in}\quad B_{\rho},\qquad\text{for}\quad \rho \ge 1,  
\end{equation}
and $|a^\flat_+ - \tilde e_1| = o_{\epk}(1)$, where $o_{\epk}(1) \to 0$ as $k\to \infty$.

Thanks to \eqref{eq:twocondbis}, \eqref{eq:twocond2bis}, and \eqref{eq:twocond3bis}, we can apply \Cref{lem:compactness2} with $d=0$ to obtain that, for any $\eta>0$, there exist $\epk$ small enough and $\ell_0$ large enough such that 
\[
\int_{r_\ell^{-1}( U_+-\zz)\cap B_{1/2}} |\tilde w_\ell - c|\, dx\le \eta, 
\]
for some $c\in \R$.  Similarly to Step 1,
after rescaling
we deduce that \eqref{industep2} holds for $m = \ell+1$, for some suitable $c_{\ell+1}\in \R$ with $|c_{\ell+1}-c_\ell|\le C2^{-{\bar \beta}\ell} \epk^{1+\cttc/3}  R_k^\flat$.
This proves \eqref{weakdecay11}, concluding the proof.
\end{proof}

\section{Proofs of \Cref{thm:main1} and  its corollaries}
\label{sec:closing}
In this section, we prove \Cref{thm:main1} and Corollaries \ref{cor:DeG_Bernoulli} and \ref{thm:main2}. As we shall see, \Cref{thm:main1} follows from \Cref{prop:decaysqrt} together with a contradiction argument.

\subsection{Remainder involving symmetric excess}

  Recall that the Weiss energy {\bf W} was introduced in \eqref{eq:W-def}. We will need two new quantities,  ${\bf M}$ and ${\bf T}$, that we now define.\footnote{The letter ${\bf M}$ is motivated by the analogies of our quantity with the so-called Monneau energy, which plays a crucial role in obstacle problems. However, contrary to that setting, now ${\bf M}$ is not a monotone quantity.}

Recalling that $u_r(x)=r^{-1}u(rx)$,  given $e\in \mathbb{S}^{n-1}$ we define
\begin{equation}\label{eq:Monneau-def}
{\bf M}(u,r, e)
:=\frac{1}{r^{n+1}}\int_{\partial B_r}
	(u-|e\cdot x|)^2
\,d\cH^{n-1},
	\quad \text{ so that } \quad
{\bf M}(u,r, e)={\bf M}(u_r,1, e)=
\int_{\partial B_1}
	(u_r-|e\cdot x|)^2
\,d\cH^{n-1},
\end{equation}
and
\begin{equation}\label{eq:T-def}
{\bf T}(u,r, e)
:=\frac{1}{r^{n+2}}\int_{B_r}
	(u-|e\cdot x|)^2
\,dx,
	\quad \text{ so that } \quad
{\bf T}(u,r, e)={\bf T}(u_r,1, e)=
\int_{  B_1}
	(u_r-|e\cdot x|)^2
\,dx.
\end{equation}
Note that $\partial_r\left(r^{n+2} {\bf T}(u, r, e)\right) = r^{n+1} {\bf M}(u, r, e)$, therefore
\begin{equation}\label{eq:T-M-rel}
r_2^{n+2}{\bf T}(u, r_2, e) - r_1^{n+2}{\bf T}(u, r_1, e) = \int_{r_1}^{r_2} s^{n+1}{\bf M}(u, s, e)\, ds\qquad \text{ for all }0<r_1<r_2.
\end{equation}

While the quantities ${\bf M}$ and ${\bf T}$ are not necessarily monotone,\footnote{The fact that we can exploit non-monotone quantities is rather remarkable, since usually the lack of monotonicity formulas makes this type of quantities useless. In this respect, our argument is very robust and we expect it to be useful in several other problems.} we can still find nice relations between them and ${\bf W}$ that will be crucial for our argument.

\begin{lem}\label{lem:Monneau}
For any $e\in \mathbb{S}^{n-1}$, it holds
\[
\partial_r {\bf W}(u,r)
\geq 2r\left (\partial_r\sqrt{{\bf M}(u,r, e)}\right )^2.
\]
Consequently, for any $r > 0$ and $\eta\in (0, 1)$, 
\[
{\bf M}(u, r, e) \le |\log \eta | \left( {\bf W}(u, r) - {\bf W}(u, \eta r)\right) + 2 {\bf M} (u, \eta r, e)
\]
and 
\[
{\bf T}(u, r, e) \le \frac{|\log(\eta)|}{n+2} ({\bf W}(u, r) - {\bf W}(u, \eta r)) + \frac{2}{n+2} {\bf M}(u, \eta r, e)  +\eta^{n+2}{\bf T}(u, \eta r, e).
\]
\end{lem}

\begin{proof}
Since  $u_r-x\cdot\nabla u_r=-r \partial_r u_r$, we have $\partial_r {\bf W}(u,r)=2r\int_{\partial B_1}(\partial_r u_r)^2\,d\cH^{n-1}$ (recall \eqref{eq:W-monotone}). By Cauchy--Schwarz,
\begin{align*}
\partial_r {\bf M}(u,r, e)
&=
2\int_{\partial B_1}
	(u_r-|e\cdot x|)\partial_r u_r
\,d\cH^{n-1}
\\
&\leq 2
	\sqrt{\int_{\partial B_1}
		(u_r-|e\cdot x|)^2
	\,d\cH^{n-1}
	}
	\sqrt{\int_{\partial B_1}
		(\partial_r u_r)^2
	\,d\cH^{n-1}
	}
=2\sqrt{
	{\bf M}(u,r, e)
	\frac{
		\partial_r {\bf W}(u,r)
	}{2r}
}.
\end{align*}
Rearranging the terms, we get the first inequality. 

Now, we integrate the first inequality between $\eta r$ and $r$,
we multiply the result by $\int_{\eta r}^r \frac{d\rho }{\rho} = |\log \eta|$, and then we apply H\"older inequality: 
\[
|\log\eta|\left({\bf W}(u, r) - {\bf W}(u, \eta r)\right) \ge 2\left(\int_{\eta r}^r \frac{d\rho }{\rho}\right) \int_{\eta r}^r \rho \left(\partial_\rho\sqrt{{\bf M}(u, \rho, e)}\right)^2\, d\rho \ge 2 \left( \int_{\eta r}^r \partial_\rho \sqrt{{\bf M}(u, \rho, e)}\, d\rho\right)^2.
\]
This gives
\[
\frac12 |\log\eta|\left({\bf W}(u, r) - {\bf W}(u, \eta r)\right)\ge \left(  \sqrt{{\bf M}(u, r, e)} - \sqrt{{\bf M}(u, \eta r, e)}\right)^2 \ge \frac12 {\bf M}(u, r, e) - {\bf M}(u, \eta r, e),
\]
which proves the second inequality.

Finally, using \eqref{eq:T-M-rel} with $r_2 = r$ and $r_1 = \eta r$, by the second inequality we get 
\[
\begin{split}
{\bf T}(u, r, e) &\le r^{-n-2}\int_{\eta r}^r s^{n+1} {\bf M}(u, s, e) ds+\eta^{n+2}{\bf T}(u, \eta r, e)\\
&\le r^{-n-2}\int_{\eta r}^r s^{n+1} \left(\log\left(\frac{s}{\eta r}\right) ({\bf W}(u, s) - {\bf W}(u, \eta r)) + 2 {\bf M}(u, \eta r, e)\right) ds+\eta^{n+2}{\bf T}(u, \eta r, e).
\end{split}
\]
Since $\log\left(\frac{s}{\eta r}\right) \leq |\log \eta|$ and ${\bf W}(u, s) \leq {\bf W}(u, r)$ for $s \in [\eta r,r]$
(recall that ${\bf W}(u,\cdot)$ is non-decreasing), we can bound the term above by
$$
\Big(|\log(\eta)| ({\bf W}(u, r) - {\bf W}(u, \eta r)) + 2 {\bf M}(u, \eta r, e) \Big) r^{-n-2}\int_{\eta r}^r s^{n+1}\,dx +\eta^{n+2}{\bf T}(u, \eta r, e),
$$
from which the third inequality follows easily.
\end{proof}

It will now be convenient to allow the center of the different quantities to vary. To this aim, we denote 
\[
{\bf W}_{x_\circ}(u, r) = {\bf W}(u(\cdot - x_\circ), r), \quad {\bf M}_{x_\circ}(u, r, e) = {\bf M}(u(\cdot - x_\circ), r, e), \quad \text{and}\quad  {\bf T}_{x_\circ}(u, r, e) = {\bf T}(u(\cdot - x_\circ), r, e).
\]

We want to show the existence of a free boundary point where the Weiss energy is close to its maximum while ${\bf T}$ and ${\bf M}$ are very small, all in terms of $\eps_k$. In this result, it will be crucial that we can prove a bound in $\epk$ with a power strictly larger than 2.
\begin{lem}
\label{lem:weiss_small}
Let $\ctta_3$ be as in \eqref{eq:alpha_n}. Then, in the setting of \Cref{prop:decaysqrt} and for $k \gg 1$, there exists   $\bar y \in B_{\Rk/32}(\zk)\cap \FB(u)$ such that 
\[
2\ctta_3 - {\bf W}_{\bar y}(u, R_k^\flat/16) \le \epk^{2+ \cttc/2}\qquad\text{and}\qquad 
{\bf T}_{\bar y}(u, R_k^\flat/4, e) + {\bf M}_{\bar y}(u, R_k^\flat/4, e) \le \epk^{2+  \cttc/2},
\]
for some $e\in \mathbb{S}^2$.
\end{lem}
 
\begin{proof}
Recall that $R_k^\flat = \epk^\cttc \Rk.$
    We divide the proof into three steps. 
    \medskip 
    
    \noindent {\bf Step 1:} Given any $\zz\in B_{\Rk/16}(\zk)\cap \cZ$, we start by proving some controls on the vectors  $a^\flat_\pm$ and the constants $b^\flat_\pm$ from \Cref{prop:decaysqrt} (whose dependence on $\zz$ is omitted). 
    
  Up to a translation, we can assume that $\zz = 0$. By \eqref{equation:weakestdecay}, if we set $r_k^\flat := \epk^{1+2\cttc} \Rk$, then
\begin{equation}
\label{eq:L1weakestdecay0}
\fint_{U_\ast\cap B_{r}} |u-a^\flat_\ast\cdot x- b^\flat_\ast|\,dx \le C \epk^{1+\cttc/3} R_k^\flat\qquad\text{for any}\quad r\in [r_k^\flat, R_k^\flat],\qquad \ast \in \{+,-\}.
\end{equation}
Since $u-a^\flat_\ast\cdot x$ is 2-Lipschitz and vanishes at 0, \eqref{eq:L1weakestdecay0} with $ r = r^\flat_k$ implies 
\begin{equation}
\label{eq:b flat}
|b^\flat_\ast| \le 2r_k^\flat + C\epk^{1+\chi/3}R_k^\flat \le C \epk^{1+\chi/3}R_k^\flat\qquad\text{for}\quad \ast\in \{+, -\}.
\end{equation}
Furthermore, since $R_k^\flat \gg \rb(0)$, \Cref{lem:X-blowdown} implies that $u$ is $L^\infty$-close to a vee $V_{0, e}$ in $B_{R_k^\flat}$. By the $L^1$-closeness condition \eqref{eq:L1weakestdecay0}, we must have $|a^\flat_+- e| \ll 1$ (up to replacing $e$ by $-e$). Consequently, $u > 0$ in the region $B_{R_k^\flat} \cap \{a^\flat_+ \cdot x > R_k^\flat / 16\}$, where it is harmonic. Thus, by $L^1$-to-$L^\infty$ estimates for harmonic functions, we have  $|u - a^\flat_+\cdot x -b^\flat_+| \le C \epk^{1+\cttc/3} R_k^\flat$ in $B_{3R_k^\flat/4} \cap \{a^\flat_+\cdot x \ge R_k^\flat/8\}$. This, together with $a^\flat_+\cdot \nabla (u-a^\flat_+\cdot x-b^\flat_+)\le 0$ 
and the bound on $b^\flat_+$, gives
\[
 u(x) \ge a^\flat_+\cdot x +b^\flat_+ -C \epk^{1+\cttc/3} R_k^\flat\ge a^\flat_+\cdot x   -C \epk^{1+\cttc/3} R_k^\flat\qquad\text{in}\quad B_{R_k^\flat/2}.
\]
By symmetry, the same bound holds for $a^\flat_-$, therefore
\begin{equation}
    \label{eq:geom_u}
 u(x) \ge \max\{a^\flat_+\cdot x, a^\flat_-\cdot x\}   -C \epk^{1+\cttc/3} R_k^\flat\qquad\text{in}\quad B_{R_k^\flat/2}.
\end{equation}
Next, we note that the closeness of $u$ to a vee implies that $|a^\flat_++a^\flat_-|\ll 1$, and we want to quantify this.
To this aim, we first note that  
\begin{equation}
\label{eq:geomcont}
\inf_{ B_{R_k^\flat/2}\cap\{e\cdot x\ge R_k^\flat/8\} } u = 0\quad\text{for any $e\perp a^\flat_+$.} 
\end{equation}
Indeed, since the sum of radii of the balls forming $U_0$ is bounded by $ C\Rk \epk^{1/\ctta} \epk^{-\frac{1+\cttb}{3\ctta}}\ll R_k^\flat$ (recall \eqref{eq:tildeA}), it follows from \Cref{lem:obsUpm} that 
we can always find free boundary points inside $B_{R_k^\flat/2}\cap\{e\cdot x\ge R_k^\flat/8\} $.

Now, assume that $|a^\flat_++a^\flat_-|>0$ (otherwise there is nothing to prove) and consider the vector 
$e = \frac{a^\flat_++a^\flat_-}{|a^\flat_++a^\flat_-|}-\frac12 |a^\flat_++a^\flat_-| a^\flat_+.$
Since $|a^\flat_++a^\flat_-|\ll 1$, $e$ is almost unitary and almost parallel to $a^\flat_++a^\flat_-$. Also, one can readily check that $e\cdot a^\flat_+=0$. Thus, 
\eqref{eq:geom_u} yields 
\[
u(x) \ge   \frac{(a^\flat_++a^\flat_-)\cdot x}{2}   -C \epk^{1+\cttc/3} R_k^\flat\ge c|a^\flat_++a^\flat_-| R_k^\flat  -C \epk^{1+\cttc/3} R_k^\flat\qquad\text{in}\quad  B_{R_k^\flat/2}\cap \{e\cdot x \ge R_k^\flat/8\},
\]
for some universal constant $c>0$.
Combining this bound with 
\eqref{eq:geomcont}, this proves that $|a^\flat_++a^\flat_-|\leq C \epk^{1+\cttc/3} R_k^\flat$. 
In particular, recalling \eqref{eq:b flat}, we conclude that
\begin{equation}
    \label{eq:apmbpm}
    |a^\flat_++a^\flat_-|\le C\epk^{1+\cttc/3}\qquad\text{and}\qquad |b^\flat_+|+|b^\flat_-|\le C\epk^{1+\cttc/3}R_k^\flat.
\end{equation}

\medskip

\noindent {\bf Step 2:} 
 We now show the existence of a point $\bar y \in B_{\Rk/32}(\zk)\cap \FB(u)$ whose distance from  $\cZ$ is comparable to~$R_k^\flat$. 
 
 Indeed, recalling that $\rb(\zz')\ll R_k^\flat$ for all $\zz' \in B_{\Rk/32}(\zk)$, it follows from \eqref{strongalternativeA}  and \Cref{lem:Nbound} that the set $\cZ\cap B_{\Rk/64}(\zk)$ can be covered by $C \epk^{-\frac{(1+\cttb)\cttc}{3}} $ balls of radius $\tfrac14  R_k^\flat$. Hence, recalling \Cref{def:UpUm}, we also have 
\[
\overline{U_0} \cap B_{\Rk/64}(\zk) \subset \bigcup_{i = 1}^N B_{\tfrac34  R_k^\flat}(\zz_i) =: S,
\]
for some  $\zz_i\in \cZ$, where $N\le C\epk^{-\frac{(1+\cttb)\cttc}{3}}$. Notice now that, thanks to  \Cref{lem:obsUpm}, $(\FB(u)\cup U_0)\cap B_{\Rk/32}(\zk)$ is contained inside a strip $W := {\rm Slab}\!\left(B_{\Rk/32}(\zk), e, \epk\right)$ of width $\Rk\epk\ll R_k^\flat$.
In addition, again by \Cref{lem:obsUpm},  $\FB(u)\cup (S\cap W)$ separates $\{u > 0\}\cap B_{\Rk/64}(\zk)$ into two disjoint regions. By projecting these sets onto the hyperplane $\{e\cdot(x-\zk) = 0\}$, we see that the area contributed by $S\cap W$ is of order $(R_k^\flat)^2 \epk^{-\frac{(1+\cttb)\cttc}{3}} \ll \Rk^2$, which means in particular that one can always find a point $\bar y\in \partial S \cap \FB(u)\cap B_{\Rk/32}(\zk)$. In particular, by the construction, it follows that 
\[
\bar y\in \FB(u) \cap B_{\Rk/32}(\zk)\cap \left\{\tfrac12 R_k^\flat\le \dist(\cdot, \cZ) \le \tfrac34 R_k^\flat\right\}.
\]
  Let $\zz\in \cZ\cap B_{\Rk/16}(\zk)$ be such that $\dist(\bar y, \cZ) = |\bar y - \zz|$. After a translation, we assume $\zz = 0$ (so we are putting ourselves in the setting of Step 1). 
From  \eqref{eq:L1weakestdecay0}  we know that
\begin{equation}
\label{eq:L1weakestdecay}
\fint_{U_\pm\cap B_{R_k^\flat/8}(\bar y)} |u-a^\flat_\pm \cdot x- b^\flat_\pm| \,dx \le C \epk^{1+\cttc/3} R_k^\flat,
\end{equation}
where $a^\flat_\pm$ and $b^\flat_\pm$ satisfy \eqref{eq:apmbpm}. Notice also that, by assumption, there are no neck centers $\cZ$ in $B_{R_k^\flat/4}(\bar y)$; therefore, $\partial U_\pm\subset \FB(u)$ and $U_+\cup U_- = \{u > 0\}$ inside $B_{R_k^\flat/8}(\bar y)$, and the restriction of $u$ to $U_\pm$ is a classical solution to the Bernoulli problem inside $B_{R_k^\flat/8}(\bar y)$. Consequently, we can apply \Cref{lem:L1Linfty_2} (rescaled) to such restrictions to obtain
\[
|u-a^\flat_\pm \cdot x - b^\flat_\pm|\le C \epk^{1+\cttc/3} R_k^\flat 
\qquad\text{and}\qquad 
|\nabla u-a^\flat_\pm |\le C \epk^{1+\cttc/3} \qquad\text{in}\quad U_\pm\cap B_{R_k^\flat/16}(\bar y).
\]\
In particular, applying the first bound above with $x=\bar y \in \FB(u)$ it follows that 
 $|a^\flat_\pm \cdot \bar y  + b^\flat_\pm|\le C \epk^{1+\cttc/3}R_k^\flat$,
 that combined with the bounds above implies
\[
|u-a^\flat_\pm \cdot (x-\bar y) |\le  C\epk^{1+\cttc/3} R_k^\flat 
\qquad\text{and}\qquad 
|\nabla u-a^\flat_\pm |\le  C\epk^{1+\cttc/3} \qquad\text{in}\quad U_\pm\cap B_{R_k^\flat/16}(\bar y).
\]
Combining this estimate with \eqref{eq:apmbpm}, we conclude that
\[
\{\pm a^\flat_+\cdot x \ge C \epk^{1+\cttc/3} R_k^\flat\}\subset U_\pm-\bar y\subset \{\pm a^\flat_+\cdot x \ge - C \epk^{1+\cttc/3} R_k^\flat\}\qquad\text{in}\quad B_{R_k^\flat/16},
\]
i.e., the two free boundaries are $C \epk^{1+\cttc/3}$-flat at scale $R_k^\flat/16$ and are at distance $C \epk^{1+\cttc/3} R_k^\flat$ from each other. 

Since $\epk \ll 1$ for $k$ sufficiently large,
combining all these estimates together  we get the desired bounds on ${\bf M}_{\bar y}$ and ${\bf T}_{\bar y}$ with $e = a^\flat_+$.

\medskip

\noindent {\bf Step 3:}
Let $\bar u(x) := \frac{16}{R_k^\flat}u\left(\bar y +\frac{R_k^\flat}{16} x\right)$ be defined in $B_1$. By Step 2 we know that $\{\bar u> 0\}$ has two flat connected components $\bar U_\pm$ and that, after a rotation
\begin{equation}
    \label{eq:baruclose}
    |\bar u\mp x_3| \le C\epk^{1+\cttc/3},\qquad |\nabla \bar u \mp e_3|\le C\epk^{1+\cttc/3}\qquad\text{in}\quad \bar U_\pm \cap B_1,
\end{equation}
and 
\begin{equation}
    \label{eq:baromflat}
\left\{\pm x_3 \ge C  \epk^{1+\cttc/3} \right\}\subset \bar U_\pm\subset \left\{\pm x_3 \ge - C \epk^{1+\cttc/3} \right \}\qquad\text{in}\quad B_{1}.
\end{equation}
Let $\bar u^\pm$ denote the restrictions of $\bar u$ to $\bar U_\pm$, so that $\bar u = \bar u^+ + \bar u^-$ and 
\[
{\bf W}_{\bar y}(u, R_k^\flat/16) = {\bf W}(\bar u, 1)  ={\bf W}(\bar u^+, 1)+{\bf W}(\bar u^-, 1).
\]
Now, given any solution $w$ to the Bernoulli problem in $B_1$,
let $\Omega:=\{w>0\}$, so that $w=0$ on $\partial \Omega$ and $\nabla w$ coincides with the inner unit normal.
Then,
using integration by parts, 
on the one hand we have 
\[
{\bf W}(w, 1) = \int_{\partial B_1} (x\cdot \nabla w - w) w \,d\mathcal H^2+ |\Omega\cap B_1|,
\]
and, on the other hand (here we use that, on $\partial \Omega$, $w=0$ and $\nabla w$ coincides with the inner unit normal),
\[
\int_{\Omega\cap \partial B_1} x\cdot \nabla w \, x_3\,d\mathcal H^2 -\int_{\partial \Omega\cap B_1} x_3 \,d\mathcal H^2 = \int_{\Omega\cap B_1} \nabla w \cdot \nabla x_3 \,dx= \int_{\Omega\cap \partial B_1} w x\cdot \nabla x_3 \,d\mathcal H^2  = \int_{\Omega\cap \partial B_1} w x_3\,d\mathcal H^2 . 
\]
Combining these two identities, we deduce that
\[
{\bf W}(w, 1) = \int_{\partial B_1} (x\cdot \nabla w - w) (w-x_3) \,d\mathcal H^2+ \int_{\partial \Omega\cap B_1} x_3\,d\mathcal H^2+|\Omega\cap B_1|. 
\]
Notice that, applying the divergence theorem to the vector field $x_3e_3$ inside a Lipschitz domain $A$, it follows that
$$
|A\cap B_1|=\int_{A\cap B_1}{\rm div}(x_3e_3)\,dx=\int_{\partial A\cap B_1}x_3 e_3\cdot \nu \,d\mathcal H^2 + \int_{\overline A\cap \partial B_1}x_3^2 \,d\mathcal H^2.
$$
Applying this estimate both with $A=B_1\setminus\Omega$ and $A=B_1^-:=B_1\cap\{x_3 \leq 0\}$, we obtain
$$
\int_{\partial \Omega\cap B_1} x_3\,d\mathcal H^2 - |B_1\setminus 
\Omega| + |B_1^-|= \int_{\partial\Omega \cap B_1}x_3(1-e_3\cdot \nu)\,d\mathcal H^2-\int_{\overline{(B_1 \setminus \Omega)}\cap \partial B_1}x_3^2d\mathcal H^2+\int_{\{x_3\leq 0\}\cap \partial B_1}x_3^2 \, d\mathcal H^2,
$$
where $\nu$ is the inner unit normal to $\Omega$, and therefore
$$
\bigg| \int_{\partial \Omega\cap B_1} x_3\,d\mathcal H^2 + | 
\Omega\cap B_1| -\tfrac12|B_1|\bigg|\leq \frac12 \int_{\partial\Omega \cap B_1}|x_3|\,|\nu -e_3|^2\,d\mathcal H^2+\int_{\left(\overline{(B_1 \setminus \Omega)}\Delta \{x_3\leq 0\}\right)\cap \partial B_1}x_3^2 \,d\mathcal H^2.
$$
Applying this bound with $w=\bar u_+$ and $\Omega =\bar U_+$, recalling \eqref{eq:baruclose} and \eqref{eq:baromflat} we get 
\[
\int_{\partial \bar U_+\cap B_1} x_3\,d\mathcal H^2 + |\Omega\cap B_1|=\tfrac12 |B_1|+ O\left(\epk^{3(1+\cttc/3)}\right),
\]
thus
\[
{\bf W}(\bar u^+, 1) = \int_{\partial B_1} (x\cdot \nabla \bar u^+ - \bar u^+) (\bar u^+-x_3) \,d\cH^2 +\tfrac12 |B_1| +  O\left(\epk^{3(1+\cttc/3)}\right).
\]
Thanks again to
 \eqref{eq:baruclose} and \eqref{eq:baromflat}, we see that the integrand above is bounded by $C\epk^{2(1+\cttc/3)}$. This implies that
\[
{\bf W}(\bar u^+, 1) =  \tfrac12 |B_1| +  O\left(\epk^{2(1+\cttc/3)}\right)\qquad\Longrightarrow\qquad {\bf W}_{\bar y}(u, R_k^\flat/16) = {\bf W}(\bar u, 1) =   |B_1| +  O\left(\epk^{2(1+\cttc/3)}\right),
\]
so the desired bound follows.
\end{proof}

\subsection{Proof of \Cref{thm:main1}} We are now ready to prove our main result. 

\begin{proof}[Proof of \Cref{thm:main1}]
We assume the statement to be false. After the reduction provided by \Cref{lem:reduction}, we can assume $u$ to have a globally bounded Hessian. Then, we can define the neck centers $\cZ$ as done in Subsection~\ref{ssec:neck-set-def}, and this set is non-empty because of \Cref{lem:Znonempty}.

We   define the symmetric excess $\bE_\zz(u, R)$ at any $\zz\in \cZ$ and $R > 0$ as in \eqref{eq:Ez_def_intro}. By \Cref{lem:zkRk}, there exist $R_k\to \infty$ and $\zz_k\in \cZ$ such that 
\[
\eps_k = \bE_{\zz_k}(u, 8R_k) \to 0\qquad\text{as}\quad k \to \infty. 
\]
Moreover, by \Cref{lem:tildezktildeRk}, there are $\Rk = \tilde \zeta_k R_k \to \infty$ and $\tilde\zz_k\in \cZ\cap B_{R_k}(\zz_k)$ such that \eqref{keyeqn} holds, with $\epk = \tilde \zeta_k^{\ctta\cttb} \eps_k$ and $\tilde \zeta_k \in (0, 1]$.

Now, by \eqref{eq:W a 2a}, the monotonicity of the Weiss energy, and \Cref{lem:weiss_small} (using the notation there, as well), there exists a point 
$\bar y\in B_{\Rk/32}(\zk) \cap \FB(u)$ such that
\[
|B_1|- \epk^{2+\cttc /2} \leq  {\bf W}_{\bar y}(u,r) \le |B_1| \qquad \text{ for all $r\geq R_k^\flat/16$}.
\]
Thus, combining this bound with Lemmas \ref{lem:weiss_small} and \ref{lem:Monneau}, we get (for $e\in \bS^2$ as in \Cref{lem:weiss_small})
\[
{\bf T}_{\bar y}(u, 32R_k, e) \le |\log (C R_k / R_k^\flat)| \epk^{2+\cttc /2} +  {\bf M}_{\bar y}(u, R_k^\flat/16, e)+ {\bf T}_{\bar y}(u, R_k^\flat/16, e)\le C|\log (  \tilde\zeta_k   \epk^{\cttc})| \epk^{2+\cttc /2}.
\]
where we used that $R_k^\flat = \epk^\chi \Rk  = \epk^\chi\tilde \zeta_k R_k$. 
Now, recalling that   $\epk = \tilde \zeta_k^{\ctta\cttb}\eps_k$,  it follows that $|\log (  \tilde\zeta_k   \epk^{\cttc})| \epk^{\chi/3} \to 0$ as $\eps_k \to 0$. Hence, in particular,
\[
{\bf T}_{\bar y}(u, 32R_k, e) \le    \eps_k^{2+\cttc /6}\qquad \text{for }k \gg  1.
\]
(Here is the only place in the paper where we are using that $\cttb>0$.) Since 
$B_{8R_k}(\zz_k)\subset B_{32R_k}(\bar y)$, for $k$ sufficiently large we get 
\begin{equation}
    \label{eq:Tbound}
\eps_k = \bE_{\zz_k}(u, 8R_k) \le \big({\bf T}_{\zz_k}(u, 8R_k, e) \big)^{1/2}\leq   \bigg( \Big(\frac{32}8\Big)^5 {\bf T}_{\bar y}(u, 32R_k, e) \bigg)^{1/2} \le 2^5 \eps_k^{1+\cttc /12}, 
\end{equation}
a contradiction.
\end{proof}

\begin{remark}
 The conditions to be satisfied by all the constants appearing throughout the paper are: 
    \[
    3\ctta \cttg  > 1, \qquad 12\ctta\beta  < 5 -\cttb,\qquad \delta_\circ = 2(2\beta - 1) > 0, \qquad \bar\gamma < \frac{1-\ctta}{\ctta},\qquad\frac34 < \ctta < \frac{5-\cttb}{6}, \qquad p < 1+ \bar\gamma,
    \]
    as well as
    \[
     4\cttc + 2 (1+2\cttc)\bar\beta < \delta_\circ, \qquad \bar \beta < \frac{p-1}{p}, \qquad \cttc < \frac{1+\bar\gamma-p}{1+p},\qquad  \frac{\delta_\circ}{8}\ge \frac{5\cttc}{2}.
    \]
    One possible choice is: 
    \[
    \cttb = \frac{1}{20}, \ \ \beta = \frac12 + \frac{1}{40}, \ \ \ctta = \frac{39}{50}, \ \ \cttg = \frac {11}{25}, \ \ p = \frac{21}{20}, \ \ \bar \beta = \frac{1}{40}, \ \ \delta_\circ = \frac{1}{10}, \ \ \cttc = \frac{1}{500}, \ \ \bar\gamma = \frac{1}{10}. 
    \]
\end{remark}

\subsection{Proofs of Corollaries \ref{cor:DeG_Bernoulli}  and \ref{thm:main2}} 
\label{ssec:proofcorollaries}  Corollaries \ref{cor:DeG_Bernoulli} and \ref{thm:main2} are now rather standard consequences of  \Cref{thm:main1}. We sketch their proofs here for the reader's convenience:

\begin{proof}[Proof of \Cref{cor:DeG_Bernoulli}] 
Let $u$ be a global classical solution of the  Bernoulli problem in $\R^4$ satisfying $\partial_{x_4} u>0$ in $\{u>0\}$. We start by noticing that $u$ is stable. Indeed, 
for every smooth compactly supported function $\xi$, let 
${\rm spt}(\xi) \subset K =K'\times[-C,C]$  for some compact set $K'\subset \R^3$. Then, whenever $K\cap\{u>0\}\neq \varnothing$ we have 
\[\inf_{K\cap \{u>0\}} \partial_{4} u \ge c(u,K) >0\]
(see, for instance, \cite[Lemma 4.1]{EFY23}), which shows that translations of the graph of $u$ in the $x_4$ direction locally foliate the graph of $u$. It is then a standard fact that this property implies the minimality of $u$ with respect to sufficiently small variations and, therefore, its stability (for a detailed proof of this fact see, for instance, \cite[Proof of Theorem 4.5]{AAC01}).\footnote{An alternative way to obtain the stability inequality from a positive (sub-)solution of the linearized equation is as follows: if $\psi>0$ satisfies $\Delta \psi=0$ in $\set{u>0}$, $\psi_\nu+H\psi=0$ on $\FB(u)$, then for $\phi\in C_c^{0,1}(\R^4)$, testing against $\phi^2/\psi$ gives
\[
\int_{\set{u>0}} |\nabla \phi|^2\,dx
-\int_{\FB(u)} H\phi^2\,d\cH^3
=\int_{\set{u>0}}
\left(|\nabla \phi|^2
-\nabla\cdot\frac{\phi^2\nabla\psi}{\psi}
\right)\,dx
=\int_{\set{u>0}} \abs{\nabla \phi-\frac{\phi\nabla\psi}{\psi}}^2\,dx
\geq 0.
\]
}

Suppose now, by contradiction, that $D^2u\not\equiv 0$ in $\{u>0\}$. By \Cref{lem:reduction}, there exists a (classical) stable solution (which, as an abuse of notation, we still denote $u$) satisfying
\[
|D^2u|\le 1 \quad \mbox{in }\{u>0\}\qquad \mbox{and}\qquad |D^2u(0)|=1.
\]
Also, as one can easily check, the proof of \Cref{lem:reduction} provides a new function that will still satisfy monotonicity  but in the weaker form
$
\partial_{x_4} u \ge 0.
$
Moreover, up to restricting $u$ to a single connected component of $\{u>0\}$, we can assume without loss of generality that $\{u>0\}$ has one connected component.

We now consider the two limits 
\[
\underline u = \underline u(x_1, x_2, x_3) : = \lim_{x_4\to-\infty } u\quad \mbox{and}\qquad \overline  u = \overline  u(x_1, x_2, x_3) : = \lim_{x_4\to+\infty } u.
\]
Thanks to the bound $|D^2u|\leq 1$ (which also gives uniform curvature bounds on the free boundary), $\underline u\le u$ is either identically zero or is a classical stable solution of Bernoulli in $\R^3$ (cf. proof of \Cref{lem:reduction}). 
Analogously,  $\overline u\ge u$ is either identically $+\infty$ or is a classical stable solution of Bernoulli in $\R^3$.

Applying \Cref{thm:main1} to both $\underline u$ and $\overline u$ we obtain that, if they are not constant (respectively equal to 0 or $+\infty$), then $\{\underline u =0\}$ and $\{\overline u =0\}$ are either a half-space or a slab (i.e., the region between two parallel hyperplanes). Since $0\le u\le \overline u$, in this second scenario also $\{u>0\}$ would be disconnected, contradicting our setup. Thus:\\
(i) $\underline u$ is either zero, or a 1D monotone minimizer (so, of the form $(x\cdot e-a)_+$), or a maximum of two minimizers with disjoint support;\\
(ii)  $\overline u$ is either a 1D monotone minimizer or $+\infty$.\\
This ensures that $\underline u$ and $\overline u$  are, respectively, lower and upper barriers for minimizers (since minimizers cannot cross). 
Thus, since the family of translated graphs $\{x_5=u(x + te_4)\}_{t\in \R}$ foliates the region 
$$
\big\{(x,x_5) \in \R^4 \times [0,+\infty) :  \underline u(x) \le x_5 \le \overline u(x)\big\},
$$
a standard foliation argument (see \cite[Proof of Theorem 1.3]{Jerison-Monneau}) implies that $u$ must be energy-minimizing in every compact subset of $\R^4$.
But then it follows from the regularity theory for minimizers for the Bernoulli problem (e.g. using \cite{DeSilva, Jerison-Savin}) that $D^2 u$ is identically zero in $\{u>0\}$, contradicting $|D^2u(0)|=1$.
\end{proof}

\Cref{thm:main2} will be obtained as a particular case of the following more technical proposition that will be useful in the sequel as well. A direct proof of the fact that \Cref{thm:main1} implies \Cref{thm:main2} is essentially contained in \cite{kamburov2022nondegeneracy}. However, the current proof of \cite[Theorem 1.2]{kamburov2022nondegeneracy} relies on \cite[Lemma 1.21]{CS05} (see the discussion before \cite[Proposition A.4]{kamburov2022nondegeneracy}), whose proof is incomplete. We fix this gap in our \Cref{lem:compact}.  

\begin{prop}
\label{prop:main2bis}
    Let  $u$ be a classical stable solution to the Bernoulli problem in $B_1\subset \R^4$    satisfying $\partial_4 u \ge0$ in $B_1$. 
    Then $|D^2u |\le C$ in $\overline {B_{1/2}\cap \{u>0\}}$, for some  $C$ universal. 
\end{prop}
\begin{proof}
We proceed as in the proof of \Cref{lem:reduction} and assume by contradiction that the statement does not hold. Then, there exists a sequence $u_k$ of classical stable solutions to the Bernoulli problem in $B_1\subset \R^4$, with $0\in \FB(u_k)$, $\partial_4 u_k \ge 0$, and such that 
\[
h_k := |D^2 u_k(x_k)| \left(\frac34 - |x_k| \right)  = \max_{x\in B_{3/4}\cap\partial\{u_k>0\}}|D^2 u_k(x)| \left(\frac34 - |x| \right) \to \infty\qquad\text{as}\quad k \to \infty. 
\]
Set $d_k := D^2 u_k(x_k)$ and $\rho_k = \frac34-|x_k|$, and define the classical stable solution
\[
\tilde u_k(x) := d_k u_k(x_k + x/d_k)\qquad\text{for}\quad x\in B_{d_k \rho_k/2}. 
\]
Since $d_k\rho_k\to \infty$, proceeding as in \Cref{lem:reduction} we can take the limit of $\tilde u_k$ (up to subsequences) to find a global classical stable solution $\tilde u_\infty$, with $0\in \FB(u_\infty)$, $\partial_4 \tilde u_\infty \ge 0$, $|D^2 \tilde u_\infty(0)| = 1$, and $|D^2 \tilde u_\infty|\le 1$ in $\{\tilde u_\infty>0\}$.

Up to restricting $\tilde u_\infty$  to one connected component of $\{\tilde u_\infty > 0\}$, we can assume that $\{\tilde u_\infty > 0\}$ has a single connected component. Then, by the strong minimum principle, either $ \partial_4 \tilde u_\infty \equiv 0$ (in which case we contradict \Cref{thm:main1}) or $ \partial_4 \tilde u_\infty > 0$ (contradicting \Cref{cor:DeG_Bernoulli}).
\end{proof}

We can now prove our second corollary. 
 \begin{proof}[Proof of \Cref{thm:main2}]
 It suffices to extend our function to $B_1\times \R \subset \R^3\times \R$ by taking it constant in the last variable, and then apply \Cref{prop:main2bis}.
 \end{proof}

\section{The Free Boundary Allen--Cahn}
\label{sec:FB_AC}
\subsection{Preliminaries} The goal of this section is to prove \Cref{thm:FBAC-main}.  To show it, we will combine the curvature estimates obtained for the free boundary in the Bernoulli problem (see \Cref{thm:main2}) with 
the Sternberg--Zumbrun stability inequality (see \Cref{lem:SZ_AC}  below). This will allow us to extend Pogorelov's argument \cite{Pog81} for stable minimal surfaces in $\mathbb{R}^3$ to our setting. 

Before beginning with the proof, let us give the definition of classical solution, in analogy with \Cref{defi:solutions}.
 Consider the energy $\cJ_1^0$ from \eqref{eq:energy-functional}. We call  $u : B_R \to [-1,1]$ a \emph{classical solution} of $\cJ_1^0$ if 
\begin{equation}\label{eq:AC-main}
\text{$\{|u| < 1\}$ is locally a smooth domain in $B_R$}\qquad\text{and}\qquad
\begin{cases}
\Delta u=0
	& \text{ in } B_R \cap \set{|u| < 1},\\
|\nabla u|=1
	& \text{ on } B_R \cap  \partial\set{|u|< 1}.
\end{cases}
\end{equation}
The set $\partial\{|u| < 1\}$ is called the \emph{free boundary} and will also be denoted $\FB(u)$. In particular, a classical solution satisfies that $\{|u| < 1\}$ is locally the subgraph of a smooth function around each free boundary point (up to a rotation).

  Classical solutions $u$ are \emph{stationary critical points} of $\cJ_1^0$, in the sense that they satisfy \eqref{eq:stationary} with $\cF = \cJ_1^0$; and stationary critical points $u$ are called \emph{stable} if they have non-negative second (inner) variations, i.e., they satisfy \eqref{eq:stability-var} with $\cF = \cJ_1^0$.  

From now on, a \emph{solution} will refer, unless otherwise stated, to the free boundary Allen--Cahn energy $\cJ^0_1$.

\begin{defn}
\label{defi:solutionsAC}
Let $n \ge 2$ and $R > 0$. In relation to the free boundary Allen--Cahn, i.e., choosing $\cF = \cJ^0_1$  in \eqref{eq:stationary}--\eqref{eq:stability-var}, we say that $u\in H^1(B_R)$ with $B_R\subset \R^n$ is:
\begin{itemize}
\item a \emph{classical solution} or \emph{classical critical point} in $B_R$ if it satisfies \eqref{eq:AC-main} (in particular, it satisfies \eqref{eq:stationary});
\item a \emph{classical stable solution} or \emph{classical stable critical point} in $B_R$ if it is a classical solution and satisfies \eqref{eq:stability-var}.
\end{itemize}
If a function satisfies one of the previous definitions for all $R > 0$, we call it \emph{global}. 
\end{defn}

The Sternberg--Zumbrun stability inequality for the free boundary Allen--Cahn is the following: 

\begin{lem}[Sternberg--Zumbrun stability inequality] \label{lem:SZ_AC} Let $n \ge 2$, and let $u$ be a classical stable critical point of $\cJ^0_1 $ (see \eqref{eq:energy-functional}) in $\R^n$. Then
\begin{equation}\label{eq:stab}
\int_{\R^n} |\cA(u)|^2 |\nabla u|^2 \zeta^2
\,dx
\leq \int_{\R^n} |\nabla u|^2 |\nabla\zeta|^2\,dx\qquad\text{for all}\quad \zeta\in C^{0,1}_c(\R^n),
\end{equation}
where
\[
|\cA(u)|^2(x) :=
\begin{cases}
|A(u(x))|^2 + |\nabla_T\log |\nabla u(x)||^2 \quad &\mbox{if }  |u|<1 \mbox{ and }|\nabla u(x)|\neq 0
\\
0& \mbox{otherwise}.
\end{cases}
\]
Here, $A(u(x))$ denotes the second fundamental form of the level set $\{u=u(x)\}$ at the point $x$  (therefore, $|A(u(x))|^2$ is the sum of the squares of the principal curvatures) and  $\nabla_T$ denotes the tangential gradient to the level sets.
\end{lem}
\begin{proof}
  As in the Bernoulli case, it follows from \eqref{eq:prev_SZ} using the identity in \cite[Lemma 2.1]{SZ98}. 
\end{proof}
 
\begin{remark}\label{rmk:test fct}
 Notice that, by approximation and smoothly extending and cutting off inside $\{|u| = 1\}$, it suffices in \eqref{eq:stab}  to consider test functions with $\zeta\in C^{0, 1}_c\big(\overline{\{|u|< 1\}}\big)$. 
\end{remark}

As a first observation, we have Modica's inequality (in analogy with its smooth counterpart \cite{Modica85}): 

\begin{lem}[Modica's inequality]
\label{lem:Modica}
Let $n \ge 2$, and let $u$ be a classical  critical point  of $\cJ^0_1$ in $\R^n$. Then  Modica's inequality takes the form
\begin{equation}\label{eq:Modica}
|\nabla u |^2\le 1.
\end{equation}
\end{lem}
\begin{proof}
    The global Lipschitz bound with a dimensional constant $C$ holds by the same proof as for the Bernoulli problem, see for example \cite[Lemma 11.19]{CS05}. Let us now prove that this constant can be chosen to be $1$.
    
    By contradiction, assume that
    \[
    \sup_{\R^n} \, |\nabla u| = 1+\kappa>1, 
    \]
    for some $\kappa \in (0, C-1]$. Then, since $|\nabla u|^2$ is subharmonic inside $\{|u|<1\},$ and $|\nabla u|=1$ on the free boundary, the  maximum principle implies that there exists a sequence $z_i\in \{|u|< 1\}$, with $|z_i|\to \infty$ as $i\to \infty$, such that $|\nabla u(z_i)|\uparrow 1+\kappa$. Let $y_i\in \partial \{|u|<1\}$ satisfy $\dist(z_i, \partial\{|u|< 1\}) = |z_i - y_i| =: \tau_i$. Notice that, since $|u|\leq 1$,  harmonic estimates imply that $|\nabla u| \leq C{\rm dist}(\cdot,\partial\{|u|<1\})^{-1}$,  so $\tau_i$ is necessarily bounded (but it could go to zero).
    
    Up to taking a subsequence and replacing $u$ with $-u$,
      we can assume that $u(y_i)= -1$ for all $i$. Then, if we define
    \[
    v_i(x) := \frac{u(z_i+\tau_i x)+1}{\tau_i}\geq 0,
    \]
    it follows that $v_i$ satisfies 
    \[
    \Delta v_i = 0\quad\text{and}\quad v_i > 0\qquad\text{in}\quad B_1, \qquad \sup_{\R^n} \, |\nabla v_i| = 1+\kappa, \qquad |\nabla v_i(0)| \uparrow_i 1+\kappa,\qquad |\nabla v_i| = 1\quad\text{on}\quad \partial\{v_i > 0\}. 
    \]
    Also, up to a rotation, we can assume that $z_i - y_i = \tau_i e_n$, therefore $v_i(-e_n) = 0$.

Then, up to a subsequence, the functions $v_i$ converge locally uniformly in $\R^n$ to a $(1+\kappa)$-Lipschitz function $v_\infty$ that is harmonic in $B_1$ and satisfies $|\nabla v_\infty(0)|=1+\kappa$,   $v_\infty \ge 0$ in $B_1$, and $v_\infty(-e_n) = 0$. Thus, by the strong maximum principle, $|\nabla v_\infty|\equiv 1+\kappa$ in $B_1$, and therefore $v_\infty (x)= (1+\kappa)(x_n +1)$ in $B_1$. By unique continuation, it follows that $v_\infty (x)= (1+\kappa)(x_n +1)$ in $\{-1 \leq x_n \leq 1\}$.

Thus, we have proved that the non-negative functions $v_i$ converge locally uniformly to  $(1+\kappa)(x_n +1)$ inside  $\{-1 < x_n < 1\}$. 
Consider now the harmonic sub-barrier
    \[
    \psi_{\kappa, \eps}(x) := (1+\tfrac{\kappa}{2})(x_n +1)  +\eps \left((x_n+1)^2-\tfrac{1}{n-1}(x_1^2+\dots+x_{n-1}^2)\right).
    \]
For $\eps$ sufficiently small (depending only on $\kappa$) and for $i$ large enough (depending on $\eps$ and $\kappa$), we see that $v_i \ge \psi_{\kappa, \eps}(x-se_n)$ on $\partial B_2(-e_n)$ for all $s\in [0, 1]$, and $\psi_{\kappa, \eps}(x- e_n)\le v_i$ in $B_2(-e_n)$. 

We now perform a sliding argument and let $s$ decrease from 1 until $\psi_{\kappa, \eps}(x-se_n)$ touches $v_i$ from below. By the previous considerations and the maximum principle, the touching point must be on the free boundary of $v_i$. But this is a contradiction, since $|\nabla v_i| = 1$ on the free boundary while $|\nabla \psi_{\kappa, \eps}| \geq \partial_n \psi_{\kappa, \eps} \geq 1+\tfrac{\kappa}{2} -C_n\eps > 1$ for $\eps$ small enough, depending only on $n$ and $\kappa$.   
\end{proof}
The stability inequality \eqref{eq:stab} will now be used in four ways:
\begin{enumerate}
    \item With a Euclidean $\log$-cut-off in $\R^3$, showing that the amount of ``bad regions'' is sublinear. This results in a clean annulus (see \Cref{lem:freeannulus}).
    \item With a Euclidean Lipschitz cut-off, ensuring good estimates in the clean annulus (see \Cref{prop:2W}).
    \item With an intrinsic $\log$-cut-off on a level set, a $2$-surface, bounding the average area near the ``bad region'' (see \Cref{lem:stab-on-bad}).
    \item With an intrinsic ``tent'' function of the form $(r-\dB)_+$, in an integral way (see \eqref{eq:stab-4}), allowing us to close a Gauss--Bonnet type estimate (see \Cref{lem:GB-rigor}).
\end{enumerate}

\subsection{Definition of $\ccB$}  
From now on, we will assume that $n = 3$ and $u$ is a global classical stable solution to the free boundary Allen--Cahn in $\R^3$ according to \Cref{defi:solutionsAC}. 
We start with the following universal derivative bounds, which follow from the curvature bounds on the free boundary of stable solutions of the one-phase Bernoulli problem.

\begin{lem}[Regularity]
\label{lem:global-curv-bd}
For any $k \geq 2$ there exists a constant $C_k>0$, depending only on $k$, such that
$|D^ku|\leq C_k$ inside ${\set{|u|<1}}$.
\end{lem}

\begin{proof}
Fix $x_\circ \in \partial\{|u|< 1\}$ and, up to replacing $u$ with $-u$, assume that $u(x_\circ) = -1$. Thanks to  \Cref{lem:Modica} we know that $|\nabla u|\le 1$, therefore $u \le 0 < 1$ in $B_1(x_\circ)$. This implies that $u+1$ is a classical stable solution to the Bernoulli problem in $B_1(x_\circ)$ (see \Cref{defi:solutions}), so we can apply \Cref{thm:main2} and \Cref{lem:eps-reg-classical-2} to deduce that $|D^k u|\le C_k$ in $B_{1/2}(x_\circ)\cap \{u +1 > 0\}$. 
Repeating this argument at every free boundary point, we obtain $|D^k u|\le C_k$
 inside  $\{ 0 < \dist(\cdot, \{|u| = 1\}) < 1/2\}$. Finally, the bound inside $\{|u|< 1\}$ follows by interior regularity estimates for harmonic functions.
 \end{proof}

Motivated by \Cref{lem:SZ_AC}, given $\delta_\circ\in(0,1)$ we define 
\begin{equation}
\label{eq:Xdef}
\cX(\delta_\circ) := \bigg\{z\in  \set{|u|<1}  : \int_{B_{2}(z)} |\cA(u)|^2  |\nabla u|^2 \,dx > \delta_\circ\bigg\}\neq\varnothing.
\end{equation}
Note that \Cref{thm:FBAC-main} is equivalent to showing that $\cX(\delta_\circ) = \varnothing$ for any $\delta_\circ>0$.
So, by contradiction, we assume that there exists $\delta_\circ \in (0,1)$ small (to be fixed later) such that $\cX(\delta_\circ) \neq  \varnothing$, and define
\begin{equation}
\label{eq:Gdef}
\cG(\delta_\circ) := \big\{z\in \{|u|< 1\}\setminus \cX(\delta_\circ) : \dist(z, \{|u|= 1 \})\le 8\big\} 
\end{equation}
and
\begin{equation}
\label{eq:Wdef_11}
\cW(\delta_\circ) := \{|u|< 1\}\setminus (\cX(\delta_\circ)\cup \cG(\delta_\circ)). 
\end{equation}

 The following result says that the set $\cG(\delta_\circ)$ locally looks like arbitrarily flat strips: 
 
\begin{lem}[Curvature bound in $\cG(\delta_\circ)$] \label{prop:good}
Given $\eta_\circ\in (0,1)$, there exists $\delta_\circ=\delta_\circ(\eta_\circ)>0$ such that 
if $x_\circ \in \cG(\delta_\circ)$, then 
\[
|D^2 u|\le \eta_\circ\quad \text{in}\quad B_{3/2}(x_\circ)\cap \{|u|< 1\},
\]
and $\nabla u$ almost achieves equality in Modica's inequality \eqref{eq:Modica}: 
\begin{equation}
\label{eq:good-grad}
1-\eta_\circ \leq |\nabla u|\leq 1
	\qquad \text{ in } \set{|u|<1}\cap B_{3/2}(x_\circ).
\end{equation}
In particular, for all $\lambda \in (-1, 1)$, the level set $\set{|u|=\lambda} \cap B_{3/2}(x_\circ)$ is a smooth surface with curvature bounded by~$\eta_\circ$. 
\end{lem}

\begin{proof}
Translating if necessary, we can assume $x_\circ=0$. We show first the bound on the Hessian.

As in the proof of \Cref{lem:Stern_Zum_App},  $|D^2 u|^2\le 3 |\cA(u)|^2 |\nabla u|^2$ wherever $u$ is harmonic. Thus, 
$$
\int_{B_2\cap \{|u| < 1\}} |D^2 u|^2 \,dx \le 3\delta_\circ. 
$$
Also, thanks to \Cref{lem:global-curv-bd},
$|D^3 u|\le C$ in $B_{3/2}\cap \{|u|< 1\}$, and the free boundaries have bounded curvature and they are uniformly separated. So, by  \Cref{lem:L1_Lip} applied to $|D^2u|$ we get
\begin{equation}\label{eq:delta1/8}
|D^2 u|\le C \delta_\circ^{1/8}\quad \text{ in }B_{3/2}\cap \{|u|< 1\}.
\end{equation}
This proves the first bound in the statement.

For the second one, we proceed by contradiction and compactness.  Let $u_k$ be a sequence of (nonconstant) stable critical points of $\cJ^0_1$ in $\R^3$ for which $0 \in \cG(\delta_k)$ with $\delta_k=\frac1k$, but $|\nabla u_k(x_k)|< 1-\eta_\circ$ for some $x_k \in \set{|u_k|<1}\cap B_{3/2}$. Note that, as a consequence of \ref{eq:delta1/8}, 
\begin{equation}\label{eq:wk-curv}
\|D^2 u_k\|_{L^\infty(B_{3/2}\cap \{|u_k|< 1\})}\le  Ck^{-1/8}\to 0,\quad\text{as}\quad k \to \infty.
\end{equation}
Thanks to \Cref{lem:global-curv-bd}, up to a subsequence the functions $u_k$ converge to $u_\infty$,
which is a classical stable solution satisfying (because of \eqref{eq:wk-curv} and unique continuation) $|D^2 u_\infty|\equiv 0$ in any connected component of $\{|u_\infty|< 1\}$
touching $B_{3/2}$. Also, there exists a point $x_\infty \in \set{|u_\infty|<1}\cap \overline{B_{3/2}}$ such that $|\nabla u_\infty(x_\infty)|\leq 1-\eta_\circ$. Since $|\nabla u_\infty|=1$ on the free boundary, this implies that every connected component of $\{|u_\infty|< 1\}$
touching $B_{3/2}$ cannot have any boundary, therefore the only option is that 
$\{|u_\infty|< 1\}=\R^3$. By the convergence of $u_k$ to $u_\infty$, this implies that $u_k$ has no free boundary point inside $B_{16}$ for $k$ large, a contradiction to the fact that $0 \in \cG(\delta_\circ)$.
\end{proof}
\begin{remark}
\label{rem:distcGcW}
As a consequence of the previous result, 
inside $\cG(\delta_\circ)$ the integral curves of $\nabla u$ are almost straight and $|\nabla u|$ is very close to $1$. Hence, by looking how the value of $u$ changes along integral curves of $\nabla u$, for any point $z\in \cG(\delta_\circ)$ it holds $\dist(z, \{|u|= 1\})\le 1-|u|+o_{\delta_\circ}(1) \le 2$, where $o_{\delta_\circ}(1)\downarrow 0$ as $\delta_\circ\downarrow 0$. In particular, by the definition of $\cW(\delta_\circ)$, it follows that $|z_g - z_w|\ge 6$ for any $(z_g, z_w)\in \cG(\delta_\circ)\times  \cW(\delta_\circ)$. Hence, these two sets are separated and since $\cW(\delta_\circ)$ is far from the free boundaries, it must always be surrounded by $\cX(\delta_\circ)$.
\end{remark}

Now, given $\alpha >0$ we define 
\begin{equation}
\label{eq:Sstardef}
\cS^*_\alpha(\delta_\circ) := \bigcup_{z\in \cX(\delta_\circ)} B_\alpha(z)=\cX(\delta_\circ)+B_\alpha. 
\end{equation}
Then, given $x\in \cG(\delta_\circ) \setminus \cS^*_4(\delta_\circ)$ and $\lambda \in [-1, 1]$, we define the $\pi_\lambda$ ``projection'' as the point on $\Sigma_\lambda := \{u = \lambda\}$ obtained from flowing $x$ perpendicularly to the level sets until it intersects $\Sigma_\lambda$. That is, let $n_x:[-1, 1]\to \overline{\{|u|< 1\}}$ be defined by 
\begin{equation}
    \label{eq:nxt}
\dot{n}_x(t) = \frac{\nabla u(n_x(t))}{|\nabla u(n_x(t))|^2}, \qquad n_x(u(x)) = x,
\end{equation}
and set 
\begin{equation}
    \label{eq:pi lambda}
\pi_\lambda(x) := n_x(\lambda).
\end{equation}
Notice that, if $x\in \cG(\delta_\circ) \setminus \cS^*_4(\delta_\circ)$, then 
\Cref{prop:good} and \Cref{rem:distcGcW} imply that
$1-\eta_\circ \leq |\nabla u| \leq 1$ in $B_4(x)\cap \{|u|<1\}$, so the map above is well defined (provided $\delta_\circ$ is sufficiently small so that $\eta_\circ \leq 1/2$).

Finally, we define
\begin{equation}
    \label{eq:Bstardef}
\ccB_*(\delta_\circ) := \cS^*_4(\delta_\circ) \cup \bigcup_{x\in \partial \cS^*_4(\delta_\circ)\cap\cG(\delta_\circ)} \left\{n_x(t) : t\in [-1,1]\right\} = \cS^*_4(\delta_\circ) \cup \bigcup_{x\in \partial \cS^*_4(\delta_\circ)\cap\cG(\delta_\circ)} \left\{\pi_\lambda(x) : \lambda\in [-1,1]\right\}.
\end{equation}
In other words, we add to $\cS^*_4(\delta_\circ)$ the image of  $\partial \cS^*_4(\delta_\circ)\cap\cG(\delta_\circ)$ through the flow $\nabla u$ across the level sets. Observe that, by construction and thanks to \Cref{prop:good}, 
\begin{equation}
\label{eq:BS}
\ccB_*(\delta_\circ)\subset \cS^*_8(\delta_\circ).
\end{equation}
Moreover, $\cG(\delta_\circ) \setminus \ccB_*(\delta_\circ)$ is ``invariant'' under the flow of $\nabla u$, that is,
\[
\text{if $x\in \cG(\delta_\circ)\setminus \ccB_*(\delta_\circ)$, then $\pi_\lambda(x)$ is well defined and $\pi_\lambda(x)\in \cG(\delta_\circ)\setminus \ccB_*(\delta_\circ)$ for all $\lambda\in [-1, 1]$.}
\]

Now, thanks to the stability of $u$, we can prove that for any $\Lambda>0$ (large) we can find  an  annulus $B_{R+\Lambda}(z) \setminus B_R(z)$ of width $\Lambda$, with  $z\in \cX(\delta_\circ)$, which does not intersect $\cS^*_8(\delta_\circ).$

\begin{lem}[Existence of a clean annulus]
\label{lem:freeannulus}
Let $\delta_\circ > 0$ be sufficiently small so that $\eta_\circ\leq 1/2$. For every $\Lambda >0$ there exist $z\in \cX(\delta_\circ)$ and  $R > 1$ such that 
\[
\ccB_*(\delta_\circ) \cap B_{R+\Lambda}(z) \subset \cS^*_8(\delta_\circ) \cap B_{R+\Lambda}(z) \subset B_R(z).
\]
\end{lem}
\begin{proof} The first inclusion follows from \eqref{eq:BS}.

For the second inclusion, assume by contradiction that it does not hold. Then, there exists $\Lambda >0$ such that for every $\bar z\in \cX(\delta_\circ)$ and $k \geq 1$, there is $z\in \cX(\delta_\circ)$ with  $B_8(z) \cap B_{(k+1)\Lambda}(\bar z) \setminus B_{k\Lambda}(\bar z)\neq \varnothing$. In particular 
$$
 B_{(k+1)\Lambda+8}(\bar z) \setminus B_{k\Lambda-8}(\bar z)\supset B_8(z) \qquad \Longrightarrow\qquad 
|\cS^*_8(\delta_\circ) \cap (B_{(k+1)\Lambda+8}(\bar z) \setminus B_{k\Lambda-8}(\bar z)) |\geq |B_8| \qquad \forall\,k \geq \frac{8}
\Lambda,
$$
from which we easily deduce that
 \begin{equation}
     \label{eq:lowercA}
    |\cS^*_8(\delta_\circ)\cap B_R(\bar z)|\ge c\frac{R}{\Lambda}\qquad\text{for all}\quad R > 1, \ \bar z\in \cX(\delta_\circ),
 \end{equation}
for some $c > 0$ universal.

On the other hand, applying the stability inequality \eqref{eq:stab} with $\zeta\in C^\infty_c(B_{2R}(\bar z))$ such that $\zeta \equiv 1$ in $B_R(\bar z)$ and $|\nabla \zeta| \le \bar C/R$ in $\R^3$, recalling that $|\nabla u|\le 1$ we get 
$$\int_{B_{R}(\bar z)} |\cA(u)|^2|\nabla u|^2 \,dx \le CR. 
$$
Now, recalling the definition of $\cX(\delta_\circ)$,  a covering argument implies that the left-hand side above is bounded from below by $c\delta_\circ |\cS^*_8(\delta_\circ) \cap B_{R-8}(\bar z)|$ for all $R > 9$. Therefore, we have proved that
 \begin{equation}
     \label{eq:uppercA}
    |\cS^*_8(\delta_\circ)\cap B_R(\bar z)|\le C\delta_\circ^{-1} R\qquad\text{for all}\quad R > 1, \ \bar z \in \cX(\delta_\circ),
 \end{equation}
 with $C$ universal. (In particular, from \eqref{eq:lowercA},  we have  $\Lambda \ge c\delta_\circ / C$.)

Now, given $\bar z\in \cX(\delta_\circ)$ and $R$ large,  for $t\ge 1$ we define
\[
A^t_{R, \bar z} := \bigcup \left\{B_t(z): \ z\in \cX(\delta_\circ) ,\, B_8(z) \cap B_R(\bar z) \neq \varnothing \right\}.
\] 
Then, for any $t\in [4, R]$, by Vitali's covering lemma we can find a disjoint subcollection of balls of radius $t/4$, centered at some $z\in \cX(\delta_\circ)$, such that the balls of radius $2t$ cover $A_{R, \bar z}^t$. Since:\\
(i) each disjoint ball of radius $t/4$ contains at least $c\Lambda^{-1}t$ mass from $\cS^*_8(\delta_\circ)$ (by \eqref{eq:lowercA});\\
(ii) these balls are all contained inside $B_{2R}(\bar z)$;\\
(iii) and $|\cS^*_8(\delta_\circ)\cap B_{2R}(\bar z)|\le C\delta_\circ^{-1}R$ (by \eqref{eq:uppercA});\\
it follows that the number of disjoint balls of radius $t/4$ is bounded by $C(\delta_\circ t)^{-1}\Lambda R$ with $C$ universal. In particular, since the balls of radius $2t$ cover $A_{R, \bar z}^t$, we get
\[
| A_{R, \bar z}^t | \leq \text{(number of disjoint balls)} \times |B_{2t}| \leq C\delta_\circ^{-1}\Lambda Rt^2 \qquad\text{for all}\quad t\in [4, R].
\]
Note that, for $t\in (0, 4]$, we can simply use the bound $|A_{R, \bar z}^t|\le C\delta_\circ^{-1}R$. 

Now, consider $R>1$ large (to be fixed later), set $d(x) :=\text{dist}(x, A^8_{R,\bar z})$, and define the test function:
\[
\zeta(x) = 
\left(
1- \tfrac{\log(1+d(x))}{\log (1+R)}\right)_+,\qquad\text{which satisfies}\quad 
|\nabla \zeta(x)| =
\begin{cases} 
\frac{1}{(1+d(x))  \log(1+R) } & \text{if } 0 < d \le R,\\
0 & \text{otherwise}.\\
\end{cases}
\]

Applying the stability inequality with $\zeta(x)$ and $R\gg 1$, we estimate the two terms as follows:
\begin{itemize}
    \item by a covering argument, the definitions of $\cX(\delta_\circ)$ and $\cS^*_8(\delta_\circ)$,  and   \eqref{eq:lowercA}, we can  bound the left-hand~side:
    \[
    \int_{\R^3}|\cA(u)|^2 |\nabla u |^2\zeta^2 \,dx 
    \ge  \int_{A^8_{R, \bar z}}|\cA(u)|^2 |\nabla u |^2 \,dx 
    \ge c\delta_\circ |\cS^*_8(\delta_\circ)\cap B_{R}(\bar z)| 
    \ge c\delta_\circ \Lambda^{-1} R;
    \]
    \item by the ``layer-cake formula'' and since $|\nabla u|\le 1$, we can bound right-hand side:
    \[
    \int_{\mathbb{R}^3} |\nabla \zeta|^2 \, dx \leq C \int_{0}^R \frac{1}{(1+t)^2 \log^2 (1+R)} \frac{d(C\Lambda R t^2)}{dt} dt  \le C\Lambda\frac{R}{\log R}. 
    \]
\end{itemize}
For $R$ large enough, this provides the desired contradiction, proving the lemma. 
\end{proof}

Now, for a given $\delta_\circ>0$ small and $\Lambda >1$ large, let $z_\Lambda\in \cX(\delta_\circ)$ and $R_\Lambda > 1$ be given by \Cref{lem:freeannulus}, and define the sets
\begin{equation}
\label{eq:cBdefAC}
\ccB = \ccB(\delta_\circ, \Lambda) := \ccB_*(\delta_\circ)\cap B_{R_\Lambda+\Lambda}(z_\Lambda),\qquad
\cS_\alpha = \cS_\alpha(\delta_\circ, \Lambda) := \bigcup_{z\in \cX(\delta_\circ)\cap \ccB(\delta_\circ, \Lambda)} B_\alpha(z). 
\end{equation}
Note that, since $\cS^*_4(\delta_\circ)\subset \ccB_*(\delta_\circ)\subset \cS^*_8(\delta_\circ)$, it follows from \Cref{lem:freeannulus} that $\cS_4\subset \ccB\subset \cS_8$. 
We now start our analysis.

\subsection{A first case: annulus formed of $\cW(\delta_\circ)$}
\label{ssec:W}

With $\delta_\circ>0$ (small) and $\Lambda >1$ (large) fixed, we recall that $z_\Lambda\in \cX(\delta_\circ)$ and $R_\Lambda > 1$ are given by \Cref{lem:freeannulus}.
Since $\dist(\cW(\delta_\circ), \{|u|= 1\}\cup \cG(\delta_\circ)) \ge 6$ (by the definition of $\cW(\delta_\circ)$ and \Cref{rem:distcGcW}), \Cref{lem:freeannulus} implies that 
\begin{align}
    \label{eq:dichotomy1}
    &\text{either}\quad B_{R_\Lambda+\Lambda}(z_\Lambda)\setminus B_{R_\Lambda}(z_\Lambda) \subset \cW(\delta_\circ),\\
    \label{eq:dichotomy2}&\text{or}\quad B_{R_\Lambda+\Lambda}(z_\Lambda)\setminus B_{R_\Lambda}(z_\Lambda) \subset \{|u| = 1\}\cup \cG(\delta_\circ).
\end{align}

We want to prove that the first case cannot occur.

\begin{prop}
\label{prop:2W}
  There exists $\Lambda_0$ sufficiently large, depending only on $\delta_\circ$, such that if $\Lambda > \Lambda_0$ then \eqref{eq:dichotomy2} holds. 
\end{prop}

In order to prove it, we will use the following: 

\begin{lem}
\label{lem:comp_Joaq0}
 Let $\delta_\circ > 0$ and $\Lambda \ge 64$, and let $\ccB$ and $\cS_8$ be as in \eqref{eq:cBdefAC}. Then
 \[
   \int_{\ccB} |\cA(u)|^2 |\nabla u|^2 \,dx \ge   c\delta_\circ  |\cS_8| \ge c \delta_\circ  |\ccB| ,
 \]
 for a universal constant $c$.
\end{lem}
\begin{proof}
     Recall that $\cS_2 \subset \ccB \subset \cS_8$. Let $\tilde{\cS_2}\subset \cS_2$ be the union of a maximally disjoint family of $N$ balls of radius 2 centered at $\cX(\delta_\circ)\cap \ccB$ such that the balls with 
     radius $2\cdot3+6$ cover $\cS_8=\cS_2+B_6$. In particular, we know that $N|B_2|\le |\cS_2| \leq |\cS_8| \le N|B_{12}|$. Moreover, by the definition of $\cX(\delta_\circ)$, 
    \[
    \int_{\tilde{\cS_2}(\delta_\circ)} |\cA(u)|^2 |\nabla u|^2 \,dx
    >N \delta_\circ \geq c |\cS_8|\delta_\circ  
    \geq c|\ccB| \delta_\circ. 
    \]
    This yields the result.  
\end{proof}

We can now prove that \eqref{eq:dichotomy1} does not occur. 
\begin{proof}[Proof of \Cref{prop:2W}]
We argue by contradiction and assume \eqref{eq:dichotomy1} holds. Then, $|u|< 1$ (and so it is harmonic)  inside $\cW(\delta_\circ)\supset B_{R_\Lambda+\Lambda}(z_\Lambda)\setminus B_{R_\Lambda}(z_\Lambda)$. Also, since $\dist(\cW(\delta_\circ), \{|u|= 1\}\cup \cG(\delta_\circ)) \ge 6$, we have $\partial \cW(\delta_\circ)\subset \partial\cX(\delta_\circ)$ (i.e., $\cW(\delta_\circ)$ is surrounded by $\cX(\delta_\circ)$). In particular,
thanks to \Cref{lem:freeannulus}, the following Lipschitz function is compactly supported inside $B_{R_\Lambda+\Lambda/2}(z_\Lambda)$ (note $|\nabla \zeta|\neq 0$ only in $\cW(\delta_\circ)$):
\[
\zeta(x) = 
\begin{cases}
1 & \text{if} \quad x\in (\ccB\cup \cG(\delta_\circ)\cup\{|u| = 1\})\cap B_{R_\Lambda}(z_\Lambda), \\
(1-\frac{2}{\Lambda}\dist(x, \ccB))_+ & \text{otherwise}. 
\end{cases}
\]
Now, consider the set $\cS_t$ as in \eqref{eq:cBdefAC} and note that $|\cS_t(\delta_\circ)|\le Ct^3 |\ccB|$ for $t\ge 2$. Also, by harmonic estimates, $|\nabla u|\le \frac{C}{t}$ inside $\{|\nabla \zeta|\neq 0\}\cap B_{R_\Lambda+\Lambda/2}(z_\Lambda) \setminus \cS_t(\delta_\circ)$. Hence, since $|\nabla \dist(\cdot, \ccB)| = 1$, by the layer-cake formula we get
\[
\int_{\R^3}|\nabla u|^2|\nabla \zeta|^2\,dx \le \frac{C}{\Lambda^2}\left( |\cS_2(\delta_\circ)|+\int_2^{\Lambda/2} t^{-2} \frac{d(Ct^3|\ccB|)}{dt}\, dt\right)\le \frac{C|\ccB|}{\Lambda}. 
\]
Thus, by the stability inequality \eqref{eq:stab}, the bound above, \Cref{lem:comp_Joaq0}, and \eqref{eq:stab}, we obtain
\[
c|\ccB|\delta_\circ \le \int_{\ccB}|\cA(u)|^2|\nabla u|^2 \,dx\le \frac{C|\ccB|}{\Lambda}.
\]
This provides the desired contradiction for $\Lambda$ sufficiently large (depending on $\delta_\circ$).
\end{proof}

Thanks to the previous proposition, from now on we will assume that \eqref{eq:dichotomy2} holds. 

\subsection{The intrinsic distance projection}

Define
\[
\Sigma_\lambda^\cG := \Sigma_\lambda\cap \cG(\delta_\circ) = \{u = \lambda\}\cap \cG(\delta_\circ),\qquad\text{for}\quad \lambda\in (-1, 1). 
\]
Also, recall \eqref{eq:pi lambda}.
We  have the following result about the comparability of the length of curves when projected onto different level sets:

\begin{lem} \label{lem:lengthcomp}
 For any $\eps_\circ > 0$ there exists $\delta_\circ> 0$ such that the following holds.
 
 Given 
 $\lambda\in (-1, 1)$ and a Lipschitz curve $\gamma_\lambda:[0,1 ]\to \Sigma^\cG_\lambda \setminus \ccB_*(\delta_\circ)$, define $\bar\gamma_\mu:[0,1 ]\to \Sigma^\cG_\lambda \setminus \ccB_*(\delta_\circ)$ as $\bar\gamma_\mu(t) := \pi_\mu(\gamma_\lambda(t))$, $\mu \in [-1,1]$. Then,
 \[ 
   (1-\eps_\circ) {\rm Length}(\bar\gamma_\mu) \le {\rm Length}(\gamma_\lambda) \le (1+\eps_\circ){\rm Length}(\bar\gamma_\mu)\qquad \forall\,\mu \in [-1,1].  
 \]
\end{lem}
\begin{proof}
    This is a direct consequence of the smallness of the curvature of the level sets given by \Cref{prop:good}. 
\end{proof}

Next, given $\lambda\in[-1,1]$, we define the intrinsic distance along $\Sigma_\lambda$ in the set $\cG(\delta_\circ)$ as follows:
\begin{equation}\label{eq:dB-def}
\dB^\lambda:B_{R_\Lambda+\Lambda}(z_\Lambda)\cap\{|u|< 1\}  \to \R\cup\{+\infty\},\qquad \dB^\lambda(x):=
\begin{cases} 
{\rm dist}_{\Sigma_\lambda}(\pi_\lambda(x), \ccB) & \text{if } x\in \cG(\delta_\circ)\setminus  \ccB,\\
0 & \text{otherwise}.
\end{cases}
\end{equation}
where ${\rm dist}_{\Sigma_\lambda}$ is the intrinsic distance inside the surface $\Sigma_\lambda$,
and $z_\Lambda\in \cX(\delta_\circ)$ and $R_\Lambda > 1$ are given by \Cref{lem:freeannulus}. 
(We have omitted in $\dB^\lambda$ the dependence on $\delta_\circ$ and $\Lambda$ for the sake of readability.) Note that $\{\dB^\lambda>0\}$ is disjoint from $\cW(\delta_\circ)$.
The next result shows how ${\rm dist}_{\Sigma_\lambda}$ changes when varying $\lambda.$

\begin{lem}[Comparison across levels]
\label{lem:comp-dist}
Let $\delta_\circ > 0$ and $\Lambda > 0$. Let $\ccB = \ccB(\delta_\circ, \Lambda)$ be as in \eqref{eq:cBdefAC}, with $\dB^\lambda$ as in \eqref{eq:dB-def}. Then, for any $\lambda,\mu\in(-1,1)$ and $0< r < \Lambda/8$, it holds:
\begin{itemize}
    \item For any $p\in\N$, there exists $\delta_\circ$ small enough depending only on $p$ such that,  
\[
\big(r-\dB^\lambda\big)_+ 
\leq \big(2^{1/p}r-\dB^\mu\big)_+,\qquad\text{in}\quad B_{R_\Lambda+\Lambda}(z_\Lambda)\cap\{|u|< 1\}. 
\]
    \item Let  $\eps_\circ$ be as in \Cref{lem:lengthcomp}. Then
    \[
\abs{\nabla \dB^\lambda } \leq 1+\eps_\circ,\qquad \text{in}\quad B_{R_\Lambda+\Lambda}(z_\Lambda)\cap\{|u|< 1\} \cap \{\dB^\lambda < \Lambda/4\}. 
\]
\end{itemize}
\end{lem}

\begin{proof}
Let $x \in B_{R_\Lambda+\Lambda}(z_\Lambda)\cap\{|u|< 1\} \cap \{\dB^\lambda < \Lambda/4\}$. We can assume that $x\in \cG(\delta_\circ)\setminus  \ccB$, otherwise $d_\ccB^\mu(x)=0$ for all $\mu$ and the result holds. 

Now, for $0 < \dB^\lambda(x) < r\le \Lambda/8$, it follows from  \Cref{lem:freeannulus} that $x\in B_{R_\Lambda+\Lambda/6}(z_\Lambda)$. Also, by \Cref{lem:lengthcomp} and the definition of $\ccB_*(\delta_\circ)$, we have $\dB^\lambda(x) \ge (1-\eps_\circ)\dB^\mu(x)$. Thus, 
\[
(r-\dB^\lambda(x))_+ \le (r-(1-\eps_\circ)\dB^\mu(x))_+ \le \left(r-\dB^\mu(x)+\frac{\eps_\circ}{1-\eps_\circ} \dB^\lambda(x)\right)_+ \le \big((1+2\eps_\circ)r-\dB^\mu(x)\big)_+ \le (2^{1/p}r-\dB^\mu(x))_+
\]
for $\eps_\circ$  small enough, depending only on $p$. 

The second part is a consequence of \Cref{lem:lengthcomp}. Indeed, recalling the validity of \eqref{eq:dichotomy2}, the distance to $\ccB$ (when nonzero) is achieved along curves fully contained inside $B_{R_\Lambda+\Lambda/2}(z_\Lambda)\cap \cG(\delta_\circ) \setminus \ccB$ (recall \Cref{lem:freeannulus}). So, we can apply \Cref{lem:lengthcomp} to deduce that $\pi_\lambda$ is $(1+\eps_\circ)$-Lipschitz near the support of minimizing curves. Since the intrinsic distance is always $1$-Lipschitz, the result follows. 
\end{proof}

\subsection{Consequences of stability}
Recall that, thanks to \Cref{prop:2W}, we can assume that \eqref{eq:dichotomy2} holds. We now show some 
 first consequences of stability.

\begin{lem}
\label{lem:comp_Joaq}
 Let $\delta_\circ > 0$ and $\Lambda \ge 64$, and let $\ccB$ and $\cS_8$ be as in \eqref{eq:cBdefAC}. Then, for any $\lambda\in (-1, 1)$ we have
 \[
  C|\ccB|\ge \cH^2(\Sigma_\lambda\cap \{0< \dB^\lambda < 2\}) \ge c\int_{\ccB} |\cA(u)|^2 |\nabla u|^2 \,dx \ge  c'\delta_\circ  |\cS_8| \ge c'\delta_\circ  |\ccB|,
 \]
 where $C$, $c,$ and $c'$ are positive universal constants.
\end{lem}
\begin{proof}
The third and fourth inequalities are from \Cref{lem:comp_Joaq0}. 

    For the second one, we apply the stability inequality \eqref{eq:stab} with $\zeta(x) = (2-\dB^\lambda(x))_+$ (recall Remark~\ref{rmk:test fct}), which is compactly supported in $B_{R_\Lambda+\Lambda/2}(z_\Lambda)$ by \Cref{lem:freeannulus}. Then, thanks to Lemmas~\ref{lem:Modica} and \ref{lem:comp-dist}  we get
    \[
    4\int_{\ccB} |\cA(u)|^2|\nabla u|^2 \,dx \le \int_{\R^3} |\cA(u)|^2|\nabla u|^2 \zeta^2 \,dx \le \int_{\{0 < \dB^\lambda < 2\}} |\nabla u|^2 |\nabla \zeta|^2 \,dx \le C|\{0 < \dB^\lambda < 2\}\cap \{|u|< 1\}|.
    \]
  Observing that $\cH^2(\Sigma_\lambda\cap \{0 < \dB^\lambda < 2\})$ is comparable to $|\{0 < \dB^\lambda < 2\}\cap \{|u|< 1\}|$ (by the curvature estimates on the level sets and lower bound on $|\nabla u|$ from \Cref{prop:good}), we get the second inequality.
  Finally, since $|\{0 < \dB^\lambda < 2\}\cap \{|u|< 1\}|\le |\cS_{16}|\le C|\cS_{2}|\le C|\ccB|$, also the first inequality follows. 
\end{proof}

We now start bounding level sets near $\ccB$:

\begin{lem}[Stability near $\ccB$, intrinsic]
\label{lem:stab-on-bad}
Let $\lambda\in (-1,1)$, $\delta_\circ > 0$, $\Lambda\geq 64$, and $\ccB = \ccB(\delta_\circ, \Lambda)$ as in \eqref{eq:cBdefAC}. Consider the following set of ``indicator functions'' on $\mathbb N$:
\begin{equation}\label{XiLambda}
\Xi(\Lambda) := 
\bigg\{
\xi : \N  \to \{0,1\} : \mbox{such that}\quad \{\xi =1\} \subset  \big[\tfrac 1 4 \log_2 \Lambda,\log_2 \Lambda-5 \big] \qquad \mbox{and}\qquad \sum_{k\in \mathbb N} \xi(k) =\ceils{\tfrac 14  \log_2 \Lambda}
\bigg\}.
\end{equation}
Then
\begin{equation}
\label{eq:bound bad stability}
\cH^2\big(\Sigma_\lambda \cap \{0<\dB^\lambda<2\}\big)
\leq \frac{C}{\delta_\circ|\log \Lambda|^2} \min_{\xi \in \Xi(\Lambda)} 
\sum_{k} \xi(k)
    \frac{\cH^2(\Sigma_\lambda \cap \{2^k\leq \dB^{\lambda}<2^{k+1}\})}{2^{2k}},
\end{equation}
for a universal constant $C$.
\end{lem}

\begin{proof}
Let $\psi:\R\to [0,1]$ be a smooth nonincreasing function satisfying  $\psi(t) =1$ for  $t\le 1$, $\psi(t) =0$ for $t\ge2$, and $|\psi'|\le 2$. Given $\xi \in \Xi(\Lambda)$ fixed, we consider the stability inequality \eqref{eq:stab} with
\[
\zeta(x) := \frac{1}{\sum_k\xi(k)}\sum_k  \xi(k)\psi(2^{-k}\dB^\lambda(x))
\]
(recall Remark~\ref{rmk:test fct}).
Notice that $\zeta$ is supported in $\{\dB^\lambda \le \Lambda/16\}\subset B_{R_\Lambda+\Lambda/8}(z_\Lambda)$ (by \Cref{lem:freeannulus} and \eqref{eq:dichotomy2}), and it is constantly equal to 1 on $\{\dB^\lambda < \Lambda^{1/4}\}$. 

Now, thanks to \Cref{lem:comp-dist}, the right-hand side of the stability inequality \eqref{eq:stab} can be bounded by
\begin{multline*}
\int_{\R^3}|\nabla \zeta|^2|\nabla u|^2 \,dx
\le 
\int_{\{|u|< 1\}}|\nabla \zeta|^2 \,dx
\le \int_{B_{R_\Lambda+\Lambda/4}(z_\Lambda)\setminus \ccB}
|\nabla \zeta|^2 \,dx
\\
\le \frac{(1+\eps_\circ)^2}{(\sum_k\xi(k))^2}\sum_k\frac{\xi(k)}{2^{2k}} \int_{B_{R_\Lambda+\Lambda/4}(z_\Lambda)\setminus \ccB} 
|\psi'(2^{-k}\dB^\lambda(x))|^2 \,dx.
\end{multline*}
(Here we used that, for each $t\ge 0$, $\psi'(2^{-k}t)$ is non-zero for a single $k=k(t)\in \N$.)  

Recalling \Cref{prop:good},
we now consider adapted coordinates   
$(y,t)\in \Sigma_\lambda \times [-1, 1] \longleftrightarrow x=n_y(t) \in \Sigma_t$ (recall \eqref{eq:nxt}), so that $dx\leq (1+C \eta_\circ)\,dt\,d\cH^{2}_y$. Thus
\[
\begin{split}
\int_{B_{R_\Lambda+\Lambda/4}(z_\Lambda)\setminus \ccB}|\psi'(2^{-k}\dB^\lambda(x))|^2 \,dx & \le \big(1+C\eta_\circ) \int_{-1}^1 \int_{\Sigma_\lambda\cap (B_{R_\Lambda+\Lambda/4}(z_\Lambda)\setminus \ccB)}|\psi'(2^{-k}\dB^\lambda(y))|^2d\cH^2_y \, dt \\
& \le 4(1+C\eta_\circ) \cH^2\left( \Sigma_\lambda\cap \{2^k\le \dB^\lambda < 2^{k+1} \}\right),
\end{split}
\]
where we used that $|\psi'|\le 2$. Thus, for $\delta_\circ$ universally small enough (so that both $\eps_\circ$ and $\eta_\circ$ are small), we get
\[
\int_{\R^3}|\nabla \zeta|^2|\nabla u|^2 \,dx
\le \frac{C}{|\log\Lambda|^2} 
\sum_k \xi(k)\frac{\cH^2\left( \Sigma_\lambda\cap \{2^k\le \dB^\lambda < 2^{k+1} \}\right)}{2^{2k}}.
\]
Combining this estimate with \eqref{eq:stab} and \Cref{lem:comp_Joaq}, the result follows. 
\end{proof}
 
Next, we prove a doubling property:

\begin{lem}[Doubling]
\label{lem:doubling-FBAC}
Given $\delta_\circ>0$ small, there exists $\Lambda_0\ge 64$, depending only on $\delta_\circ$, such that the following holds whenever $\Lambda\ge\Lambda_0$.

Let $\ccB = \ccB(\delta_\circ, \Lambda)$ be as in \eqref{eq:cBdefAC}, $\Xi(\Lambda)$ as in \eqref{XiLambda}, and fix $p\geq 16$. Then, for any given  $\lambda\in (-1,1)$ there exists  $r\in  (\Lambda^{1/4}, \Lambda/ 8 )$ such that the following two inequalities hold simultaneously:
\begin{equation}\label{eq_doubl1}
    \cH^2\left(\Sigma_\lambda \cap \set{0<\dB^\lambda < 2^{1/p}r}\right)
\leq 2\cH^2\left(\Sigma_\lambda \cap \set{0<\dB^\lambda < r}\right)
\end{equation}
and
\begin{equation}\label{eq_doubl2}
\frac{1}{|\log_2\Lambda|} \min_{\xi \in \Xi(\Lambda)}
\sum_{k} \xi(k)
    \frac{\cH^2(\Sigma_\lambda \cap \{2^k\le\dB^\lambda<2^{k+1}\})}{2^{2k}}
\leq  16\frac{\cH^2(\Sigma_\lambda \cap \set{0<\dB^\lambda < r})}{r^2}.
\end{equation}
\end{lem}

\begin{proof}
Fix $\lambda \in (-1,1)$ and define 
\[
\Theta(r) : = \cH^2(\Sigma_\lambda \cap \{0<\dB^\lambda<r\}).
\]
Note that, by the curvature estimates of the level sets and lower bound on $|\nabla u|$ from \Cref{prop:good}, $\Theta(r)$ is comparable to $|\{0 < \dB^\lambda < r\}\cap \{|u|< 1\}|$. Hence, 
recalling \eqref{eq:cBdefAC} and noticing that $|\{0 < \dB^\lambda< r\}\cap \{|u| < 1\}|\le |\cS_{r+2}|\le C |\cS_r|$, by the Euclidean cubic volume growth of balls and \Cref{lem:comp_Joaq} we get
\begin{equation}
\label{cubicgrowth}
\Theta(r) \le C|\cS_r | \le C|\cS_4 | r^3 \le C \delta_\circ^{-1} \Theta(2)r^3 \qquad \text{for all }4 \le r \le \Lambda/2.
\end{equation}
Recalling the definition of $\Xi(\Lambda)$ in \eqref{XiLambda}, we define 
 \[
K : = \N \cap \big[ \tfrac 1 4 \log_2 \Lambda,\log_2 \Lambda-5 \big],\qquad a(k) : = \frac{\Theta(2^{k+1}) -\Theta(2^{k})}{2^{2k}}\quad \text{for} \quad k\in K. 
\]
Let $M$ be the median value of $a$ within $K$,
\[
M : = {\rm median}\big( \{a(k)\, :\, k\in K \}\big),
\]
and note that, from the definition of  $\Xi(\Lambda)$, we have
\begin{equation}\label{minmedian}
    \frac{1}{\log_2 \Lambda}\min_{\xi \in \Xi(\Lambda)}
\sum_{k} \xi(k)
    \frac{\cH^2(\Sigma_\lambda \cap \{2^k\le\dB^\lambda<2^{k+1}\})}{2^{2k}} =\frac{1}{\log_2 \Lambda}\min_{\xi \in \Xi(\Lambda)}
\sum_{k} \xi(k)a(k)< M
\end{equation}
Define $K' : = \big\{ k\in K \ : \  a(k)\ge M\big\}$
and notice that
\begin{equation}\label{thers}
M \le a(k) \le \frac{\Theta(2^{k+1})}{2^{2k}} \le 16 \frac{\Theta(r)}{r^2} \qquad \mbox{for all $k\in K'$ and $r\in [2^{k+1}, 2^{k+2}]$}.
\end{equation}
Hence, to show that \eqref{eq_doubl1} and \eqref{eq_doubl2} hold simultaneously at some scale $r\in \big(\Lambda^{1/4} , \Lambda/8\big)$, we only need to find an   $r$ as in \eqref{thers} for which \eqref{eq_doubl1} holds. 

To prove it, we will consider $r$ as in \eqref{thers} of the form  $r=2^{\ell/p}$ for $\ell \in \mathbb N$. So, we define  
\[
L : = \{ \ell\in \N \ : \ \ell/p\in [k+1, k+2) \quad \mbox{for some } k \in K'\},
\]
and we notice that 
\begin{equation}
\label{eq:L p}
\# L \ge p\,\# K' \ge\frac p 2 \# K \ge \frac p 4 \log_2 \Lambda.
\end{equation}
To conclude the proof, we claim that there exists $\ell\in L$ such that
$
\Theta(2^{(\ell+1)/p})\le 2 \Theta(2^{\ell/p}).$
Indeed, if the claim were false, 
we would have that $\Theta(2^{(\ell+1)/p})> 2 \Theta(2^{\ell/p})$ for all $\ell \in L$.
Thus, since $\ell \mapsto \Theta(2^{\ell/p})$ is nondecreasing,
setting $\ell_{*} := p (\ceils{\log_2 \Lambda}-4)$ we would get
\[
\Theta(2^{\ell_*/p}) > 2^{\# L} \Theta(2) > (2^{\ell_*/p})^{p/4}  \Theta(2) \geq  (2^{\ell_*/p})^{4}  \Theta(2),
\]
where we used \eqref{eq:L p} and that $p/4\ge 4$. This quartic growth contradicts the cubic growth bound in \eqref{cubicgrowth} if $2^{\ell_*/p}\sim \Lambda$ is large enough, depending only on $\delta_\circ$, so the claim holds.
\end{proof}

 \subsection{Integrated Gauss--Bonnet result} 
 Our next result is an estimate on the areas of sublevel sets of the distance to a compact set. Here we crucially use the fact that we consider 2-dimensional surfaces, since we exploit  Gauss--Bonnet on level sets of the distance. Our proof is inspired by a classical argument of Pogorelov \cite{Pog81}, but requires a much more refined analysis due to potential singularities of the distance function. We recall that the distance function to a set is always semiconcave (namely, in any chart, it can be written as the sum of a concave and a smooth function), see \cite{MM03}, therefore its distributional Riemannian Hessian is a measure whose singular part is negative definite.

\begin{lem}
\label{lem:GB-rigor}
Let $\Sigma\subset \R^3$ be a smooth embedded $2$-dimensional Riemannian surface,  $\cK\Subset \Sigma$, and $d_\cK:={\rm dist}_\Sigma(\cdot,\cK)$, where ${\rm dist}_\Sigma$ is the intrinsic distance to $\cK$ on $\Sigma$. 
Then, for a.e. $r_1,r_2>0$ with $r_1 < r_2$ and 
$\{r_1 < d_\cK < r_2\}\Subset \Sigma$ (namely,
$\{r_1 < d_\cK < r_2\}$ is compactly contained in $ \Sigma$), we have
\[
\cH^2(\{r_1<d_\cK<r_2\})\leq
\int_{r_1}^{r_2}\hspace{-1mm}\int_{r_1}^s\hspace{-1mm}\hspace{-0mm}\int_{\{ r_1<d_\cK<t\}}\hspace{-4mm}(K_\Sigma)_-\,d\cH^2\,dt\, ds
+(r_2-r_1)\cH^1(\{d_\cK=r_1\})+\frac{(r_2-r_1)^2}{2}  \int_{\{ d_\cK= r_1\}} \hspace{-1mm}(\Delta d_\cK)_+\,d\cH^1,
\]
where $(\Delta d_\cK)_+$ denotes the positive part of the (Riemannian) Laplacian of the semiconcave function $d_\cK$ , 
and $(K_\Sigma)_-$ is the negative part of the  Gauss curvature. 
\end{lem}
\begin{proof}

Throughout the proof, all the differential operators are the Riemannian ones on $\Sigma$. 

We start with the following considerations: 
\begin{itemize}
    \item[(i)] We observe that $d_\cK (x) = \min_{y\in \cK} d(y,x)$ is a semiconcave function. Hence, since the surface is two-dimensional, the set $\cC\subset (0,\infty)$ of critical values (i.e., images by $d_\cK$ of points where $0$ belongs to the super-differential) is closed and has measure zero (see, for instance, \cite{RZ2020}). 

    \item[(ii)] Let us fix an arbitrary  regular value $\bar r_1\in (0,\infty)\setminus \cC$, and let $\eta_\circ>0$ be  such that $(\bar r_1-\eta_\circ, \bar r_1+\eta_\circ)\cap \cC = \varnothing$. In particular, for all $r_1\in (\bar r_1-\eta_\circ, \bar r_1+\eta_\circ)$ the level set $\{d_\cK=r_1\}$ is a union of Lipchitz closed curves and, for any given $r_2>\bar r_1 +\eta_\circ$,  $\{r_1 < d_\cK < r_2\}$   has the same number of connected components for all $r_1$ in the interval. 
    
     \item[(iii)]  Up to reducing to a connected component of $\{\bar r_1 < d_\cK < r_2\}$, we may (and do) assume $\{r_1 < d_\cK < r_2\}$ is connected for all $r_1\in (\bar r_1-\eta_\circ, \bar r_1+\eta_\circ)$. 
    
    \item[(iv)] Also, up to making $r_2$ smaller and taking a limit in the inequality to be shown, we may assume that $\{d_\cK > r_2\}\neq \varnothing$.  
\end{itemize}   
    
\medskip 

 Since $|\nabla d_\cK|=1$ a.e., by the coarea formula we have 
$$
\cH^2(\{r_1<d_\cK<r_2\})=\int_{r_1}^{r_2}\cH^1(\{d_\cK=s\})\,ds.
$$
We   observe that, for a.e. $s>0$,  $\nabla d_\cK$  is $\cH^1$-a.e. equal to the outer normal of $\{d_\cK<s\}$  on the level set $\{d_\cK=s\}$.  Hence,
$$
\cH^1(\{d_\cK=s\})-\cH^1(\{d_\cK=r_1\})
=\int_{\partial \{r_1<d_\cK<s\}} \nu\cdot \nabla d_\cK\,d\cH^1=\int_{\{r_1<d_\cK<s\}} \Delta d_\cK\, d\cH^2
$$
where $\Delta d_\cK$ is the distributional Laplacian. Recalling that the distance function is always locally semiconcave, $\Delta d_\cK$ is a locally finite measure whose positive part has bounded density with respect to $\cH^2$. This implies that, for a.e. $r_1<r_2$,
$$
\cH^2(\{r_1<d_\cK<r_2\})=(r_2-r_1)\cH^1(\{d_\cK=r_1\})+\int_{r_1}^{r_2}\int_{\{r_1<d_\cK<s\}} \Delta d_\cK\, d\cH^2\,ds.
$$

Now, given a smooth function $\varphi:\Sigma\to \R$, consider the expression
\begin{equation}
\label{eq:K GB 0}
\int_{\{r_1<\varphi<s\}} |\nabla\varphi|\,{\rm div}\biggl(\frac{\nabla\varphi}{|\nabla\varphi|}\biggr)\,d\cH^2=
\int_{r_1}^s\int_{\{\varphi=t\}}{\rm div}\biggl(\frac{\nabla\varphi}{|\nabla\varphi|}\biggr)\,d\cH^1\,dt,
\end{equation}
where the equality follows by the coarea formula.
We  observe that, by Sard's Theorem, for a.e. $t$ the level set $\{\varphi=t\}$ is a smooth curve without critical points of $\varphi$, and ${\rm div}\bigl(\frac{\nabla\varphi}{|\nabla\varphi|}\bigr)$ corresponds to its geodesic curvature.   Then, by Gauss--Bonnet, for a.e. $r_1$ and $t$ it holds
\begin{equation}
\label{eq:K GB}
\int_{\{\varphi=t\}}{\rm div}\biggl(\frac{\nabla\varphi}{|\nabla\varphi|}\biggr)\,d\cH^1 = 
-\int_{\{  r_1<\varphi<t\}}K_\Sigma\,d\cH^2+
\int_{\{\varphi=  r_1\}}{\rm div}\biggl(\frac{\nabla\varphi}{|\nabla\varphi|}\biggr)\,d\cH^1
+2\pi \chi(\{  r_1<\varphi<t\}),
\end{equation}
 where $\chi(S)$ denotes the Euler characteristic of the complete orientable surface $\overline S$.

We are now going to apply this identity (integrated for $t$ between $r_1$ and $s$)
to a smoothed version of the distance function, and then let the regularization parameter go to zero.
More precisely, 
fix a compact neighborhood of $\{r_1<d_\cK<s\}$ and cover it with a finite atlas \( \{(U_m, \phi_m)\}_{m=1}^N \).
Then consider a partition of unity \( \{\psi_m\}_{m=1}^N \) subordinate to this atlas, and fix \( \rho: \mathbb{R}^n \to [0, \infty) \) a smooth compactly supported mollifier.
Then, for \( \eta > 0 \), set \( \rho_\eta(z) := \eta^{-n} \rho(z / \eta) \), we define the following local smoothing operator for functions $f:\Sigma\to \R$:
\[
f \mapsto [f]_\eta(x) := \sum_{m=1}^N \psi_m(x) \big(( f \circ \phi_m^{-1}) * \rho_\eta\big)(\phi_m(x)).
\]
While this regularization does not commute with derivatives, it does in the limit. More precisely we recall that, in local coordinates, the Hessian of a function $f$ is given by 
$
(D^2f)_{ij} = \partial^2_{ij} f - \Gamma_{ij}^k \partial_kf,
$
where $\Gamma_{ij}^k$ are the Christoffel symbols of the Riemannian metric \( g \) on $\Sigma$, and we adopt the Einstein convention of summation over repeated indices. Hence, when locally mollifying in charts as done above, for any given smooth $h: U_m \to \R$,  $x\in \phi_m (U_m)$ and $\eta>0$ sufficiently small, we have
\[
(D^2(h \ast \rho_\eta))_{ij}(x) - (D^2h)_{ij} * \rho_\eta(x) = \int_{\phi_m(U_m)} \left[ \Gamma_{ij}^k(y)  - \Gamma_{ij}^k(x) \right] \partial_k h(y) \rho_\eta(x - y) \, dy.
\]
If we define $\varphi_\eta:=[d_\cK]_\eta$, applying the  formula above to $h=d_\cK \circ \phi_m^{-1}$,  $m=1,\ldots,N$, since $d_\cK$ is $1$-Lipschitz it follows easily that
\begin{equation}
\label{eq:D2 convolve}
(D^2\varphi_\eta)_{ij}=[(D^2d_\cK)_{ij}]_\eta + O(\eta).
\end{equation}
Now, since $\varphi_\eta \to d_\cK$ locally uniformly and $\Delta \varphi_\eta \rightharpoonup^* \Delta d_\cK$ in the sense of measures, using \eqref{eq:K GB 0}-\eqref{eq:K GB} with $\varphi=\varphi_\eta$, for a.e. $r_1<s$ we have
\begin{align*}
\int_{\{r_1<d_\cK<s\}} \Delta d_\cK\,d\cH^2
&=\lim_{\eta \to 0} \int_{\{r_1<\varphi_\eta<s\}} \Delta \varphi_\eta \,d\cH^2\\
&\hspace{-0.5cm}\leq \limsup_{\eta \to 0} \int_{\{r_1<\varphi_\eta<s\}} |\nabla\varphi_\eta|\,{\rm div}\biggl(\frac{\nabla\varphi_\eta}{|\nabla\varphi_\eta|}\biggr) \,d\cH^2
+\limsup_{\eta \to 0}\int_{\{r_1<\varphi_\eta<s\}}  \bigg\langle D^2\varphi_\eta \cdot \frac{\nabla\varphi_\eta}{|\nabla\varphi_\eta|},\frac{\nabla\varphi_\eta}{|\nabla\varphi_\eta|}\bigg\rangle  \,d\cH^2\\
&\hspace{-0.5cm}\le-\int_{r_1}^s \int_{\{ r_1<d_\cK<t\}}K_\Sigma\,d\cH^2\,dt
+ (s-r_1) \limsup_{\eta \to 0}\int_{\{\varphi_\eta= r_1\}}{\rm div}\biggl(\frac{\nabla\varphi_\eta}{|\nabla\varphi_\eta|}\biggr)\,d\cH^1\,\\
&\hspace{-0.5cm}\quad+ 2\pi \limsup_{\eta \to 0}\int_{r_1}^s \chi(\{ r_1<\varphi_\eta<t\})\,dt
+\limsup_{\eta \to 0}\int_{\{r_1<\varphi_\eta<s\}}  \bigg\langle D^2\varphi_\eta \cdot \frac{\nabla\varphi_\eta}{|\nabla\varphi_\eta|},\frac{\nabla\varphi_\eta}{|\nabla\varphi_\eta|}\bigg\rangle   \,d\cH^2.
\end{align*}

We observe that, by the initial assumptions (i)-(ii) and Fatou's lemma, we have:
\[
\limsup_{\eta \to 0}\int_{r_1}^s \chi(\{ r_1<\varphi_\eta<t\})\,dt \le \int_{r_1}^s \chi(\{ r_1<d_\cK<t\})\,dt.
\]
Now, since $D_{r_1, r_2} := \{r_1 < d_\cK < r_2\}$ is connected by (iii) and $\{d_\cK > r_2\}\neq \varnothing$ by (iv), we must have $\chi(D_{r_1, r_2}) \le 0$ (recall  that $\chi = 2-2g-b$ for connected surfaces with $b$ boundary components and genus $g$, with $g \ge0$ and $b\ge2$ in this case). Furthermore, we must also have $\chi(D_{r_1, t}) \le 0$ for a.e. $t\in (r_1, r_2)$. Indeed, otherwise, there would be a connected component $U\subset D_{r_1, t}$ with $\chi(U) > 0$, i.e., with $b \le 1$. But, since $D_{r_1, r_2}$ is connected, $U$ must always have two boundary components (one of them on $\{d_\cK = r_1\}$).  In all, $\chi(D_{r_1, t}) \le 0$ for a.e. $t\in (r_1, r_2)$. 

  Also, as discussed before,  the semiconcavity of $d_\cK$ implies that $D^2d_\cK$ is a matrix-valued measure whose singular part is negative. Hence,
 if $D^2_ad_\cK$ denotes the absolutely continuous part of the Hessian, recalling \eqref{eq:D2 convolve} we get 
$$
\limsup_{\eta \to 0}\int_{\{r_1<\varphi_\eta<s\}}  \bigg\langle D^2\varphi_\eta \cdot \frac{\nabla\varphi_\eta}{|\nabla\varphi_\eta|},\frac{\nabla\varphi_\eta}{|\nabla\varphi_\eta|}\bigg\rangle  \,d\cH^2
\leq \limsup_{\eta \to 0}\int_{\{r_1<\varphi_\eta<s\}}   [(D^2_ad_\cK)_{ij}]_\eta \cdot \frac{\partial_i\varphi_\eta}{|\nabla\varphi_\eta|}\frac{\partial_j\varphi_\eta}{|\nabla\varphi_\eta|}  \,d\cH^2.
$$
Since $[(D^2_ad_\cK)_{ij}]_\eta \to (D^2_ad_\cK)_{ij}$ in $L^1_{\rm loc}$ (because $(D^2_ad_\cK)_{ij}$ is a locally integrable function) and $\nabla\varphi_\eta \to \nabla d_\cK$ $\cH^2$-a.e., by dominated convergence we deduce that the limsup in the right-hand side above is equal to
$$
\int_{\{r_1<d_\cK<s\}}  \big\langle D^2_ad_\cK\cdot \nabla d_\cK,\nabla d_\cK\big\rangle  \,d\cH^2.
$$
Also, because $D^2_ad_\cK\cdot \nabla d_\cK=(\frac12 \nabla |\nabla d_\cK|^2)_a=0$ a.e. (where the subscript $a$ denotes the absolutely continuous part, and the derivative is zero because $|\nabla d_\cK|^2=1$ a.e.), the integral above is zero.
Hence, we proved that for a.e. $r_1$,
\begin{align*}
\int_{\{r_1<d_\cK<s\}} \Delta d_\cK \,d\cH^2
&\leq -\int_{r_1}^s \int_{\{ r_1<d_\cK<t\}}K_\Sigma\,d\cH^2\,d\tau\,dt 
+  (s-r_1) \limsup_{\eta \to 0}  \int_{\{\varphi_\eta=  r_1\}}{\rm div}\biggl(\frac{\nabla\varphi_\eta}{|\nabla\varphi_\eta|}\biggr)\,d\cH^1.
\end{align*}

Finally, arguing exactly as before, we have
$$
|\nabla\varphi_\eta|\, {\rm div}\biggl(\frac{\nabla\varphi_\eta}{|\nabla\varphi_\eta|}\biggr)
=\bigg\langle D^2\varphi_\eta \cdot \frac{\nabla^\perp\varphi_\eta}{|\nabla\varphi_\eta|},\frac{\nabla^\perp\varphi_\eta}{|\nabla\varphi_\eta|}\bigg\rangle  \leq 
[(D^2_ad_\cK)_{ij}]_\eta \cdot \frac{(\nabla^\perp\varphi_\eta)^i}{|\nabla\varphi_\eta|}\frac{(\nabla^\perp\varphi_\eta)^j}{|\nabla\varphi_\eta|}+O(\eta) ,
$$
and therefore, for a.e. $r_1\in (\bar r_1- \eta_\circ,\bar r_1+ \eta_\circ)$,
\[
\limsup_{\eta \to 0}  \int_{\{\varphi_\eta=  r_1\}}{\rm div}\biggl(\frac{\nabla\varphi_\eta}{|\nabla\varphi_\eta|}\biggr)\,d\cH^1\, \le  \int_{\{d_\cK=  r_1\}}\big\langle D^2_ad_\cK \cdot \nabla^\perp d_\cK,\nabla^\perp d_\cK\big\rangle\,d\cH^1 .
\]
 
Noticing that
$$
(\Delta d_\cK)_a=\big\langle D^2_ad_\cK \cdot \nabla^\perp d_\cK,\nabla^\perp d_\cK\big\rangle+\big\langle D_a^2d_\cK \cdot \nabla d_\cK,\nabla d_\cK\big\rangle=\big\langle D^2_ad_\cK \cdot \nabla^\perp d_\cK,\nabla^\perp d_\cK\big\rangle
$$
(recall that $\big\langle D_a^2d_\cK \cdot \nabla d_\cK,\nabla d_\cK\big\rangle= 0$ a.e.), the result follows, noticing that $(\Delta d_\cK)_a \le (\Delta d_\cK)_+$ because of the semiconcavity of the distance.
 \end{proof}

\subsection{Proof of \Cref{thm:FBAC-main} and its corollaries} \label{proofs:FBAC} We can now proceed with the proof of our main result that, as explained before, directly implies \Cref{thm:FBAC-main}.

\begin{prop}
\label{prop:Xisempty}
For every $\delta_\circ \in (0,1)$, the set $\cX(\delta_\circ)$ is empty. 
\end{prop}

\begin{proof}
Since the sets $\cX(\delta_\circ)$ are monotonically decreasing, it suffices to prove the result for all $\delta_\circ$ sufficiently small.

So, assume by contradiction that $\cX(\delta_\circ)\neq \varnothing$, then for $\Lambda > 64$ we  construct the set $\ccB\neq\varnothing$ as in \eqref{eq:cBdefAC}. Also, by choosing $\Lambda$ sufficiently large (depending on $\delta_\circ)$, we can assume \eqref{eq:dichotomy2} holds (recall \Cref{prop:2W}).

Now, let $\Sigma$ denote one of the level sets of $u$ inside $B_{R_\Lambda+\Lambda}(z_\Lambda)\cap \cG(\delta_\circ) \setminus \ccB$ (note that this is a smooth surface), and apply \Cref{lem:GB-rigor} with $\cK=\ccB$,
$r_1\in (1/4, 2)$,
and $r_2=r < \Lambda/8$
(recall $\ccB$ surrounds $\cW(\delta_\circ)$ and separates it from $\cG(\delta_\circ)$ on $\Sigma$).
For a.e. $r_1<r$, since $\Delta\,\dB \leq C$ on $\{\dB = r_1\}$ by the semiconcavity of the distance (see for instance \cite{MM03}), we have
\begin{equation}
\label{eq:boundtousenow}
\begin{split}
    \cH^2(\Sigma\cap \{r_1<\dB<r\})& \leq 
\int_{r_1}^{r}\int_{r_1}^s
\int_{\Sigma\cap \{ r_1<\dB<t\}}|K_\Sigma|\,d\cH^2 \,dt\, ds   
+C{r^2}  \cH^1(\Sigma\cap \{\dB = r_1\}).
\end{split}
\end{equation}
We emphasize that this holds regardless of whether $\Sigma$ is connected.

Since $\int_{r_1}^r \int_{r_1}^s
        \mathbbm{1}_{ \{\dB<t\}}
    \,dt\,ds = \tfrac12 (r-\dB)^2_+$ for $r_1<\dB$ and $|K_\Sigma|\le \tfrac12 |A_\Sigma|^2$ (where $|A_\Sigma|^2$ denotes the sum of squares of principal curvatures), using Fubini we get
\begin{align*}
 \int_{r_1}^r\int_{r_1}^s \int_{\Sigma\cap \{  r_1<\dB<t\}}
    |K_\Sigma|
\,\,d\cH^2\,dt\,ds
&\leq\frac12\int_{r_1}^r\int_{r_1}^s \int_{\Sigma\cap \{r_1 <\dB<r\}}
    |A_\Sigma|^2\,\mathbbm{1}_{ \{ \dB<t\}} 
\,d\cH^2\,dt\,ds\\
&=\frac14 \int_{\Sigma\cap \{r_1<\dB<r\}}
    |A_\Sigma|^2 (r-\dB)_+^2
\,d\cH^2.
\end{align*}
Thanks to this bound, averaging \eqref{eq:boundtousenow} over $r_1\in [1,2]$, since $\int_{1}^{2}
    \cH^1(\Sigma \cap \set{\dB= r_1})
\,d r_1
= \cH^2(\Sigma \cap \set{1<\dB<2})$
we obtain
\[
\cH^2(\Sigma\cap \{2 <\dB<r\})
\leq
\frac{1}{4}\int_{\Sigma \setminus \ccB} |A_\Sigma|^2 (r-\dB)_+^2 \,d\cH^2
+Cr^2 \cH^2(\Sigma\cap \{1<\dB<2\}) 
\]
for a.e. $r \in(1,\Lambda/8)$,
which also implies (up to replacing $C$ with $C+1$ in the right-hand side)
\[
\cH^2(\Sigma\cap \{0<\dB<r\})
\leq
\frac{1}{4}\int_{\Sigma \setminus \ccB} |A_\Sigma|^2 (r-\dB)_+^2 \,d\cH^2
+Cr^2 \cH^2(\Sigma\cap \{0<\dB<2\}).
\]
Now choose $\nu \in (-1,1)$ such that
$$
\int_{\Sigma_\nu}
    |A_{\Sigma_\nu}|^2 \,
    \big(r-\dB^\nu\big)_+^2
\,d\cH^2
\leq 
\fint_{-1}^{1}\int_{\Sigma_\lambda}
    |A_{\Sigma_\lambda}|^2\,
    \big(r-\dB^\lambda\big)_+^2
\,d\cH^2\,d\lambda,
$$
and apply the bound above to the level set $\Sigma=\Sigma_\nu$.
Then,
thanks to \Cref{lem:stab-on-bad}  we get
\begin{multline}\label{eq:change-levels-now}
    \frac{\cH^2(\Sigma_\nu \cap \set{0<\dB^\nu<r})}{r^2}
\leq \frac{1}{4r^2} \fint_{-1}^{1}\int_{\Sigma_\lambda}
    |A_{\Sigma_\lambda}|^2\,
    \big(r-\dB^\lambda\big)_+^2
\,d\cH^2\,d\lambda+C \cH^2(\Sigma_\nu\cap \{0<\dB^\nu<2\}) \\
\leq \frac{1}{4r^2} \fint_{-1}^{1}\int_{\Sigma_\lambda}
    |A_{\Sigma_\lambda}|^2\,
    \big(r-\dB^\lambda\big)_+^2
\,d\cH^2\,d\lambda +\frac{C}{\delta_\circ|\log \Lambda|^2} \min_{\xi \in \Xi(\Lambda)} 
\sum_{k} \xi(k)
    \frac{\cH^2(\Sigma_\nu \cap \{2^k\leq \dB^{\nu}<2^{k+1}\})}{2^{2k}}.
\end{multline}
Note now that, thanks to \Cref{lem:comp-dist} with $p=16$, the coarea formula, and \Cref{prop:good}, for any $\mu\in(-1,1)$ it holds
\begin{align*}
\int_{-1}^{1}\int_{\Sigma_\lambda\setminus\ccB}
    |A_{\Sigma_\lambda}|^2 \,
    \big(r-\dB^\lambda\big)_+^2
\,d\cH^2\,d\lambda
&\leq 
\int_{-1}^{1}
\int_{\Sigma_\lambda\setminus\ccB}
    |A_{\Sigma_\lambda}|^2 \,
    \big(2^{1/p}r-\dB^\mu\big)_+^2
\,d\cH^2\,d\lambda\\
&\leq (1+C\eta_\circ) \int_{ \set{|u|<1}\setminus\ccB }
    |A(u)|^2\,
    |\nabla u|^2\,
    \big(2^{1/p}r-\dB^\mu\big)_+^2
\,dx,
\end{align*}
for some universal $C$. 
Next, we apply the stability inequality \eqref{eq:stab}  with test function $\big(2^{1/p}r-\dB^\mu(x)\big)_+$ (which is admissible for $r\le  \Lambda/8$, due to \Cref{lem:freeannulus} for $r\le  \Lambda/8$ and \eqref{eq:dichotomy2}, recall also \Cref{rmk:test fct}), giving
\begin{equation}\label{eq:stab-4}
\begin{split}
\int_{-1}^{1}\int_{\Sigma_\lambda}
    |A_{\Sigma_\lambda}|^2\,
    \big(r-\dB^\lambda\big)_+^2
\,d\cH^2\,d\lambda
&\leq (1+C\eta_\circ)
\int_{ \set{|u|<1} }
    |\nabla u|^2 
    \abs{\nabla\big(2^{1/p}r-\dB^\mu\big)_+}^2
\,dx\\
&\leq (1+C(\eta_\circ+\eps_\circ)) \, \left|\{|u|< 1\}\cap \{0 < \dB^\mu < 2^{1/p} r\}\right|\\
&\leq 2(1+C(\eta_\circ+\eps_\circ))\cH^2(\Sigma_\mu \cap \{0<\dB^\mu<2^{1/p}r\} ),
\end{split}\end{equation}
for some universal $C$ (that can be different line to line). In the last inequality we have again used the flatness of level sets given by \Cref{prop:good}. 
 
Combining this bound with \eqref{eq:change-levels-now}, we obtain
\begin{equation} 
\begin{split}
    \frac{\cH^2(\Sigma_\nu\cap \set{0<\dB^\nu<r})}{r^2}
&\leq \frac{1+C(\eta_\circ+\eps_\circ)}{4}\cdot
    \frac{\cH^2(\Sigma_\nu \cap \set{0<\dB^\nu < 2^{1/p}r})}{r^2}
\\
&\qquad  +\frac{2C}{\delta_\circ|\log \Lambda|^2} \min_{\xi \in \Xi(\Lambda)} 
\sum_{k} \xi(k)
    \frac{\cH^2(\Sigma_\nu \cap \{2^k\leq \dB^{\nu}<2^{k+1}\})}{2^{2k}}.
\end{split}
\end{equation}
Recalling \Cref{lem:doubling-FBAC},
this implies the existence of $r\in (\Lambda^{1/4},\Lambda/8)$ such that
\begin{equation} 
\begin{split}
    \frac{\cH^2(\Sigma_\nu\cap \set{0<\dB^\nu<r})}{r^2}
&\leq \frac{1+C(\eta_\circ+\eps_\circ)}{2}\cdot
    \frac{\cH^2(\Sigma_\nu \cap \set{0<\dB^\nu < r})}{r^2}+\frac{C}{\delta_\circ |\log\Lambda|}
    \frac{\cH^2(\Sigma_\nu \cap \set{0<\dB^\nu < r})}{r^2}.
\end{split}
\end{equation}
Fixing $\eta_\circ$ and $\eps_\circ$ (and thus $\delta_\circ$) sufficiently small so that $C(\eta_\circ+\eps_\circ) < \tfrac14$, and then fixing $\Lambda$ sufficiently large so that $\frac{C}{\delta_\circ|\log\Lambda|}\le \frac18$ and \Cref{prop:2W} holds, we deduce $\cH^2(\Sigma_\nu\cap \set{0<\dB^\nu<r}) = 0$. This means that $\Sigma_\nu\cap \set{0<\dB^\nu<r} = \varnothing$ for some $\nu\in (-1, 1)$, and therefore for all $\nu \in (-1,1)$ (this follows, for instance, by \Cref{lem:lengthcomp}). 

Since $\Lambda$ (and therefore $r$) can be chosen arbitrarily large, we have shown that $\{|u|< 1\}$ has a bounded connected component $\cC$ that is contained inside $\ccB\subset B_R(z)$. However this is impossible, as one can show, for instance, by taking a test function
$\zeta$ in \eqref{eq:stab}  such that  $\cC\subset \{\zeta \equiv 1\}$ and ${\rm supp}(\nabla \zeta)\subset \{|u| = 1\}$. This contradiction proves that $\cX(\delta_\circ)  = \varnothing$, as desired.
\end{proof}

Thanks to this last result, our main theorem follows immediately:

\begin{proof}[Proof of \Cref{thm:FBAC-main}]
Thanks to \Cref{prop:Xisempty},  the sets   $\cX(\delta_\circ) = \varnothing$ are empty for every $\delta_\circ>0$.
Recalling their definition, this implies that locally either $\mathcal A(u)\equiv 0$ or $\nabla u\equiv 0$. Thus, by unique continuation, either the solution is constant, or all level sets are flat and the solution is one-dimensional.
\end{proof}

Next, we want to prove \Cref{cor:DeG_main}. The following lemma will be useful:

\begin{lem} \label{lem:extrconccomp}
Let $n\ge 2$, and let $u$ be a global classical solution to the free boundary Allen--Cahn problem in $\R^n$ (see \Cref{defi:solutionsAC}). 
Given $V$ any connected component of $\{|u|<1\}$, there is a unique global classical solution to the free boundary Allen--Cahn problem  $\tilde u$  such that:\\
- $\tilde u=u$ in $V$;\\ - $\tilde u$ restricted to $\R^n\setminus \overline V$ takes values in $\{\pm 1\}$.
\end{lem}
\begin{proof}
    The only delicate part is to show that we can assign a constant value (either $+ 1$ or $-1$) to each connected component $W$ of $\R^n\setminus \overline V$ in a way that it agrees with $u$ on $\partial W$.
    Although intuitive, to rigorously justify it, we use that the two free boundaries $\Gamma_\pm : = \partial V \cap \{u = \pm 1\}$ are smooth submanifolds of $\R^n$ (in particular, they are oriented and embedded).     We want to show that, given a connected component $W$ of $\R^n\setminus \overline V$,  the boundary $\partial W$ is either fully contained in  $\Gamma_+$ (and then we assign $+1$ to $u$ in $W$) or in $\Gamma_-$ (and then we assign $-1$). 
    
    Assume by contradiction the existence of two points $p_\pm\in \Gamma_\pm$  such that $\hat p_\pm := p_\pm  - t \nu(p_\pm) \in W$ for  $t>0$ small enough,  where $\nu$ is the inward unit normal to $V$. 
    Since $W$ is open and connected, there is a smooth curve $\gamma_1$ joining $\hat p_+$ and $\hat p_-$ that does not intersect $\partial W$---in particular, it does not intersect $\Gamma_+$. 
    On the other hand, since $V$ is open and connected, there is another curve $\gamma_2$ contained in $V$ and joining the two points $\tilde p_\pm = p_\pm  + t \nu(p_\pm) \in V$ (for $t$ small) that does not intersect $\partial V$---in particular, it does not intersect $\Gamma_+$. But then the concatenation of $\gamma_1$ and $\gamma_2$ with the two segments $\tilde p_+ \hat p_+$ and  $\tilde p_- \hat p_-$ would give a closed curve intersecting $\Gamma_+$ exactly once  (notice that the segment $\tilde p_- p_-$  intersects $\Gamma_-$ and not $\Gamma_+$). However, by the invariance of the mod 2 self-intersection number (see, e.g., \cite[Chapter 2]{Guillemin-Pollack}), any closed curve has to intersect $\Gamma_+$ an even number of times (being homologous to zero in $\mathbb Z/2$-homology), a contradiction.
\end{proof}

The proof of \Cref{cor:DeG_main} now follows through rather standard arguments, which we sketch for the reader's convenience:
\begin{proof}[Proof of \Cref{cor:DeG_main}]
By \Cref{lem:extrconccomp} we can assume that $\{|u|<1\}$ is connected. Similarly to the proof of \Cref{cor:DeG_Bernoulli} (see Subsection~\ref{ssec:proofcorollaries}), thanks to the monotonicity assumption, the solution is stable and we have universal curvature estimates for the free boundary (see \Cref{prop:main2bis}), so the limits 
    \[
    \underline u (x_1,x_2, x_3) : = \lim_{x_4\to -\infty} u \qquad \mbox{and}\qquad \overline u (x_1,x_2, x_3) : =\lim_{x_4\to +\infty} u.
    \]
    are classical stable solutions in $\R^3$. Thus, by \Cref{thm:FBAC-main}, $\underline u$ and $\overline u$ depend only on one Euclidean variable. 
     
We now claim $u$ is an energy minimizer.
Indeed, since $u$ is monotone in the $x_4$ direction, there are three cases to consider (up to rotation in the first three variables and replacing $u(x',x_4)$ by $-u(x',-x_4)$, if needed):
    \begin{enumerate}[(i)]
    \item  $\underline  u \equiv -1$ and $\overline u \equiv +1 $;
    \item  $\underline  u \equiv -1$ and $\overline u = \overline h(x_1)$ for some $\overline{h}: \R\to [-1,1] $ not constant;
    \item  $\underline  u = \underline h(x_1)$ and $\overline u = \overline h(x_1)$ for some $\underline h, \overline h: \R\to [-1,1]$ not constant.
    \end{enumerate}
    Now, the assumptions that $\{|u|<1\}$ is connected and  $-1\le \underline  u\le u\le \overline u\le 1$ imply that $\{\underline  u=1\}$ cannot contain a slab $\{a\le  x_1\le b\}$ with $a\le b$ finite, as otherwise we would have that $u=1$ in such a slab, thus disconnecting $\{|u|<1\}$. 
    Symmetrically, there cannot be a slab  $\{a\le  x_1\le b\}$ where $\overline  u=-1$.

It follows that, in the previous possible scenarios:
either $\underline u$ is a minimizer (identically $-1$, or a 1D monotone solution)  or a maximum of two minimizers, and  $\overline u$ is either a minimizer (identically $+1$, or a 1D monotone solution)  or a minimum of two minimizers.
This ensures that $\underline u$ and $\overline u$  are, respectively, lower and upper barriers for minimizers. Therefore, since the family of translated graphs $\{x_5=u(x + te_4)\}_{t\in \R}$ foliates the region 
\[
\{ \underline u(x) \le x_5 \le \overline u(x)\}\subset \R^4 \times [-1,+1],
\]
via a standard foliation argument (see \cite[Proof of Theorem 1.3]{Jerison-Monneau}) it follows that $u$ must be an energy minimizer in every compact subset of $\R^4$, as claimed. 
Thanks to the energy minimality, we can apply \cite[Theorem 3]{Val06} to conclude that  $u$ is one-dimensional. 
\end{proof}

Finally, we provide the proof of \Cref{cor:curvest}, relying on the $C^{1,1}$ to $C^{2,\alpha}$ estimate established in \cite{An25}. Notably, the estimate in \cite{An25} is significantly more elementary than its Allen–Cahn counterpart in \cite{CM20, WW19}, as it does not need to account for sheet interactions.
\begin{proof}[Proof of \Cref{cor:curvest}]
We claim that
\begin{equation}\label{keyptcurvest}
   \sup_{B_{3/4}\cap \{|u_\eps|<1\}} \ep |D^2 u_\eps|\le C,
\end{equation}
with a universal constant~$C$.
Note that since $|u_\eps|\le 1$ and  $|\nabla u_\eps|=1/\ep$ on $\partial\{|u_\eps|<1\}$, the estimate~\eqref{keyptcurvest} provides universal curvature bounds for the free boundary and all level sets of $u_\eps$ within~$B_{3/4}$, for $\eps$ sufficiently small.

To prove~\eqref{keyptcurvest}, we argue by contradiction, combining the $C^{1,1}$-to-$C^{2,\alpha}$ estimates from~\cite{An25} with \Cref{thm:FBAC-main}. This approach is a standard scaling-compactness argument analogous to the curvature estimate proofs in~\cite{Chodosh19,WW19}.
Indeed, suppose---for the sake of contradiction---that there exists a sequence $u_k$ of classical stable critical points of $\cJ^0_{\eps_k}$ in $B_1$ for some $\eps_k\in (0, 1)$  such that 
    \[
       \sup_{B_1 \,\cap\, \{|u_k| < 1\}} 
           \eps_k|D^2 u_k(x)|\,(1 - |x|) 
       \;\ge\; k \;\longrightarrow\;\infty.
    \]
Let $x_k$ be a point where the maximum is attained and set
\[
h_k := \eps_k|D^2 u_k(x_k)|\big(1 - |x_k|\big) = \max_{x\in B_1\cap \overline{\{|u_k| < 1\}}} \eps_k|D^2 u_k(x)|\big(1 -|x|\big)\to \infty,\quad\text{as}\quad k \to \infty. 
\]
Let $d_k := \eps_k |D^2 u_k(x_k)|$ and $\rho_k = 1-|x_k|$, so that $h_k = d_k \rho_k$ and $d_k \to \infty$. Notice that, by \Cref{lem:global-curv-bd} applied to $u_{\eps_k}(\eps_k\, \cdot )$, we have $\eps_k d_k \le C$ with $C$  universal. 

Now, choose any sequence $\tau_k\downarrow 0$ such that $\tau_k h_k\to \infty$ and define 
\[
\widetilde  u_k(y) :=  u_k\bigg(x_k + \frac{y}{d_k}\bigg)\qquad\text{for}\quad y\in B_{\tau_k d_k \rho_k }. 
\]
Then  $\widetilde u_k$ is  a classical stable critical point of $\cJ_{\epk}^0$ in $B_{\tau_k d_k\rho_k}$, where $\epk = \eps_k d_k \le C$, with $0\in \overline{\{|\widetilde u_k| < 1\}}$ and $\epk |D^2 \widetilde u_k(0)| = 1$. Also, by definition of $h_k$, 
for $x=x_k + \frac{y}{d_k}\in \{u_k>0\}$  with $|y|<  \tau_k d_k \rho_k$ we have 
\[
\eps_k\left|D^2 u_k\left(x_k + \frac{ y}{d_k}\right)\right| \le \eps_k |D^2 u_k(x_k)| \frac{ 1-|x_k|  }{1-|x_k + y/d_k|} \le d_k \frac{\rho_k}{\rho_k -\tau_k\rho_k}.
\] 
Therefore, 
\[
\widetilde \ep_k |D^2 \widetilde u_k(y)| = \frac{\eps_k }{d_k} \abs{D^2 u_k\left(x_k + \frac{ y}{d_k}\right)}
\le \frac{1}{1- \tau_k} \qquad\text{for}\quad y\in B_{\tau_kd_k \rho_k} \cap \overline {\{\widetilde u_k > 0\}}. 
\]
By construction, the radius of the ball  $\tau_k d_k\rho_k = \tau_k h_k$ goes to infinity as $k \to \infty$.

We now distinguish two cases: 
\begin{enumerate}
    \item[(a)] If  $\lim_{k\to \infty } \widetilde \eps_k =0$,  then using the $C^{1,1}$-to-$C^{2,\alpha}$ estimates in \cite{An25}---similarly to \cite{CM20, WW19}---we obtain that the free boundaries of $\widetilde u_k$ converge (with local graphical $C^2$ convergence) to a complete stable minimal surface in the Euclidean space $\R^3$ with non-zero second fundamental form at the origin. This contradicts the classical classification of stable minimal surfaces in $\R^3$ (see \cite{DP79, FS80, Pog81}), stating that such surfaces must be flat.
    
    \item[(b)]  Otherwise, up to passing to a subsequence, we have $\widetilde\eps : = \lim_{k\to \infty }\widetilde \eps_k >0$. Then the functions $\widetilde u_k(\widetilde \eps\,\cdot\,)$ must converge---similar to the proof of \Cref{lem:reduction}---to a classical stable critical point of $\cJ_1^0$  in the whole  $\R^3$ with nonzero Hessian at the origin, contradicting \Cref{thm:FBAC-main}.
\end{enumerate}
This completes the proof.
\end{proof}

\begin{remark}
    With arguments similar to the ones above, one could also show the following Riemannian version of \Cref{cor:curvest}: \\
{\it Let $g$ be a Riemannian metric on the Euclidean ball $B_1 \subset \mathbb{R}^3$, and assume that $\|g\|_{C^2(B_1)}+\|g^{-1}\|_{C^2(B_1)} \leq M$. Let $u_\varepsilon: B_1 \to [-1,1]$ be a classical stable critical point of 
$$
\mathcal{J}_{g,  \eps}^0(u;B_1)=\int_{B_1}\left\{ \varepsilon \,g^{ij}\partial_iu\partial_ju+ \frac{1}{\varepsilon} \mathbbm{1}_{(-1,1)}(u)\right\} \, d{\rm vol}_g,\qquad \eps \in (0,1).
$$
 Then the principal curvatures of the level sets of $u_\varepsilon$ are bounded in $B_{1/2}$ by a constant depending only on $M$.}
\end{remark}

\appendix

\section{Some classical results}
\label{app:classical-results}

We begin by recalling the following quantitative version of Hopf's lemma.

\begin{lem}[Hopf]
\label{lem:Hopf-quant}
Suppose $B_1(e_n)$ touches $\Omega\subset\R^n$ from the interior at $0$. Suppose
\begin{equation*}
\left\{
\begin{array}{rcll}
	-\Delta u &\geq& 0 & \text{ in } B_1(e_n),\\
	u&>&0 & \text{ in } B_{1}(e_n),\\
	u(0)&=&0.&
\end{array}
\right.
\end{equation*}
Then there exist dimensional constants $c_1,c_2>0$ such that
\begin{equation*}
\partial_\nu u(0)\geq c_1\inf_{B_{1/2}(e_n)}u\geq c_2 \fint_{B_{1/2}(e_n)}u\,dx.
\end{equation*}
Here $\nu$ is the inward unit normal of $\Omega$ (as consistent with the Bernoulli problem). 
\end{lem}

\begin{proof}
Define
$$
\Gamma_n:B_1(e_n) \setminus \{e_n\}\to \R,\qquad \Gamma_n(x):=\begin{cases} 
-\frac{\log|x-e_n|}{\log 2}& \text{if } n=2,\\
\frac{|x-e_n|^{2-n}-1}{2^{n-2}-1} & \text{if }n \geq 3,\\
\end{cases}
$$
so that 
$$
\Delta \Gamma_n=0\quad \text{ in $B_1(e_n)\setminus \{e_n\}$},\qquad \Gamma_n|_{\partial B_1(e_n)}=0,\qquad \Gamma_n|_{\partial B_{1/2}(e_n)}=1.
$$
Then $v(x)=\left(\inf_{\partial B_{1/2}(e_n)} u\right) \Gamma_n(x)$
is a lower barrier for $u$ inside $B_1(e_n)\setminus B_{1/2}(e_n)$, thus $\partial_\nu u(0) \geq \partial_\nu \Gamma_n(0)=c(n)\inf_{\partial B_{1/2}(e_n)} u$. This proves the first inequality. The second follows from the mean value inequality for superharmonic functions.
\end{proof}

We also present a useful interpolation inequality between $L^1$ and ${\rm Lip}$:  

\begin{lem}[Interpolation]
\label{lem:L1_Lip}
Let $n \ge 2$, and let $\Omega := \{(x', x_n)\in \R^{n-1}\times \R : x_n > \phi(x')\}$ with $\phi(0) = 0$ and $|\nabla \phi|\le C_\circ$. Let $u \in {\rm Lip}(\Omega\cap B_1)$. Then,  either $\|u \|_{L^\infty(\Omega\cap B_1)}\le C \|u\|_{L^1(\Omega\cap B_1)}$ or 
\[
\|u \|^{n+1}_{L^\infty(\Omega\cap B_1)}\le C \|u\|_{L^1(\Omega\cap B_1)} \|\nabla u\|^n_{L^\infty(\Omega\cap B_1)}, 
\]
for some $C$ depending only on $n$ and $C_\circ$. In particular, for any $\eps > 0$,
\[
\|u \|_{L^\infty(\Omega\cap B_1)}\le C_\eps \|u\|_{L^1(\Omega\cap B_1)} +\eps \|\nabla u\|_{L^\infty(\Omega\cap B_1)}, 
\]
for some  $C_\eps > 0$ depending only on $n$, $C_\circ$, and $\eps$.
\end{lem}
\begin{proof}
Let $h = \|u \|_{L^\infty(\Omega\cap B_1)}$,  $V = \|u\|_{L^1(\Omega\cap B_1)}$, $L = \|\nabla u\|_{L^\infty(\Omega\cap B_1)}$, and let $x_\circ\in \Omega\cap B_1$ be such that $|u(x_\circ)| \ge \frac{h}{2}$.  
Then, since $u$ is $L$-Lipschitz, we have
\[
|u(x)| \ge \frac{h}{2}- L|x-x_\circ|.
\]
If $L \leq \frac{h}{8}$, we easily have $V \geq ch$. Otherwise, denoting $r = \frac{h}{4L} \leq 2$, we have $|u| \ge \frac{h}{4}$ in $B_r(x_\circ) \cap \Omega$ and therefore
\[
V = \int_{\Omega\cap B_1} |u|\ge \int_{B_r(x_\circ)\cap \Omega} |u| \ge \frac{h}{4}|B_r(x_\circ)\cap \Omega|\ge c h r^{n}=c4^{-n}\frac{h^{n+1}}{L^n},
\]
which gives the first result. The second estimate then follows from Young's inequality. 
\end{proof}

In the following result, we will use the stability inequality for classical solutions to the Bernoulli problem in $B_1$, which reads as 
\begin{equation}
\label{eq:stab_ineq_H}
\int_{\partial\{u > 0\}} H \xi^2\, d\mathcal{H}^{n-1} \le \int_{\{u > 0\}} |\nabla\xi|^2\, dx,\qquad\text{for all}\quad \xi\in C^\infty_c(B_1), \ \xi\ge 0, 
\end{equation}
(see \cite[Lemma 1]{CJK04}). We recall that $H$ denotes the mean curvature of the free boundary, and $H(x) = -\partial^2_{\nu\nu} u(x)$ for $x\in \FB(u)$, where $\nu$ is the inward unit normal vector field to $\FB(u)$ (cf. \Cref{lem:vbounds}).

Setting $\xi=\bc \eta$ in \eqref{eq:stab_ineq_H} and integrating by parts, we get the following equivalent formulation:
\begin{equation}
\label{eq:stab_ineq_H_2}
\int_{\set{u>0}}
	\bc\Delta\bc\,
	\eta^2
\,dx
+\int_{\partial\{u > 0\}}
	\bc(\bc_\nu+H\bc)
	\eta^2
\, d\mathcal{H}^{n-1}
\le \int_{\{u > 0\}} 
	\bc^2 |\nabla\eta|^2
\, dx,
	\qquad\text{for all}\quad \bc,\eta\in C^\infty_c(B_1), \ \bc\ge 0.
\end{equation}
One can then prove the following version of the stability inequality.

\begin{lem}[Sternberg--Zumbrun inequality]
\label{lem:Stern_Zum_App}
Let $n \ge 2$,  and let $u$ be a classical stable solution to the Bernoulli problem in $B_1\subset \R^n$. Then
\[
\int_{B_1\cap \{u > 0\}} |D^2 u|^2 \eta^2 \,dx \le n \int_{B_1}|\nabla u|^2|\nabla\eta|^2 \,dx\qquad \text{for all}\quad \eta\in C^\infty_c(B_1).
\]
\end{lem}
\begin{proof}
The proof follows along the lines of \cite[Theorem 1.9]{FR19} (cf. \cites{SZ98, FV09, Ca10} for the semilinear case), using the stability condition \eqref{eq:stab_ineq_H_2} with $\bc=|\nabla u|$. More precisely, by harmonicity of $u$ we have
\[
|\nabla u|\Delta|\nabla u|
=\frac{1}{2}\Delta\big(|\nabla u|^2\big)
	-\big|\nabla |\nabla u|\big|^2
=|D^2 u|^2
	-\big|\nabla |\nabla u|\big|^2
\quad \text{ inside } \set{u>0}\cap \set{|\nabla u|>0}.
\]
Setting $\nu=\frac{\nabla u}{|\nabla u|}$ (which is the inward unit normal of super-level sets of $u$ and extends the inward unit normal on $\FB(u)$ to $\overline{\set{u>0}}\cap\set{|\nabla u|>0}$), we note that
\[
\nabla|\nabla u|
=\frac{(\nabla|\nabla u|^2\cdot \nu)\nu}{2|\nabla u|}
=(\partial^2_{\nu\nu}u)\nu,
	 \quad \text{ therefore }
\qquad
\partial_{\nu}|\nabla u|
=-H
	\quad \text{ on } \partial\{u>0\}.
\]
Thus, thanks to \eqref{eq:stab_ineq_H_2}, we have
\begin{equation}
\label{eq:prev_SZ}
\int_{B_1}|\nabla u|^2|\nabla \eta|^2 \,dx \ge \int_{B_1\cap \{u > 0\}\cap \{|\nabla u| > 0\}}\left(|D^2 u|^2 -\big|\nabla|\nabla u|\big|^2\right)\eta^2 \,dx \qquad\text{for any}\quad \eta\in C^\infty_c(B_1). 
\end{equation}
Notice that, since $u$ is harmonic in $\{u > 0\}$, the set $\{|\nabla u| = 0\}\cap \{u > 0\}$ has zero measure (by unique continuation) and therefore the right integral above is in fact inside $B_1\cap \{u > 0\}$. 
Now, given any point $x_\circ\in \{u > 0\}$, up to a rotation we can assume then $\nabla u(x_\circ) = e_1|\nabla u(x_\circ)|$. Then, at such point, the previous integrand equals
\[
\left(|D^2 u|^2 - \big|\nabla |\nabla u|\big|^2\right)(x_\circ) = \sum_{i, j = 1}^n (\partial_{ij}^2 u(x_\circ))^2 - (\partial_{11}^2 u(x_\circ))^2 = \sum_{\substack{i, j = 1\\(i, j)\neq (1, 1)}}^n (\partial_{ij}^2 u(x_\circ))^2.
\]
Notice that, by harmonicity, 
\[
\begin{split}
(\partial_{11}^2 u(x_\circ))^2 & = \left(\partial_{22}^2 u(x_\circ)+\partial_{33}^2 u(x_\circ)+\dots +\partial_{nn}^2 u(x_\circ)\right)^2 \\
& \le (n-1)\left((\partial_{22}^2 u(x_\circ))^2+(\partial_{33}^2 u(x_\circ))^2+\dots +(\partial_{nn}^2 u(x_\circ))^2\right), 
\end{split}
\]
and so, for any $\tau\in [0, 1]$, 
\[
 \sum_{\substack{i, j = 1\\(i, j)\neq (1, 1)}}^n (\partial_{ij}^2 u(x_\circ))^2 \ge  (1-\tau)\sum_{\substack{i, j = 1\\(i, j)\neq (1, 1)}}^n (\partial_{ij}^2 u(x_\circ))^2 + \frac{\tau}{n-1} (\partial_{11}^2 u(x_\circ))^2.
\]
Choosing $\tau = \frac{n-1}{n}$, this proves that 
\[
 \left(|D^2 u|^2 - \big|\nabla |\nabla u|\big|^2\right)(x_\circ)   \ge  \frac{1}{n}\sum_{i, j = 1}^n (\partial_{ij}^2 u(x_\circ))^2 = \frac{1}{n}|D^2 u(x_\circ)|^2.
\]
Combining this inequality with \eqref{eq:prev_SZ}, we get the desired result.
\end{proof}

\normalcolor
	
\section{Linear estimates for the Bernoulli problem}
\label{app:linear-estimate} 

In this appendix, we prove some linear estimates for the Bernoulli or one-phase problem that are useful throughout the work. 
In the following, we keep in mind the equivalence:

\begin{lem}
\label{lem:equivv}
Let $n \ge 2$, $e\in\mathbb{S}^{n-1}$, and let $u$ be a classical solution to the Bernoulli problem in $B_1\subset \R^n$ with $0\in \FB(u)$. Then, the following are equivalent for any $\eps_\circ = \eps_\circ(n)$ small  enough: 
\begin{enumerate}[(i)]
\item $|u-e\cdot x| \leq \eps_\circ$ in $B_1\cap \set{u>0}$;
\item $(e\cdot x-\eps_\circ)_+ \leq u \leq (e\cdot x+\eps_\circ)_+$ in $B_1$.
\end{enumerate}
\end{lem}
\begin{proof}
    All implications are elementary, except that (i) implies $u \ge (e\cdot x - \eps_\circ)_+$ in $B_1$. By contradiction, we would have $u \equiv 0$ in $B_1 \cap \{|x\cdot e|\ge \eps_\circ\}$. However, since $ 0\in \FB(u)$, this contradicts \Cref{lem:cleanball} (or \Cref{rem:nondeg}) if $\eps_\circ$ is small enough.
\end{proof}
\begin{remark}
    In the previous statement, the hypothesis $0\in \FB(u)$ can be replaced,  e.g., with  $B_{7/8}\cap \{u > 0\}\neq \varnothing$.
\end{remark}

\begin{thm} \label{thm:DeSilva-estimate} Given $n \ge 2$, there exists $\eps_\circ> 0$ small enough depending only on $n$  such that the following holds. 

Let $u$ be a classical solution to the Bernoulli problem in $B_1\subset \R^n$, and  suppose that 
\[
\|u - e\cdot x-b\|_{L^\infty(B_1\cap \{u > 0\})} \le \eps_\circ\qquad\text{for some}\quad e\in \mathbb{S}^{n-1}, \ b\in \R. 
\]
Then, for any $\ctta\in (0, 1)$ we have 
\[
\|\nabla u - e\|_{L^\infty(B_{1/2}\cap \{u > 0\})}+[\nabla u]_{C^{\ctta}(B_{1/2}\cap \{u > 0\})}\le C \|u - e\cdot x-b\|_{L^\infty(B_1\cap \{u > 0\})},
\]
for some constant $C$ depending only on $n$ and $\ctta$. 
\end{thm}

\begin{proof}
Denote $\eps := \|u - e\cdot x - b\|_{L^\infty(B_1\cap \{u > 0\})}\le \eps_\circ$. If there are no free boundary points in $B_{3/4}$ we are done, either by harmonic estimates if $u > 0$ in $B_{3/4}$, or because $ u \equiv 0$ in $B_{3/4}$. Thus, let us assume $x_\circ\in B_{3/4} \cap \FB(u)$, and consider $\bar u(x) = 8u\left(x_\circ+\frac{x}{8}\right)$, which is a classical solution to the Bernoulli problem in $B_1$ such that 
\[
\left\|\bar u - e\cdot (8x_\circ+x) - 8b\right\|_{L^\infty(B_1\cap \{\bar u > 0 \})}\le 8 \eps\le 8 \eps_\circ. 
\]
In particular, since $\bar u(0)=0$, we have $|8e\cdot x_\circ +8b|\le 8\eps$, and therefore
\[
\left\|\bar u - e\cdot x\right\|_{L^\infty(B_1\cap \{\bar u > 0 \})}\le 16 \eps\le 16 \eps_\circ. 
\]
Recalling \Cref{lem:equivv},
for $\eps_\circ$ small enough we can iteratively apply \cite[Lemma 4.1]{DeSilva} (cf. \cite[Proof of Theorem 1.1]{DeSilva}) to get
\[
[\nabla \bar u]_{C^{1,\alpha}(B_{1/2}\cap \{\bar u > 0\})}\le C \eps.
\]
In particular, setting $z_\circ:=\frac{1}{4}e$,
\[
\|\nabla \bar u - \nabla \bar u(z_\circ)\|_{L^\infty(B_{1/2}\cap \{u > 0\})}\le C\eps. 
\]
Since the free boundary is flat, we can use harmonic estimates for $\bar u -e\cdot x$ in $B_{1/8}(z_\circ)$, to deduce 
\[
|\nabla \bar u(z_\circ) - e|
\le C\|\bar u - e\cdot x\|_{L^\infty(B_{1/8}(z_\circ))} 
\le C\eps.
\]
Combining this bound with the above estimate at $z_\circ$, we obtain
\[
\|\nabla \bar u - e\|_{L^\infty(B_{1/2}\cap \{\bar u > 0\})}
\le \|\nabla\bar  u - \nabla \bar u(z_\circ) \|_{L^\infty(B_{1/2}\cap \{\bar u > 0\})} 
+ |\nabla \bar u(z_\circ) - e| \le C \eps.
\]
By rescaling back and a covering argument, we get the desired result. 
\end{proof}
 \begin{remark}
 \label{rem:DeSilva}
Since $u$ is  Lipschitz, the estimate
 \[
\|\nabla u - e\|_{L^\infty(B_{1/2}\cap \{u > 0\})} \le C \|u - e\cdot x-b\|_{L^\infty(B_1\cap \{u > 0\})},
\]
holds with $C=2\eps_\circ^{-1}$ when the right-hand side is not smaller than $\eps_\circ$. 
 \end{remark}
 
  By a standard interpolation argument, we can now show that $L^1$-flatness implies $L^\infty$-flatness: 

\begin{prop} \label{prop:L1estimate} Let $n \ge 2$, and let $u$ be a classical solution to the Bernoulli problem in $B_1\subset \R^n$. There exists a dimensional constant $C$ such that, for any $e\in\bS^{n-1}$ and $b\in\R$, 
\[
\|u - e\cdot x-b\|_{L^\infty(B_{1/2}\cap \{u > 0\})} \le C \|u - e\cdot x-b\|_{L^1(B_{1}\cap \{u > 0\})} ,
\]
for some constant $C$ depending only on $n$.
\end{prop}

\begin{proof}
From \Cref{thm:DeSilva-estimate} and \Cref{rem:DeSilva}, we know that for any $e\in\bS^{n-1}$ and $b\in\R$,
\[
\|\nabla u - e\|_{L^\infty(B_{1/2}\cap \{u > 0\})} \le C \|u - e\cdot x-b\|_{L^\infty(B_{1}\cap \{u > 0\})}.
\]
Combining this bound with the interpolation \Cref{lem:L1_Lip}, for any $e\in\bS^{n-1}$, $b\in\R$, and $\delta > 0$, 
\[
\|u - e\cdot x-b\|_{L^\infty(B_{1/2}\cap \{u > 0\})} \le C_\delta \|u - e\cdot x-b\|_{L^1(B_{1/2}\cap \{u > 0\})} + \delta \|u - e\cdot x-b\|_{L^\infty(B_{1}\cap \{u > 0\})}, 
\]
for some $C_\delta > 0$. Now, for any $B_{r}(z)\subset B_1$, applying this estimate to $u_{r,z}=\frac{u(z+r\cdot)}{r}$ with $b$ replaced by $\frac{b+e\cdot z}{r}$, we deduce that
\[
r^n \|u - e\cdot x-b\|_{L^\infty(B_{r/2}(z)\cap \{u > 0\})} \le C_\delta  \|u - e\cdot x-b\|_{L^1(B_{1/2}\cap \{u > 0\})} + \delta r^n \|u - e\cdot x-b\|_{L^\infty(B_{r}(z)\cap \{u > 0\})}. 
\]
Now, choosing $\delta$ sufficiently small, 
we can apply a standard covering trick to reabsorb the $L^\infty$-term in the right-hand side (see, for example, \cite[Lemma 2.27]{FR22}) and we deduce that, for any $e\in\bS^{n-1}$ and $b\in\R$, it holds
\[
\|u - e\cdot x-b\|_{L^\infty(B_{1/4}\cap \{u > 0\})} 
\le C\|u - e\cdot x-b\|_{L^1(B_{1/2}\cap \{u > 0\})}.
\]
After a final covering and scaling argument, this proves our desired result. 
\end{proof}

Finally, exploiting the recent results in \cite{Lian-Zhang}, one obtains the linear estimates for higher-order derivatives of solutions to the Bernoulli problem:

\begin{prop} \label{prop:linear-estimate} Let $n \ge 2$ and $k\in \N$.  
There exists $\eps_\circ = \eps_\circ(n)> 0$ small enough   such that the following holds. 

Let $u$ be a classical solution to the Bernoulli problem in $B_1\subset \R^n$, and let us suppose that 
\[
\|u - e\cdot x-b\|_{L^\infty(B_1\cap \{u > 0\})} \le \eps_\circ\qquad\text{for some}\quad e\in \mathbb{S}^{n-1}, \ b\in \R. 
\]
Then, we have 
\[
\|D^k u\|_{L^\infty(B_{1/2}\cap \{u > 0\})} \le C \|u - e\cdot x-b\|_{L^\infty(B_1\cap \{u > 0\})},
\]
for some constant $C$ depending only on $n$ and $k$.
\end{prop}

\begin{proof}
Let us denote $\eps := \|u - e\cdot x - b\|_{L^\infty(B_1\cap \{u > 0\})} \le \eps_\circ$. The proof follows by tracking the dependence on $\eps$ in \cite[Theorem 1.31]{Lian-Zhang} (with right-hand side $f=0$). 
More precisely, after a translation, rotation, and a covering argument, thanks to \Cref{thm:DeSilva-estimate} we can assume that 
\[
\norm[C^{1,1/2}(B_{5/6} \cap \{u>0\})]{u-x_n} \le C \eps, 
\]
where $\{u > 0\}$ coincides inside $B_{5/6}$ with the epigraph $\{x_n>\varphi(x')\}$, where $\varphi$ is uniformly $C^{1,1/2}$. 
From this point, the proof follows by induction as in \cite[Theorem 1.31]{Lian-Zhang}. Namely, by repeated applications of \cite[Theorem 1.29]{Lian-Zhang} we deduce that, for any $k\geq 2$,
\[
\norm[C^{k,1/2}(\Omega\cap B_{\rho_k})]{u-x_n}
\leq C_{k}\eps,
	\qquad \text{ for some radii } \quad 
\tfrac12<\rho_{k+1}<\rho_k<1,
\]
which gives the desired result.
\end{proof}

\section{Compactness of stable solutions}
\label{sec:compact-stable-sol}

In this appendix we show the compactness of sequences of stable solutions to the Bernoulli problem, as stated in \Cref{lem:compact}. Before that, we need an auxiliary lemma:

\begin{lem}
\label{lem:halfspacestable}
    Let $n\ge 2$, $\delta>0$, and let $u$ be a classical stable solution to the Bernoulli problem in $B_2\subset \R^n$ with $0\in \FB(u)$. Assume, in addition:
    \[
    \begin{cases}
     u > 0 \quad & \mbox{in }B_2\cap \{x_n>\delta\} \\
     u\le \delta &\mbox{in } B_2\cap \{x_n<-\delta\}.
    \end{cases}
    \]
    Then $u=0$ in $B_1\cap \{x_n<-C\delta\}$, where $C$ is a dimensional constant. 
    \end{lem}
\begin{proof}
    Combining our assumption with \Cref{lem:nondegen}, it follows that the free boundary of $u$ is contained in a strip  $\{-C_0\delta \le x_n\le \delta\}$ inside $B_{7/4}$, with $C_0$ dimensional. 
    Hence, to prove the result, we assume by contradiction that 
    \begin{equation}
    \label{eq:contr}
    0<u\leq \delta\qquad\text{inside}\quad  \{x_n<-C_0\delta\}\cap B_1.
    \end{equation}
    By \Cref{lem:nondegen} and  Lipschitz estimates (see, e.g., \cite[Lemma 11.19]{CS05}),  we have  $\sup_{B_1} u \ge c_1>0$ and  $|\nabla u|<C_2$ in $B_{3/2}$, where $c_1$ and $C_2$ are dimensional constants. Hence, there exists $y\in B_1$ such that $\min_{ B_{r/2}(y)} u \ge c_1/2$ where $r= c_1/C_2>0$. Assuming that $\delta$ is sufficiently small (if not, we take $C=1/\delta$ in the conclusion) and recalling that  $u\le \delta$ in  $B_2\cap \{x_n<-\delta\}$, it follows that $y_n>r/4$.
    
    Since $\Delta u =0$ in  $B_2\cap \{x_n>\delta\}\subset \{u>0\}$, 
    by Harnack's inequality we obtain
    \[
    u \ge c>0 \qquad \mbox{in } B_{7/4}\cap \{x_n>1/8\}
    \]
    and therefore, by a standard barrier argument,
    \begin{equation}\label{recallthis}
          u(x) \ge c (x_n-\delta) >0  \qquad \mbox{for all } x\in B_{3/2}\cap \{x_n>\delta\}.
    \end{equation}
  We now want to exploit \Cref{lem:Stern_Zum_App}.
  Recall that, by contradiction, we are assuming \eqref{eq:contr}. Thus, by Fubini's theorem we have
\[
\begin{split}
\int_{B_1\cap\{u  >0\}}|D^2 u|^2\,dx
& \ge \int_{B'_{1/2}}\int_{[-1/2, 1/2]\cap \{u(\sigma,t) > 0\}} |D^2u|^2(\sigma, t)\, dt\, d\sigma \ge \int_{B'_{1/2}}\int_{-1/2}^{t_{\sigma}} |D^2u|^2(\sigma,t)\, dt\, d\sigma,
\end{split}
\]
where:\\
- $B_r'\subset \R^{n-1}$  denotes the ball of radius $r$ in $\R^{n-1}$;\\
- given $\sigma \in B_{1/2}'$, $t_{\sigma}$ denotes the maximal value $t_*\in [-1/4, 1/4]$   such that $(-1/2, t_*)\subset \{u(\sigma, \cdot) > 0\}$. \\
Now, let $\Pi_n:\R^n\to \R^{n-1}$ denote the orthogonal projection onto the first $n-1$ variables, and define 
  \[
  A : = \Pi_n\big(\FB(u)\cap(B'_{1/2}\times (-1/2,1/2))\big)\subset B_{1/2}'\subset \R^{n-1}.
  \]
Notice now that, by \eqref{eq:contr} and harmonic estimates, there exists $\overline C>1$ dimensional such that $|\nabla u(\sigma, -\overline C\delta) |^2\le 1/2$ for all $\sigma \in B_{1/2}'$.
Also, for $\sigma \in A$, we have $-C_0\delta \le t_\sigma\le \delta$.
Hence, since $|\nabla u|^2=1$ on $\FB(u)$ and
 $\bigl|\nabla |\nabla u|^2\bigr| \le 2|D^2 u|$, 
 it follows that
\[
\frac 1 2 |A| \le \int_{A} \int_{-\overline C\delta }^{t_{\sigma}}
\bigl|
	\partial_n |\nabla u|^2(\sigma, t)\big|  dt\, d\sigma\le  2 \int_{A} \int_{-\overline C\delta }^{\delta}
|D^2u|  dt\,d\sigma.
\]
Thus, applying Cauchy--Schwarz and \Cref{lem:Stern_Zum_App}, we obtain
\[
\frac 1 2 |A| \le C (|A|\delta)^{1/2}\left( \int_{B_1\cap\{u  >0\}}|D^2 u|^2\,dx\right)^{1/2} \le C(|A|\delta)^{1/2} \qquad \Longrightarrow\qquad |A|\le C\delta.
\]
On the other hand, for $\sigma \in B_{1/2}' \setminus A$ we have
$t_\sigma =\tfrac14 \ge \delta^{1/2}$. Thus,  from \eqref{recallthis} and the fact that $0 \le u \le \delta$ for $\{x_n < -\delta\}$, for $\delta$ sufficiently small we obtain
\[
\begin{split}
\int_{-\delta^{1/2}}^{\delta^{1/2}} |\partial^2_{nn} u(\sigma, t)|dt & \ge  \left|\frac{u(\sigma, \delta^{1/2}) - u(\sigma, -C \delta)}{\delta^{1/2} + C\delta} - \frac{u(\sigma, -C\delta) - u(\sigma, -\delta^{-1/2})}{- C\delta + \delta^{1/2}} \right|\\ 
& \ge  \left(\frac{c (\delta^{1/2} - \delta) - \delta}{2\delta^{1/2}}-\frac{\delta-0}{\tfrac12\delta^{1/2}}\right)\ge \frac{c}{4}.
\end{split}
\]
Hence, arguing similarly to before, we get
\[
\frac{c}{4} |B'_{1/2}\setminus A| \le \int_{B'_{1/2}\setminus A}\hspace{-1mm}\int_{-\delta^{1/2}}^{\delta^{1/2}}  |\partial^2_{nn} u(\sigma, t)|dt\le  C \big(|B'_{1/2}\setminus A|\delta^{1/2}\big)^{1/2}  
\left( \int_{B_1\cap\{u  >0\}}\hspace{-2mm}|D^2 u|^2\,dx\right)^{1/2}\hspace{-2mm}\le  C \big(|B'_{1/2}\setminus A|\delta^{1/2}\big)^{1/2},  
\]
which proves that $|B_{1/2}'\setminus A|\le C \delta^{1/2}$.
Combining the bounds that we have obtained, we get $|B_{1/2}'|\leq |A|+|B_{1/2}'\setminus A| \leq C(\delta+\delta^{1/2})$, a contradiction for $\delta$ small enough.  
\end{proof}

We can now give the proof of \Cref{lem:compact}. This fixes a small gap in  \cite[Theorem 1.2]{kamburov2022nondegeneracy}, since the authors rely on \cite[Lemma 1.21]{CS05} and the proof there is incomplete, as one can see by comparing their argument with ours below.

\begin{proof}[Proof of \Cref{lem:compact}] We prove the three points separately.

\smallskip

\noindent {\it (1)}
Since $\norm[L^\infty(B_{k/2})]{\nabla v_k} \leq C$ (by  Lipschitz regularity of classical solutions, see e.g. \cite[Lemma 11.19]{CS05}), for any $\ctta\in(0,1)$ we have
$
v_k \to v_\infty
$ in $C^{0,\ctta}_{\loc}(\R^n),
$
where $\|\nabla v_\infty\|_{L^\infty(\R^n)}\le C(n)$ (in fact, $v_\infty$ is 1-Lipschitz by \Cref{lem:Lipbound}). Also, since $v_k$ is subharmonic, so is $v_{\infty}$
and we have
\[\nabla v_k \to \nabla v_{\infty} \qquad \mbox{strongly in } L^1_{\rm loc}(\R^n),
\]
(see for instance \cite[Lemma A.1(b1)]{CFRS}).
Hence, thanks to the bound $\|\nabla (v_k - v_\infty)\|_{L^\infty(B_{k/2})} \le   C$, it follows by interpolation that $v_k \to v_{\infty}$ strongly in $H^1_{\loc}(\R^n)$. 

\smallskip

\noindent {\it (2)} We now prove the Hausdorff convergence of the different sets. \\
{\it $\bullet$ Hausdorff convergence of free boundaries.}\\
Thanks to  \eqref{eq:density-B}, given $x_k\in \FB(v_k)$ with $x_k \to z_\infty$ we have
\[
\norm[L^\infty(B_r(x_k))]{v_k}\geq c(n)r\qquad \Longrightarrow \qquad \norm[L^\infty(B_r(x_\infty))]{v_\infty}\geq c(n)r.
\]
In particular, since $v_\infty(x_\infty)=\lim_{k\to\infty}v_k(x_k)=0$, it follows that $x_\infty\in \FB(v_\infty)$. 

Conversely, let $x_\infty\in \FB(v_\infty)$ and assume by contradiction that there is no free boundary point for $v_k$ in a neighborhood, for all $k$ large. Then the functions $v_k$ are all harmonic around $x_\circ$ (they are either identically zero, or positive and harmonic), and thus $v_\infty$ would be harmonic in a uniform neighborhood around $x_\circ$; impossible.

\smallskip

\noindent
{\it $\bullet$ Hausdorff convergence of $\overline{\{v_k = 0\}}$ to $\overline{\{v_\infty = 0\}}$.}\\
If $x_k \in \{v_k = 0\}$ and $x_k\to x_\infty$, then $v_\infty(x_\infty) = 0$. Conversely, if $v_\infty(x_\infty) = 0$, we want to prove that there exist points $x_k\in \{v_k = 0\}$ such that $x_k\to x_\infty$. This is the main part of the proof. 

Let $\cC\subset \{v_\infty=0\}$ denote the set of zero points of $v_\infty$ that are also accumulation points of convergent sequences  $x_k$ with $v_k(x_k)>0$. Note that $\cC$ is closed and that, by the Hausdorff convergence of the free boundaries, $\partial \cC\subset  \FB(v_\infty)$. We need to prove that the interior of $\cC$ is empty. 

If not, by contradiction, there exists $x_\circ \in \partial \cC$ such that the open sets ${\rm int}\, \cC\cap B_\varrho(x_\circ)$ and $\{v_\infty>0\}\cap B_\varrho(x_\circ)$ are both nonempty. 
By \Cref{lem:perbound} and \Cref{lem:nondegen}, the sets $\FB(v_k)$ are ``equi-uniformly'' Alfohrs--David regular: namely, there exists a dimensional constant $C_1>1$ such that, for all $k$, 
\[
\frac{1}{C_1} r^{n-1} \le \cH^{n-1} (\FB(v_k)\cap B_r(y)) \le C_1 r^{n-1}\qquad \text{for all}\quad y\in \FB(v_k), \ r > 0.
\]
This implies, by standard covering arguments  using Besicovitch theorem,
\[
\cH^{n} ((\FB(v_k) +B_t)\cap B_r(y)) \le C_2 r^{n-1} t\qquad \text{for all}\quad y\in \R^n, \ r > 0,\ t\in (0, r).
\]
This last inequality   is stable under Hausdorff convergence, giving
\[
 \cH^{n} ((\FB(v_\infty) +B_t)\cap B_r(y)) \le C_2 r^{n-1} t\qquad \text{for all}\quad y\in \R^n, \ r > 0,\ t\in (0, r).
\]
This proves that the $(n-1)$-dimensional upper Minkowski content of the boundary
\[
\partial \cC\cap B_\varrho(x_\circ) = \FB(v_\infty)\cap B_\varrho(x_\circ)
\] 
is finite, and therefore the set $\{v_\infty>0\}\cap B_\varrho(x_\circ)$ has finite (relative) perimeter in $B_\varrho(x_\circ)$. 
Thus, as a consequence of De Giorgi's structure theorem, \cite[Chapter 15]{Mag12}, there is a point $y_\circ$ belonging to the reduced boundary of $\cC$ in $B_{\rho/2}(x_\circ)$, $y_\circ\in \partial^*\cC\cap B_{\rho/2}(x_\circ)$. Hence, by zooming enough around $y_\circ,$  both $\cC$ and $\{v_\infty >0\}$ look locally like half-spaces around $y_\circ$, and therefore we can apply \Cref{lem:halfspacestable} 
to deduce that the functions $v_k$ have to vanish somewhere near $y_\circ$, for $k$ large enough. This provides a contradiction and concludes the proof.

\smallskip

\noindent
{\it $\bullet$   Hausdorff convergence of $\{v_k > 0\}$ to $\{v_\infty > 0\}$.}\\
This follows from the convergence of the free boundaries and the closures of the contact sets.

\smallskip

\noindent {\it (3)}  Given $u\in H^1_{\rm loc}(\R^n)$ and a smooth compactly supported vector field $\Psi\in C^\infty_c(\R^n; \R^n)$, it is a direct computation to obtain the first and  second inner variation of the energy in the direction $\Psi$: 
\[
\frac{d}{dt}\bigg|_{t = 0} \hspace{-0.5mm}(E(u(\cdot+t\Psi )); \R^n)   =  \int_{\R^n}  \left\{ -2\nabla u \, D\Psi (\nabla u)^\top +|\nabla u|^2{\rm div}(D\Psi) \right\}\,dx+\int_{\{u > 0\}} {\rm div}(\Psi) \,dx,
\]
and 
\begin{multline*}
\frac{d^2}{dt^2}\bigg|_{t = 0} \hspace{-0.5mm}(E(u(\cdot+t\Psi )); \R^n) 
 = \int_{\R^n}  \hspace{-1mm}
\nabla u  \hspace{-0.5mm}
\left[
	4(D\Psi)^2
	+2D\Psi (D\Psi)^\top 
	-4(\Div\Psi)D\Psi 
	+(\Div\Psi)^2{\rm Id}
	-{\rm tr}\big((D\Psi)^2\big)
\right](\nabla u)^\top \,dx
\\
+ \int_{\{u > 0\}} \left(
	(\Div\Psi)^2
	-{\rm tr}\big((D\Psi)^2\big)
\right)\,dx.
\end{multline*}
Thanks to the convergences proved in points (1) and (2) above, together with the fact that $\norm[L^\infty(B_{k/2})]{\nabla v_k} \leq C$,
we can let $k \to \infty$ in the formulas for the first and second variation to deduce that $v_\infty$ is stationary and that
\[
0\le \frac{d^2}{dt^2}\bigg|_{t = 0} \hspace{-0.5mm}(E(v_k(\cdot+t\Psi )); \R^n) \to \frac{d^2}{dt^2}\bigg|_{t = 0} \hspace{-0.5mm}(E(v_\infty(\cdot+t\Psi )); \R^n)\qquad\text{as}\quad k\to \infty, 
\]
for any $\Psi\in C^\infty_c(\R^n; \R^n)$ fixed. This proves that $v_\infty$ is a stable solution.  
\end{proof}

\section{Estimates for positive harmonic functions in a flat-Lipschitz domain}
\label{appendixD}

\begin{proof}[Proof of \Cref{lem_basic1}]
The statement is scale-invariant, so we can fix $r = 1$. Let $B^{+, t}_2 := B_2\cap \{x_n \ge t\}$, and consider $P_t(x, y):B_2^{+, t}\times \partial B_2^{+, t} \to [0, \infty)$ the Poisson kernel for the domain $B_2^{+, t}$. We note that there exists some dimensional constant $c_n > 0$ such that $P_t\left(\tfrac54 e_n, y\right) \ge c_n > 0$ for any $t\in [0, 1]$ and $y\in B_{3/2}\cap \{x_n = t\}$ (this can be seen, for example, by comparing $P_t$ to the Poisson kernel of the half-space). Therefore, 
\[
w\left(\tfrac54 e_n\right) \ge \int_{|y'|\le 3/2} P_t\left(\tfrac54 e_n, (y', t)\right)\, w(y', t)\, dy' \ge c_n \int_{{|y'|\le 3/2}}  w(y', t)\, dy'.
\]
Since the values $w\left(\tfrac54 e_n\right)$ and $w(e_n)$ are comparable (by Harnack inequality), the result follows.
\end{proof}

\medskip

{  
\begin{proof}[Proof of \Cref{lem_basic2}]
We divide the proof into two steps. 
\medskip 

\noindent{\bf Step 1:} By scaling invariance, we fix $r = 1$. Let $r_k=2^{-k}$ and split $B_1 \cap D = \bigcup_{k\ge 1} S_k$, where $S_k = \{x\in B_1:r_k< \dist(x, D^c) \le r_{k-1}\}$ can be covered by a union of balls $\bigcup_{i\in I_k}B_{r_{k+1}}(x_i)$ with bounded overlapping. In particular, $\# I_k\leq Cr_k^{-(n-1)}$. 

Now, using Harnack inequality and interior estimates, \eqref{eq:twoeqs} holds in $S_1$ (in place of $B_1\cap D$). Also, again by interior estimates,  for $k\geq 2$ we have
\begin{align*}
\int_{B_1\cap S_k}|D^2 w|^{\gamma'} \,dx &
\le \sum_{i\in I_k} 
\int_{B_{r_{k+2}}(x_i)} |D^2 w|^{\gamma'} \,dx 
\le C \sum_{i\in I_k} r_{k}^{-2\gamma'} \int_{B_{r_{k+1}}(x_i)} w^{\gamma'} \,dx 
\\ 
&\le C  r_{k}^{-2\gamma'+n}  \# I_k\,|B_{r_{k+1}}| \bigg( \frac{1}{\# I_k\,|B_{r_{k+1}}|} \sum_{i\in I_k} \int_{B_{r_{k+1}}(x_i)} w \,dx\bigg)^{\gamma'}
\\
&\le C  r_{k}^{-2\gamma'} \big(\# I_k\,|B_{r_{k+1}}|\big)^{1-\gamma'}  
\bigg(  \int_{S_{k-1}\cup S_k \cup S_{k+1}} w \,dx\bigg)^{\gamma'}\\
&\leq   C  r_{k}^{1-3\gamma'}  
\bigg(  \int_{S_{k-1}\cup S_k \cup S_{k+1}} w \,dx\bigg)^{\gamma'},
\end{align*}
where, in the second line, we applied Jensen's inequality (note that $t\mapsto t^\gamma$ is concave).

\noindent{\bf Step 2:} Let $\tau\ll 1$ be universally small (to be fixed later) and assume that $c_\circ \ll \tau^2$. 
For $i\in\N$, consider the scales $\rho_i := \frac{1}{8}   \tau^i$ and indices $j\in I^{(i)}$ so that the ``graphical lattice'' $p_j^{(i)} \in \partial D \cap B_{3/2}$ projects along $e_n$ to $(p_j^{(i)})' \in \frac{1}{16}\rho_i \Z^{n-1}\subset \{x_n=0\}$. Then, consider the covering by spherical caps,
\begin{equation}\label{eq:W1p-A1-100}
D\cap B_{3/2}\cap  \set{x_n\leq \tfrac{1}{16}} \subset \bigcup_{i=0}^{\infty}\bigcup_{j\in I^{(i)}}D_j^{(i)},
    \quad \mbox{ for } \quad 
D_j^{(i)}:=p_j^{(i)}+\set{x_n\geq \tau^2\rho_i} \cap B_{\rho_i}\subset D.
\end{equation}
Using \Cref{lem_basic1}  at scale $\rho_i$ and integrating over $t\in[0,1]$, since $c_\circ \ll \tau^2$) we deduce that
\begin{equation}\label{eq:W1p-A1-200}
\int_{D_j^{(i)}}w\,dx
\leq C\int_{\tilde{D}_j^{(i)}}w\,dx,
    \quad \mbox{ for } \quad
\tilde{D}_j^{(i)}:=p_j^{(i)}+\set{x_n \geq \tfrac{1}{4}\rho_i} \cap B_{2\rho_i}.
\end{equation}
We now define the slabs 
$$
S_j^{(i)}:=p_j^{(i)}+\set{\tfrac{\tau}{8}\rho_i \leq x_n \leq 4\tau \rho_i} \cap B_{\rho_i/2}
$$
and note that 
\[
S_j^{(i)}
\supset \bigcup_{\ell\in I^{(i+1)}_j}\tilde{D}_\ell^{(i+1)}\quad
    \mbox{ for some family of indices $I_j^{(i+1)}$ satisfying } 
\bigcup_{j\in I^{(i)}}I_j^{(i+1)}=I^{(i+1)}.
\]
Applying \Cref{lem_basic1} again at scale $\rho_i$, but this time integrating over $t \in [0,4\tau]$), we have
\begin{equation*}
\sum_{\ell\in I_j^{(i+1)}}
    \int_{\tilde{D}_\ell^{(i+1)}}w\,dx 
\leq C\int_{S_j^{(i)}}w\,dx
\leq \tilde{C} \tau  \int_{\tilde{D}_j^{(i)}}w\,dx,
\end{equation*}
so $\sum_{j\in I^{(i)}}\int_{\tilde{D}_j^{(i)}}w\,dx$ decays geometrically as long as $\tau< \frac{1}{\tilde{C}}$.
Hence, recalling \eqref{eq:W1p-A1-200} we get
\[
\int_{B_{3/2}\cap \tilde S^\tau_i} w\le \sum_{j\in I^{(i)}} \int_{D_J^{(i)}} w\, dx \le (\tilde C\tau)^{i+1} \int_{\tilde{D}_j^{(0)}}w\,dx
\leq C\int_{\{x_3 \ge 1/{64}\}\cap B_{7/4}} w\, dx \le C (\tilde C\tau)^{i+1} w(e_n),
\]
where $\tilde S_i^\tau =  \{\tau^{i+2}/8 < \dist(\cdot, D^c) < \tau^i/16\}$, and where we have also used  Harnack inequality. 

Observe now that  $S_{k-1}\cup S_k\cup S_{k+1}\subset B_{3/2}\cap (\tilde S^\tau_i\cup \tilde S^\tau_{i+1})$ as long as $2^{-k-1} > \tau^{i+2}/8$ and $2^{-k+1}< \tau^{i-1}/16$. This holds, for instance, for $i = \lfloor  {k}/{|\log_2(\tau)|}\rfloor$ with $\tau$ universally small.
Hence, by the previous inequality, we get
\[
\int_{S_{k-1}\cup S_k\cup S_{k+1}} w\, dx \le \int_{B_{3/2}\cap (\tilde S^\tau_i\cup \tilde S^\tau_{i+1})} w \, dx \le C(\tilde C \tau)^{k/|\log_2(\tau)|} w(e_n) \le C 2^{k\frac{C}{|\log \tau|}} 2^{-k} w(e_n).
\]
Applying Step 1 and adding over $k$, we finally get
\[
\int_{B_1\cap D}|D^2 w|^{\gamma'} \, dx \le C (w(e_n))^{\gamma'} \sum_{k \ge 1} r_k^{1-2\gamma'-\frac{C\gamma'}{|\log \tau|}}.
\]
Note that the previous sum is finite as long as $-2\gamma'+1-\frac{C\gamma'}{|\log \tau|}>0$, which holds for any $\gamma' < \tfrac12$ by choosing $\tau$ sufficiently small (depending on $\gamma'$). This concludes the proof.
\end{proof}

\begin{proof}[Proof of \Cref{lem:abst_osc1}]
 By scale invariance, we fix $r = 1$. Proceeding exactly as in Step 2 of the proof of \Cref{lem_basic2} (see above), by applying \Cref{lem:abst_osc2} instead of \Cref{lem_basic1} (which is integrable in $t$ as long as $(n-1) (1-q) > -1$) we obtain 
\begin{equation*}
\sum_{\ell\in I_j^{(i+1)}}
    \int_{\tilde{D}_\ell^{(i+1)}}|\nabla w|^q\,dx 
\leq C\int_{S_j^{(i)}}|\nabla w|^q\,dx
\leq C_q \tau^{n-(n-1)q} \int_{\tilde{D}_j^{(i)}}|\nabla w|^q\,dx.
\end{equation*}
As before,  $\sum_{j\in I^{(i)}}\int_{\tilde{D}_j^{(i)}}|\nabla w|^q\,dx$ decays geometrically (now for $q\in(1,\frac{n}{n-1})$ and $\tau=\tau(n,q)$ fixed). Thus,
\[
\begin{split}
\int_{D\cap B_{3/2} \cap \set{x_n\leq \frac{1}{16}}}|\nabla w|^q\,dx
& \leq C\sum_{i=0}^{\infty}\sum_{j\in I^{(i)}} \int_{\tilde{D}_j^{(i)}}|\nabla w|^q\,dx
\\
& \leq C\sum_{i=0}^{\infty}\left(C_q \tau^{n-(n-1)q}\right)^{i} \sum_{j\in I^{(0)}}
\int_{\tilde{D}_j^{(0)}}|\nabla w|^q\,dx
\leq C_q\int_{B_{7/4}\cap \set{x_n \ge \frac{1}{64}}} |\nabla w|^q\, dx,
\end{split}
\]
from which it follows that
$$
\int_{D\cap B_{3/2}}|\nabla w|^q\,dx
\leq C_q\int_{B_{7/4}\cap \set{x_n \ge \frac{1}{64}}} |\nabla w|^q\, dx.
$$
Using again \Cref{lem:abst_osc2} (to replace $\set{x_n \ge \frac{1}{64}}$ with $\set{x_n \ge \frac{1}{4}}$), the result follows.
\end{proof}
}

\bibliographystyle{plain}
\bibliography{mybib}

\end{document}